\documentclass[reqno,12pt]{amsart}
\usepackage{amsthm,amsfonts,amssymb,euscript,mathrsfs,graphics,color,amsmath,amssymb,latexsym,marginnote}
\usepackage[dvips]{graphicx}
\usepackage[margin=1in]{geometry}

\usepackage{hyperref}

\theoremstyle{plain}
\newtheorem{theorem}{Theorem}

\newtheorem{remark}[theorem]{Remark}
\newtheorem{proposition}[theorem]{Proposition}
\newtheorem{lemma}[theorem]{Lemma}

\numberwithin{equation}{section} 
\numberwithin{theorem}{section}

\newcommand{\comment}[1]{\vskip.3cm
\fbox{%
\parbox{0.93\linewidth}{\footnotesize #1}}
\vskip.3cm}

\newcommand{\mf}{\widetilde}

\def\wtF{\widetilde{\mathcal{F}}}
\def\whF{\widehat{\mathcal{F}}}

\def\supp{\mathrm{supp}}
\def\e{{\varepsilon}}

\def\eps{{\epsilon}}
\def\epss{{\epsilon_1,\epsilon_2}}
\def\epsss{{\epsilon_1,\epsilon_2,\epsilon_3}}

\def\C{{\mathbb C}}
\def\Z{{\mathbb Z}}

\def\R{{\mathbb R}}

\def\jxi{\langle \xi \rangle}
\def\jeta{\langle \eta \rangle}
\def\jsigma{\langle \sigma \rangle}
\def\jsig{\langle \sigma \rangle}

\def\jrho{\langle \rho \rangle}

\def\jt{\langle t \rangle}
\def\js{\langle s \rangle}
\def\jx{\langle x \rangle}

\def\what{\widehat}

\def\Ftil{\wt{\mathcal{F}}}

\def\s{\sigma}

\def\jnab{\langle \nabla \rangle}

\def\pv{\mathrm{p.v.}}

\def\wt{\widetilde}

\def\tofill{\vskip20pt $\cdots$ To fill in $\cdots$ \vskip20pt}

\definecolor{red}{rgb}{1.0, 0.0, 0.7}

\definecolor{bluegreen}{rgb}{0.0, 0.3, 0.9}

\newcommand{\fp}[1]{\color{bluegreen} { #1} \color{black}  }

\begin{document}

\author{Tristan L\'eger}
\address{Tristan L\'eger, Princeton University, Department of Mathematics, Fine Hall, 
  Washington Road, Princeton, NJ, 08544, United States.}
\email{tleger@princeton.edu}

\author{Fabio Pusateri}
\address{Fabio Pusateri, University of Toronto, Department of Mathematics, 40 St George Street, 
  Toronto, ON, M5S 2E4, Canada.}
\email{fabiop@math.toronto.edu}

\title[Internal modes for quadratic KG]{Internal modes and radiation damping for
\\ quadratic Klein-Gordon in $3$d}

\begin{abstract}
We consider Klein-Gordon equations with an external potential $V$ and a quadratic nonlinearity in $3+1$ space dimensions.
We assume that $V$ is regular and decaying and that the (massive) Schr\"odinger operator $H = -\Delta + V + m^2$ 
has a positive eigenvalue $\lambda^2 < m^2$ with associated eigenfunction $\phi$. 
This is a so-called internal mode and gives rise to 
time-periodic and spatially localized solutions of the linear flow. 
We address the classical question of whether such solutions persist under 
the full nonlinear flow,
and describe the behavior of all solutions in a suitable neighborhood of zero.

Provided a natural Fermi-Golden rule holds, 
our main result shows that
a solution to the nonlinear Klein-Gordon equation can be decomposed into a discrete component $a(t)\phi$
where $a(t)$ decays over time, and a continuous component $v$ which has some weak dispersive properties.
We obtain precise asymptotic information on these components 
such as the sharp rates of 
decay  $|a(t)| \approx t^{-1/2}$ and ${\| v(t) \|}_{L^\infty_x} \approx t^{-1}$, 
(where the implicit constants are independent of the small size of the data)
as well as the growth of a natural weighted norm of the profile of $v$.


In particular, our result extends the seminal work of Soffer-Weinstein \cite{SWmain} 
for the cubic Klein-Gordon, 
and shows that radiation damping also occurs in the quadratic case.





\end{abstract}

\maketitle

\tableofcontents

\medskip
\section{Introduction}

\subsection{Description of the problem and difficulties}
\color{black}
We consider the following initial value problem:\footnote{\textbf{MSC classes:} 35Q40, 35L70, 35P25, 37K55

\textbf{Subjects:} Partial differential equations, Mathematical Physics

}
\begin{equation}\label{introKG}
\begin{cases}
\partial_t ^2 u + (-\Delta + V(x) + m^2)u = u^2, 
\\
u(0,x) = u_0, \quad \partial_t u(0,x) = u_1,
\end{cases}
\end{equation}
for an unknown $u: \R_t \times \R^3_x \rightarrow \R$,
with a sufficiently regular and decaying external potential $V:\R^3\rightarrow \R$, 
and sufficiently smooth and localized initial data $u_0$ and $u_1$.
Precise assumptions on $V$ and the data will be given further below.
In the nonlinear Klein-Gordon equation \eqref{introKG} 
one can rescale the mass of the scalar field and therefore, without loss of generality, we may assume that $m\equiv 1$  
from now on.

Under our assumption on the decay of $V$, 
the spectrum of the Schr\"odinger operator $H:= -\Delta + V$ consists of purely absolutely continuous 
spectrum in $[0,\infty]$ and a finite number of negative eigenvalues, 
with corresponding smooth and fast decaying eigenfunctions \cite{SimonSpec}.
We then further assume that 
the operator $L^2 := -\Delta + V + 1$ has a unique strictly positive eigenvalue\footnote{Generalizations 
of our results to multiple eigenvalues may be possible under some suitable conditions, but we do not pursue this here. 
See the work of Bambusi and Cuccagna \cite{Bambusi-Cuccagna} for the cubic problem.}
with a corresponding normalized eigenfunction $\phi$: 
\begin{align}\label{introim}
(-\Delta + V + 1)\phi = \lambda^2 \phi, \qquad  1/2 < \lambda < 1, \qquad {\| \phi \|}_{L^2} = 1.
\end{align}
In particular we have $\lambda \notin \sigma_{ac}(L) 
$ and $2 \lambda \in \sigma_{ac}(L).$
The eigenvalue $\lambda$ is usually called the ``internal frequency of oscillation'',
and $\phi$ is referred to as an ``internal mode'' of the dynamics.

\eqref{introim} gives rise to a two parameter family of solutions to the linear equation 
$\partial_t ^2 u + (-\Delta + V + 1)u=0$ of the form:
\begin{align}\label{phi0}
\phi_{A,\theta} (t,x) = A \cos(\lambda t + \theta) \phi(x), \qquad A,\theta \in \R.
\end{align}
These solutions - referred to as ``bound states'', or ``internal modes'' with a slight abuse -
are time-periodic, oscillating and spatially localized.
One is then interested in understanding what happens to these states under
the full nonlinear flow of \eqref{introKG}, and how general small solutions of \eqref{introKG} behave.
In this paper we answer these questions. Our main result can be informally stated as follows:

\begin{theorem}\label{thmintro}
Consider \eqref{introKG}-\eqref{introim}, for regular and decaying $V,$ 
under generic assumptions on $H,$ 
including the nonlinear Fermi-Golden rule \eqref{introFGR} below.

Then, for all data that are sufficiently small in a weighted Sobolev space,
solutions of \eqref{introKG}-\eqref{introim} decay over time.

More precisely, if we decompose $u(t,x) = a(t)\phi(x) + v(t,x)$ with $v\perp \phi$, we have
\begin{align}
|a(t)| \approx t^{-1/2}, 
  \qquad {\| v(t) \|}_{L^{\infty}} \approx t^{-1}, \quad t\gg 1.
\end{align}

\end{theorem}

Theorem \ref{thmintro} implies that, in a neighborhood of zero, the bound states \eqref{phi0} 
do not persist under the nonlinear flow of \eqref{introKG}, 
nor continue to quasi-periodic solutions or other non-decaying solutions.
However, the decay of solutions of \eqref{introKG}-\eqref{introim}, 
is much slower compared to linear solutions of $(\partial_t^2 - \Delta + 1 + V)u=0$
(for generic $V$ with no bound states), or to nonlinear solutions of
$(\partial_t^2 - \Delta + 1) u = u^2$,
which all decay at the rate of $t^{-3/2}$ (in $L^\infty_x$).

\medskip
To the authors' knowledge, this result is the first construction and asymptotic description of dispersive solutions in $d=2,3$ for low powers (below the Strauss exponent) and non trivial spectrum, that is $\sigma(-\Delta + V) \neq \sigma(-\Delta).$ We remark that the linearization of the $\phi^4$ model near its kink solution in $1d$ fits that setting. This has prompted several groups to investigate the small data regime of such equations in $1d$, although the results currently available are either only valid on shorter than global time scales, or for local-in-space norms. We refer to Section \ref{sseclit} for more information on the related literature. 

\medskip
The question of the persistence for \eqref{phi0} is classical, 
and was first addressed by Sigal \cite{Sigal}, who showed instability
of the bound states for very general classes of equations. 
Soffer and Weinstein \cite{SWmain} then obtained asymptotics
in the case of cubic Klein-Gordon equations, that is, \eqref{introKG} with a $u^3$ nonlinearity.
We will discuss the literature around this problem and related ones later in Subsection \ref{sseclit}.
For the moment, we just mention that the case of a quadratic nonlinearity is substantially more difficult
than the cubic case, as it is to be expected in the context of the long-time analysis of dispersive and wave equations.
It is also worth noting that quadratic models appear naturally in the linearization around 
(topological) soliton solutions
of many physical systems,
and, therefore, are often more relevant in applications. 

\medskip
In short, the main difficulties in the analysis of \eqref{introKG}-\eqref{introim} are:



\setlength{\leftmargini}{1.5em}
\begin{itemize}

\smallskip
\item[1.] A complicated ODE analysis for the amplitude of the discrete component, 
where separating the leading order damping dynamics
from remainders terms requires a more careful study of resonances and oscillations.

\smallskip
\item[2.] Analyzing localized solutions
on the background of a large potential, and propagating weighted estimates
for low power nonlinearities. 
In fact even when $H:=-\Delta + V$ does not have bound states, asymptotic stability was not previously known for quadratic Klein-Gordon with $V \neq 0.$

\smallskip
\item[3.] The slow decay of the discrete component creating a singularity in frequency space 
in the continuous component of the solution, that is, $v.$
This originates precisely from 
the feedback of the discrete mode of the solution into the PDE satisfied by the radiation. 
\end{itemize} \color{black}
\smallskip
Let us briefly elaborate on these points before explaining our approach to overcome them in the next subsection.

\setlength{\leftmargini}{1.5em}
\begin{itemize}

\smallskip
\item[(1)] The presence of a bound state for $H$ is an obvious 
difficulty and can potentially
generate a non-decaying state within the full nonlinear flow, 
thus making the dynamics too complicated to control for long times.
Fortunately, it is well known, following the seminal work of Sigal \cite{Sigal},
that the solutions \eqref{introim} are unstable for the flow of \eqref{introKG}
under a suitable generic assumption, the so-called nonlinear `Fermi Golden rule' (see \eqref{introFGR}).

It turns out that the states \eqref{phi0} are in fact meta-stable, 
in the sense that the component of $u$ along $\phi$ will decay over time, 
but at an anomalously slow rate compared to the $t^{-3/2}$ decay
of linear solutions of the Klein-Gordon equation (KGE);
in the case of \eqref{introKG}, the rate that we prove is $t^{-1/2}$.
This meta-stability phenomenon was rigorously proven in \cite{SWmain} for the cubic KGE
(see also Remark \ref{remSW}).

To capture this 
in the quadratic model \eqref{introKG}, one can start (as in the cubic case) by decomposing
\begin{align}\label{avintro}
u = a(t) \phi + v, \qquad a(t) := ( u(t), \phi ), \qquad v\perp \phi \quad (v = {\bf P}_c v).
\end{align}
See \eqref{proj} for definitions.
This leads to analyze an ODE for the amplitude of the internal oscillations 
(or discrete component) $a(t)$,
and a PDE for the `continuous component' $v$. Both equations are obviously quadratic; 
see \eqref{sysav} below.
After standard normal form transformations, and upon plugging in the leading order (singular component)
of $v$ in the ODE, this equation assumes 
the form $\dot{x} = c x^2 + R$, with $\Re(c) <0$,
where $x := |a|^2$ and $R$ 
includes several types of nonlinear expression in $a$ 
and (nonlinear expression in) $v$ (projected along $\phi$).
Since the equation is quadratic, many of these terms
are slowly decaying and cannot be treated perturbatively simply using linear estimates 
(e.g., linear decay, Strichartz, or local decay) as in the cubic case.

\smallskip
\item[(2)] 
In the basic case where $H=-\Delta + V$ has no bound states (for example when $V \geqslant 0$)
a strong notion of asymptotic stability for small solutions of quadratic KGE
in $3$d is not known, to the best of our knowledge. 
The existence of global solutions follows from local-wellposedness and conservation of energy;
however, proving time-decay and asymptotics is a non-trivial task.

One should note that the case $V=0$ is classical (see Klainerman \cite{KKGE} and Shatah \cite{shatahKGE}),
but relies on methods (vectorfields or normal forms) which are not available as soon as $V\neq 0$.
Indeed, the presence of a (large) potential is a well-known difficulty in the analysis of the long-time 
behavior of nonlinear PDEs as it simultaneously destroys the invariance of the equation 
(e.g., the Lorentz invariance in the case of KGE) ruling out the direct use of vectorfields methods,
and perturbs the frequencies of interacting waves, 
thereby preventing standard (nonlinear) Fourier methods from being effective.

The combination of a potential and a quadratic nonlinearity 
naturally requires using a weighted norm which
(a) is strong enough to imply dispersive properties and (b) is sufficiently well-adapted to the problem
so that it can be controlled along the flow.
Our choice is to measure $\partial_\xi (\wtF{f})(t,\xi)$ in $L^2$, 
where $f$ is the profile of $v$ 
defined in \eqref{introf}, 
and $\wtF$ denotes the distorted Fourier transform associated to the Schr\"odinger operator $H$ (see Section \ref{secdFT}).
This choice has been used in the $3$d works \cite{GHW,PS}, as well as in the $1$d works \cite{GPR,ChPu};
it essentially corresponds to 
standard weighted norms measuring $x f$ in $L^2$ (in physical space). 
Such a norm fairly easily takes care of (a) above.
Controlling it over time turns out to be a difficult task, as the internal mode causes various growth phenomena.

\smallskip
\item[(3)] Indeed in our problem \eqref{introKG}-\eqref{introim} there are additional difficulties 
coming from the feedback of the internal mode $a(t)\phi$ 
into the equation governing the behavior of the continuous (spectral projection) component $v(t,x).$ 
The slow decay of $a$ in turn causes $v$ to have anomalously slow decay. 
More precisely, its $L^{\infty}$ norm decays at the rate $1/t$ (without smallness) 
instead of $\varepsilon /t^{3/2},$ where $\varepsilon$ denotes the size of the initial data, 
as would be the case for linear Klein-Gordon equations \cite{JSS,KKbook,Sch1}
as well as for quadratic KG with\footnote{We 
also expect the same rate of decay for \eqref{introKG} when $V$ has no bound states; see Remark \ref{KGnone}} 
$V=0$ \cite{shatahKGE}.

Furthermore, it turns out that the slow decay of $a(t)$ also
creates a singularity in (distorted) frequency space
(essentially corresponding to a lack of good localization in physical space).
In fact, we prove that 
the natural $L^2(\jx^2 dx)$ weighted norm of the profile $f$
is lower bounded by
$\sqrt{t}$. 
This means that parts of the solution are very badly behaved when measured in the
most natural norms that one would like to control. 
By contrast, in the case where $H$ does not have bound states, 
this weighted norm is expected to stay small for all times.
See Remark \ref{rembad1} for more on this.
\end{itemize}

\subsection{Strategy}
To address the issues described above we propose an approach based on some new ideas which we now describe.

\smallskip
\subsubsection{Discrete and Continuous decomposition, ODE dynamics, damping and remainders} 
As it is standard, we begin by decomposing 
the solution as  in \eqref{avintro}. 
This leads to an ODE for $a$ and a PDE for $v$:
\begin{equation}\label{sysav}
\begin{cases}
\ddot{a} + \lambda^2 a = \big((a \phi + v)^2 ,\phi\big) 
\\
\partial_t ^2 v + L^2 v = {\bf{P_c}} \big((a \phi + v)^2 \big), \qquad L = \sqrt{-\Delta + V + 1}.
\end{cases}
\end{equation}
The analysis of the ODE in \eqref{sysav} is conceptually similar to that of Soffer-Weinstein \cite{SWmain}.
However we follow a slightly different approach based solely on the method of averaging for ODEs. 
This allows us to derive both upper and lower bounds for $a,$ proving sharpness of the decay rate.
More importantly, as pointed out above, 
it is quite harder here to separate the leading order (damping) dynamics,
that are responsible for the decay of $|a(t)|$, from the remainder terms. 
This is due to the quadratic nature of the nonlinearity.

To explain this, let us filter out the linear oscillations passing to the profile $A(t)$ via
\begin{align}\label{profiles0A}
\begin{split}
& A(t) := \frac{1}{2i\lambda}e^{-i\lambda t} (\dot{a} + i \lambda a),
\\
& (a(t) = A(t) e^{i\lambda t} + \overline{A}(t) e^{-i\lambda t},
  \quad \dot{A}(t) e^{i\lambda t} + \dot{\overline{A}}(t) e^{-i\lambda t} = 0).
\end{split}
\end{align}
The ODE in \eqref{sysav} is 
\begin{align}\label{Duhamela0}
\begin{split}
\dot{A}(t) & = \frac{1}{2i\lambda} e^{-i\lambda t} \Big( \big(a(t) \phi + v(t) \big)^2, \phi \Big),
\\
& = \frac{1}{2i\lambda} e^{-i\lambda t} a^2(t) \int_{\R^3} \phi^3\,dx 
  + \frac{1}{i\lambda} e^{-i\lambda t}  a(t) \int_{\R^3} v(t) \phi^2 \, dx + \cdots ,
\end{split}
\end{align}
having retained only some of the main terms for the purpose of our explanation.
The terms that are quadratic in $a$ on the right-hand side of \eqref{Duhamela0}
can be made cubic using standard ODE normal forms. 
Those can in turn be eliminated in favor of higher order remainder terms,
with the exception of the `integrable' ones, \textit{i.e} of the form $ i c |A|^2A$ for a constant $c\in\R$; 

The leading order dynamics are coming from the second integral in \eqref{Duhamela0}.
In a perturbative regime one may expect the solution of the PDE in \eqref{sysav} 
to be given at leading order by the solution $z$ of $\partial_t^2 z + L^2 z = a^2 {\bf{P_c}}\phi^2$
(with the same initial data of $v$):
\begin{align}\label{introvmain0}
\begin{split}
z(t) & = \frac{1}{L}\Im \big(e^{itL}(v_t(0) + i Lv(0) \big)  
  + \frac{1}{L} \Im \int_0^t e^{i(t-s)  L  } a^2(s) \, ds \, {\bf{P_c}}\phi^2
  \\
  & = \frac{1}{L} \Im e^{itL} \int_0^t e^{-is(L-2\lambda)} A^2(s) \, ds \, {\bf{P_c}}\phi^2 + \cdots,
\end{split}
\end{align}
where in the last identity we used \eqref{profiles0A} and only wrote out explicitly the leading order\footnote{This 
term corresponds to the first integral in the right-hand side of \eqref{Duhamelw0}} term.
Upon integration by parts in $s$, and disregarding less important contributions, we can write
\begin{align}\label{introvmain}
\begin{split}
\int_0^t e^{-is(L-2\lambda)} A^2(s) \, ds \, {\bf{P_c}}\phi^2 
  & = \lim_{\eta \rightarrow 0^+} 
  i\frac{e^{2i\lambda t} e^{-\eta t}}{L-2\lambda+i\eta} \, {\bf{P_c}}\phi^2 \, A^2(t) + \cdots
  \\
  & =  \pi e^{2i\lambda t} A^2(t) \delta(L-2\lambda) {\bf{P_c}}\phi^2 + \cdots.
\end{split}
\end{align}
Plugging \eqref{introvmain} into \eqref{introvmain0}, and then $z$ in place of $v$ 
in the last integral of \eqref{Duhamela0} leads to
\begin{align}\label{Duhamela0'}
\begin{split}
\dot{A}(t) = 
  -\frac{\pi}{(2\lambda)^2} |A(t)|^2 A(t) \int \delta(L-2\lambda)  {\bf{P_c}}\phi^2 \, \phi^2 \, dx + R,
\end{split}
\end{align}
for some 
$R$.
One then imposes that the Fermi Golden rule 
\begin{align*}
\Big({\bf{P_c}} \phi^2, \delta(L-2\lambda) {\bf P_c} \phi^2 \Big) > 0
\end{align*}
holds, so as to obtain damping for $|A|$ from the ODE \eqref{Duhamela0'},
provided $R$ is shown to be a remainder.

This turns out to be a non-trivial task. Indeed $R$ includes a term of the form
\begin{align}\label{RFtolaterintro0}
\int_{\R^3} \phi^2(x) \, \frac{1}{L} \Im \int_0^t e^{i(t-s)L} {\bf P_c} \big( v(s,x) \big)^2  ds \, dx,
\end{align}
which leads to an issue that is specific to the quadratic case: in order for this term to give a perturbative contribution to the dynamics of $A$ in \eqref{Duhamela0'}
one needs to show that it decays faster than $1/t$ (that is, $|A|^2$).
This amounts to showing 
\begin{align}\label{RFtolaterintro}
{\Big\| \int_0^s e^{i(s-\tau)L} {\bf P_c} \big( v(\tau) \big)^2  d\tau \Big\|}_{L^\infty_x} 
  \lesssim  \frac{1}{s^{1+b}},
\end{align}
for some $b>0$, when $s \gg 1$.
\eqref{RFtolaterintro} is akin to a decay estimate for a quadratic Klein-Gordon equation,
which is classical for $V=0$ but not known for any $V\neq 0$.
A simple bootstrap argument is not sufficient, therefore we rely on 
multilinear analysis techniques in distorted Fourier space. 


\smallskip
\subsubsection{Harmonic Analysis tools}
Our approach to 
the PDE part of the problem is based on the use of (nonlinear) 
Fourier analysis techniques adapted to the context of the Schr\"odinger operator;
it follows the general philosophy 
of many recent works that used Fourier analysis in the flat case and were inspired 
by the space-time resonance method.
However this requires developing new tools adapted to the non flat background, 
namely a generalization of the Fourier transform, 
as well as a precise understanding of the interaction of three generalized eigenfunctions 
of the perturbed Laplacian $H =-\Delta + V.$ This is detailed in the remainder of this subsection.

\subsubsection*{Distorted Fourier Transform (dFT) and `continuous' interactions}
Given a regular and decaying $V$, under our assumption \eqref{introim},
we have the canonical decomposition of $L^2$ into absolutely continuous and pure point subspaces:
$L^2(\R^d) = L^2_{ac}(\R^d) \oplus L^2 _{pp}(\R^d)$ where $L^2 _{pp}(\mathbb{R}^d) = \mathrm{span}(\phi)$. 
Letting $\psi(x,\xi)$ denote the generalized eigenfunctions 
solving
\begin{align}\label{psieq0}
\begin{split}
& (-\Delta + V ) \psi(x,\xi) = \vert \xi \vert^2 \psi(x,\xi), \qquad \xi,x \in \R^3,
  \\
  & \mbox{with} \quad \big| \psi(x,\xi) - e^{ix\cdot \xi} \big| \longrightarrow 0,
  \quad \mbox{as} \quad |x|\rightarrow \infty,
\end{split}
\end{align}
the familiar formulas relating the Fourier transform and its inverse 
hold if one replaces (up to constants) $e^{ix\cdot \xi}$ by $\psi(x,\xi)$:
for any function $g\in L^2_{ac}$, there exists a unitary operator $\Ftil$ (the dFT) defined as 
\begin{align*}
\wt{\mathcal{F}}f(\xi) := \wt{f}(\xi) = \frac{1}{(2\pi)^{3/2}} \lim_{R \to + \infty} 
\int_{\vert x \vert \leqslant R} f(x) \overline{\psi(x,\xi)} dx,
\end{align*}
with inverse 
\begin{align*}
\big[\wt{\mathcal{F}}^{-1}f\big](x) = \frac{1}{(2\pi)^{3/2}} \lim_{R \to + \infty} 
\int_{\vert \xi \vert \leqslant R} f(\xi) \psi(x,\xi) d\xi.
\end{align*}
In particular, for any $g\in L^2$ we can write
\begin{align}\label{dFTinvintro}
\begin{split}
& g(x) = \wtF^{-1} \big( \wt{{\bf P}_c g} \big) + (g, \phi) \phi,
\end{split}
\end{align}
where ${\bf P_c}$ denotes the projection onto the continuous subspace $L^2_{ac}$.
Also, $\wtF$ diagonalizes the Schr\"odinger operator (on the continuous spectrum), $\Ftil {\bf{P_c}} H = |\xi|^2 \Ftil$,
and associated functional calculus ensues.


We then look at the `profile' or `interaction variable' given by
\begin{align}\label{introf}
\begin{split}
& f := e^{-itL} w = e^{-itL} (\partial_t + iL) v, \qquad v = \Im \big( \frac{1}{L} e^{itL} f \big).
\end{split}
\end{align}
Duhamel's formula for $f$ reads 
\begin{align}\label{Duhamelw0}
\begin{split}
f(t) - f(0) & = \int_0^t e^{-isL}  {\bf{P_c}} \Big( a(s) \phi + \frac{1}{2iL} (w(s)-\bar{w}(s)) \Big)^2 \, ds
\\
& = \int_0^t e^{-is(L-2\lambda)} A(s)^2 \,ds \, {\bf{P_c}} \phi^2 
  - \frac{1}{4} \int_0^t e^{-isL} \big( L^{-1} e^{isL} f(s) \big)^2 \, ds + \cdots,
\end{split}
\end{align}
having kept only some of the main terms above for the sake of explanation.
In distorted frequency space $\wt{f}(t,\xi) = e^{-it\jxi} \wt{w}(t,\xi), \jxi := \sqrt{1+|\xi|^2}$, 
so that
\begin{align}\label{introD}
\begin{split}
& \wt{f}(t,\xi) - \wt{f}(0,\xi)  = \int_0^t e^{-is(\jxi-2\lambda)} A(t)^2 {\bf{P_c}} \wt{\phi^2} \, ds 
  + F(t,\xi) + \cdots
\end{split}
\end{align}
where
\begin{align}\label{introD1}
F(t,\xi) & := -\frac{1}{4} \int_0^t \iint_{\R^3 \times \R^3} 
	e^{is (-\jxi + \jeta + \jsig)} \wt{f}(s,\eta) \widetilde{f}(s,\s)
  \, \frac{1}{\jeta\jsig}\mu(\xi,\eta,\s) \, d\eta d\s ds,
\\
\label{intromu}
\mu(\xi,\eta,\s) & := \frac{1}{{(2\pi)}^{9/2}} \int_{\R^3} \overline{\psi(x,\xi) }\psi(x,\eta) \psi(x,\s) \, dx.
\end{align}

\eqref{introD}-\eqref{intromu} is the starting point for our analysis of the continuous component of the solution. 
For the moment let us just point out that there are two main terms:
the first time integral which gives the (main) feedback of the internal mode into the continuous part of the solution,
and the term $F$ in \eqref{introD1}-\eqref{intromu} which contains the main 
nonlinear interaction in the continuous component.

\smallskip
\subsubsection*{The Nonlinear Spectral Distribution (NSD)}
The distribution $\mu$ in \eqref{introD1}-\eqref{intromu} 
characterizes the interaction between the generalized eigenfunctions, 
and we call it the ``{\it Nonlinear Spectral Distribution}'' (NSD),
following \cite{GPR,PS}.
Note that in the unperturbed case $V=0$ the NSD is just a delta function $\delta(\xi-\eta-\sigma)$.
In contrast with this, in equations \eqref{introD1}-\eqref{intromu} 
all frequencies interact with each other without any a priori constraint.

To handle \eqref{introD1} and prove a priori estimates in weighted spaces one needs
to  exploit the oscillations of the phase in distorted frequency space,
through integrations by parts in $\eta, \s$. 
This requires information on the regularity of the integrand, and in particular of the distribution 
$\mu.$ 
As it turns out, $\mu$ is not regular in any direction but can be decomposed in many pieces  
that are regular in at least one direction.
Oversimplifying, one can think that $\mu$ looks as follows
\begin{align}\label{intromu'}
\begin{split}
\mu(\xi,\eta,\s) = \delta(\xi-\eta-\s) + \pv \frac{1}{|\xi-\s| - |\eta|} + \pv \frac{1}{|\xi| - |\eta+\s|}
  + \cdots.
\end{split}
\end{align}
A precise version of this was shown by the second author and Soffer in \cite{PS}.
In the Klein-Gordon case we need to refine that analysis
and adapt the decomposition of $\mu$ 
to the geometry of the equation 
and of the nonlinear 
phases $\Phi(\xi,\eta,\s):= -\jxi \pm \jeta \pm \jsig$.

One useful realization in dealing with the term $F$ above is that when $\mu$ is close to being singular 
(for example, looking at the second term on the right-hand side of \eqref{intromu'}, when $|\eta| \approx |\xi-\sigma|$)
then $\Phi$ is essentially lower bounded, and one can exploit oscillations in $s$.
When instead 
$\mu$ is regular, decay can be obtained from the integral in $\eta,\s$.

Once we have dealt with regularity issues for the measure $\mu,$ we turn to the other terms that make 
up the integrand in \eqref{introD1}, namely the profiles $\wt{f}.$
Unlike in the case where no internal mode is present, $\partial_\xi \wt{f}$ is not that regular at all: 
it cannot be controlled uniformly along the flow, and in fact we show that it grows in time like $\sqrt{t}$.
This leads us to our last point in the strategy.

\smallskip
\subsubsection{Good vs. bad decomposition and the `Fermi frequency'} 
An important idea in our analysis is to further decompose the profile $f$ (see \eqref{introf})
into a `good' part (that is regular in distorted frequency space), which we call $h$, 
and a `bad' or `Fermi' component (singular in Fourier space), which we call $g$.
Looking at \eqref{Duhamelw0} it is natural to let
\begin{align}\label{introf=g+h}
\begin{split}
f & = g + h, \qquad g(t,x) = \chi(|L-2\lambda|) \int_0^t e^{-is(L-2\lambda)} A(s)^2 \,ds \, {\bf{P_c}} \phi^2,
\end{split}
\end{align}
where $\chi$ is a cutoff with small support 
around the origin. Notice that $g$ has a fairly simple and explicit expression.
Passing to Fourier space and differentiating we have
\begin{align}\label{introwtg}
\begin{split}
\partial_\xi \wt{g}(t,\xi) \approx \chi(|\jxi-2\lambda|) \int_0^t s\, 
  e^{-is(\jxi-2\lambda)} A(s)^2 \,ds \, \wt{{\bf{P_c}}\phi^2}(\xi).
\end{split}
\end{align}
We can then see that $g$ behaves poorly when the size of its (distorted) frequencies 
approaches the solution of $\jxi = 2\lambda$, i.e., the bad `Fermi frequency' $\xi_0 := \sqrt{(2\lambda)^2-1}$.
From rigorous asymptotics for $A$ 
we show that 
\begin{align}\label{introggrowth}
{\| \partial_\xi \wt{g}(t,\xi) \|}_{L^2} \gtrsim \sqrt{t}.
\end{align}
See Remark \ref{rembad1} for more on this.
In particular $\wt{g}$ is quite irregular, which complicates the estimates of expressions like \eqref{introD1}.

Once the bad component has been removed, the hope is that
the remaining component $h$ of the solution is more regular and has better localization properties.
We are indeed eventually able to show that, for $t\gg 1$,
${\| \partial_\xi \wt{h}(t,\xi) \|}_{L^2} \lesssim t^\beta \e^{\beta}$, for some small $\beta$.

\smallskip
\subsection{The main result}
Let us introduce some notation and state our precise assumptions before giving the details of our main result.
We let ${\bf P_c}$ denote the projection onto the continuous spectral subspace, 
namely, for every $\psi \in L^2(\R^3;\R)$
\begin{align}\label{proj}
{\bf P_c} \psi := \psi - (\phi,\psi)\phi, \qquad (\psi_1,\psi_2) := \int_{\R^3} \psi_1 \psi_2 \, dx.
\end{align}
We make the assumptions that:

\begin{itemize}

\medskip
\item {\em Regularity and decay of $V$}: The potential $V$ is sufficiently smooth and decaying.
For simplicity we assume $V \in \mathcal{S}$, but a finite amount of regularity 
and decay would suffice.\footnote{The main property we need to guarantee is that wave operators 
are bounded on sufficiently regular weighted Sobolev spaces, 
as in the statement of the Theorem \ref{Wobd}.} 

\medskip
\item {\em Coupling to continuous spectrum}: $\lambda \in (1/2,1)$ so that $\lambda \notin \sigma_{ac}(-\Delta + V + 1)$ 
  and $2 \lambda \in \sigma_{ac}(-\Delta + V + 1).$ 
  
\medskip
\item {\em Fermi Golden Rule}: The ``Fermi Golden Rule'' resonance condition holds:
\begin{align}\label{introFGR}
\Gamma:=\frac{\pi}{2 \lambda} \Big({\bf{P_c}} \phi^2, \delta(L-2\lambda) {\bf P_c} \phi^2 \Big) > 0.
\end{align}

\medskip
\item {\em Genericity of $V$}: the $0$ energy level is regular for $H := -\Delta+V$, 
that is, $0$ is not an eigenvalue, 
nor a resonance, i.e., there is no $\psi \in \jx^{1/2+} L^2(\R^3)$ such that $H\psi = 0$. 
Such a potential $V$ is said to be `generic'.

\end{itemize}

\smallskip
The first two conditions above are the same as in \cite{SWmain}, naturally transposed to the quadratic case.
The last condition ensures that one has the same linear decay for the propagator $e^{itL}{\bf P}_c$ 
as the one for $e^{it\langle \nabla \rangle}$.

\smallskip
Recall that we decompose a solution $u$ into a discrete and a continuous component:
\begin{align}\label{uav}
u(t) = a(t) \phi + v(t),
\end{align}
with $(v(t),\phi)=0$ for all $t.$ 
This yields the coupled system \eqref{sysav}. 
We then pass to a first order system by introducing the new unknowns $A$ (see \eqref{Adotest}) and  
$w:=\partial_t v + i L v$ which satisfy \eqref{sysav}.

We call $v$ or, with a slight abuse of notation, $w$, the {\it radiation} or {\it field}
component of the solution $u$, and will call $a$ the (amplitude of the) 
{\it discrete component} or {\it internal mode} of the solution.

\begin{theorem}\label{maintheo}
Consider \eqref{introKG}, with initial data $u(0,x) = u_0(x)$, $u_t(0,x)=u_1(x)$ such that: 
\begin{align}\label{mtdata}
\begin{split}
& | ( u_0, \phi) | + | ( u_1, \phi) | \leqslant \e_0,
\\
& {\|  (\jnab u_0, u_1) \|}_{W^{3,1}}+  {\| \jx (\jnab u_0, u_1) \|}_{H^3} + {\| (\jnab u_0, u_1) \|}_{H^N}  
  \leqslant \e_0 
\end{split}
\end{align}
for\footnote{See \eqref{mteps} for a more quantitative, albeit far from optimal, choice.} $N \gg 1$.
Then, there exists $\overline{\e} \in (0,1)$ such that, for all $\e_0 \leqslant \bar{\e}$,

\setlength{\leftmargini}{1.5em}
\begin{itemize}

\smallskip
\item The equation \eqref{introKG} under the assumptions stated above, has a unique global solution 
$u \in C(\R; H^{N+1}(\R^3))$ such that the following hold:
$u = a(t) \phi  + v(t)$ with
\begin{align}
\label{mtadecay}
& |a(t)| \approx \e_0 (1+\e_0^2 t)^{-1/2}, 
\end{align}
and
\begin{align}
\label{mtvdecay}
& \jt {\big\| \partial_t v + i L v \big\|}_{L^\infty_x} \approx 1, 
  \qquad t \gtrsim \e_0^{-2}, 
\\
\label{mtSobolev}
& {\big\|\partial_t v + i L v \big\|}_{H^N_x} \lesssim  \e_0^{1-\delta},
\end{align}
for arbitrary small $\delta > 0$.

\smallskip
\item
Furthermore, we have the following asymptotic behavior: 
Let 
$f:= e^{-itL}(\partial_t + i L)v.$
Then, there exists $f_\infty \in \e_0^{1-\delta} H^N_x$, with $\delta >0$ arbitrarily small as above,
such that
\begin{align}\label{mtscattHN}
{\| f(t) - f_\infty \|}_{H^N_x} \lesssim \e_0^{1-\delta} \jt^{-\delta'}
\end{align}
for some small $\delta' >0$.
Finally, we have the asymptotic growth
\begin{align}\label{mtgrowth}
{\big\| \partial_\xi \wt{f}(t) \big\|}_{L^2} \gtrsim \jt^{1/2}, \qquad t \gtrsim \e_0^{-2}.
\end{align}

\end{itemize}

\end{theorem}


\medskip
The proof of our main theorem is based on a bootstrap argument for the 
amplitude of the internal mode $a(t)$ and for the `good' component of the profile $f$,
which we call $h$. Proposition \ref{mtboot} is our main bootstrap.

The bad part of the profile $g=h-f$ (see \eqref{def-h} and \eqref{defB} for the definition,
as well as the simplified version in \eqref{introf=g+h})
is the one responsible for the weak decay rate \eqref{mtvdecay} 
and the growth \eqref{mtgrowth}. 
We discuss this in more details in Remark \ref{rembad1} below.

The decay \eqref{mtadecay} is part of the main bootstrap; see \eqref{main-amplconc0}.

The proof of \eqref{mtvdecay} is given in \S\ref{mtvdecaypr} as a consequence of the main bootstrap proposition.

\eqref{mtSobolev} is a direct consequence of \eqref{mtscattHN}.
The proof of the latter is discussed in \S\ref{mtSobolvpr};
there we show an analogous scattering statement with $g$ instead of $f$, 
while Proposition \ref{propHN} gives the (stronger) scattering statement for $h$.

The growth estimate \eqref{mtgrowth} is a consequence of the inequality
\begin{align}\label{introggrowth'}
{\big\| \partial_\xi \wt{g}(t) \big\|}_{L^2} \gtrsim \jt^{1/2},
\end{align}
which is proven in \S\ref{ssecgrowth}
and the bound ${\| \partial_\xi \wt{h}(t) \|}_{L^2} \lesssim \e^\beta \jt^\beta$, for small $\beta >0$,
in the main bootstrap Proposition \ref{mtboot}; see \eqref{bootconc0}.

\medskip
\subsection{Some remarks}
Here are a few remarks about our result and some related problems.


\smallskip
\begin{remark}[About energy and decay]
The energy of the solution, initially split between
discrete and continuous modes, slowly migrates to the continuous spectrum,
since the $H^N_x$ norm of $a(t)\phi$ decays to zero. 

The full solution $u$ decays over time as slowly as $a(t)$, that is, at the rate of $t^{-1/2}$, $t\gg1$. 
The internal mode is responsible for the slow decay,
while the radiation part decays faster, but still at a rate that is not integrable in time. 
Moreover, the $L^{\infty}_x$ decay rate comes without smallness.
On the other hand we can prove an (almost optimal) $L^6_x$ decay estimate 
with some smallness:
\begin{align}\label{mtv=6decay}
{\big\| \partial_t v + i L v \big\|}_{L^6_x} \lesssim \e_0^\beta \jt^{-1+\delta}, 
\end{align}
for small $\delta \geqslant \beta + 3/N$, where $N$ is the Sobolev regularity
and $\beta$ is a small parameter (see \eqref{mteps}). 
\eqref{mtv=6decay} is proven in \S\ref{mtv=6decaypr}.

Note that we prove sharpness of the decay both for the discrete 
component and for the radiation in $L^{\infty}_x.$ 
The Fermi Golden Rule is ultimately the reason for this sharpness, 
stressing the fact that the interaction between the discrete and continuous parts 
of the spectrum dictates the exact behavior of the solution.

\end{remark}

\smallskip
\begin{remark}[The bad component: A $1$d ``singularity'' and \eqref{mtgrowth}]\label{rembad1}
As already mentioned, a key part of our analysis is to isolate the bad contributions 
to the radiation by splitting the profile $f$ into $h$ and $g$, see \eqref{def-h}. 
The bad part $g$ is quite explicit and it is apparent that it behaves badly
for small $\jxi - 2\lambda$, that is, as $|\xi| \longrightarrow \xi_0 := \sqrt{4\lambda^2-1}$,
as the integral lacks oscillations.
Notice that this is a $1$d singularity in 
$|\xi|-\xi_0$.
In this respect our $3$d problem shares many features with its $1$d analogue
which arises in the study of the stability of the $1$d kink for the $\phi^4$ model.
See \S\ref{sseclit} for some literature and recent advances on this.

The growth of the bad component $g$ of the radiation profile $f$, as given in \eqref{mtgrowth},
can be shown through accurate asymptotics for the amplitude of the internal mode $A$
(or, equivalently, of the renormalized amplitude $B$).
As expected, this growth happens around the Fermi frequency $\xi_0$.
In \eqref{growth} we prove a similar upper bound up to a logarithm using just basic estimates. 


Similar growth phenomena (but with a small factor of $\e$)
were previously observed (though not rigorously proven) for equations without external potentials
in $2$d; 
see \cite{DIP,DIPP} and references therein.
A related growth phenomenon is also expected in $1$d problems (with potentials) 
in the presence of internal modes
(see \cite{DMKink}) or zero-energy resonances (see the discussions in the works \cite{LLS} and \cite{GPKG}).
\end{remark}

\begin{remark}[Pointwise asymptotics and scattering]\label{remasyf}
It is also possible to show a more refined scattering statement in $L^\infty_\xi.$ 
We only state a result below and prove it rigorously in a separate note \cite{LPnote}.

First, by using the estimates (and variants of those) in Sections \ref{SecF}-\ref{secmixed}
one can show that there there exist $\mathcal{N}(\wt{f}) \in L^{\infty}_t  L^{\infty}_{\xi}$ and $h_\infty \in \e^\delta L^\infty_\xi$ such that 
\begin{align} \label{hinfty}
{\big\| \wt{h}(t) -  \mathcal{N}(\wt{f})(t) - h_\infty \big\|}_{L^\infty_\xi} \lesssim \e^\delta \jt^{-\delta}
\end{align}
for some $\delta >0.$

Then, combining this with the asymptotics \eqref{gasy},
letting $F_\infty(t) :=\mathcal{N}(\wt{f})(t) +  h_\infty - g_\infty,$ we can show that 
\begin{align}\label{fasy}
\begin{split}
\wt{f}(t,\xi) = \int_0^t \exp \Big(is(\jxi-2\lambda)
  + i \tfrac{2c_2\lambda}{\Gamma}\log\big(1 + \tfrac{\Gamma}{\lambda} Y_0^2 s\big) \Big) y(s)^2 ds \, e^{i 2\Psi_\infty} 
  \, \wt{{\bf P}_c\phi^2}(\xi)
  \\
  + F_\infty(t,\xi) 
  + \, O_{L^\infty_\xi}\big(\varepsilon^\delta \jt^{-\delta}\big),
\end{split}  
\end{align}
where $y(s) := Y_0 \, (1 + (\Gamma/\lambda) s Y^2_0)^{-1/2}$,
with $Y_0$ that can be determined in terms of the initial data 
(through \eqref{BYphi0} with \eqref{BW} and \eqref{defB}),
the phase $\Psi_\infty \in \e^\delta L^\infty$ 
(see \eqref{BPsi} with \eqref{Bphilim} and \eqref{Bpsiinfty}),
and $c_2$ is the constant in \eqref{BWdot}. 
Relying on new bilinear lemmas for operators with singular kernels, we show in \cite{LPnote}
that the correction $\mathcal{N}(\wt{f})$ is of size $\e^{\delta}.$

Next, note that the time integral in \eqref{fasy} is large when $\jxi \approx 2\lambda$: 
indeed, when $|\jxi - 2\lambda| \ll t^{-1}$, $t\gg1$, we have
\begin{align}\label{fasy1}
\begin{split}
& \int_0^t \exp \Big(is(\jxi-2\lambda)
  + i 2c_2\tfrac{\lambda}{\Gamma} \log\big(1 + \tfrac{\Gamma}{\lambda} Y_0^2 s\big) \Big) y(s)^2 ds \, e^{i 2\Psi_\infty} 
  \, \wt{{\bf P}_c\phi^2}(\xi)
  \\ 
  & \approx \int_0^t
  \exp \Big(i 2c_2\tfrac{\lambda}{\Gamma} \log\big(1 + \tfrac{\Gamma}{\lambda} Y_0^2 s\big) \Big) 
  \frac{Y_0^2}{(1 + \tfrac{\Gamma}{\lambda} s Y^2_0)} \, ds \, e^{i 2\Psi_\infty} 
  \, \wt{{\bf P}_c\phi^2}(\xi) + o(1),
  \\
  & = \frac{1}{2ic_2} \Big[ \exp \Big(i 2c_2\tfrac{\lambda}{\Gamma} 
  \log\big(1 +\tfrac{\Gamma}{\lambda} Y_0^2 t \big) \Big) - 1 \Big] e^{i 2\Psi_\infty} 
  \, \wt{{\bf P}_c\phi^2}(\xi) + o(1).
\end{split}  
\end{align}
The quantity above has modulus $O(1)$ for (arbitrarily large) times $t \gtrsim Y_0^{-2} \approx \e_0^{-2}$. 
Together with \eqref{hinfty} and the smallness of the correction, this entails that $\wt{f}(t,\xi)$ 
is also of size $O(1)$ around the frequency $\vert \xi \vert  = \sqrt{4 \lambda^2 - 1}.$
Note, however, that its $H^N_x$ norm is small,
as in \eqref{mtSobolev} and \eqref{mtscattHN}, since the singularity/growth happens on a lower dimensional set.
\end{remark}

\smallskip
\begin{remark}[The case without bound states]\label{KGnone}
If one considers \eqref{introKG} when $H=-\Delta + V$ has no bound states
(for example when $V \geqslant 0$ or has a small negative part) the problem is substantially easier. 
Nevertheless, as mentioned in the introduction, global decay of solutions in this case 
was not known before this work.
It is certainly possible that, in this special case one could actually prove stronger decay estimates
than just the almost sharp $L^6_x$ bound we have in \eqref{mtvdecay}.
The case of the Schr\"odinger equation with a $u^2$ nonlinearity was treated by
the second author and  Soffer in \cite{PS} where integrable-in-time decay
was obtained.
\end{remark}

\smallskip
\begin{remark}[Comparing to Soffer-Weinstein 1: weighted norms]\label{remSW}
In the seminal paper \cite{SWmain}, the authors treated the cubic problem.
Our result settles the quadratic problem, and therefore 
the case of a generic (semilinear) Hamiltonian perturbation of the linear Klein-Gordon equation.

Our analysis is very much complicated by the need to estimate weighted norms to control quadratic interactions.
And, as observed, weighted norms naturally grow due to the interaction between discrete and continuous spectrum.
For the cubic problem, if one were to measure weighted norms, one would also see a large growth over time.
However, the cubic problem can be handled just using suitable $L^p_x$ type norms 
(for example bootstrapping sharp $L^8_x$ decay),
so that the main problems we face can be bypassed.
\end{remark}

\smallskip
\begin{remark}[Comparing to Soffer-Weinstein 2: an apparent paradox]\label{remSW2}
In comparing the result of \cite{SWmain} on the cubic problem
with our result, there is one apparent conceptual paradox that it is worth addressing.
One could think that the objective here is to destroy
coherent states (the bound states \eqref{introim}) and, therefore, a stronger perturbation (quadratic)
should be more favorable than a weaker one (cubic).
However, as it is clear from the approach that one naturally needs to follow, 
the heart of the matter is to control the dynamics of \eqref{sysav} for long-times
which this is substantially harder for a lower power nonlinearity. 
\end{remark}

\smallskip
\begin{remark}
The analysis can be extended to the case where the nonlinearity is multiplied by a smooth bounded coefficient, provided the Fermi Golden rule is suitably adapted.
\end{remark}

\medskip
\subsection{Some literature}\label{sseclit}
As already mentioned at the beginning of this introduction,
our work is greatly inspired by the seminal paper \cite{SWmain} where the cubic KG problem was treated,
and the foundational paper by Sigal \cite{Sigal}.
We refer to the beautiful introductions of \cite{Sigal,SWmain} 
for a more thorough discussion about the history of the problem and of the physical background.
The problem of meta-stability in the presence of multiple bound states 
in $3$d was treated by Bambusi-Cuccagna \cite{Bambusi-Cuccagna}.
For the nonlinear Schr\"odinger equation (NLS) the problem of meta-stability (for excited states) was
considered by Tsai-Yau \cite{TsaiYau}; see also the recent advances by Cuccagna-Maeda \cite{CuMa2},
their survey \cite{CuMaSurII} and references therein.
In one dimension, nonlinear Klein-Gordon equation with bound states were studied by 
Komech-Kopylova \cite{KomKop,KomKop2} in the context of the kink stability problem
for relativistic Ginzburg-Landau theories with sufficiently flat wells.
In a recent monograph \cite{DMKink}, Delort-Masmoudi studied this problem in the context of the
one dimensional $\phi^4$ equation and the stability of its kink solution under odd perturbations;
these authors were able to prove a meta-stability theorem for times 
of $O(\e_0^{-4})$, where $\e_0$ is the size of the perturbation.
See also the recent work by Cuccagna-Maeda on NLS \cite{CuMaNLS} and references therein.

Concerning the general theory of nonlinear PDE with potentials, let us focus on recent
works about the study of nonlinear resonances\footnote{Since 
it would be impossible for us to give a complete account of the literature,
we refer to the surveys \cite{WeinSur,Sof06,SchSur} and reference therein
for more  general discussions on PDE with potentials and various applications.}.
Germain-Hani-Walsh \cite{GHW} treated the quadratic NLS $i\partial_t u +(-\Delta+V)u = \bar{u}^2$ in $3$d and 
established a first set of useful Coifman-Meyer type multilinear bounds for the NSD $\mu$ (see \eqref{intromu}), 
together with a commutation identity for the weight $|x|$.
The first author proved global decay in the case of a $u^2$ nonlinearity with small $V$ \cite{L}.
Note that while a $\bar{u}^2$ nonlinearity is non-resonant (for any type of evolution), a $u^2$ nonlinearity has 
classical space-time resonances even in the case $V=0$.
See also \cite{Leger2} for the case of electromagnetic perturbations.


More recently, Soffer and the second author \cite{PS}, following the initial ideas of \cite{GHW},
laid down the foundations for some of the harmonic analysis tools that we also use in
the present paper; these include a precise study of the singularities of $\mu$
and product estimates for the operators associated to the various terms in the expansion of $\mu$ 
(see \eqref{intromu'}).

While the approach based on the use of the distorted Fourier transform 
to study nonlinear oscillations in $d=3$ \cite{GHW} had not been fully developed prior to \cite{PS} and the present work, 
it was successfully used in several $1$d problems in recent years.
We remark that in one spatial dimension the analysis of $\mu$ is simpler
since the generalized eigenfunctions \eqref{psieq0} solve an ODE and differ from plane waves $e^{ix\xi}$ 
by a fast decaying term. Of course, in one dimension 
the dispersive properties of waves are weaker compared to the $3$d case. 
Concerning $1$d works that use the dFT in a spirit similar to ours,
we mention the papers on NLS \cite{GPR,ChPu,ChenNLSV} and on KG \cite{GPKG,GPZ}.
For other approaches to similar or related problems,
we refer the reader to the works 
by Delort \cite{DelortNLSV} and Naumkin \cite{NaumkinNLSV} on NLS, 
and by Lindblad-L\"uhrmann-Soffer \cite{LLS} and Lindblad-L\"uhrmann-Schlag-Soffer \cite{LLSS} on KG.
We also mention the major recent advances on the related problem of 
the stability of kinks by Kowalczyk-Martel-Mu\~noz \cite{KowMarMun},
Kowalczyk-Martel-Mu\~noz-Van de Bosch \cite{KMMV}, Delort-Masmoudi \cite{DMKink} and L\"uhrmann-Schlag \cite{LSSineG}.

\medskip
\subsection{The main bootstrap and proof of Theorem \ref{maintheo}}\label{secmtboot}
Our analysis will be carried out for the system for the profiles (of the first order equations) 
given by \eqref{profiles0A}-\eqref{Duhamela0} and \eqref{introf}-\eqref{Duhamelw0}.
In particular, our main bootstrap argument will be based on estimates for 
the amplitude $|A| \approx |\dot{a}| + |a|$, 
and norms of the distorted transform of $f = e^{-itL}w = e^{-itL}(\partial_t + iL)v$; 
see the main bootstrap Propositions \ref{main-ampl} and \ref{propboot}.
The estimates obtained through this argument will imply the 
decay properties in \eqref{mtadecay}-\eqref{mtvdecay} 
and the statements \eqref{mtscattHN} and \eqref{mtgrowth} as explained just after the main theorem.
We explain in detail our bootstrap below, together with the global structure of the proof.


Since $u = a(t) \phi  + v(t)$ with $(v(t),\phi) = 0$,
the first inequality in \eqref{mtdata} implies
\begin{align}\label{mtdatapr1}
|a(0)| + | \dot{a}(0) | \leqslant \e_0,
\end{align}
and the second gives 
\begin{align}\label{mtdatapr2}
\begin{split}
& \Vert \partial_t v(0) + i \langle \nabla \rangle v(0) \Vert_{W^{3,1}} 
  + {\big\| \jx \big( \partial_tv(0) + i \jnab v(0) \big) \big\|}_{H^3} 
  \\
& \leqslant (1 + 2{\| \jx \phi \|}_{H^4} + 2 \Vert \phi \Vert_{W^{4,1}} ) \, \e_0,
\\
& {\big\| \partial_tv(0) + i \jnab v(0) \big\|}_{H^N} \leqslant (1 + 2 {\| \phi \|}_{H^{N+1}}) \e_0 .
\end{split}
\end{align}
To transfer the assumption to $w$ (in Fourier space) we use
the boundedness of wave operators on $H^s$ (see Theorem \ref{Wobd}) which in particular implies,
for all $f \in L^2_{ac}$,
\begin{align}\label{mtdatawo}
{\| \jxi \nabla_\xi \wt{f} \|}_{L^2} + {\| \jxi \wt{f} \|}_{L^2} \lesssim {\| \jx Lf \|}_{L^2}.
\end{align}
Then, from \eqref{mtdatapr1}-\eqref{mtdatapr2} and \eqref{profiles0A} and \eqref{introf},
we have
\begin{align}\label{mtdatapr3}
\begin{split}
|A(0)| \lesssim \e_0, \qquad {\| \jxi^3 \nabla_\xi \wt{w}(0) \|}_{L^2} + {\| w(0) \|}_{H^N}  \lesssim \e_0,
\end{split}
\end{align}
with $w(0) = w(0,x) = f(0,x)$.

We define 
\begin{align}\label{mtrho}
\rho(t) := |A(t)|^2, \qquad \rho_0 := \rho(0) \lesssim \e_0^2. 
\end{align}
Recall that $f = e^{-itL} w$ and let
\begin{align}\label{mth}
\begin{split}
& \wt{h} := \wt{f} + \wt{g},
\\
& \wt{g}(t,\xi) := -\chi_C(\xi)
  \int_{0} ^t B^2(s) e^{-is(\jxi-2\lambda)} ds \, \wt{{\bf P}_c\phi^2}(\xi),
  \qquad \chi_C(\xi):=\varphi_{\leqslant -C}(\jxi - 2 \lambda)
\end{split}
\end{align}
for $C$ sufficiently large so that $2^{-C+10} \leqslant 2\lambda - 1 $,
where $B$ is defined from $A$ through \eqref{defB}; see \eqref{def-h}.
Notice that $\wt{g}(t,\xi)$ can be estimated using only information about $A(t)$.
Therefore, we can use $A$ and $\wt{h}$ as our main unknowns.

We pick a large enough absolute constant $C_0$ and set
\begin{align}\label{mteps}
\e := C_0\e_0, \qquad \beta=\frac{1}{300}, \qquad N = 10^6, \qquad \delta_N := 5/N.
\end{align}
$C_0$ is chosen to be at least larger than the implied constants in \eqref{mtdatapr3}.

We assume that local-in-time solutions for the system \eqref{sysav} 
with $(a,v) \in C^1([0,1]; \mathbb{R}) \times C([0,1];H^{N+1}_x)$ and $v_t \in C([0,1];H^N_x)$,
have been constructed by standard methods.
The following proposition is our main bootstrap. 

\begin{proposition}\label{mtboot}
Let the following a priori assumptions hold for some $T>0:$
\begin{align}\label{bootstrap-A-A0}
\sup_{t \in [0,T]} 
  \left| \big(1+\rho_0 \frac{\Gamma}{\lambda} t \big) \frac{\rho(t)}{\rho_0} -1 \right|
  < \frac{1}{2},
\end{align}
and
\begin{align}\label{boot0}
\begin{split}
\sup_{t \in [0,T]}  \langle t \sqrt{\rho(t)} \rangle^{-1} \rho^{\beta-1/2}(t) {\big\| \nabla_{\xi} \wt{h}(t) \big\|}_{L^2} 
  \leqslant 2\e^{\beta}, \\
\sup_{t \in [0,T]} {\| h(t) \|}_{H^N}  \leqslant 2 \e.
\end{split}
\end{align}
Then, there exists $\bar{\e}>0$ such that if $C_0^{-1}\e = \e_0 \leqslant \bar{\e}$, 
we have the improved estimates
\begin{align}\label{main-amplconc0}
\sup_{t \in [0,T]} \Big| 
  \big(1+\rho_0 \frac{\Gamma}{\lambda} t \big) \frac{\rho(t)}{\rho_0} -1 \Big| < \frac{1}{4},
\end{align}
and
\begin{align}\label{bootconc0}
\begin{split}
\sup_{t \in [0,T]}  \langle t \sqrt{\rho(t)} \rangle^{-1} \rho^{\beta-1/2}(t) {\big\| \nabla_{\xi} \wt{h}(t) \big\|}_{L^2} \leqslant \e^{\beta}, \\
\sup_{t \in [0,T]} {\| h(t) \|}_{H^N}  \leqslant \e.
\end{split}
\end{align}
\end{proposition}

Note that \eqref{bootstrap-A-A0} implies
\begin{align*}
\rho(t) \approx \frac{\e_0^2}{1+\e_0^2 t}, \qquad |a(t)| + |\dot{a}(t)| \approx \frac{\e_0}{\sqrt{1+\e_0^2 t}}.
\end{align*}
Proposition \ref{mtboot} is proved through Proposition \ref{main-ampl} (which proves \eqref{main-amplconc0}) 
and Proposition \ref{propboot} (which proves \eqref{bootconc0}).
Note that the assumptions \eqref{bootstrap-A-w} of Proposition \ref{main-ampl} 
are stated in terms of norms of $w$ rather than $h$ for convenience; 
in Remark \ref{rembootAw} we show how \eqref{boot0} implies \eqref{bootstrap-A-w}.

\begin{remark}\label{remtimew}
The expression for the time-weights in \eqref{boot0} may seem slightly complicated.
This is due to the fact that $\rho(t) \approx \rho_0 (1+ \rho_0 t)^{-1}$
does not decrease until times of the order $\rho_0^{-1} \gtrsim \e_0^{-2}$,
and that we want to use a single bound to bootstrap for all times.
To fix ideas, notice that, in view of \eqref{bootstrap-A-A0}, we have the following bounds:

\begin{itemize}

\medskip
\item For times up to the standard local existence time $T_{loc} = O(\e_0^{-1})$ we have 
$t \sqrt{\rho(t)} \approx \sqrt{\rho_0} t \lesssim 1$ and therefore
\begin{align}\label{boot0'loc}
\begin{split}
& {\big\| \partial_{\xi} \wt{h}(t) \big\|}_{L^2} 
  \lesssim \rho^{1/2-\beta}(t) \e^\beta \lesssim \e_0^{1-2\beta} \e^{\beta} \approx \e_0^{1-\beta}. 
\end{split}
\end{align}

\medskip
\item For intermediate times $t \in [c\e_0^{-1}, C\e_0^{-2}]$ we have
$t\sqrt{\rho(t)} \approx t \sqrt{\rho_0} \gtrsim 1$, and therefore
\begin{align*}
\begin{split}
& {\big\| \partial_{\xi} \wt{h}(t) \big\|}_{L^2} 
 \lesssim t \sqrt{\rho_0} \cdot \rho^{1/2-\beta}(t) 
 \cdot \e^{\beta} \lesssim 
 t \e_0^{2-\beta}.
\end{split}
\end{align*}
Note how the $-\beta$ exponent here gives us room to fit a $\log t$ or 
small polynomial-in-time growth for the norms above, as $t$ approaches $\e_0^{-2}$.

\medskip
\item For large times $t \geqslant C\e_0^{-2}$ we have $\rho(t) \approx t^{-1}$
and the bounds \eqref{boot0} read instead
\begin{align}\label{boot0'}
\begin{split}
& {\big\| \partial_{\xi} \wt{h}(t) \big\|}_{L^2} 
 \lesssim \sqrt{t} \cdot \rho^{1/2-\beta}(t) \cdot \e^{\beta}
 \lesssim t^{\beta} \e_0^\beta.
\end{split}
\end{align}

\end{itemize}

In the course of the proof we will prove bounds taking care of all these time scales, 
and will only a few times simplify our notation by letting $t \gtrsim \e_0^{-2}$
when it is obvious how to take care of smaller times.
Nevertheless, we invite the reader to think that time $t$ is always large enough so that 
one can effectively work with the simpler looking bounds \eqref{boot0'}.
\end{remark}

Once the main bootstrap is proven, a standard continuity argument
provides us with global solutions satisfying, in particular,
\begin{align}\label{main-amplconc0'}
\rho(t) \leqslant \frac{\rho_0}{1+\rho_0 \frac{\Gamma}{\lambda} t }
\end{align}
and
\begin{align}\label{bootconc0'}
{\big\| \partial_{\xi} \wt{h}(t) \big\|}_{L^2} \leqslant \langle t \sqrt{\rho(t)} \rangle \rho^{1/2-\beta}(t) \e^{\beta},
  \qquad {\| h(t) \|}_{H^N} \leqslant \e,
\end{align}
for all $t\in [0,\infty)$. 
Then, \eqref{main-amplconc0'} and the definitions \eqref{mtrho} and \eqref{profiles0A} 
immediately imply \eqref{mtadecay}. 

We complete the proof of Theorem \ref{maintheo} below by showing \eqref{mtvdecay}, 
\eqref{mtSobolev} and \eqref{mtscattHN} (or indicating where they are proven later in the paper). 
The lower bound \eqref{mtgrowth} is proven in \S\ref{ssecgrowth} through the proof of \eqref{introggrowth'}.

Without loss of generality we can restrict to times $t \geqslant 1$.
Recall that $(\partial_t+iL)v = w = e^{itL}f = e^{itL}h-e^{itL}g$, see \eqref{def-h}.

\subsubsection{Proof of \eqref{mtSobolev} and \eqref{mtscattHN}}\label{mtSobolvpr}
Recall from \eqref{mth} that $f=g-h$.
For $h$ we can use the statement of Proposition \ref{propHN} to see that there exists
$h_\infty \in \e H^N_x$ such that, for some small enough $\delta>0$,
\begin{align}\label{hscattmt}
{\| h(t) - h_\infty \|}_{H^N_x} \lesssim \e^{1+\delta} \jt^{-\delta}.
\end{align}
For $g$ we proceed as follows: from the definition \eqref{mth} 
using the boundedness of wave operators (Theorem \ref{Wobd}),
using $|B(s)|^2 \lesssim |A(s)|^2 = \rho(s)$ (see Lemma \ref{renorm-A} and \eqref{defB}), 
and the Strichartz estimate of Lemma \ref{Strichartz} with $(\wt{q},\wt{r},\gamma)=(2,\infty,0)$ 
(see the notation in Subsection \ref{secnotation}), we have, for all $0<t_1<t_2\leqslant T$,
\begin{align} \label{HNg}
\begin{split}
{\| g(t_1) - g(t_2) \|}_{H^N_x}
  & \lesssim {\Big\| 
  \int_0^t \mathbf{1}_{[t_1,t_2]}(s) e^{-isL} \big( L^{N} {\bf P}_c \phi^2 \big) \, | B(s)|^2 \, ds \Big\|}_{L^2_x}
  \\ 
  & \lesssim
  {\big\| {\bf P}_c \phi^2 \big\|}_{W^{N+2,1}_x} \, {\| \rho \|}_{L^2_t([t_1,t_2])}
  \lesssim \e_0^{1-\delta} t_1^{-\delta/2};
\end{split}
\end{align}
for the last inequality we have used $\rho(t) \lesssim \jt^{-1/2-\delta/2}\e_0^{1-\delta}$.
In particular $g$ is Cauchy in $H^N_x$ and 
there exists $g_\infty \in \e^{1-\delta} H^N_x $ such that
\begin{align}\label{gscattmt}
{\| g(t) - g_\infty \|}_{H^N_x} \lesssim \e_0^{1-\delta} \jt^{-\delta/2}.
\end{align}
\eqref{hscattmt} and \eqref{gscattmt} give us \eqref{mtscattHN} with $f_\infty := g_\infty - h_\infty$.
$\hfill \Box$

\medskip
\subsubsection{Proof of \eqref{mtvdecay}}\label{mtvdecaypr}
Note that this $L^\infty_x$ decay estimate is not part of the bootstrap 
and it is just obtained a posteriori.
We start from Duhamel's formula for $w=(\partial_t+iL)v$ (see \eqref{sysav})
which, up to irrelevant constants, is
\begin{align}\label{mtvpr1}
\begin{split}
& w(t) = e^{itL} w(0) + D_1(t) + D_2(t) + D_3(t),
\\
& D_1(t) = e^{itL} \int_{0}^t e^{-isL} \, a^2(s) ds \, \theta, \qquad \theta := {\bf P}_c\phi^2,
\\
& D_2(t) = \int_{0}^t e^{i(t-s)L} \, a(s) {\bf P}_c \Big(\phi \, \Im \frac{w(s)}{L} \Big) ds,
\\
& D_3(t) = \int_{0}^t e^{i(t-s)L} \, {\bf P}_c \Big( \frac{\Im w(s)}{L} \Big)^2 ds.
\end{split}
\end{align}
The bound for the linear solution follows from \eqref{decay6}:
\begin{align}\label{D0}
\begin{split}
{\big\| e^{itL} w(0) \big\|}_{L^\infty} & \lesssim \jt^{-3/2} \Vert w_0 \Vert_{W^{3,1}} \lesssim \e_0 \jt^{-3/2}.
\end{split}
\end{align}

From the linear estimate \eqref{decay} and the bound \eqref{main-amplconc0'} we have
\begin{align}\label{D1}
\begin{split}
{\| D_1(t) \|}_{L^\infty_x} & \lesssim \int_0^t \frac{1}{\langle t-s \rangle^{3/2}}|a(s)|^2 ds 
  {\big\| {\bf P}_c\phi^2 \big\|}_{H^2 \cap W^{3,1}}
  \\
  & \lesssim \int_0^t \frac{1}{\langle t-s \rangle^{3/2}} \frac{\e_0^2}{1+\e_0^2 s} ds 
  \lesssim \jt^{-1}.
\end{split}
\end{align}
Similarly, we can estimate
\begin{align*}
\begin{split}
{\| D_2(t) \|}_{L^\infty_x} & \lesssim \int_0^t \frac{1}{\langle t-s \rangle^{3/2}}|a(s)|
  {\big\| {\bf P}_c(\phi \, L^{-1} \Im w(s) \big\|}_{W^{3,1} \cap H^2} ds 
  \\
  & \lesssim \int_0^t \frac{1}{\langle t-s \rangle^{3/2}} \frac{\e_0}{\sqrt{1+\e_0^2 s}} 
  {\| w(s) \|}_{W^{2,6}} ds.
\end{split}
\end{align*}
The bound 
\begin{align}\label{D2}
{\| D_2(t) \|}_{L^\infty_x} \lesssim \jt^{-5/4},
\end{align} 
would then follow from the estimate
\begin{align}\label{mtvpr5}
{\| w(s) \|}_{W^{2,6}} \lesssim \js^{-1+\delta}. 
\end{align}
A stronger version of \eqref{mtvpr5} is \eqref{mtv=6decay} proven in \S\ref{mtv=6decaypr} below. 
%

An estimate for the last term in \eqref{mtvpr1}, in fact the stronger inequality 
\begin{align}\label{D3}
{\| D_3(t) \|}_{L^\infty_x} \lesssim \frac{\rho^{0+}(t)}{\jt},
\end{align}
is given in \eqref{RFtolater} and proven in Section \ref{SecF} (see Proposition \ref{mainquaddecay}).
The bounds \eqref{D0}-\eqref{D2} and \eqref{D3} give us one side of the approximate equality \eqref{mtvdecay}.

In view of \eqref{mtvpr1}, in order to prove the sharpness of the $L^{\infty}_x$ decay
it will suffice to show that ${\| D_1(t) \|}_{L^\infty_x} \gtrsim \jt^{-1}$, $t \gtrsim \e^{-2}$,
which in turn follows from the stronger inequality
\begin{align}\label{D1lb}
{\big\| \jx^{-5} D_1(t,\cdot) \big\|}_{L^\infty_x} \gtrsim \jt^{-1}.
\end{align}
We start by first expanding $a^2$ as in \eqref{profiles0A}, 
and subsequently integrate by parts in time, to find that
\begin{align}\label{D1split}
\begin{split}
D_1 & = i D_{1,0} -i D_{1,1} -i D_{1,2},
\\
& D_{1,0}(t) := A_0^2 \frac{e^{itL} \theta}{L-2\lambda}  
  + e^{itL} \Big(  2 \vert  A_0 \vert ^2 \frac{\theta}{L} + \overline{A_0}^2\frac{\theta}{L+2\lambda} \Big),
  \\
& D_{1,1}(t) := e^{2i\lambda t} A^2(t) \frac{\theta}{L-2\lambda} 
  + 2\vert A(t) \vert^2 \frac{\theta}{L} + e^{-2i\lambda t} \overline{A(t)}^2 \frac{\theta}{L+2\lambda},
  \\
& D_{1,2}(t) := e^{itL} \int_0 ^t e^{-isL} \bigg(e^{2i\lambda s} \frac{ (\dot{A} A)(s)}{L-2\lambda} + \frac{(\dot{A} \bar{A})(s) 
  + (\dot{\bar{A}}A)(s)}{L} + e^{-2i\lambda s} \frac{(\dot{\bar{A}} \bar{A})(s)}{L+2\lambda} \bigg) ds \, \theta.
\end{split}
\end{align}

It is not hard to see that the first and third terms are lower order.
Indeed, from the equation for $\dot{A}$ in \eqref{Duhamela0} with \eqref{mtvpr5} 
and \eqref{bootstrap-A-A0} we see that $\vert \dot{A} \vert \lesssim s^{-1}$; 
therefore, using the local decay estimates \eqref{localdecay} we get
\begin{align*}
  {\big\| \jx^{-5} D_{1,0}(t) \big\|}_{L^\infty} 
  +
  {\big\| \jx^{-5} D_{1,2}(t) \big\|}_{L^\infty} \lesssim \jt^{-3/2}.
\end{align*}

To deal with the leading order term $D_{1,1}$, 
assume by contradiction that there exists $\zeta \in \mathbb{C}, \vert \zeta \vert = 1$, such that 
\begin{align*}
\zeta^2 \frac{\theta}{L-2\lambda} + 2 \frac{\theta}{L} + \bar{\zeta}^2 \frac{\theta}{L+2 \lambda} = 0.
\end{align*}
Then 
\begin{align*}
\theta =-\zeta^{-2} (L-2\lambda) \bigg[ 2 \frac{\theta}{L} + \bar{\zeta}^2 \frac{\theta}{L+2 \lambda}  \bigg].
\end{align*}
On the Fourier side (see Proposition \ref{propdFT}),
after evaluating at the Fermi frequency $\vert \xi_0 \vert 
:= \sqrt{4 \lambda^2 -1},$ the above implies that $\wt{\theta} (\xi_0) = 0,$
which violates the Fermi Golden rule \eqref{introFGR}.
From this we infer
\begin{align*}
\inf_{\zeta \in \mathbb{C}, \vert \zeta \vert =1} \Bigg \Vert \zeta^2 \frac{\theta}{L-2\lambda} + 2 \frac{\theta}{L} 
+ \bar{\zeta}^2 \frac{\theta}{L+2 \lambda}  \Bigg \Vert_{\langle x \rangle^{5}L^{\infty}} \gtrsim 1.
\end{align*}
Therefore, applying this to $\zeta = e^{i\lambda t}A(t) \rho^{-1/2}$
(recall $\rho(t) := |A(t)|^2$), we have the lower bound
\begin{align*}
\rho^{-1}(t) {\big\| \jx^{-5} D_{1,1}(t) \big\|}_{L^\infty} \gtrsim 1
\end{align*}
and hence obtain \eqref{D1lb}. We have proved \eqref{mtvdecay}. $\hfill \Box$

\smallskip
\subsubsection{Proof of \eqref{mtv=6decay} and \eqref{mtvpr5}}\label{mtv=6decaypr}
It suffices to show the slightly stronger estimate
\begin{align}\label{mtv=6decay'}
{\big\| w(t) \big\|}_{W^{2,6}_x} \lesssim \e_0^\beta \jt^{-1+\delta}, 
\end{align}
for, recall, $w = \partial_t v + i L v = e^{itL}h-e^{itL}g$ and some $\delta > \beta > 0$. 

Using Bernstein's inequality and the bounds \eqref{bootconc0'}
(here $P_k$ is the distorted Littlewood-Paley projection, see \eqref{defLP})
we can estimate
\begin{align}\label{mrvpr10}
\begin{split}
{\| e^{itL}h \|}_{W^{2,6}_x} & \lesssim 
  \sum_{0 \leqslant k \,:\, 2^k \lesssim t^\gamma} 2^{2k} {\| e^{itL}P_k h(t) \|}_{L^6_x} 
    + {\big\| e^{itL}P_{\geqslant t^\gamma}h(t) \big\|}_{H^3} 
  \\
  & \lesssim t^{11\gamma/3} t^{-1} {\| \partial_{\xi} \wt{h}(t) \|}_{L^2_\xi} + t^{-\gamma (N-3)} {\| h(t) \|}_{H^N}
  \\
  & \lesssim t^{-1+\beta+11\gamma/3} \e^\beta + t^{-\gamma (N-3)} \e.
\end{split}
\end{align}
This is consistent with the estimate in \eqref{mtv=6decay'} provided 
$\gamma$ is chosen small enough so that
$11\gamma/3 + \beta \leqslant \delta$ and $N$ is chosen large enough so that $\gamma (N-3) \geqslant 1$.

For the contribution from $g$ we can use the linear estimate \eqref{decay} together with Sobolev's embedding,
and the bound $|B(s)|^2 \lesssim |A(s)|^2 = \rho(s)$, see Lemma \ref{renorm-A} and  \eqref{defB}, 
to see that 
\begin{align}\label{mrvpr11}
\begin{split}
{\| e^{itL}g \|}_{W^{2,6}_x} & = {\Big\| 
  e^{itL} \int_{0}^t e^{-is(L-2\lambda)} \, B^2(s) ds \, {\bf P}_c \phi^2 \Big\|}_{W^{2,6}_x}
  \\
  & \lesssim \int_{0}^t \frac{1}{\langle t - s \rangle} \, \rho(s) ds \, 
  {\big\| \phi^2 \big\|}_{W^{11/3,6/5}_x \cap H^2_x} \lesssim \frac{\log(1+t)}{\jt}.
\end{split}
\end{align}
This proves \eqref{mtv=6decay'} if $t \gtrsim \e_0^{-1}$.
When instead $t \lesssim \e_0^{-1},$ \eqref{mtv=6decay'} holds in view of Sobolev's embedding,
and the a priori bounds, see \eqref{main-amplconc0} and \eqref{bootconc0},
provided we let $\delta \geqslant 2\beta$. $\hfill \Box$


\medskip
\subsection{Structure of the paper}\label{ssecproof}
Section \ref{secdFT} contains some preliminary material concerning the distorted Fourier transform
and linear estimates (pointwise decay and Strichartz).

Section \ref{secD} is dedicated to the discrete mode ODE analysis and the proof of radiation damping.
The main estimate is stated in the bootstrap Proposition \ref{main-ampl}.
To complete its proof we use the estimate \eqref{RFtolater} which is proven later on 
in Sections \ref{SecF}-\ref{secFR} along with weighted estimates.

Section \ref{secSetup} sets up the analysis for the continuous component of the solution $v$.
We split the profile into a `good' and a `bad' component and write out Duhamel's formula for 
the good component $h$; see Lemma \ref{decomposition}.
The main bootstrap estimate on weighted and Sobolev norms of $h$ is stated in Proposition \ref{propboot}.
The proof of this proposition occupies the rest of the paper.

Sections \ref{SecF}, \ref{secFS} and \ref{secFR} constitute the bulk of the proof and deal with the quadratic terms in the continuous component.
Here we prove the main weighted estimate as well as the decay estimates used 
in the ODE analysis in Section \ref{secD}, 
under the assumptions that all frequencies are upperbounded by a small power of the time variable. 
The argument uses a splitting of the NSD $\mu$ (see \eqref{intromu}) 
into a singular and a regular part which is described in \S\ref{SecFmu}-\ref{secFSRsplit}. 
We also give bilinear estimates for the corresponding pieces,
as well as pointwise estimates for the regular part. 
In Section \ref{secFS} we treat the terms corresponding to the singular part of the measure. 
In Section \ref{secFR} we handle the terms that come from the regular part. 
In both of these sections we distinguish the different types of interactions 
between the good and bad parts of the solution, since they require different strategies.

In Section \ref{secmixed} we estimate the mixed-type nonlinearity which is quadratic 
in the continuous-discrete components. This is not as hard as the quadratic terms that only
involve the continuous components, but it is still non-trivial; one of the main difficulties
is the need to apply normal forms which delocalize the bilinear operators 
and make the combined use of local-decay and Strichartz-type estimates 
less efficient.

In Section \ref{energyest} we deal with the high frequency estimates. 
First we control the Sobolev norm of $h$. 
Since the decay is not integrable we need to resort to Strichartz estimates here,
and exploit the fact that the nonlinearity is strongly semilinear. 
Then we prove a result which controls the interactions when the frequencies involved are 
larger than a small power of the time variable.
 
Appendix \ref{Appmu} and Appendix \ref{Appmu2} 
contain results on the nonlinear spectral distribution, bilinear product estimates 
for the relevant operators appearing in our analysis, and some other supporting material. 
The results there are mostly borrowed from \cite{PS} and \cite{GHW}.


\medskip
\subsection{Notation}\label{secnotation}

We adopt the following notation, most of which are standard.

\smallskip
\noindent
$\langle x \rangle =: \sqrt{1+|x|^2}$.

\smallskip
\noindent
$a \vee b$ denotes the largest number between $a$ and $b$, $\max(a,b)$; 
$a \wedge b$ denotes the smallest number between $a$ and $b$, $\min(a,b)$.

\smallskip
\noindent
We let $a^+ = a \vee 0$ and $a^- = a \wedge 0$.

\smallskip
\noindent
For $x \in \R$ we denote its integer part (the largest integer smaller than $x$) by $\lfloor x \rfloor$.

\smallskip
\noindent
$\Re$ and $\Im$ denote the real and imaginary parts of a complex number.

\smallskip
\noindent
We use $a\lesssim b$ when $a \leqslant Cb $ for some absolute constant $C>0$ independent on $a$ and $b$.
$a \approx b$ means that $a\lesssim b$ and $b\lesssim a$.
When $a$ and $b$ are expressions depending on variables or parameters, the inequalities
are assumed to hold uniformly over these.

\smallskip
\noindent
We denote by $a+$, resp. $a-$, a number $b>a$, resp. $b<a$, that can be chosen arbitrarily close to $a$.


\smallskip
\noindent
We use standard notation for Lebesgue, Sobolev and Besov norms, such as $L^p$, $W^{s,p}$ and $B^{s,p}_q$
with $H^s = W^{s,2}$.

\smallskip
\noindent
We denote by 
\begin{align}
\what{f} = \whF(f):= \frac{1}{(2\pi)^{d/2}} \int_{\R^d} e^{-ix\xi} f(x) \, dx
\end{align}
the standard Fourier transform of $f$.

\smallskip
\noindent
For a kernel $K:(x,y) \in \R^3 \times \R^3 \rightarrow \C$,
we denote its Schur norm by
\begin{align}\label{defSchur}
{\| K \|}_{Sch} := \Big( \sup_{x \in \R^3} \int_{\R^3} |K(x,y)| \, dy \Big)^{1/2}
  \cdot \Big( \sup_{y\in\R^3} \int_{\R^3} |K(x,y)| \, dx \Big)^{1/2}. 
\end{align}
This will always be used in conjunction with Schur's test:
if $(T_Kf)(x) := \int_{\R^3}K(x,y) f(y) \, dy$, then ${\|T_K\|}_{L^2(\R^3)\rightarrow L^2(\R^3)} \leqslant {\| K \|}_{Sch}$.

\smallskip
\noindent
We will often use a dot symbol `$\cdot$' to distinguish the bounds on various quantities 
involved in multilinear estimates; see for example the estimates after \eqref{ggbdpr4} or \eqref{ghbdpr12}.
This is meant to help the reader navigate some of the longer bounds.

\medskip
\noindent
{\it Cutoffs}.
We fix a smooth even cutoff function  $\varphi: \R \to [0,1]$ 
supported in $[-8/5,8/5]$ and equal to $1$ on $[-5/4,5/4]$.
For $k \in \Z$ we define $\varphi_k(x) := \varphi(2^{-k}x) - \varphi(2^{-k+1}x)$, 
so that the family $(\varphi_k)_{k \in\Z}$ forms a partition of unity,
\begin{equation*}
 \sum_{k\in\Z}\varphi_k(\xi)=1, \quad \xi \neq 0.
\end{equation*}
We let
\begin{align}\label{cut0}
\varphi_{I}(x) := \sum_{k \in I \cap \Z}\varphi_k, \quad \text{for any} \quad I \subset \R, \quad
\varphi_{\leqslant a}(x) := \varphi_{(-\infty,a]}(x), \quad \varphi_{> a}(x) = \varphi_{(a,\infty)}(x),
\end{align}
with similar definitions for $\varphi_{< a},\varphi_{\geqslant a}$.
We will also denote $\varphi_{\sim k}$ a generic smooth cutoff function 
that is supported around $|\xi| \approx 2^k$, e.g. $\varphi_{[k-2,k+2]}$ or $\varphi'_k$.


\smallskip
\noindent
We denote by $P_k$, $k\in \Z$, the Littlewood-Paley projections adapted to the distorted Fourier transform:
\begin{equation}\label{defLP}
\wt{P_k f}(\xi) = \varphi_k(\xi) \wt{f}(\xi), \quad \wt{P_{\leqslant k } f}(\xi) 
  = \varphi_{\leqslant k}(\xi) \wt{f}(\xi), \quad \textrm{ etc.}
\end{equation}
We will also sometimes use the notation $f_k$ for $P_kf$.
We will avoid using, as a recurrent notation, the flat analogue of these projections;
they may still appear but only when there is no risk of confusion,
or when they are equivalent to their distorted analogues in view of the boundedness of wave operators
(Theorem \ref{Wobd}).

\smallskip
\noindent
We also define the cutoff functions
\begin{equation}\label{cut1}
\varphi_k^{(k_0)}(\xi) = 
\left\{
\begin{array}{ll}
\varphi_k(\xi) \quad & \mbox{if} \quad k>\lfloor k_0 \rfloor,
\\        
\varphi_{\leqslant  \lfloor k_0 \rfloor }(\xi) \quad & \mbox{if} \quad k=\lfloor k_0 \rfloor,
\end{array}\right. 
\end{equation}
and
\begin{equation}\label{cut2}
\varphi_k^{[k_0,k_1]}(\xi) = 
\left\{
\begin{array}{ll}
\varphi_k(\xi) \quad & \mbox{if} \quad k \in (\lfloor k_0 \rfloor ,\lfloor k_1 \rfloor) \cap \Z,
\\        
\varphi_{\leqslant  \lfloor k_0 \rfloor }(\xi) \quad & \mbox{if} \quad k=\lfloor k_0 \rfloor,
\\
\varphi_{\geqslant \lfloor k_1 \rfloor}(\xi) \quad & \mbox{if} \quad k=\lfloor k_1 \rfloor.
\end{array}\right. 
\end{equation}
Note that the indexes $k_0$ and $k_1$ above do not need to be integers.

\smallskip
\noindent
We will denote by $T$ a positive time, and always work on an interval $[0,T]$
for our bootstrap estimates; see for example Proposition \ref{propboot}.
To decompose the time integrals, such as \eqref{introD1}, for any $t \in [0,T]$,
we will use a suitable decomposition of the indicator function $\mathbf{1}_{[0,t]}$
by fixing functions $\tau_0,\tau_1,\cdots, \tau_{L+1}: \R \to [0,1]$, 
for an integer $L$ with $|L-\log_2 (t+2)| < 2$,
with the properties that 
\begin{align}\label{timedecomp}
\begin{split}
& \sum_{n=0}^{L+1}\tau_n(s) = {\bf{1}}_{[0,t]}(s),  
\qquad \supp (\tau_0) \subset [0,2], \quad  \supp (\tau_{L+1}) \subset[t-2,t],
\\
& \mbox{and} \quad \supp(\tau_n) \subseteq [2^{n-1},2^{n+1}], 
  \quad |\tau_n'(t)|\lesssim 2^{-n}, \quad \mbox{for} \quad n= 1,\dots, L.
\end{split}
\end{align}

\smallskip
\subsection*{Acknowledgements}
T.L. was supported by the Simons collaborative grant on weak turbulence.  
F.P. was supported in part by a start-up grant from the University of Toronto, 
and NSERC grant RGPIN-2018-06487. 

\noindent
We thank J. L\"uhrmann  and A. Soffer for useful comments on an earlier version of this paper.


\medskip
\section{Preliminaries}\label{secdFT}
In this section we gather several tools from the linear theory. 
In Subsection \ref{ssecdFT} we state basic properties of the distorted Fourier transform, 
and introduce the wave operators. We state some of their boundedness properties on Sobolev and weighted spaces.
In Subsection \ref{ssecdecay} we state some $L^p_x$ decay estimates for the linear propagator $e^{itL}$,
classical Strichartz estimates, and a useful local decay estimate.

\subsection{Linear theory}\label{ssecdFT}
Let $\psi(x,\xi)$ denote the generalized eigenfunction (or modified complex exponential), 
that is, the solution to 
\begin{align}\label{psieq}
(-\Delta + V ) \psi(x,\xi) = \vert \xi \vert^2 \psi(x,\xi), \qquad \xi,x \in \R^3
\end{align}
with $\vert \psi(x,\xi) - e^{i x \cdot \xi} \vert \longrightarrow 0 $ as $\vert x \vert \rightarrow + \infty.$
We have the canonical decomposition $L^2(\mathbb{R}^d) = L^2_{ac}(\mathbb{R}^d) \oplus L^2 _{pp}(\mathbb{R}^d).$ 
Note that given our assumption on the potential and \eqref{introim},
we have $L^2 _{pp}(\mathbb{R}^d) = \mathrm{span}(\phi).$ 
In this context we can define a modified Fourier transform:

\begin{proposition}\label{propdFT}
Assume that $V = O(\vert x \vert^{-1-}).$ Let $f \in L^2_{ac}(\R^3) \cap \mathcal{S}.$ 
Define the distorted Fourier Transform (dFT) as
\begin{align}\label{Ftildef}
\wt{\mathcal{F}}f(\xi) := \wt{f}(\xi) = \frac{1}{(2\pi)^{3/2}} \lim_{R \to + \infty} 
\int_{\vert x \vert \leqslant R} f(x) \overline{\psi(x,\xi)} dx.
\end{align}
Then, the Plancherel identity holds: for any $f,g \in L^2_{ac}(\R^3) \cap \mathcal{S}$
\begin{align}\label{plancharel}
\langle \wtF f, \wtF g \rangle = \langle f, g \rangle,
\end{align}
and $\wtF$ has an isometric extension to $L^2 _{ac}(\mathbb{R}^d)$ with inverse
\begin{align}\label{Ftilinvdef}
\big[\wt{\mathcal{F}}^{-1}f\big](x) = \frac{1}{(2\pi)^{3/2}} \lim_{R \to + \infty} 
\int_{\vert \xi \vert \leqslant R} f(\xi) \psi(x,\xi) d\xi.
\end{align}

Moreover, the restriction of $\widetilde{\mathcal{F}}$ to $L^2_{ac}(\mathbb{R}^d)$ diagonalizes $-\Delta+V,$ 
in the sense that $\widetilde{\mathcal{F}} (-\Delta+V) \widetilde{\mathcal{F}}^{-1} = \vert \xi \vert^2.$
\end{proposition}

The above result is due several authors, including 
Ikebe \cite{Ikebe}, Alsholm-Schmidt \cite{AlS} and Agmon \cite{Agmon}.
See also \cite[Theorem 2.1]{GHW} and additional references therein. 
We will use the inversion formula \eqref{Ftilinvdef} to write Duhamel's formula for $f$ in distorted Fourier
space as in Subsection \ref{ssecDuh}.

A useful tool in the study of the operator $-\Delta+V$ is the wave operator, 
defined on $\wtF^{-1} L^2_{ac}(\mathbb{R}^d)$ as
\begin{align}\label{Wodef}
\mathcal{W} = \wtF^{-1} \whF, \qquad \mathcal{W}^\ast = \whF^{-1} \wtF,
\end{align}
where $\widehat{\mathcal{F}}$ denotes the flat ($V=0$) Fourier transform. 
The following result follows from the work of Yajima \cite{Y} and Germain-Hani-Walsh \cite{GHW}:

\begin{theorem}[Boundedness of Wave Operators]\label{Wobd}
Assume that $V$ is generic, that for some $\delta>0,$ and for all $\vert \alpha \vert \leqslant l,$
\begin{align}\label{WobdV}
\langle x \rangle^{1+\delta} D^{\alpha} V \in L^2(\mathbb{R}^3), 
  \quad \langle x \rangle^{5 + \delta} D^{\alpha} V \in L^{\infty}(\R^3).
\end{align}
Then, the wave operators \eqref{Wodef} are bounded on $W^{k,p}$ for $0 \leqslant k \leqslant l$
and $1 \leqslant  p  \leqslant \infty$:
\begin{align}\label{WobdSob}
{\big\| \mathcal{W} f \big\|}_{W^{k,p}} + {\big\| \mathcal{W}^\ast f \big\|}_{W^{k,p}} 
  \approx {\| f \|}_{W^{k,p}}.
\end{align}

\medskip
Moreover, under the assumptions \eqref{WobdV} the wave operators are bounded on weighted $L^2$-spaces:
for all $f \in L^2_{ac}$
\begin{align}\label{Wobdx}
{\big\| \jx \mathcal{W} f \big\|}_{L^2} + {\big\| \jx \mathcal{W}^\ast f \big\|}_{L^2} \approx {\| \jx f \|}_{L^2}.
\end{align}
\end{theorem}

\begin{proof}
The results \eqref{WobdSob} is classical and, under the assumptions \eqref{WobdV},
it is due to Yajima \cite[Theorem 1.1]{Y}. 

The proof of \eqref{Wobd} is essentially a direct consequence of Theorem 3.4 in \cite{GHW} which states the following:
for all radial functions $a(x)$ that satisfy $|\nabla a(x)|\leqslant 1$ one has
\begin{align}\label{WoGHW1}
[a(x), \mathcal{W}], \, [a(x), \mathcal{W}^\ast] :  L^p\rightarrow L^q, 
  \qquad 1/p - 1/q \leqslant \min(1/6, 1/p), \quad 1\leqslant p < q \leqslant \infty,
\end{align}
where $[a,b] := ab - ba$. 
Using \eqref{WoGHW1} with $q=2$ and $p=2-$ we see that
\begin{align*}
{\| [\jx, \mathcal{W}] f\|}_{L^2} + {\| [\jx, \mathcal{W}^\ast] f\|}_{L^2}
  \lesssim {\| f \|}_{L^{2-}} \lesssim {\| \jx f \|}_{L^2}.
\end{align*}
In particular, the left-hand side of \eqref{Wobdx} is bounded by the right-hand side
thanks to the boundedness of wave operators on $L^2$.
The equivalence follows by letting $f$ be $\mathcal{W}f$ (or $\mathcal{W}^\ast f$).
\end{proof}

We have the following consequence of Theorem \ref{Wobd}:
\smallskip
\begin{lemma}[Weighted norms and complex conjugate]\label{lemconj}
Assume that \eqref{Wobdx} holds.
Then, for all $f \in L^2_{ac}$
\begin{align}\label{conjconc}
{\big\| \nabla_\xi \wt{\overline{f}} \big\|}_{L^2} + {\| f \|}_{L^2} 
  \approx {\big\| \nabla_\xi \wt{f} \big\|}_{L^2} + {\| f \|}_{L^2},
\end{align}
or, in other words, ${\| \wt{f} \|}_{H^1_\xi} \approx {\| \wt{\overline{f}} \|}_{H^1_\xi}$.
\end{lemma}

\begin{proof}
We first write 
\begin{align}\label{conj1}
\begin{split}
{\big\| \nabla_\xi \wt{\overline{f}} \big\|}_{L^2}
 & = {\big\| \nabla_\xi \whF \mathcal{W}^\ast \overline{f} \big\|}_{L^2} 
\\
 & \lesssim {\big\| x \mathcal{W}^\ast \overline{f} \big\|}_{L^2} 
 \lesssim {\big\| \jx \overline{f} \big\|}_{L^2} = {\big\| \jx f \big\|}_{L^2},
\end{split}
\end{align}
having used the boundedness of wave operators on weighted spaces for the second inequality.
Then, using again \eqref{Wobdx}, we see that
\begin{align}\label{conj2}
\begin{split}
{\big\| \jx f \big\|}_{L^2} & \lesssim {\big\| \jx \whF^{-1} \wt{f} \big\|}_{L^2}
  \lesssim {\big\| \nabla_\xi \wt{f} \big\|}_{L^2} + {\| f \|}_{L^2}.
\end{split}
\end{align}
Putting together \eqref{conj1} and \eqref{conj2} shows one of the inequalities in \eqref{conjconc}.
Exchanging the role of $f$ and $\overline{f}$ proves the equivalence.
\end{proof}

\begin{remark}\label{remconj}
Lemma \ref{lemconj} is useful since our bootstrap assumptions \eqref{boot0} 
involve weighted norms of $\wtF h$ in Fourier space, 
but the nonlinearity will also contain $\wtF(\overline{h})$; see for example \eqref{fDuh}. 
Thanks to Lemma \ref{lemconj} our bootstrap assumptions also control ${\| \nabla_\xi \wtF(\overline{h})\|}_{L^2}$
and, therefore, we may essentially disregard the conjugation.
Note that this simplifies the handling of expressions like those in the last line of \eqref{fDuh};
since the conjugation appears before the distorted transform,
we do not need to conjugate 
the generalized eigenfunctions $\psi(x,\eta)$ and $\psi(x,\sigma)$.
\end{remark}

\medskip
\subsection{Decay, Strichartz estimates, and local decay}\label{ssecdecay}

An important identity related to the Wave operators is the intertwining property
\begin{align}\label{Woid}
a(-\Delta+V) = \mathcal{W} a(-\Delta) \mathcal{W}^{\ast},
  \qquad \wtF a(-\Delta+V)\wtF^{-1} = \whF a(-\Delta) \whF^{-1}. 
\end{align}
Combined with the boundedness of wave operators on $W^{k,p}$,
this formula allows us to bring linear estimates known in the flat case to the perturbed setting.   

We recall first some linear dispersive estimates.

\begin{lemma}[Linear dispersive estimates]  \label{decay}
We have
\begin{align}
\label{decayp}
\Vert e^{itL} f_k \Vert_{L^p _x} & \lesssim \langle 2^{5k/2(1/p'-1/p)} \rangle t^{-3(1/2-1/p)} \Vert f_k \Vert_{L^{p'}_x}, 
\\
\label{decay6}
\Vert e^{itL} f_k \Vert_{L^6_x} & \lesssim \langle 2^{5k/3} \rangle t^{-1} \Vert \partial_{\xi} \mf{f} \Vert_{L^2_\xi},
\end{align}
where $f_k := P_k f.$
\end{lemma}
\begin{proof}
The first inequality is a consequence of the standard $L^{p}$ decay estimate
for the Klein-Gordon propagator, and boundedness of wave operators on $W^{k,p}.$

For the second, we have to show the analog in the flat case, which is slightly less classical than the first. 
Following \cite{KMV}, we write
\begin{align*}
\Vert e^{it L} f_k \Vert_{L^6_x} & \lesssim \Vert e^{itL} f_k \Vert_{L^{6,2}_x} 
                                 \lesssim \langle 2^{5k/3} \rangle t^{-1} \Vert f_k \Vert_{L^{6/5,2}_x} \\
                                & \lesssim \langle 2^{5k/3} \rangle t^{-1} \Vert \vert x \vert^{-1} \Vert_{L^{3,\infty}_x} \Vert \vert x \vert f \Vert_{L^2_x} \\
                                & \lesssim \langle 2^{5k/3} \rangle t^{-1} \Vert \partial_{\xi} \widehat{\mathcal{F}}f \Vert_{L^2_{\xi}}.
\end{align*}
We used successively the following facts about Lorentz spaces:
\begin{itemize}
\item $L^{p,p}_x = L^p_x$ for $1\leqslant p < \infty, a<b.$
\item $L^{p,a}_x \hookrightarrow L^{p,b}_x$ for $1\leqslant p < \infty, a<b.$
\item $\Vert fg \Vert_{L^{p,a}_x} \lesssim \Vert f \Vert_{L^{q,b}_{ x }} \Vert g \Vert_{L^{r,c}_{ x }} $ where $1 \leqslant p,q,r < \infty$ and $1\leqslant a,b,c \leqslant \infty$ satisfy $\frac{1}{p}=\frac{1}{q}+\frac{1}{r}$ and $\frac{1}{a} = \frac{1}{b} + \frac{1}{c}.$ 
\end{itemize}
\end{proof}
We also record Strichartz estimates.
\begin{lemma}[Strichartz estimates] \label{Strichartz}
Let $u$ be a solution to the inhomogenenous equation
\begin{align*}
\partial_{t}^2 u + L^2 u = F, \ \ u(0)=u_0, \ \ \partial_t u(0) = u_1
\end{align*}
Let $0 \leqslant \gamma \leqslant 1,$ $2 \leqslant q, \widetilde{q} \leqslant \infty,$ and $2 \leqslant r, \widetilde{r} < \infty $ be exponents satisfying the scaling and admissibility conditions
\begin{align*}
\frac{1}{q} + \frac{d}{r}= \frac{d}{2}-\gamma = \frac{1}{\widetilde{q}'} + \frac{d}{\widetilde{r}'}-2, \ \ and \ \ \frac{1}{q} + \frac{d-1}{2r}, \frac{1}{\widetilde{q}} + \frac{d-1}{2 \widetilde{r}} \leqslant \frac{d-1}{4}.
\end{align*}
Then 
\begin{align*}
& \Vert \langle \nabla \rangle^{\gamma} u \Vert_{C_t L^2_x} + \Vert \langle \nabla \rangle^{\gamma-1} \partial_t u \Vert_{C_t L^2_x} 
  + \Vert u \Vert_{L^q_t L^r_x} 
\\
& \lesssim \Vert \langle \nabla \rangle^{\gamma} u_0 \Vert_{L^2_x} + \Vert \langle \nabla \rangle^{\gamma-1} u_1 \Vert_{L^2_x} + \Vert F \Vert_{L^{\widetilde{q}'}_t L^{\widetilde{r}'}_x}.
\end{align*}
\end{lemma}
\begin{proof}
Follows from the analog in the flat case, see for example \cite{KSV}, Lemma 2.1.
\end{proof}

\medskip
Finally, we will also need the following local decay estimate.
\begin{lemma} \label{localdecay}
We have the bound, for $\phi, \zeta \in {\bf P_c} \mathcal{S}$,
\begin{align*}
\bigg \Vert \phi \frac{e^{itL}}{L-2\lambda+i0}  \zeta  \bigg \Vert_{L^1_x} 
  \lesssim t^{-3/2} \Vert \langle x \rangle^4 L^5  \zeta \Vert_{L^2_x} \Vert \langle x \rangle^4 L^5 \phi \Vert_{L^2_x} .
\end{align*}
\end{lemma}

Note that the norms on the right-hand side of the above estimate are, of course, 
  far from optimal, but suffice for our purposes.

\begin{proof}
Introducing a cut-off $\varphi_{\leqslant -C}(\jxi-2\lambda)$ where $2^{-C+10} \leqslant 2\lambda-1,$
we split the above into a singular and a non-singular part. 
Standard dispersive estimates allow us to write that, for the non-singular part,
\begin{align*}
\bigg \Vert \phi (1-\varphi_{\leqslant -C}(L-2\lambda)) \frac{e^{itL}}{L-2\lambda+i0}  \zeta \bigg \Vert_{L^1_x} & \leqslant \Vert \phi \Vert_{L^1_x} \bigg \Vert (1-\varphi_{\leqslant -C}(L-2\lambda)) \frac{e^{itL}}{L-2\lambda+i0} \zeta \bigg \Vert_{L^{\infty}} \\
& \lesssim t^{-3/2} \bigg \Vert L^3 \frac{\varphi_{> -C}(L-2\lambda)}{L-2\lambda}  \zeta \bigg \Vert_{L^1_x}.
\end{align*}
For the more challenging singular part, we introduce a regularization of the above term. By the limiting absorption principle, it is sufficient to obtain estimates that are uniform in $\eta$ to conclude. We write, integrating by parts in $\xi$, that
\begin{align*}
&\varphi_{\leqslant -C}(L-2\lambda)\frac{e^{itL}}{L-2\lambda+i\eta} \zeta = \int_{\R^3} 
\psi(x,\xi) \frac{e^{it\langle \xi \rangle}}{\langle \xi \rangle - 2 \lambda+i\eta} \varphi_{\leqslant -C}(\jxi-2\lambda) 
\widetilde{ \zeta}(\xi) d\xi 
\\
& = -i e^{2it\lambda} \int_{t}^{\infty}  \int_{\R^3} \psi(x,\xi) e^{i\tau(\langle \xi \rangle -2 \lambda)-\tau \eta} 
\varphi_{\leqslant -C}(\jxi-2\lambda) \widetilde{ \zeta}(\xi) d\xi d\tau 
\\
& = i   e^{2it\lambda} \int_{t}^{\infty}  \int_{\R^3} \frac{e^{i\tau(\langle \xi \rangle -2 \lambda)-\tau \eta}}{(i \tau)^3} 
  \mathcal{L}^3 \bigg( \psi(x,\xi) \varphi_{\leqslant -C}(\jxi-2\lambda) \widetilde{\zeta}(\xi)  \bigg) d\xi d\tau,
\end{align*}
where the operator $\mathcal{L}$ is defined by its action on $\mathcal{S}$ as 
\begin{align*}
\mathcal{L}f:= {\rm{div}} \bigg( \xi \frac{\jxi}{\vert \xi \vert^2} f(\xi) \bigg).
\end{align*}
From Lemma 3.2 in \cite{PS},
\begin{align*}
\vert \nabla_{\xi}^{A} \psi(x,\xi) \vert 
  \lesssim (\vert \xi \vert / \langle \xi \rangle)^{1-\vert A \vert} \langle x \rangle^{\vert A \vert}, \qquad A \in \mathbb{N}.
\end{align*}
Moreover
\begin{align*}
\Bigg \vert \nabla_{\xi}^A \bigg( \big( \frac{\langle \xi \rangle}{\vert \xi \vert} \big)^A \varphi_{\leqslant -C}(\jxi-2\lambda)\bigg) \Bigg \vert \lesssim 1,
\end{align*}
hence we can conclude that for any $\eta>0,$
\begin{align*}
\Bigg \Vert \phi \frac{\varphi_{\leqslant -C}(L-2\lambda)}{L-2\lambda+i\eta} e^{itL} \zeta
  \Bigg \Vert_{L^1_x} 
  \lesssim \Vert \langle x \rangle^4 \phi \Vert_{L^2_x} \Bigg \Vert \langle x \rangle^{-4} 
  \varphi_{\leqslant -C}(L-2\lambda)\frac{e^{itL}}{L-2\lambda+i\eta}  \zeta \Bigg \Vert_{L^{2}_x} 
  \\
  \lesssim t^{-2} \Vert \langle x \rangle^4 \zeta \Vert_{L^2_x} .
\end{align*}
The desired result follows.
\end{proof}



\medskip
\section{Analysis for the discrete component}\label{secD}

In this section we analyze the ODE for the amplitude of the discrete component, and extract its driving behavior. 

\subsection{Set-up}
The computations in this section are analogous to those in \cite{SWmain}.
Anticipating that the discrete component will oscillate at frequency $\lambda,$ we start by writing that
\begin{align}\label{aAprof}
a(t) = A(t) e^{i\lambda t} + \overline{A}(t) e^{-i\lambda t},
\qquad A(t) = \frac{1}{2i\lambda} e^{-it\lambda} (\dot{a} + i\lambda a(t))
\end{align}
imposing that $A$ satisfies
\begin{align} \label{cancel-A}
\dot{A}(t) e^{i\lambda t} + \dot{\overline{A}}(t) e^{-i\lambda t} =0.
\end{align}
This yields the following equation for the profile of the amplitude of the discrete component:
\begin{align}\label{amplitude}
\dot{A}(t) = \frac{1}{2i\lambda} e^{-i\lambda t} \big((a \phi + \frac{1}{L} \Im w)^2 ,\phi\big).
\end{align}
To extract the leading order dynamics from \eqref{amplitude}
we first need to look at the equation for $w$, and identify the leading order terms there.

Duhamel's formula for the profile $f$ of the radiation $w$, see \eqref{Duhamelw0}, yields 
\begin{align}
\notag w(t) & = e^{iLt} w_0 
\\
\label{linear-res}
& + \int_0 ^t e^{i(t-s)L} \Big( a^2(s) {\bf{P_c}} \phi^2 \Big) ds 
\\
\label{quadratic}
&+2 \int_0 ^t a(s) e^{i(t-s)L} {\bf{P_c}}\big(  \frac{\phi}{L} \frac{w-\overline{w}}{2i} \big) ds 
+ \int_0 ^t e^{i(t-s) L}{\bf{P_c}} \bigg( \frac{\Im w}{L} \bigg)^2 ds. 
\end{align}
We further inspect \eqref{linear-res}.  Denoting $\theta:={\bf P_c} \phi^2,$
\begin{align}\label{main-linear-res} 
\eqref{linear-res} 
  = e^{itL} \int_0^t e^{-is(L-2\lambda)} A^2(s) \theta ds +  e^{itL} R_w, 
\end{align}
where 
\begin{align}\label{Rw}
R_w 
  := 2 \int_0^t e^{-isL} 
  \,\vert A \vert^2(s) \, ds \, \theta 
  + \int_0^t e^{-is(L+2\lambda)} 
  \,\overline{A}^2(s) \, ds \, \theta .
\end{align}
After introducing a regularization of the above term ($\eta>0$) and integrating by parts in time we obtain 
\begin{align*}
\eqref{main-linear-res} = \lim_{\eta\rightarrow 0^+}\Big[ 
  i\frac{e^{2i\lambda t} e^{-\eta t}}{L-2\lambda+i\eta} \, \theta \, A^2(t)
 - i e^{itL} \frac{A_0^2}{L-2\lambda+i\eta}\theta 
 \\
 -i e^{itL} \int_0 ^t \frac{e^{-is(L-2\lambda)}e^{-s\eta}}{L-2\lambda+i\eta}  2 A(s) \dot{A}(s) \theta ds \Big]
 + e^{itL} R_{w}.
\end{align*}
We then use 
\begin{align*}
\frac{1}{x \pm i0^+} = \lim_{\eta \rightarrow 0^+} \frac{1}{x \pm i\eta} = \pv \frac{1}{x} \mp i \pi \delta,
\end{align*}
and go back to \eqref{amplitude}, extract the dominant term in the right-hand side, and write
\begin{align}\label{amplitude-bis}
& \dot{A}(t) := P + \frac{1}{2i\lambda} (\Lambda - i \Gamma ) \vert A \vert^2 A + R ,
\\
\notag 
& P:= \frac{1}{2i\lambda} e^{-i\lambda t} (A e^{i\lambda t} + \overline{A} e^{-i\lambda t})^2 \int \phi^3, 
\end{align}
where we define the real constants 
\begin{align*}
\Lambda & := \Bigg(\theta; \frac{1}{L} \pv \frac{1}{L-2\lambda} \theta \Bigg) , \qquad
\Gamma:= \frac{\pi}{2 \lambda} \Bigg(\theta; \delta(L-2\lambda) \theta \Bigg) , 
\end{align*}
and the remainder term
\begin{align}\label{Drem}
R & := R_D + R_F,
\end{align}
with
\begin{align}
\label{DremD}
& R_D(t) := \frac{1}{2i\lambda} \bigg( e^{2i\lambda t} (\Lambda-i\Gamma) A^3-(\Lambda +i\Gamma) 
  \big( e^{-4i\lambda t} \overline{A}^3 +e^{-2i\lambda t} \vert A \vert^2 \overline{A} \big) \bigg),
\end{align}
and
\begin{align}\label{DremF}
\begin{split}
R_F(t) := \frac{1}{2i\lambda} e^{-i\lambda t}  
  & \Bigg[ \int_{\R^3} \big(\frac{1}{L} \Im w \big)^2 \phi 
  \\
  & + 2a(t) \int_{\R^3} \phi^2 \frac{1}{L} \,\Im \bigg( - ie^{itL} \frac{A_0^2}{L-2\lambda + i0^+} \theta 
  \\
  & - i e^{itL} \int_0^t \frac{e^{-is(L-2\lambda +i0^+)} 
  }{L-2\lambda +i0^+} 2 A(s) \dot{A}(s) \, \theta \, ds \bigg)
  \\
  & + 2a(t) \int_{\R^3} \phi^2 \, \frac{1}{L} \,\Im \big( e^{itL} R_w
 + e^{itL} w_0 + \eqref{quadratic} \big) dx \Bigg] .
\end{split}
\end{align}


\smallskip
\subsection{Behavior of the modulus of the discrete component}
In this subsection we show the expected behavior $|a(t)|\approx \e (1+\e^2t)^{-1/2}$. 
The analysis differs from \cite{SWmain} or \cite{AS} in that more refined estimates are needed, including 
some that we postpone to Section \ref{SecF} where we will analyze more in details the 
interactions that are quadratic in the radiation component. Indeed, unlike in \cite{SWmain}, a bootstrap argument does not provide enough decay, and we have to rely on new ideas.

We start by multiplying \eqref{amplitude-bis} by $2 \overline{A}$ and taking the real part of the resulting equation
to obtain
\begin{align}\label{eqmodA0}
\frac{d}{dt} \big( \vert A \vert^2 \big) = 2 \Re (P \overline{A}) - \frac{\Gamma}{\lambda} \vert A \vert^4 
  + 2 \Re (R \overline{A}).
\end{align}
Let us denote 
\begin{align}
\rho(t) := |A(t)|^2, \qquad \rho_0:=\rho(0) \approx \e_0^2.
\end{align}
Then we have 
\begin{align} \label{eq-short-time}
-\frac{\rho'}{\rho^2} = -2\rho^{-2} \Re (P \overline{A}) + \frac{\Gamma}{\lambda} -2 \rho^{-2} \Re (R \overline{A}),
\end{align}
and after integrating, we obtain
\begin{align} \label{main-eq-rho}
\frac{1}{\rho(t)} - \frac{1}{\rho_0} = \frac{\Gamma}{\lambda} t - 2 \int_0 ^t \rho^{-2}(s) \Re (P \overline{A}(s)) + \rho^{-2}(s) \Re (R\overline{A}(s)) ds.
\end{align}
We will prove that the last two terms on the right-hand side are perturbations of the first term.
In particular the behavior of $\rho$ will be proven to be a perturbation of 
\begin{align}\label{rhoexp}
\frac{\rho_0}{1+\frac{\Gamma}{\lambda}\rho_0 t},
\end{align}
which is the solution of the above equation if the perturbative terms are set to 0.

Here is the key bootstrap proposition that gives us control on the amplitude of the discrete component:

\begin{proposition}\label{main-ampl}
With $\rho = |A|^2$, assume that for $T>0$
\begin{align} \label{bootstrap-A-A}
\sup_{t \in [0,T]} \bigg \vert \big(1+\rho_0 \frac{\Gamma}{\lambda} t \big) \frac{\rho(t)}{\rho_0} -1 \bigg \vert < \frac{1}{2},
\end{align}
and
\begin{align}\label{bootstrap-A-w}
\begin{split}
\sup_{t \in [0,T]} \rho^{\frac{2\beta-1}{2}}(t) \langle t \sqrt{\rho(t)} \rangle^{-1} 
 \jt^{1-10 \delta_N} \big \Vert w(t) \Vert_{L^6_x} 
  \lesssim \varepsilon^{\beta}, 
  \\
  \sup_{t \in [0,T]}   \Vert w(t) \Vert_{H^N_x} \lesssim \e^{1-\delta},
\end{split}
\end{align}
for $N \gg 1$, $\delta_N$ as in \eqref{mteps}, 
$0< \beta \ll 1$, and $\delta$ sufficiently small.
Then, there exists $\varepsilon_0>0$ such that if $\varepsilon,\rho_0 < \varepsilon_0,$ then
\begin{align}\label{main-amplconc}
\sup_{t \in [0,T]} \bigg \vert \big(1+\rho_0 \frac{\Gamma}{\lambda} t \big) \frac{\rho(t)}{\rho_0} -1 \bigg \vert 
< \frac{1}{4}.
\end{align}
\end{proposition}

\begin{remark}\label{remwa}
Note that the quantity in \eqref{rhoexp} is approximately $\rho_0 = |A(0)|^2$ when $t \lesssim 1/\rho_0$, 
and it is $1/t$ if $t \gtrsim 1/\rho_0$. 

For simplicity, the reader is invited to think that time is large enough (past the `local' time-scale $1/\rho_0$)
so that $\rho(t) \approx \jt^{-1}$, and \eqref{bootstrap-A-w} implies that
for all $t \in [0,T]$
\begin{align*}
{\| w(t) \|}_{L^6_x} & \lesssim \jt^{-1+\beta + 10\delta_N} \varepsilon^{\beta}. 
\end{align*}
In particular, one should always think that the continuous/radiation component
is smaller (in the $L^6$ norm) than the amplitude of the discrete component:
${\| w(t) \|}_{L^6} = o(\rho^{1/2}(t))$.
\end{remark}

\begin{remark}\label{rembootAw}
Note that the assumptions \eqref{bootstrap-A-w} are stated in terms of $w$ rather than
in terms of $h$ as is done in the main bootstrap in Proposition \ref{mtboot}.
However it is easy to see how \eqref{bootstrap-A-w} follow from \eqref{boot0}.  
For the $L^6$ estimate the proof is identical to that of \eqref{mtv=6decay} in \S\ref{mtv=6decaypr}. 
For the $H^N$ estimate, recall that $w = e^{itL} f = e^{itL} g + e^{itL} h$; 
then we can conclude in view of \eqref{HNg}. 
\end{remark}

We start by proving some estimates on remainder terms $R_F$ in \eqref{DremF}; 
this will allow us to restrict our attention to terms involving only the discrete component in the equation 
\eqref{amplitude-bis} for the proof of Proposition \ref{main-ampl}.
We naturally have that a term is a remainder for the leading order cubic dynamics in \eqref{amplitude-bis} 
essentially if it is $o(\rho^{3/2})$.
Correspondingly, when looking at \eqref{main-eq-rho} we aim to say that the lower order terms on the right-hand side 
are $o(\rho^{-1})$.

\begin{lemma} \label{remain-ampl}
Under the assumptions \eqref{bootstrap-A-A} and \eqref{bootstrap-A-w}, 
there exists $1/2<a<1$ such that we have for all $s \in [0,T]$
\begin{align}\label{remain-amplbd}
\big| R_F(s) \big| \lesssim \rho^{3/2-a}(s)\js^{-1}. 
\end{align}

\end{lemma}

\begin{proof}
We may assume $s\geqslant 1$, since for $s<1$ all estimates are simpler to establish.
Using the a priori estimate \eqref{bootstrap-A-w} we can bound the first term in \eqref{DremF} by
\begin{align*}
\int_{\R^3} \Big( \frac{\Im w(s) }{L} \Big)^2 \phi \ dx 
  & \lesssim {\| w(s) \|}_{L^6_x}^2 \Vert \phi \Vert_{L^{3/2}_x} 
  \lesssim \rho(s)^{2-2\beta} s^{20 \delta_N} \varepsilon^{2 \beta}.
\end{align*}
For the second term in \eqref{DremF}, 
we use the local decay estimate from Lemma \ref{localdecay}:
\begin{align*}
& |a(s)| \Big| \int_{\R^3} \phi^2 \, \frac{1}{L} \Re \Big( e^{isL} \frac{A_0^2}{L-2\lambda+i0^+} \theta \Big) \, dx\Big|
  \lesssim \rho^{1/2}(s) 
  \cdot \rho_0 \, \js^{-3/2},
\end{align*}
which is enough.\\
To estimate the third term in \eqref{DremF}, we first observe that 
\begin{align}\label{Adotest}
\begin{split}
| \dot{A}(t) | & \lesssim |a^2(t)| + |a(t)| {\big\|\phi^2 L^{-1} \Im w(t) \big\|}_{L^1} 
  + {\big\|\phi \big(L^{-1} \Im w(t)\big)^2 \big\|}_{L^1}
  \lesssim \rho(t),
\end{split}
\end{align}
having used \eqref{bootstrap-A-A}-\eqref{bootstrap-A-w}; see also Remark \ref{remwa}.
Then, using the local decay estimate from Lemma \ref{localdecay} we write
\begin{align*}
& |a(s)| \Big| \int_{\R^3} \frac{\phi^2}{L} \Big( e^{isL} \int_0^s \frac{e^{-i\tau(L-2\lambda)} 
  }{L-2\lambda+i0^+} 
  \dot{A}(\tau) A(\tau) d\tau \, \theta \Big) \, dx \Big|
\\
& \lesssim \rho^{1/2}(s) 
  \int_0^s \Bigg \Vert \phi \frac{e^{-i(s-\tau)(L-2\lambda)}e^{-\eta \tau}}{L-2\lambda+i0^+} 
  \theta \Bigg \Vert_{L^1_x} \rho^{3/2}(\tau) d\tau 
\\
&  \lesssim \rho^{1/2}(s) \int_0^s \langle s-\tau \rangle^{-3/2} \rho^{3/2}(\tau) \, d\tau
\lesssim \rho^{2}(s).
\end{align*}
The terms involving $R_w$ are easier to estimate than those just treated since 
we can integrate by parts in $s$ in the formula \eqref{Rw} without introducing any singular factor of $L-2\lambda$;
this gives quadratic boundary terms which contribute $O(\rho^{3/2})$ to $R_F$,
and cubic terms integrated in time which, upon using the linear estimate of Lemma \ref{decay},
produce a contribution to $R_F$ that is $O(\rho^2)$.

For the term involving the initial data, we write 
\begin{align*}
\vert a(t) \vert \bigg \vert \int_{\mathbb{R}^3} \phi^2 \frac{1}{L} \Im\big( e^{itL} w_0 \big) dx \bigg \vert 
  \lesssim \rho^{1/2}(t) \Vert e^{itL} w_0 \Vert_{L^{\infty}} \lesssim \rho^{1/2} (t) \jt^{-3/2} \varepsilon_0,
\end{align*}
which is sufficient to conclude.

We move on to more difficult terms in $R_F$ involving \eqref{quadratic} and begin by estimating 
\begin{align}
 \notag 
 & |a(s)| \Big| \int_{\R^3}  \frac{\phi^2}{L} \eqref{quadratic} dx \Big|
 \\
 & \lesssim \rho^{1/2}(s) \Bigg \Vert  \int_0 ^s e^{i(s-\tau)L} 
 a(\tau) {\bf P_c} \big( \phi \frac{w-\overline{w}}{2iL} \big) d\tau \Bigg \Vert_{L^{\infty}_x} 
 + \rho^{1/2}(s) \Bigg \Vert \int_0^s e^{i(s-\tau)L} 
 {\bf P_c} \bigg( \frac{\Im w}{L} \bigg)^2 d\tau \Bigg \Vert_{L^{\infty}_x} .
 \label{RF9}
\end{align}
We may assume $s\geqslant 2$, as the case of $s<2$ is easy to handle.
For the first term above, using the dispersive $L^1$-$L^\infty$ estimates from Lemma \ref{decay}
and Sobolev's embedding, we find that
\begin{align}
\nonumber
& \rho^{1/2}(s) \Bigg \Vert \int_0^s e^{i(s-\tau)L} a(\tau) {\bf P_c} 
  \big( \phi \frac{w-\overline{w}}{2iL} \big) 
  d\tau \Bigg \Vert_{L^{\infty}_x}  
  \\
\label{RF10'}
& \lesssim \rho^{1/2}(s) \int_0^1 \rho(\tau)^{1/2} \vert s-\tau \vert^{-3/2} {\big\| w (\tau) \big\|}_{H^2_x} d\tau 
 \\ 
\label{RF10}
 & + \rho^{1/2}(s) \int_1^{s-1} \rho(\tau)^{1/2} 
 \vert s-\tau \vert^{-3/2} {\big\| L^{5/2} w (\tau) \big\|}_{L^6_x} d\tau
 \\
\label{RF10''}
& + \rho^{1/2}(s) \int_{s-1}^s \rho(\tau)^{1/2} {\big\| L^{5/2} w (\tau) \big\|}_{L^6_x} d\tau,
\end{align}
where the implicit constant depends on $\| (1-\Delta) \phi\|_{L^{6/5}\cap L^2}$.
The term in \eqref{RF10'} is bounded by 
$\rho^{1/2}(s) \cdot \rho_0^{1/2} \cdot \langle s \rangle^{-3/2} \cdot \varepsilon^{1-\delta}$, 
which is sufficient.
For the remaining terms, we notice that by the Gagliardo-Nirenberg-Sobolev interpolation inequality 
and the a priori bounds \eqref{bootstrap-A-w} we have for $\tau \geqslant 1$
\begin{align}\label{RF11}
\begin{split}
{\big\| L^{5/2} w (\tau) \big\|}_{L^6_x} & \lesssim 
  {\| w(\tau) \|}_{L^6_x}^{1-\alpha} {\| w(\tau) \|}_{H^N_x}^\alpha 
  \\
&  \lesssim \tau^{-1 + \alpha} \tau^{10(1-\alpha)\delta_N} \rho(\tau)^{(1/2-\beta)(1-\alpha)}
\langle \tau \sqrt{\rho(\tau)} \rangle^{1-\alpha} \e^{\beta(1-\alpha)}  \cdot \e^{\alpha(1-\delta)},
\end{split}
\end{align}
with $\alpha = 5/(2N-4)$.
Then, taking $N$ large enough to make $\alpha$ sufficiently small,
and plugging this into \eqref{RF10''} we find that it is bounded by
\begin{align*}
\rho^{1/2}(s) \cdot \rho^{1/2}(s) \cdot s^{-1+\alpha + 10(1-\alpha) \delta_N} \cdot 
  \langle s \sqrt{\rho(s)} \rangle^{1-\alpha} \cdot \rho(s)^{(1/2-\beta)(1-\alpha)} 
  \e^{\beta + \alpha(1-\delta-\beta)},
\end{align*}
which yields the desired bound by the right-hand side of \eqref{remain-ampl}.

Next, we use inequality \eqref{RF11} in \eqref{RF10};
we estimate the integral in three parts, assuming that $s \geqslant \rho_0^{-1/2}$ and leave the 
easier case $s<\rho_0^{-1/2}$ to the reader.
First, we have
\begin{align*}
\rho^{1/2}(s)
& \int_1^{\rho_0^{-1/2}} \rho(\tau)^{1/2 + (1/2-\beta)(1-\alpha)} \vert s-\tau \vert^{-3/2} 
  \langle \tau \rangle^{-1 + \alpha + 10 \delta_N} 
  d\tau \, \e^{\beta + \alpha(1-\delta -\beta)} 
  \\
& \lesssim \rho^{1/2}(s) \rho_0^{1-\beta} \int_1^{\rho_0^{-1/2}} 
  \vert s-\tau \vert^{-3/2} \tau^{-1+ \alpha + 10 \delta_N} d\tau
  \, \varepsilon^{\beta} 
  \\
& \lesssim  \rho^{1/2}(s) \rho_0^{1-\beta} \cdot s^{-3/2} \cdot \rho_0^{-5 \delta_N-\alpha/2} \varepsilon^{\beta}.
\end{align*}

Second, we can also estimate 
\begin{align*}
& \rho^{1/2}(s)
\int_{\rho_0^{-1/2}}^{s/2} \rho(\tau)^{1/2} \vert s-\tau \vert^{-3/2} \tau^{-1+\alpha} 
  \cdot \tau^{10 \delta_N} \rho(\tau)^{(1/2 - \beta)(1-\alpha)} 
  \tau^{1-\alpha} \rho(\tau)^{\frac{1-\alpha}{2}} d\tau \, \e^{\beta + \alpha(1-\delta-\beta)} 
  \\
 & \lesssim s^{-3/2} \rho^{1/2}(s) \int_{\rho_0^{-1/2}}^{s/2} \rho(\tau)^{1/2 + (1-\beta)(1-\alpha)} 
 \tau^{10 \delta_N} d\tau \, \e^{\beta + \alpha(1-\beta)}  
 \\
& \lesssim  s^{-3/2} \rho(s)^{1/2}  \int_{\rho_0^{1/2}}^{s \rho_0/2} \bigg( \frac{\rho_0}{1 + \tau} \bigg)^{1/2 
  + (1-\beta)(1-\alpha)} \rho_0^{-10\delta_N}  \tau^{10 \delta_N} \frac{d\tau}{\rho_0} 
  \, \e^{\beta + \alpha(1-\delta-\beta)} 
  \\
& \lesssim s^{-3/2} \rho(s)^{1/2} \rho_0^{-1/2 + (1-\beta)(1-\alpha)-10\delta_N}
  \, \e^{\beta + \alpha(1-\delta-\beta)},
\end{align*}
which is sufficient to conclude, provided $\alpha,\beta$ and $\delta_N$ are small enough.
For the remaining piece the bound reads 
\begin{align*}
& \rho^{1/2}(s)
\int_{s/2}^{s-1} \rho(\tau)^{1/2} \vert s-\tau \vert^{-3/2} 
  \tau^{-1+\alpha} \cdot \tau^{10 \delta_N} \tau^{1-\alpha} \rho(\tau)^{\frac{1-\alpha}{2}} 
  \cdot \rho(\tau)^{(1/2 - \beta)(1-\alpha)} 
  d\tau \, \e^{\beta + \alpha(1-\delta-\beta)} 
  \\
 & \lesssim \rho^{1/2}(s) \rho(s)^{\frac{1}{2}+\frac{1-\alpha}{2} 
  + (1/2-\beta)(1-\alpha)} s^{10 \delta_N} \int_{s/2}^{s-1} \vert s-\tau \vert^{-3/2} d\tau.
\end{align*}
This is consistent with the desired \eqref{remain-amplbd}.


We are then left with the quadratic term in $w$ in \eqref{RF9}. We observe that
\begin{align*}
|a(s)| \Bigg| \int \frac{\phi^2}{L} 
\int_0 ^s e^{i(s-\tau)L} {\bf P_c} \big(\frac{\Im w}{L} \big)^2 d\tau dx \bigg \vert 
  \lesssim \rho(s)^{1/2}
  \bigg \Vert 
  \int_0^s e^{i(s-\tau)L} {\bf P_c} \big(\frac{\Im w}{L} \big)^2 d\tau \bigg \Vert_{L^{\infty}_x}.
\end{align*}

Therefore, the desired bound will follow from proving that there exists $\frac{1}{2}<a<1$ such that
\begin{align}
\label{RFtolater}
\bigg \Vert \int_0 ^s e^{i(s-\tau)L} {\bf P_c} \big(\frac{\Im w}{L} \big)^2  d\tau \bigg \Vert_{L^{\infty}_x} 
  \lesssim  \frac{\rho(s)^{1-a}}{\langle s \rangle }. 
\end{align}
This estimate will be a consequence of the analysis carried out in Section \ref{SecF}, 
and is postponed until that section.
\end{proof}

We move on to the proof the main result of this section:

\begin{proof}[Proof of Proposition \ref{main-ampl}]

We proceed in three steps.

\medskip
\textit{Step 1: Acceptable growth rate}. 
The result will follow from showing that 
\begin{align} \label{aim-A}
\Big| \Re \int_0^t  \rho^{-2}(s) (P(s) \overline{A}(s))  
  + \rho^{-2}(s) (R(s) \overline{A}(s)) ds \Big| \lesssim \rho^{-a}(t),
\end{align}
where $\frac{1}{2}<a<1$ and the implicit constant only depends on $\phi.$
Indeed, assume this result is proved. Then, we have, 
for some constant $C_1>0$ independent of $\rho_0,$ using \eqref{main-eq-rho} that 
\begin{align*}
\frac{1}{\rho_0} - C_1 \rho_0^{-a} \Big(\displaystyle 1 + \frac{\Gamma \rho_0}{ \lambda} t\Big)^{a} 
  - C_1 + \frac{\Gamma}{\lambda} t \leqslant \frac{1}{\rho(t)} 
  \leqslant \frac{1}{\rho_0} + C_1 \rho_0^{-a} \Big(\displaystyle 1 + \frac{\Gamma \rho_0}{ \lambda} t\Big)^{a} 
  + \frac{\Gamma}{\lambda} t + C_1,
\end{align*}
and we conclude
\begin{align*}
\frac{\rho_0}{1+\frac{\Gamma \rho_0}{2\lambda} t 
+  C_1 \rho_0^{1-a} \big(\displaystyle 1 + \frac{\Gamma \rho_0}{\lambda} t\big)^{a}  + C_1 \rho_0} 
\leqslant \rho(t) \leqslant  \frac{\rho_0}{1+\frac{\Gamma \rho_0}{\lambda} t 
-  C_1 \rho_0^{1-a} \big(\displaystyle 1 + \frac{\Gamma \rho_0}{ \lambda} t\big)^{a}  - C_1 \rho_0}. 
\end{align*}
The desired result \eqref{main-amplconc} follows provided $\rho_0$ is small enough.

\medskip
\textit{Step 2: Lower order terms}. 
We treat the easier terms on the right-hand side of \eqref{aim-A},
which are the ones involving $R$, see \eqref{Drem}.
Lemma \ref{remain-ampl} directly gives acceptable bounds for the $R_F$ part. 
For the discrete part,  $\int_0^t \rho^{-2}(s) \vert \Re(\overline{A}(s) R_D(s)) \vert ds$
can be written as a sum of terms of the form
\begin{align}\label{aim-A-integrals}
\int_0 ^t e^{i \alpha \lambda s} \frac{A_{+}^i A_{-}^j}{\rho^3 (s)} ds, 
  \qquad i+j=6, \quad \alpha \in \mathbb{Z} \setminus \lbrace 0 \rbrace, \quad A_{+}:=A, \quad A_{-}:=\overline{A}.
\end{align}
Therefore it is sufficient to show how to deal with such oscillatory integrals.
Integrating by parts, with $\widetilde{A}$ denoting either $A_+$ or $A_-$, we obtain
\begin{align*}
\int_0 ^t e^{i\alpha\lambda s} \frac{\widetilde{A}^6(s)}{\rho^3(s)} ds 
  = \frac{e^{i\alpha\lambda t}}{i\alpha \lambda} \frac{\widetilde{A}^6(t)}{\rho(t)^3} 
  - \frac{1}{i\alpha \lambda} \frac{\widetilde{A}^6(0)}{\rho^3(0)} 
  - \int_0 ^t \frac{e^{i\alpha \lambda s}}{i \alpha\lambda} \partial_s \bigg( \frac{\widetilde{A}^6}{\rho^3} \bigg) ds.
\end{align*}
Bounding the remainder, using $\widetilde{A}' = O(\rho)$ (from \eqref{Adotest}), 
we see that the integrand is $O(\rho^{1/2}).$
We can conclude since, using \eqref{bootstrap-A-A}, 
\begin{align} \label{lot}
\int_0^t \rho^{1/2}(s) ds & \lesssim 
\frac{1}{\Gamma \sqrt{\rho_0}} \sqrt{1+ \frac{\Gamma \rho_0}{\lambda} t} \lesssim \rho^{-1/2}(t).
\end{align}

\medskip
\textit{Step 3: Normal forms}. 
From \eqref{amplitude-bis} we see that 
the leading order terms on the right-hand side of \eqref{aim-A} are of the form 
\begin{align}
\int_0 ^t e^{is\alpha \lambda} A_{+}^i(s) A_{-}^j(s) \rho(s)^{-3} ds,
  \qquad i+j=5, \quad \alpha \in \mathbb{Z} \setminus \lbrace 0 \rbrace, \quad A_{+}:=A, \quad A_{-}:=\overline{A},
\end{align}
and a crude bound gives a growth at the rate $t^{3/2},$ which is not acceptable.  
To conclude the proof we show that these terms
can be written in the form of integrals treated in Step 2.

We use the definition of $P$ in \eqref{amplitude-bis}, to see that
\begin{align}\label{RPA}
2 \Re(P \overline{A} ) 
  & = -\frac{c_0}{2i\lambda} \big(A^3 e^{3i\lambda s} - \overline{A}^3 e^{-3i\lambda s}\big)
  -\frac{c_0}{2i\lambda} |A|^2 \big(A e^{i \lambda s} - \overline{A}e^{-i\lambda s} \big),
  \qquad c_0 := \int \phi^3,
\end{align}
and write
\begin{align*}
& -2\int_0^t  \rho^{-2} \Re (P \overline{A}) ds 
= c_0 \big( T_1 + T_2 \big),
\\
& T_1 := \frac{1}{2i\lambda}\int_0^t \rho^{-2}(s) \big(A^3 e^{3i\lambda s} 
  - \overline{A}^3 e^{-3i\lambda s}\big) ds, 
\qquad T_2 := 
   \frac{1}{2i\lambda} \int_0^t \rho^{-1}(s) \big(A e^{i \lambda s} - \overline{A}e^{-i\lambda s} \big) ds. 
\end{align*}
Both terms can be dealt with in the same way. 
Integrating by parts, we obtain
\begin{align}\label{T2}
T_2 = 
  \frac{1}{2i\lambda} \Big[\rho^{-1} \frac{A e^{i\lambda s}}{i \lambda} 
  + \rho^{-1} \frac{\overline{A}e^{-i\lambda s}}{i \lambda} \Big]_0^t 
  & - \frac{1}{2i\lambda} \int_0^t \partial_s(\rho^{-1}) 
  \Big( \frac{A e^{i\lambda s}}{i\lambda} + \frac{\overline{A}e^{-i\lambda s}}{i \lambda} \Big) ds
\end{align}
where we used \eqref{cancel-A}.
The boundary terms satisfy the required bound. 
For the other term in \eqref{T2} we use \eqref{eq-short-time} and \eqref{RPA} to calculate
\begin{align*}
& \frac{1}{2 \lambda^2} \int_0^t \partial_s(\rho^{-1}) \big(A e^{i\lambda s} + \overline{A} e^{-i\lambda s} \big) ds 
  \\
  & =
  \frac{-i c_0}{4 \lambda^3} \int_0 ^t \rho^{-2}(s) \big(A e^{i \lambda s} + \overline{A} e^{-i \lambda s} \big) 
  \big( A^3 e^{3i\lambda s} - \overline{A}^3 e^{-3i\lambda s} + \rho^{-1}(s) (A e^{i \lambda s} 
  - \overline{A} e^{-i\lambda s})  \big) ds 
 + L_1(t)  
\\
& = \frac{-ic_0}{4 \lambda^3} \int_0^t 2 \rho^{-1}(s) \big( A^2 e^{2i\lambda s} 
- \overline{A}^2 e^{-2i\lambda s} \big) + \rho^{-2}(s) \big( A^4 e^{4i\lambda s} - \overline{A}^4 e^{-4i\lambda s} \big) ds 
 + L_1(t), 
 \end{align*}
where
\begin{align*}
L_1(t) := \frac{1}{2 \lambda^2} \int_0^t \rho^{-2}(s) a(s) \bigg(\frac{\Gamma}{ \lambda} \rho^2(s) 
  - \Re (R \overline{A}(s)) \bigg) ds.
\end{align*}
All the integrals above are of the form \eqref{aim-A-integrals} as desired
(note that there is a key cancellation, which makes the end result of the acceptable form with non-vanishing oscillations),
and $L_1$ is a lower order term satisfying an acceptable bound, also in view of Step 2. 


We can treat $T_1$ similarly to $T_2$. Integrating by parts we get
\begin{align*}
2i\lambda T_1 & = 
B - T_{21} -T_{22},
\\
B & := \frac{1}{3 i \lambda \rho^2(s)} \big(A^3(s) e^{3i\lambda s} 
  + \overline{A}^3 e^{-3i\lambda s} \big) \Big|_0^t,
\\
T_{21} &:= \frac{1}{3i \lambda} \int_0 ^t \partial_s \big( \rho^{-2} (s) \big) \big( A^3(s) e^{3i\lambda s} 
  + \overline{A}^3 e^{-3i\lambda s} \big) ds,
  \\
T_{22} &:= \frac{1}{i \lambda} \int_0^t \rho^{-2} (s) 
  \big(A'(s) A^2(s) e^{3i\lambda s} + \overline{A}'(s) \overline{A}^2(s) e^{-3i\lambda s} \big) ds.
\end{align*}
Then, using \eqref{eq-short-time}, \eqref{RPA}, and \eqref{amplitude-bis}, we get
(notice again a key cancellation of the terms without oscillations)
\begin{align*}
T_{21} & =- \frac{c_0}{3 \lambda ^2} \Bigg[ \int_0 ^t \rho^{-3} (s) \big( A^6(s) e^{6i \lambda s} 
  - \overline{A}^6 e^{-6i\lambda s} \big) + \rho^{-2}(s) \big( A^4(s) e^{4i \lambda s} 
  - \overline{A}^4 e^{-4i\lambda s} \big) 
  \\ 
& + \rho^{-1}(s) \big( - A^2(s) e^{2i \lambda s} + \overline{A}^2 e^{-2i\lambda s} \big) ds \Bigg]
  + L_2(t), 
\\
T_{22} & = \frac{c_0}{2\lambda^2} 
  \Bigg[ \int_0 ^t \rho^{-2} (s) \big( A^4(s) e^{4i\lambda s} + e^{-4i\lambda s} \overline{A}^4(s)  \big) 
  + 2 \rho^{-1}(s) \big( A^2(s) e^{2i\lambda s} + \overline{A}^2(s) e^{-2i\lambda s} \big)  ds \Bigg]
  \\
& + L_3(t), 
\end{align*}
where 
\begin{align*}
L_2(t) & := -\frac{2}{3i\lambda} \int_0 ^t \rho^{-3}(s) \big(A^3(s) e^{3i\lambda s} + \overline{A}^3(s)e^{-3i\lambda s} \big) 
  \bigg(-\frac{\Gamma}{\lambda} \rho^2(s) +2 \Re (R \overline{A}(s)) \bigg)  ds,   
\\
L_3(t) & := \frac{1}{i\lambda} \int_0^t \rho^{-2}(s) 
  2 \Re \Big[ A^2(s) e^{3i\lambda s}
  \Big( \frac{1}{2i\lambda} (\Lambda - i \Gamma) \vert A \vert^2 A + R \Big) \Big] ds.
\end{align*}
The integrals in $T_{21}$ and $T_{22}$ are of the acceptable form \eqref{aim-A-integrals}.
The integrands in the terms $L_2(t)$ and $L_3(t)$ that involve the part of $R$
coming from $R_D = O(\rho^{3/2})$ are $O(\rho^{1/2})$, so we can conclude given \eqref{lot}.
The parts involving $R_F$ can be seen to satisfy an acceptable estimate by the right-hand side of 
\eqref{aim-A} using the estimate \eqref{remain-amplbd}.
\end{proof}

\smallskip
\subsection{Renormalized equation}
To simplify the dispersive analysis, we renormalize the equation for the amplitude 
$A$ by introducing a new variable which varies more slowly than $A$.
We will denote by $B$ this renormalized profile of the amplitude of the internal mode. 

\begin{lemma} \label{renorm-A}
Let
\begin{align}\label{defB}
B = F(A) :=  A + c_0\Big( \frac{e^{i\lambda t}}{2 \lambda ^2}A^2 
  - \frac{e^{-3i\lambda t}}{6 \lambda^2} \overline{A}^2 - \frac{e^{-i\lambda t}}{2 \lambda^2} |A|^2\Big), 
  \qquad c_0 := \int \phi^3.
\end{align}
Then
\setlength{\leftmargini}{2em}
\begin{itemize}
\medskip
\item[(i)] $F(A)/A \rightarrow 1$ as $\vert A \vert \to 0$;

\medskip
\item[(ii)] $\vert (d/dt)F(A) \vert \lesssim \rho(t)^{3/2} + \rho(t)^{3/2-a} \jt^{-1},$ 

\end{itemize}
where $\frac{1}{2}<a<1$ is as in Lemma \ref{remain-ampl}.

\end{lemma}

\begin{proof}
The first property in the statement clearly holds true by definition.
Integrating \eqref{amplitude-bis} we obtain
\begin{align}\label{renormAeq}
\begin{split}
A(t)-A_0 &= \frac{1}{2i\lambda} \int_{0}^t \bigg( e^{i \lambda s} A^2 
  + e^{-3 i\lambda s} \overline{A}^2 + 2 e^{-i\lambda s} \vert A \vert^2 \bigg) ds \int \phi^3 
  \\
& + \frac{1}{2i\lambda} (\Lambda - i \Gamma) \int_0 ^t \vert A \vert^2 A  ds + \int_0 ^t R ds 
\\
& =: \bigg[-\frac{e^{i \lambda s}}{2 \lambda^2} A^2 + \frac{e^{-3i\lambda s}}{6 \lambda^2} \overline{A}^2
  + \frac{1}{\lambda ^2} e^{-i \lambda s} \vert A \vert^2(s)  \bigg]_0^t \, \int \phi^3
  + \int_0^t R_1(s)ds,
\end{split}
\end{align} 
that is, $F(A)-F(A_0) = \int_0^t R_1(s)ds$,
with
\begin{align}\label{renormAR}
\begin{split}
R_1(s) := \Big( \frac{e^{i \lambda s}}{2 \lambda^2} \frac{d}{ds}A^2(s) 
  - \frac{e^{-3i\lambda s}}{6 \lambda^2} \frac{d}{ds}\overline{A}^2(s) 
  - \frac{1}{\lambda ^2} e^{-i \lambda s} \frac{d}{ds}\vert A \vert^2(s) \Big) \int \phi^3
  \\ 
  + \frac{1}{2i\lambda} (\Lambda - i \Gamma) \vert A \vert^2 A + R (s) .
\end{split}
\end{align}
We can then verify (ii) noticing that 
\begin{align*}
\frac{d}{dt}F(A) = R_1 = R_F + O \big(|A| |\dot{A}|\big) + O(|A|^3),
\end{align*}
where $R_F$ is defined in \eqref{DremF}, and $O(|A|^3)$ also includes $R_D$ in \eqref{DremD}.
The a priori assumption $|A|\lesssim \rho^{1/2}$, the estimate \eqref{Adotest} 
and Lemma \ref{remain-ampl} give the conclusion.
\end{proof}

\smallskip
\subsection{Asymptotics for the internal mode}\label{secasyB}
We now obtain asymptotics for $A$, and then use this to get asymptotics for 
parts of $w$ as well. We proceed in a few steps.

\smallskip
{\it Equation for $B$}.
First, let us derive an equation for the renormalized amplitude $B$ as defined in \eqref{defB}.
In view of \eqref{renormAeq}-\eqref{renormAR}, and \eqref{Drem},
\begin{align}\label{B0}
\begin{split}
\dot{B}(s) =  c_0 \Big( \frac{e^{i \lambda s}}{\lambda^2} A(s)\frac{d}{ds}A(s) 
  - \frac{e^{-3i\lambda s}}{3 \lambda^2} \overline{A}(s)\frac{d}{ds}\overline{A}(s)
  - \frac{2}{\lambda ^2} e^{-i \lambda s} \Re\big( \overline{A}(s)\frac{d}{ds}A(s) \big) \Big)
  \\ 
  + \frac{1}{2i\lambda} (\Lambda - i \Gamma) |A|^2 A + R_D(s) + R_F(s).
\end{split}
\end{align}
Using the equation \eqref{amplitude-bis} we write the above identity more explicitly
by separating the terms that are cubic in $A$ from those that are faster decaying:
\begin{align}\label{B1}
\dot{B} & = C(A) + R_2,
\end{align}
where (recall also \eqref{DremD})
\begin{align}\label{BCA}
\begin{split}
C(A) & := c_0 \Big( \frac{e^{i \lambda s}}{\lambda^2} A(s)P(s)
  - \frac{e^{-3i\lambda s}}{3 \lambda^2} \overline{A}(s) \overline{P}(s)
  - \frac{2}{\lambda ^2} e^{-i \lambda s} \Re \big( \overline{A}(s) P(s) \big) \Big) 
  \\
  & + \frac{1}{2i\lambda} (\Lambda - i \Gamma) \vert A \vert^2 (s) A (s)
  + R_D(s)
\end{split}
\end{align}
and $R_2$ is defined via \eqref{B0}-\eqref{BCA} and \eqref{amplitude-bis}-\eqref{DremF} and satisfies
\begin{align*}
R_2 & = c_0\Big( \frac{e^{i \lambda s}}{\lambda^2} A(s) R_F
  - \frac{e^{-3i\lambda s}}{3 \lambda^2} \overline{A}(s)\overline{R_F}(s)
  - \frac{2}{\lambda ^2} e^{-i \lambda s} \Re \big( \overline{A}(s) R_F \big) \Big) 
  + R_F(s) + O(|A|^4).
\end{align*}
In particular the $O(|A|^4)$ appearing above has an explicit expression that can be calculated, but 
its exact form is unimportant. We will use such a notation for various other remainder terms in the 
rest of the argument below.
In view of \eqref{remain-amplbd}, we have
\begin{align}\label{BR2}
R_2(s) =  O(\rho^{3/2-a}(s)\js^{-1}), \qquad 1/2<a<1.
\end{align}
In what follows, we will consider the right-hand side of \eqref{BR2} as an acceptable remainder.

\smallskip
{\it Cubic Normal Form and renormalization}.
Let us analyze more closely the cubic terms in \eqref{BCA}.
Recall that, see \eqref{amplitude-bis}
\begin{align*}
P & = \frac{c_0}{2i\lambda} \big( e^{i\lambda t} A^2 + 2e^{-i\lambda t}|A|^2 + e^{-3i\lambda t} \overline{A}^2 \big), 
\end{align*}
so that, gathering terms with the same oscillation we can write
\begin{align}\label{CA}
C(A) & = \sum_{\epsss} c_\epsss A_{\eps_1} A_{\eps_2} A_{\eps_3} e^{i(-1+\eps_1+\eps_2+\eps_3)t}, 
  \qquad A_+ := A, \quad A_- := \overline{A}
\end{align}
where $c_\epsss \in \C$ are certain constants.
We now write
\begin{align}\label{Bconst}
\begin{split}
C(A) &= c_1 \vert A \vert^2 A + \sum_{\epsilon_1 + \epsilon_2 + \epsilon_3 \neq 1} 
  c_{\epsss} A_{\epsilon_1} A_{\epsilon_2} A_{\epsilon_3} e^{i(-1 + \epsilon_1 + \epsilon_2 + \epsilon_3) t},
\\
c_1 & := \sum_{\epsilon_1 + \epsilon_2 + \epsilon_3 = 1} c_{\epsss} 
  = \frac{1}{2i\lambda}(\Lambda - i\Gamma) + \frac{8}{3}\frac{c_0^2}{i\lambda^3}.
\end{split}
\end{align}


We then eliminate all the cubic non-resonant terms from the right-hand side of \eqref{BCA},
that is, all oscillating terms in \eqref{CA}.
We let 
\begin{align}\label{BW}
\begin{split}
& W := B - N(A),
\\
& N(A)(t) := \sum_{\epsilon_1 + \epsilon_2 + \epsilon_3 \neq 1} \frac{c_\epsss}{i(-1+\eps_1+\eps_2+\eps_3)} 
  A_{\eps_1} A_{\eps_2} A_{\eps_3} e^{i(-1+\eps_1+\eps_2+\eps_3)t}, 
\end{split}
\end{align}
so that
\begin{align}\label{BWdot}
\begin{split}
\dot{W} & = c_1 |A|^2 A  + R_2 - \sum_{\epsilon_1 + \epsilon_2 + \epsilon_3 \neq 1} \frac{c_\epsss}{i(-1+\eps_1+\eps_2+\eps_3)} 
  \frac{d}{dt}\big( A_{\eps_1} A_{\eps_2} A_{\eps_3}\big)  e^{i(-1+\eps_1+\eps_2+\eps_3)t} 
  \\
  & =  -\frac{\Gamma}{2\lambda} |W|^2 W + ic_2 |W|^2W + R_2 + O(|A|^4), 
  \qquad c_2:= -\frac{\Lambda}{2\lambda} - \frac{8}{3}\frac{c_0^2}{\lambda^3},
\end{split}
\end{align}
\def\cone{-\frac{\Lambda}{2\lambda} - \frac{17}{6}\frac{c_0^2}{2\lambda^3}}
having used $B-A = O(A^2)$, $W-B = O(A^3)$, $\dot{A} = O(A^2)$, and \eqref{Bconst}.

We then gauge away the `rotating' terms $ic_2|W|^2W$ by setting
\begin{align}\label{BZ}
Z(t) := W(t) \exp \Big( -i c_2\int_0^t |W|^2(s) \,ds \Big).
\end{align}
Notice 
\begin{align}\label{modZ}
|Z| = |W| = O(\rho^{1/2}), \qquad |W - B| = O(\rho^{3/2}), \qquad |B - A| = O(\rho).
\end{align}
We then have
\begin{align*}
\begin{split}
& \dot{Z} =  -\frac{\Gamma}{2\lambda} |Z|^2 Z + R_2' + O(|Z|^4),
  \qquad R_2' := \exp \Big( -i c_2\int_0^t |W|^2(s) \,ds \Big) R_2,
\end{split}
\end{align*}
which we rewrite as
\begin{align}\label{B11}
\begin{split}
\dot{Z} =  -\frac{\Gamma}{2\lambda} |Z|^2 Z + O(\rho^{3/2-a}\js^{-1}).
\end{split}
\end{align}

\smallskip
{\it Asymptotic equations}.
To find a limit equation for $Z$ we let 
\begin{align}\label{BZY}
Z(t) = Y(t) e^{i\vartheta(t)}, \qquad Y,\vartheta \in \R.
\end{align}
Observe that \eqref{modZ} gives $Y = O(\rho^{1/2})$ 
(recall that $\rho(t) := |A|^2$) and $\big| Y - \rho^{1/2} | = O(\rho)$;
in particular, we may assume that $Y_0 := Y(0) > 0$.
Taking real and imaginary parts in \eqref{B11}, we find
\begin{align}
\label{BYdot}
\dot{Y} & = -\frac{\Gamma}{2\lambda} Y^3 + O(\rho^{3/2-a}\js^{-1}),
\\
\label{Bphidot}
\dot{\vartheta} & = O(\rho^{1-a}\js^{-1}).
\end{align}
The data for these equations can be found from the initial data for the amplitude $(a(0),\dot{a}(0))$ 
(see \eqref{aAprof}) and the relations \eqref{defB}, \eqref{BW} and \eqref{BZ}:
\begin{align}
\label{BYphi0}
\begin{split}
Y_0 := Y(0) & = \big( F(A(0)) - N(A(0)) \big) e^{-i\vartheta(0)},
\\
\vartheta_0 := \vartheta(0) & = \mathrm{arg} \big( F(A(0)) - N(A(0)) \big).
\end{split}
\end{align}

From \eqref{BYdot} we can write
\begin{align*}
-2\frac{\dot{Y}}{Y^3} & = \frac{\Gamma}{\lambda} + O(\rho^{-a}\js^{-1}),
\end{align*}
and integrating this we obtain
\begin{align}\label{y}
\frac{Y(t)}{y(t)} = \Big( 1 + y^2(t)\int_0^t O\big(\rho(s)^{-a} \js^{-1}\big) ds \Big)^{-1/2},
\qquad y(t) := \frac{Y_0}{\sqrt{1 + \tfrac{\Gamma}{\lambda} t Y^2_0}}.
\end{align}
Notice that 
\begin{align*}
y(t)^2\int_0^t O\big(\rho(s)^{-a} \js^{-1}\big) ds = O (\rho^{1-a}(t)) \, |\log \e|,
\end{align*}
and, therefore,
\begin{align}\label{BY/y}
\Big| \frac{Y(t)}{y(t)} - 1 \Big| = O\big(\rho^{1-a}(t) |\log\e|\big).
\end{align}

Next, since
\begin{align*}
\int_{t_1}^{t_2} \rho(s)^{1-a} \js^{-1} ds = O\big(\rho(t_2)^{1-a} \, |\log \e| \big)
\end{align*}
we see from \eqref{Bphidot} that we can define a limit
\begin{align}\label{Bphiinfty}
\vartheta_\infty := \int_0^\infty \dot{\vartheta}(s)ds, \qquad \vartheta_\infty = O(\e^{2(1-a)}|\log\e|),
\end{align}
and that we have
\begin{align}\label{Bphilim}
| \vartheta(t) - \vartheta_\infty | = O\big( \rho(t)^{1-a} |\log\e|\big).
\end{align}

We now go back to the complex variables $Z$ and $W$.
From \eqref{BZY} and the asymptotics \eqref{BY/y} and \eqref{Bphilim} it follows that, recall \eqref{y},
\begin{align}\label{BZlim}
  \qquad \big| Z(t) - y(t) e^{i \vartheta_\infty} \big| = O\big(\rho^{3/2-a}(t) |\log\e|\big). 
\end{align}
This implies
\begin{align*}
\big| |W(t)|^2 - y(t)^2 \big|  = O\big(\rho^{2-a}(t) |\log\e|\big)
\end{align*}
and therefore
\begin{align*}
\int_{t_1}^{t_2} \big| |W(s)|^2 - y(s)^2 \big| \,  ds = O\big(\rho^{1-a}(t_1) |\log\e|\big)
\end{align*}
so that we can extract a limit
\begin{align}\label{Bpsiinfty}
\begin{split}
& \psi_\infty := \int_0^\infty \big( |W(s)|^2 - y(s)^2 \big) \, ds, \qquad \psi_\infty = O(\e^{2(1-a)}|\log\e|),
\end{split}
\end{align}
and
\begin{align}\label{Bpsilim}
& \Big| \int_0^t \big( |W(s)|^2 - y(s)^2 \big) \, ds - \psi_\infty \Big| 
  \lesssim \rho(t)^{1-a} |\log\e| \lesssim \frac{\e^\delta}{\jt^\delta},
\end{align}
for some small $\delta \in (0,1-a)$.
In particular, calculating the integral of $y^2$ (see \eqref{y}) we see that
\begin{align}\label{Bpsilim'}
& \int_0^t |W(s)|^2 ds =  \frac{\lambda}{\Gamma} \log \Big( 1 + \frac{\Gamma}{\lambda} t Y_0^2 \Big) 
  + \psi_\infty + O\big(\e^\delta \jt^{-\delta}\big).
\end{align}

With \eqref{Bphiinfty} and \eqref{Bpsiinfty} we define the total (constant) phase shift by
\begin{align}\label{BPsi}
\Psi_\infty = \vartheta_\infty + \psi_\infty.
\end{align}
Finally, putting together \eqref{modZ} (which gives us $|A-W| = O(\rho)$),
and the relation between $W$ and $Z$ in \eqref{BZ} with the asymptotics for $Z$ in \eqref{BZlim}
and for the nonlinear phase in \eqref{Bpsilim'} (see also \eqref{BPsi}),
we can conclude that
\begin{align}\label{BAlim}
\Big| A(t) - y(t) \exp\Big( i \Psi_\infty + i\frac{c_2 \lambda}{\Gamma} 
  \log\Big( 1+ \frac{\Gamma}{\lambda} Y_0^2 t\Big)
  \Big)  \Big| = O\Big( \rho^{1/2}(t) \e^{\delta}\jt^{-\delta} \Big),
\end{align}
for some small $\delta \in (0,1-a)$,
where we recall that $y$ is given by \eqref{y} with \eqref{BYphi0}, 
and $c_2 \in \R$ is defined in \eqref{BWdot}.

\smallskip
\subsection{Asymptotic behavior of the Fermi component}\label{ssecgasy}
In this last subsection we derive asymptotics in $L^\infty_\xi$ 
for the bad part of the continuous component of the solution, 
that is, $g$ as defined in \eqref{mth}. 
We then prove 
the growth of the weighted norm \eqref{introggrowth} and hence \eqref{mtgrowth}.

\smallskip
\subsubsection{Pointwise asymptotics for $\wt{g}$}
In view of the definition \eqref{mth}, \eqref{BAlim} and \eqref{modZ},
we define 
\begin{align}\label{defG}
G(t,\xi) := -\int_0^t \exp\Big(is(2\lambda-\jxi)
  + i\frac{2 c_2 \lambda}{\Gamma} \log \Big( 1+ \frac{\Gamma}{\lambda} Y_0^2 s\Big) \Big) 
  y(s)^2 ds 
  \, e^{2i\Psi_\infty} \wt{{\bf P_c} \phi^2}(\xi).
\end{align}

Then, we see that, for $t_1<t_2$, we have
\begin{align}\label{gasy'}
\begin{split}
{\big\| \wt{g}(t_1,\xi) - G(t_1,\xi) - \big( \wt{g}(t_2,\xi) - G(t_2,\xi) \big) \big\|}_{L^\infty_\xi}
  \lesssim \e^\delta t_1^{-\delta}.
\end{split}
\end{align}
In particular $\wt{g}-G$ is Cauchy in time and there exists $g_\infty \in \e^\delta L^\infty_\xi$ such that
\begin{align*}
{\| \wt{g}(t,\xi) - G(t,\xi) - g_\infty \|}_{L^\infty} \lesssim \e^\delta t^{-\delta}, \quad t\geqslant 1.
\end{align*}
or, more explicitly, for all $\xi$,
\begin{align}\label{gasy}
\begin{split}
\wt{g}(t,\xi) 
  = - 
  \int_0^t \exp\Big(is(2\lambda-\jxi)
  + i\frac{2c_2 \lambda}{\Gamma} \log \Big( 1+ \frac{\Gamma}{\lambda} Y_0^2 s\Big) \Big) 
  y(s)^2 ds \, e^{2i\Psi_\infty} \wt{{\bf P_c} \phi^2}(\xi)
  \\
  + g_\infty(\xi) + O_{L^\infty_\xi}(\e^\delta) \jt^{-\delta}.
\end{split}
\end{align}

\subsubsection{Growth of the weighted norm}\label{ssecgrowth}
We now want to prove that the bad component $g$ as defined in \eqref{def-h} has a fast 
growing weighted norm, and the following lower bound holds:
\begin{align}\label{gg0}
{\| \partial_\xi \wt{g}(t) \|}_{L^2_\xi} \gtrsim \sqrt{t}, \qquad t \geqslant C\e^{-2}.
\end{align}
for some large constant $C$, so that $t \geqslant \lambda/(\Gamma Y_0^2)$.

Recall first that, using \eqref{defB}, $B$ satisfies the same asymptotics of $A$ in \eqref{BAlim} with \eqref{y}:
\begin{align}\label{ggb}
\begin{split}
& \big| B(t) - B_\infty(t) \big| \lesssim \rho^{1/2}(t) \frac{\e^\delta}{\jt^\delta},
\\
& B_\infty(t) := y(t) \exp\Big( i \Psi_\infty 
  + i\frac{c_2 \lambda}{\Gamma} \log\Big( 1+ \frac{\Gamma}{\lambda} Y_0^2 t\Big) \Big).
\end{split}
\end{align}

Let $\ell$ be such that $2^\ell = c\jt^{-1}$ with $c < \Gamma/\lambda$ sufficiently small.
Differentiating the formula \eqref{def-h} we then write, 
\begin{align}\label{ggdxi}
\begin{split}
& \varphi_\ell(\jxi-2\lambda) \partial_\xi \wt{g} = G_{1,\ell} + G_{2,\ell} + G_{3,\ell},
\\
& G_{1,\ell}(t,\xi) := \varphi_\ell(\jxi - 2\lambda) \frac{\xi}{i\jxi}
  \int_0^t \, s \, B_\infty^2(s) e^{-is(\jxi-2\lambda)} ds \, \wt{\phi^2}(\xi),
\\
& G_{2,\ell}(t,\xi) := \varphi_\ell(\jxi - 2\lambda)
  \int_0^t \,B^2(s) e^{-is(\jxi-2\lambda)} ds \, \partial_\xi \wt{\phi^2}(\xi),
\\
& G_{3,\ell}(t,\xi) :=	\varphi_\ell(\jxi - 2\lambda) \frac{\xi}{i\jxi}
  \int_0^t \, s \,\big( B^2(s) - B^2_\infty(s) \big) e^{-is(\jxi-2\lambda)} ds \, \wt{\phi^2}(\xi).
\end{split} 
\end{align}
We first notice that last two terms above satisfy the estimates
\begin{align*}
{\| G_{2,\ell}(t) \|}_{L^2} \lesssim 
  \int_0^t \, |B(s)|^2 ds \, {\| \partial_{\xi} \wt{\phi^2}(\xi)\|}_{L^2} \lesssim \log (1+t)
\end{align*}
and, using \eqref{ggb},
\begin{align*}
{\| G_{3,\ell}(t) \|}_{L^2} & \lesssim 
 2^{\ell/2} \int_0^t \, s \rho(s) \frac{\e^\delta}{\js^\delta} ds \, {\| \wt{\phi^2}(\xi) \|}_{L^\infty} 
 \lesssim \e^\delta \jt^{1/2-\delta}.
\end{align*}
Therefore, using also $L^2$-orthogonality, to show main bound \eqref{gg0} it will suffice to prove
\begin{align}\label{gg0main}
{\| G_{1,\ell}(t) \|}_{L^2} & \gtrsim \jt^{1/2}.
\end{align}

\smallskip
\begin{proof}[Proof of \eqref{gg0main}]
Using the explicit expression in \eqref{ggb},
and $\wt{\phi^2}(\xi_0) \neq 0, \vert \xi_0 \vert = \sqrt{4 \lambda^2-1}$ by the FGR \eqref{introFGR}, we write
\begin{align}\label{ggpr1}
\begin{split}
& \Big| G_{1,\ell}(t,\xi) \Big| 
\\
& \gtrsim \varphi_\ell(\jxi-2\lambda)
 \Big| \int_0^t \, \exp\Big(-is(\jxi - 2\lambda) 
 + ic_2 \frac{\lambda}{\Gamma} \log \big( 1 + \frac{\Gamma}{\lambda} s Y^2_0 \big) \Big) 
 \, \frac{sY_0^2}{1 + \frac{\Gamma}{\lambda} s Y^2_0} ds \Big|
 \\
 & \gtrsim \varphi_\ell(\jxi-2\lambda)
 \Big| \int_0^t \, \exp\Big(-is(\jxi - 2\lambda) 
 + ic_2 \frac{\lambda}{\Gamma} \log \big( 1 + \frac{\Gamma}{\lambda} s Y^2_0 \big) \Big) ds \Big|
 \\
 & - \varphi_\ell(\jxi - 2\lambda) O(\log t).
\end{split}
\end{align}
From now on we disregard the lower order $O(\log t)$ term above.
We then change variables to $u := Y_0^{-2} + \tfrac{\Gamma}{\lambda} s$ and denote 
\begin{align*}
X = X(\xi) = -(\jxi-2\lambda) \frac{\lambda}{\Gamma}, 
  \qquad c_3 := c_2 \frac{\lambda}{\Gamma},
  \qquad T := Y_0^{-2} + \frac{\Gamma}{\lambda} t,
\end{align*}
so that
\begin{align}\label{ggpr2}
\Big| G_{1,\ell}(t,\xi) \Big| \gtrsim \varphi_\ell(\jxi-2\lambda)
 \Big| \int_{Y_0^{-2}}^T \, \exp\Big(iuX(\xi) + ic_3 \log u \Big) du \Big| 
\end{align}
where the implicit constant is independent of $Y_0 =O(\e)$.
Changing variable to $v = uX(\xi)$ we get
\begin{align}\label{ggpr3}
\Big| G_{1,\ell}(t,\xi) \Big| \gtrsim \varphi_\ell(\jxi-2\lambda) \, 2^{-\ell} \,
  \Big| \int_{Y_0^{-2}X}^{TX} \, \exp\big(iv + ic_3 \log v \big) dv \Big|.
\end{align}
Finally we look at the integral in \eqref{ggpr3} and write,
for $X\jt \approx XT \in [c/2,2c]$, for small $c$,
\begin{align}\label{ggpr4}
\begin{split}
\int_{Y_0^{-2}X}^{TX} \, \exp\Big(iv + ic_3 \log v \Big) dv 
 & = \int_{Y_0^{-2}X}^{TX} \, e^{ic_3 \log v} dv  
 + \int_{Y_0^{-2}X}^{TX} \, \big[ e^{iv} - 1 \big] e^{ic_3 \log v} dv 
 \\
& \gtrsim TX + O(TX)^2
\end{split}
\end{align}
where the (real and imaginary parts of the) first of the two integrals 
above can be explicitly calculated with the substitution $\log v = w$ for example.
Therefore
\begin{align}\label{ggpr5}
{\big\| G_{1,\ell}(t,\xi) \big\|}_{L^2} \gtrsim 2^{-\ell/2} \, TX \approx \jt^{1/2}
\end{align}
as claimed.
\end{proof}

\medskip
\section{Dispersive analysis setup}\label{secSetup}
In this section we set up the analysis for the PDE in \eqref{sysav}.
In Subsection \ref{ssecDuh} we first write out Duhamel's formula for the profile $f=e^{-itL} w$ on the distorted Fourier side.
We then define a new variable $h$ by subtracting the most singular contribution, which we call $g$, from $f$. 
Lemma \ref{decomposition} rearranges all the terms according to the nature of the interactions
and the way they will be treated in the rest of the paper.
In \ref{ssecboot} we explain our main bootstrap argument for $h$.
We then also give some simple consequences of our bootstrap assumptions
in Lemmas \ref{dispersive-bootstrap} and \ref{decay-der}.

\medskip
\subsection{Duhamel formula}\label{ssecDuh}
We start by writing Duhamel's formula in Fourier variables and in terms of the profile $f(t):=e^{-itL} w(t).$ 
For any function $a$, let us denote
\begin{align*}
a_+ = a, \qquad a_- = \bar{a}.
\end{align*}
With $\psi$ the generalized eigenfunctions from \eqref{psieq}, let us define
\begin{align}
\label{muDuh}
& \mu(\xi,\eta,\sigma) 
  := \frac{1}{(2 \pi)^{9/2}} \int_{\R^3} \overline{\psi(x,\xi)} \psi(x,\eta) \psi(x,\sigma) dx,
\\
\label{nuDuh}
& \nu(\xi,\eta) := \frac{1}{(2 \pi)^3} \int_{\R^3} \overline{\psi(x,\xi)} \psi(x,\eta) \phi(x) dx.
\end{align}

Using the inversion formula from Proposition \ref{propdFT}, and the definitions \eqref{muDuh}-\eqref{nuDuh},
we write Duhamel's formula \eqref{Duhamelw0} in terms of $\mf{f}$:
\begin{align}\label{Duhamelf}
\begin{split}
\mf{f}(t) & = \mf{f_0} + \int_0 ^t e^{-is\jxi} a^2 (s) \mf{{\bf P_c  }\phi^2}(\xi) ds 
\\
& -\frac{i}{2} \int_0^t a(s) e^{-is \jxi} \int_{\R^3} \bigg( \frac{e^{is \jeta} \mf{f}(s,\eta)}{\jeta} 
  - \frac{e^{-is \jeta} \mf{\overline{f}}(s,\eta)}{\jeta} \bigg) \nu(\xi,\eta) d\eta \, ds 
\\
& -\frac{1}{4} \int_0 ^t e^{-is\jxi} \int_{\R^6} \frac{\big(\mf{w}(\eta) - \mf{\overline{w}}(\eta) \big) \big(\mf{w}(\sigma) 
  - \mf{\overline{w}}(\sigma) \big)}{\jeta \langle \sigma \rangle} \mu(\xi,\eta,\sigma) d\eta d\sigma \, ds.
\end{split}
\end{align}

We renormalize the $a$ terms via $a(t) = A(t) e^{i\lambda t} + \overline{A}(t) e^{-i\lambda t}$,
see \eqref{profiles0A}, and keeping in mind \eqref{defB}, and write \eqref{Duhamelf} as
\begin{align}\label{fDuh}
\begin{split}
\mf{f}(t) & = \mf{f_0}  
+ \int_0 ^t e^{-is(\jxi-2\lambda)} B^2 (s) \mf{{\bf P_c  }\phi^2}(\xi) ds 
\\
&  
+ \int_0 ^t e^{-is(\jxi-2\lambda)} (A^2(s)-B^2 (s)) \mf{{\bf P_c  }\phi^2}(\xi) ds 
\\
 & 
  + 2 \int_0 ^t e^{-is\jxi} \vert A \vert^2 (s) \mf{{\bf P_c  }\phi^2}(\xi) ds 
  + \int_0 ^t e^{-is(\jxi+2\lambda)} \overline{A}^2 (s) \mf{{\bf P_c  }\phi^2}(\xi) ds
 \\
& - \frac{i}{2} \int_0^t B(s) \int_{\R^3} e^{-is(\jxi - \jeta - \lambda)} 
  \jeta^{-1} \wt{f}(s,\eta) \nu(\xi,\eta) d\eta \, ds
  \\
& - \frac{i}{2} \int_0^t (A(s)-B(s)) \int_{\R^3} e^{is(\jxi - \jeta - \lambda)} \jeta^{-1} \wt{f}(s,\eta) \nu(\xi,\eta) 
  d\eta ds
\\
& - \frac{i}{2} \int_0^t \overline{A}(s) \int_{\R^3} e^{-is(\jxi - \jeta + \lambda)} 
  \jeta^{-1} \wt{f}(s,\eta) \nu(\xi,\eta) d\eta \, ds 
\\
& + \frac{i}{2} \int_0^t a(s) e^{-is \jxi} \int_{\R^3} \frac{e^{-is \jeta} \mf{\overline{f}}(s,\eta)}{\jeta} \nu(\xi,\eta) 
  d\eta \, ds 
\\
& - \frac{1}{4} \int_0 ^t e^{-is\jxi} \int_{\R^6} \frac{\big(e^{is\jeta}\mf{f}(\eta) - e^{-is\jeta}\mf{\overline{f}}(\eta) \big) 
  \big(e^{is\jsig}\mf{f}(\sigma) - e^{-is\jsig}\mf{\overline{f}}(\sigma) \big)}{\jeta \langle \sigma \rangle} 
  \mu(\xi,\eta,\sigma) d\eta d\sigma \, ds.
\end{split}
\end{align}

We introduce the new variable
\begin{align}\label{def-h}
\begin{split}
& \wt{h} := \wt{f} + \wt{g},
\\
& \wt{g}(t,\xi) := -\chi_C(\xi)
  \int_{0} ^t B^2(s) e^{-is(\jxi-2\lambda)} ds \, \wt{\phi^2}(\xi),
  \qquad \chi_C(\xi):=\varphi_{\leqslant -C}(\jxi - 2 \lambda)
\end{split}
\end{align}
for $C$ sufficiently large so that $2^{-C+10} \leqslant 2\lambda - 1$.

We now rephrase the above computations in terms of $h$ and $g$, 
and organize the various contributions into leading and remainder terms,
and according to the nature of the interactions.


\begin{lemma}\label{decomposition}
We have the decomposition
\begin{align}\label{decompositionh}
\mf{h}(t) & = \mf{f_0} + S + M + F,
\end{align}
where:

\setlength{\leftmargini}{1.5em}
\begin{itemize}

\item The source terms are defined as, for $\theta := {\bf P_c}\phi^2$, 
\begin{subequations}
\label{S}
\begin{align}
\nonumber
S & = S_1 + S_2 + S_3
\\
\label{Scubic}
& S_1 := 
  \int_0 ^t e^{-is(\jxi-2\lambda)} (A^2(s)-B^2 (s)) \, \mf{\theta}(\xi) ds,
\\
\label{S1-chi}
& S_2 := (1-\chi_C(\xi))   
  \int_0 ^t e^{-is(\jxi-2\lambda)} B^2 (s) \, \mf{\theta}(\xi) ds,   
\\
\label{Snr}
& S_3 := \int_0 ^t e^{-is\jxi} \vert A \vert^2 (s) \, \mf{\theta}(\xi) ds 
  + \int_0 ^t e^{-is(\jxi+2\lambda)} \overline{A}^2 (s) \, \mf{\theta}(\xi) ds.
\end{align}
\end{subequations}

\smallskip
\item
The mixed terms are defined as 
\begin{subequations}
\label{M}
\begin{align}
\nonumber
M & = -\frac{i}{2} \big( M_1 - M_2 + RM \big),
\\
\label{M1}
& M_1 := \int_0 ^t B(s) \int_{\R^3} e^{-is(\jxi - \jeta - \lambda)} \jeta^{-1} \wt{h}(s,\eta) \nu(\xi,\eta) d\eta \, ds,
\\ 
\label{M2}
& M_2 := \int_0 ^t B(s) \int_{\R^3} e^{-is(\jxi - \jeta - \lambda)} \jeta^{-1} \wt{g}(s,\eta) \nu(\xi,\eta) d\eta \, ds,
\end{align}
\end{subequations}
where $RM$ denotes remainder terms given by
\begin{subequations}
\label{RM}
\begin{align}
\nonumber
& RM := RM_1 + RM_2 + RM_3,
\\
\label{RM1}
& RM_1 := \int_0 ^t \overline{A(s)} \int_{\R^3} e^{-is(\jxi - \jeta + \lambda)} 
  \jeta^{-1}\wt{f}(s,\eta) \nu(\xi,\eta) d\eta \, ds,
\\
\label{RM2}
& RM_2 := -\int_0^t a(s) \int_{\R^3} 
  e^{-is(\jxi + \jeta)} \jeta^{-1} \overline{\wt{f}}(s,\eta) \nu(\xi,\eta) d\eta \, ds 
\\
\label{RM3}
& RM_3 := \int_0 ^t (A(s)-B(s)) \int_{\R^3} e^{-is(\jxi - \jeta - \lambda)} \jeta^{-1} \wt{f}(s,\eta) \nu(\xi,\eta) d\eta \, ds .
\end{align}
\end{subequations}

\smallskip
\item The field self-interaction terms are defined by 
\begin{align}
\label{Fepssquad}
\begin{split}
& F := -\frac{1}{4}\sum_{\epss \in \lbrace \pm \rbrace} \eps_1\eps_2 F_{\epss}\big(f_{\eps_1},f_{\eps_2}\big), \qquad f_+=f, \,\,\, f_- = \bar{f},
\\
& F_{\epss}(a,b) := \int_0^t \int_{\R^6} \frac{e^{-is \Phi_\epss (\xi,\eta,\sigma)}}{\jeta \jsig} 
    \wt{a}(s,\eta) \wt{b}(s,\sigma) \mu(\xi,\eta,\sigma) d\eta d\sigma {\fp \, 
     ds},  
\end{split}
\end{align}
where the phases are defined by 
\begin{align}\label{phepss}
\Phi_{\epss}(\xi,\eta,\sigma) = \jxi -\eps_1\jeta -\eps_2\jsig.
\end{align}

\begin{proof}
We just need to verify that the terms on the right-hand side \eqref{decompositionh}
equal all the terms on the right-hand side of \eqref{fDuh} plus $\wt{g}$ as defined in \eqref{def-h}.
This is easily checked since the terms in \eqref{S} added to $\wt{g}$ match the first four integrals
on the right-hand side of \eqref{fDuh},
the term $F$ in \eqref{Fepssquad} matches the last integral in \eqref{fDuh},
and the terms \eqref{M} and \eqref{RM} match the remaining four integrals in \eqref{fDuh} (after using $f=h-g$).
\end{proof}

\end{itemize}

\end{lemma}

\medskip
\subsection{The main bootstrap}\label{ssecboot}

Here is our main bootstrap proposition for the radiation term:

\begin{proposition}\label{propboot}
Let $u = a(t)\phi + v(t)$ be a solution of \eqref{sysav},
let $f = e^{-itL}(\partial_tv + iLv)$, 
and recall the definition of $h$ in \eqref{def-h}. 
Let $\rho(t) = |A(t)|^2$,
where $A$ is defined as in \eqref{profiles0A}, and assume that \eqref{bootstrap-A-A} holds, that is, 
for $T>0,$
\begin{align}\label{bootstrap-A-A'}
\sup_{t \in [0,T]} \Big| \big(1+\rho_0 \frac{\Gamma}{\lambda} t \big) \frac{\rho(t)}{\rho_0} - 1 \Big| < \frac{1}{2}.
\end{align}
Then, there exists $\bar{\e}>0$ such that, if $C_0^{-1}\e = \e_0 \leqslant \bar{\e}$, with $C_0 > 1$
(sufficiently large) and
\begin{align} \label{boot}
\begin{split}
\sup_{t \in [0,T]} \langle t \sqrt{\rho(t)} \rangle^{-1} \rho^{\frac{2\beta-1}{2}}(t)
 \Vert \partial_{\xi} \widetilde{h}(t) \Vert_{L^2}
  \leqslant 2\varepsilon^{\beta}, 
  \\
\sup_{t \in [0,T]} \Vert h(t) \Vert_{H^N} \leqslant 2 \varepsilon,
\end{split}  
\end{align}
for some small $\beta>0$, 
then
\begin{align} \label{bootconc}
\begin{split}
\sup_{t \in [0,T]} \langle t \sqrt{\rho(t)}  \rangle^{-1} \rho^{\frac{2\beta-1}{2}}(t)
  \Vert \partial_{\xi} \widetilde{h}(t) \Vert_{L^2}
  \leqslant \varepsilon^{\beta}, 
  \\
\sup_{t \in [0,T]} \Vert h(t) \Vert_{H^N} \leqslant \varepsilon.
\end{split}  
\end{align}
\end{proposition}

Let us now explain the structure of the proof of Proposition \ref{propboot}.

\medskip
\begin{proof}[Proof of Proposition \ref{propboot}] 
First, when the time is smaller than the natural local time of existence $1/\e_0$
it is not hard to show the following local-in-time bootstrap:
assume that $T \leqslant 1/\e_0$, that \eqref{bootstrap-A-A0} holds and that 
\begin{align}\label{locboota}
\sup_{t \in [0,T]} \Vert \partial_{\xi} \wt{h}(t) \Vert_{L^2} \leqslant 2 \e^{1-\beta} , \qquad
  \sup_{t \in [0,T]} \Vert h(t) \Vert_{H^N} \leqslant 2 \e;
\end{align} 
then we have \eqref{main-amplconc} and 
\begin{align}\label{locbootc}
\sup_{t \in [0,T]} \Vert \partial_{\xi} \wt{h}(t) \Vert_{L^2} \leqslant \e^{1-\beta} , \qquad
  \sup_{t \in [0,T]} \Vert h(t) \Vert_{H^N} \leqslant \e.
\end{align} 
The proof of this is much simpler than for large times. 
In particular no precise information about the measure $\mu$ is needed
and the crude Lemma \ref{derM} suffices. See Subsection \ref{SsecH}. 


The rest of the paper is then dedicated to proving the result for the remaining times.
Note that in the regime $t > \varepsilon_0^{-1},$ the bootstrap assumptions \eqref{boot} read  
\begin{align*}
\sup_{t \in [\e_0^{-1},T]} t^{-1} \rho^{\beta-1}(t) \Vert \partial_{\xi} \wt{h} \Vert_{L^2_{\xi}} \leqslant 2 \e^{\beta},  
\qquad \sup_{t \in [\e_0^{-1},T]} \Vert h(t) \Vert_{H^N} \leqslant 2 \e.
\end{align*}

The weighted estimates for the three types of nonlinear terms on the right-hand side of \eqref{decompositionh},
that is, $S$, $M$ and $F$,
are performed, respectively, in \S\ref{source-terms}, Section \ref{secmixed} and Section \ref{SecF}. 
Together with \eqref{mtdatapr3}, 
the main inequalities \eqref{mainsource}, \eqref{mainquadbound} and \eqref{mixedest}, 
show that there exists a constant $C_1 >0$ and some $\delta>0$ such that
\begin{align*}
\Vert \partial_{\xi} \wt{h} \Vert_{L^2} \leqslant \e_0 + C_1 t \rho^{1-\beta + \delta}(t) \e^{\beta} 
  \leqslant  C_1 \bigg(\e_0 + \e_0^{2\delta} \big(t \rho^{1-\beta}(t) \e^{\beta} \big) \bigg).
\end{align*}
Therefore, since $t \rho^{1-\beta}(t) \geqslant (1/2)\e_0^{1-2\beta}$,
(for $\e_0$ small enough) 
it suffices to choose $\bar{\e}$ such that $C_1 \big( (\overline{\e}/C_0)^{\beta} + \overline{\e}^{2\delta} ) < 1,$
to conclude that the first inequality in \eqref{bootconc} holds.

Finally, the bound on the $H^N$ norm in \eqref{bootconc} is a consequence of Proposition \ref{propHN}.
\end{proof}

\begin{remark}\label{localT}
In view of the local bootstrap \eqref{locboota}-\eqref{locbootc}, 
from now on we will always assume that $t \geqslant \e_0^{-1}  \approx \rho_0^{-1/2}.$
Note that past this local time of existence, the first estimate in
\eqref{bootconc}
is implied by 
$\Vert \partial_{\xi} \wt{h} \Vert_{L^2} \leqslant t \rho^{1-\beta}(t) \e^{\beta}$,
and it suffices to show that
there exists $\delta>0$ such that $\Vert \partial_{\xi} \wt{h} \Vert_{L^2} 
  \lesssim t \rho^{1-\beta+\delta}(t) \e^{\beta}$ (since $\rho(t) \lesssim \e_0^2$).
  
\end{remark}


\medskip
\subsection{Preliminary estimates}

\smallskip
\subsubsection{Consequences of the a priori bounds}
We now summarize the main dispersive estimates that are a consequence of the a priori assumption \eqref{boot}.

\begin{lemma}\label{dispersive-bootstrap}
We have
\begin{align}\label{dispbootg}
\begin{split}
& \Vert e^{itL} g(t)  \Vert_{L^6_x}   \lesssim  \rho(t) \ln (1+t) \big( \Vert (1-\Delta) \phi \Vert_{L^{12/5}}^2 
  + \Vert \phi \Vert_{H^3_x}^2  \big),  
\\
& \Vert e^{itL} g(t)  \Vert_{L^3_x} \lesssim \rho(t) \ln (1+t) \langle \sqrt{t} \rangle 
  \big( \Vert (1-\Delta) \phi \Vert_{L^{12/5}}^2 + \Vert \phi \Vert_{H^3_x}^2  \big).
\end{split}
\end{align}
More generally,
\begin{align} \label{dispbootgbis}
\Vert e^{itL} g(t)  \Vert_{L^p_x}  \lesssim \rho(t) t^{1-3(1/2-1/p)} \ln (1+t),
\end{align}
and
\begin{align}\label{dispbootf}
\begin{split}
& \Vert e^{itL} h(t)  \Vert_{L^q_x}, \quad \Vert e^{itL} f(t)  \Vert_{L^q_x} 
	\lesssim \rho(t)^{1-\beta} t^{\frac{5(q-2)}{2q}\delta_N + \frac{6-q}{2q}}  \varepsilon^{\beta},
\end{split}
\end{align}
for $2 \leqslant q \leqslant 6.$
\end{lemma}

\begin{proof}
The proofs are standard. 
For \eqref{dispbootg}, splitting the integral $\int_0 ^t = \int_0^{t-1} + \int_{t-1}^{t}$ we treat the main piece using Lemma \ref{decay}:
\begin{align*}
\Vert e^{itL} g(t)  \Vert_{L^6_x} &\lesssim \int_0 ^{t-1} \frac{1}{t-s} \rho(s) ds \Vert (1-\Delta) \phi \Vert_{12/5}^2 
\\
& \lesssim \bigg[ \frac{1}{t} \int_0 ^{t/2} \rho(s)ds + \rho(t) \int_{t/2}^{t-1} \frac{1}{t-s} ds \bigg] \Vert (1-\Delta) \phi \Vert_{12/5}^2 .
\end{align*}
For the other piece we use Sobolev embedding to obtain the result.  The proof is the same for the second inequality on $g$.

The inequality on $e^{itL} h$ in \eqref{dispbootf} in the case $q=6$
is a consequence of the second decay estimate from Lemma \ref{decay} for frequencies less than $t^{\delta_N}$.
The inequality for the remaining frequencies is obtained using Sobolev's embedding and Bernstein's inequality
with the a priori bound on the Sobolev norm in \eqref{boot}.
The estimate for $f$ is obtained from summing the estimates for $h$ and $g$. 

Finally, to prove the $L^q$ estimates, we use the interpolation inequality
$$\Vert f \Vert_{L^{q}} \lesssim \Vert f \Vert_{L^6}^{\frac{3(q-2)}{2q}} \Vert f \Vert_{L^2}^{\frac{6-q}{2q}},$$ 
and the conclusion follows.
\end{proof}

The above bootstrap assumptions also imply a decay bound on $\Vert \partial_s \wt{f} \Vert_{L^2}:$

\begin{lemma} \label{decay-der}
We have 
\begin{align}\label{estdsh}
\Vert \partial_s h \Vert_{L^2_x} \lesssim \rho(s).
\end{align}
Also,
\begin{align}\label{estdsg}
{\big\| e^{isL} \partial_s g \big\|}_{L^p_x} \lesssim \rho(s), \qquad 1 \leqslant p \leqslant \infty.
\end{align}
\end{lemma}

\begin{proof}
From the definition \eqref{def-h}, and using the equation $\partial_s f = e^{-isL} (\partial_s -iL)w$,
we have
\begin{align}
\begin{split}
\partial_s \wt{h}(s,\xi) & = \partial_s \wt{f} - \chi_C(\xi) B^2(s) e^{-is(\jxi-2\lambda)} \, \wt{\phi^2}(\xi)
 \\
 & = e^{-is\jxi} \wtF \Big( \big(a \phi + \dfrac{1}{L} \Im w \big)^2 \Big) 
 - \chi_C(\xi) B^2(s) e^{-is(\jxi-2\lambda)} \, \wt{\phi^2}(\xi).
\end{split}
\end{align}
Then, from H\"older, the apriori bounds \eqref{boot}
and Lemma \ref{dispersive-bootstrap} we obtain
\begin{align}\label{estdsh0}
\Vert \partial_s \widetilde{h} \Vert_{L^2} \lesssim \rho(s) 
  +\rho(s)^{\frac{3-\beta}{2}} s^{3\delta_N} \varepsilon^{\beta} 
  +\rho(s)^{2-\beta} s^{1/2+3\delta_N}  \varepsilon^{2\beta},
\end{align}
from which \eqref{estdsh} directly follows.

\eqref{estdsg} follows directly from the definition \eqref{def-h} and the a priori bound \eqref{bootstrap-A-A'}
transferred to $B$ via \eqref{defB}.
\end{proof}

\smallskip
\subsubsection{Basic bounds on the Fermi/bad component} 
We collect basic properties satisfied by $g$:
\begin{lemma}
For $s\approx 2^m$, $m=1,2, \dots$, we have
\begin{align} \label{inftyfreqg}
{\Vert \wt{g}(s) \Vert}_{L^{\infty}_{\xi}} \lesssim m \rho(2^m) 2^m,
\end{align}
and, for integers $0 \leqslant n \leqslant N$,
\begin{align}\label{eng}
{\Vert D^n \widehat{\mathcal{F}}^{-1} \wt{g}(s) \Vert}_{L^2} \lesssim  \e^{1-}, 
\qquad 
{\Vert D^n \widehat{\mathcal{F}}^{-1} \partial_\xi \wt{g}(s) \Vert}_{L^2} \lesssim 2^{m/2} . 
\end{align}
Moreover
\begin{align}\label{growth}
{\Vert \partial_{\xi} \wt{g}(s,\xi) \Vert}_{L^2_{\xi }} \lesssim m^2 2^{3m/2} \, \rho(2^m),
\end{align}
and
\begin{align}\label{dxiwtgest}
{\Vert \partial_{\xi} \wt{g}(s,\xi) \Vert}_{L^1_{\xi}} \lesssim 2^m \rho(2^m) m^2 .
\end{align}
Finally
\begin{align}\label{dxi2g}
\sup_{s\approx 2^m} {\| \partial_\xi^2 \wt{g}(s,\xi) \|}_{L^2_{ \xi}} \lesssim \rho(2^m) 2^{3m}.
\end{align}
\end{lemma}

\begin{proof}
The property \eqref{inftyfreqg} is straightforward from the definition of $g$ in \eqref{def-h} and using \eqref{bootstrap-A-A}.

The first inequality in \eqref{eng} follows from \eqref{HNg}, and the second can be proved similarly.

To prove \eqref{growth} we split, for some $A<-C-1$ to be determined later,
\begin{align*}
\partial_{\xi} \wt{g}(s,\xi)  &= I_A(s,\xi) +  \sum_{k \geqslant A+1}^{-C} II_k(s,\xi) + \lbrace \textrm{easier terms} \rbrace, 
\\
I_A(s,\xi) &:=  \frac{\xi}{\jxi} \varphi_{\leqslant A}(\jxi - 2 \lambda)  \int_0^s \tau 
  B^2 (\tau) e^{-i\tau(\jxi - 2 \lambda)} d\tau \wt{\theta} (\xi) 
  \\
II_k(s,\xi) &:= \frac{\xi}{\jxi}  \varphi_{k}(\jxi - 2 \lambda)
  \int_0^s \tau B^2 (\tau) e^{-i\tau(\jxi - 2 \lambda)} d\tau \wt{\theta} (\xi),
\end{align*}
where $\textrm{easier terms}$ refers to cases where the derivative falls 
on either the $\varphi_{\leqslant -C}$ cut-off or on $\wt{\theta}.$
For the first piece we have the estimate $ \Vert I(s,\xi) \Vert_{L^2_{\xi}} \lesssim 2^{A/2+m}.$
For the second we integrate by parts in time and bound the two resulting pieces using Lemma \ref{renorm-A} and \eqref{bootstrap-A-A}:
\begin{align*}
\Vert II_k (s,\xi) \Vert_{L^2_{\xi}}  
  & \lesssim \bigg \Vert \frac{e^{-is(\langle \xi \rangle - 2 \lambda)}}{\langle \xi \rangle -2 \lambda} 
  \frac{\xi}{\langle \xi \rangle} s B^2(s) \widetilde{\theta}(\xi) \varphi_k (\jxi - 2 \lambda) \bigg \Vert_{L^2_{\xi}} 
  \\ 
&+ \bigg \Vert  \frac{\xi}{\langle \xi \rangle} \varphi_k (\jxi - 2 \lambda) 
  \int_0^s \tau B'(\tau) B(\tau) e^{-i\tau(\langle \xi \rangle - 2 \lambda)}  d\tau 
  \frac{\widetilde{\theta}(\xi)}{\langle \xi \rangle - 2 \lambda} \bigg \Vert_{L^2_{\xi}}
  \\
& \lesssim \bigg( \int_{\vert \langle \xi \rangle - 2\lambda \vert \approx 2^k} d\xi \bigg)^{1/2} 
  2^{-k} \bigg[ s \rho(s) + \ln (s \rho_0+1)  \bigg].
\end{align*}
Optimizing, we choose $A = -m,$ and obtain \eqref{growth}.

For the \eqref{dxiwtgest} we have similarly
\begin{align*}
\Vert I_A (s,\xi) \Vert_{L^1_{\xi}} \lesssim 2^{A+m},
  \qquad \Vert II_k (s,\xi) \Vert_{L^1_{\xi}}  \lesssim \big[s \rho(s) + \ln (s\rho_0 + 1) \big] (-A),
\end{align*}
and the result follows with the same choice of $A$ as above.

The last bound is obtained more crudely just by differentiating the definition of $\wt{g}$ in \eqref{def-h}.
\end{proof}

\medskip
\subsubsection{Source terms} \label{source-terms}
We have the following bound on the `good' source terms \eqref{S}:

\begin{lemma} \label{source}
Under the a priori assumptions \eqref{bootstrap-A-A0}, we have, for some $\delta \in (0,\beta/2)$,
\begin{align} \label{mainsource}
 \Vert \partial_{\xi} S  \Vert_{L^2} \lesssim t \rho(t)^{1-\beta+\delta} \e^\beta.
\end{align}
\end{lemma}

\begin{proof}
We start with the easier term $S_2$. 
We localize in time using \eqref{timedecomp} and treat the main contribution
integrating by parts in $s$ and using Lemma \ref{renorm-A}:
\begin{align*}
&\bigg \Vert (1- \chi_{C}(\xi)) \int_0^t \tau_m(s) s \frac{\xi}{\jxi}\, e^{-is(\langle \xi \rangle -2 \lambda)} B^2(s) 
\,  \widetilde{\theta(\xi)}  ds \bigg \Vert_{L^2} 
  \\
& \lesssim \bigg \Vert \tau_m(t) t B^2(t) \frac{\xi}{\langle \xi \rangle} 
  \frac{e^{-it(\langle \xi \rangle - 2 \lambda)}}{\langle \xi \rangle -2 \lambda} 
  (1-\chi_C(\xi)) \, \widetilde{\theta(\xi)} \bigg \Vert_{L^2} 
  \\
&+ \bigg \Vert  \frac{1-\chi_C(\xi)}{\langle \xi \rangle - 2 \lambda} 
  \frac{\xi}{\langle \xi \rangle} \int_0^t e^{-is(\langle \xi \rangle -2\lambda)} 
 \big( 2s B'(s) B(s) \tau_m(s) + s B^2(s) 2^{-m} \tau'_m(s) + B^2(s)\tau_m(s)\big) ds \, \widetilde{\theta(\xi)} \bigg \Vert_{L^2} 
  \\
& \lesssim \int_{s \approx 2^m} 
 \rho(s) ds \lesssim \ln \langle 2^m \rho_0 \rangle.
\end{align*}
Summing over $m$ gives a bound by $\log \langle t \rho_0^2 \rangle$
which is acceptable for the right-hand side of \eqref{mainsource}.
The above integration by parts in time 
is the typical strategy that we can use to treat non-resonant terms with non-vanishing phases. 

For the term $S_1$ in \eqref{Scubic} we can estimate 
\begin{align*}
\Vert \partial_{\xi} S_1 \Vert_{L^2} \lesssim \sum_m \Bigg \Vert \frac{\xi}{\jxi} 
  \int_0^t \tau_m(s) s e^{-is(\jxi-2\lambda)} \big( A^2(s)- B^2(s) \big) \wt{\theta}(\xi) ds \Bigg \Vert_{L^2},
\end{align*} 
having disregarded the term where the derivative falls on $\wt{\theta}$ which is faster decaying
and easy to estimate. 
Now we write $A^2 - B^2 = 2B (A-B) +(A-B)^2.$
Using \eqref{defB}, we see that $(A-B)^2 = O(A^4),$ hence
\begin{align} \label{Seasy}
\Bigg \Vert \frac{\xi}{\jxi} \int_0 ^t s \tau_m(s) e^{-is(\jxi-2\lambda)} \big( A(s) - B(s) \big)^2 
  \wt{\theta}(\xi) ds \Bigg \Vert_{L^2} \lesssim 2^{2m} \rho^2(2^m),
\end{align}
which is sufficient.  For the other term with $B(A-B)$, we also use \eqref{defB} to expand $A-B$ and obtain three terms. 
The cases of the terms $B\overline{A}^2$ and $B \vert A \vert^2$ are non-resonant,
and the associated phase does not vanish 
and can be handled as $S_2$ above. 
We focus on the more difficult term with $BA^2$:
\begin{align*}
\Vert \partial_{\xi} S_1 \Vert_{L^2}  \lesssim \sum_m \Bigg \Vert \frac{\xi}{\jxi} 
\int_0^t s \tau_m(s) e^{-is(\jxi-3\lambda)} B(s) A^2(s) \wt{\theta}(\xi) ds \Bigg \Vert_{L^2}.
\end{align*}
As we did above, we can write $B A^2 = B (A^2 - B^2) + B^3,$ the first term having enough 
decay to be treated as \eqref{Seasy} above. 
For the remaining piece, we can proceed exactly as in the proof of 
\eqref{growth} above, replacing $2 \lambda$ by $3 \lambda$,
and obtain the desired bound.

For the remaining term $S_3$ in \eqref{Snr}, we start by replacing the $A$ terms by $B$ terms. 
The error terms are then of the same type as $S_1,$ but easier to treat since the phases are not resonant. 
The terms that remain are then of the same type as $S_2,$ in the sense that they 
are quadratic in $B$ with a non-resonant phase. Therefore, the same argument applies
and we can skip the details.
\end{proof}

\medskip
\section{Field self-interactions}\label{SecF}
Our main goal is to prove the main weighted bound for the terms that 
are quadratic in the continuous component $f$ of the solution. 
In (distorted) Fourier space this is the term $F$ from \eqref{Fepssquad} 
(or equivalently the last term of \eqref{Duhamelf})
which, recall, is given by
\begin{align}\label{secF00}
\begin{split}
& F = F(t,\xi) 
  = \wtF\Big( \int_0^t e^{-isL} \Big( \frac{\Im w(s)}{L}\Big)^2 ds \Big)(\xi), \qquad w = e^{itL} f.
\\
\end{split}
\end{align}
The main result we want to prove is:

\begin{proposition}\label{mainquad}
Under the a priori assumptions \eqref{bootstrap-A-A'}-\eqref{boot} we have
\begin{align}\label{mainquadbound} 
{\| \partial_{\xi} F(t) \|}_{L^2} \lesssim \e^{\beta} t \rho^{1-\beta + \delta}(t),
\end{align}
for some $\delta \in (0,\beta/2).$
\end{proposition}

Moreover, as mentioned in Section \ref{secD}, we also need to prove
the estimate \eqref{RFtolater}, which was used to obtain Lemma \ref{remain-ampl},
and whose proof was postponed to the current section. 
Expressed in terms of $F$, see \eqref{secF00}, the estimate we need is contained in the following:

\begin{proposition}\label{mainquaddecay}
Under the a priori assumptions \eqref{bootstrap-A-A'}-\eqref{boot}
there exists $\frac{1}{2}<a<1$ such that 
\begin{align}\label{quaddecaybound}
  {\big\| e^{itL}  \wtF^{-1}(F)(t) \big\|}_{L^\infty} \lesssim \frac{\rho_0^{a}}{\jt^{1+a}}.
\end{align}
\end{proposition}

\smallskip
The proof of the main Propositions \ref{mainquad} and \ref{mainquaddecay} occupies all of 
the current section and Sections \ref{secFS} and \ref{secFR}.

\smallskip
\begin{remark}
Note that using the sharp linear estimate \eqref{decay6} (and Bernstein's inequality), 
together with \eqref{mainquadbound}, we could obtain a bound similar to \eqref{quaddecaybound} 
but with a decay rate slightly slower than $\jt^{-1}$.
Unfortunately, this is insufficient for the arguments in Section \ref{secD} where the rate needs to be integrable.
As it turns out, proving \eqref{quaddecaybound} requires additional arguments 
from those in the proof of \eqref{mainquadbound}. 
\end{remark}

\medskip
\subsection{Set-up}\label{secFsetup}
In this set-up section we first explain how to deal with high and (very) low frequencies.
More precisely, we state results that prove \eqref{mainquadbound} and \eqref{quaddecaybound} 
when the input and output frequencies
are larger than a small power of time or smaller than a large negative power of time.
This is convenient as it will reduce the analysis to the case when all 
the input and output frequencies involved are bounded above by a small power of the time variable,
and are lower bounded as well, so that summations over their respective dyadic scales 
can then be done at the expense of innocuous logarithmic factors in time.

At the end of this subsection we discuss the different types of quadratic interactions that we are going
to need to estimate, and the strategy for the proofs of the main Propositions \ref{mainquad} and \ref{mainquaddecay}.

\medskip
\subsubsection{High frequencies}\label{easy-field-quad}
We distinguish between (very) high frequency and mid-low frequencies 
by introducing 
a parameter $M_0 \in \mathbb{N}$ such that $2^{M_0} \approx \js^{\delta_N}$
\begin{align}\label{HLsplit}
M_0 = \lfloor \delta_N \log_2 \js \rfloor , \qquad \delta_N = 5/N, 
\end{align}
and splitting  
$F$ as follows:
\begin{align}\label{FHL}
\begin{split}
& -4 F := F_H + F_L,
\\
& F_L(t,\xi) 
  := -4 \int_0^t \varphi_{\leqslant M_0}(\xi) \, e^{-is\jxi} \wtF \Big( \frac{P_{\leqslant M_0} \Im w(s)}{L}\Big)^2 ds
\end{split}
\end{align}
with the natural definition of $F_H$.
Recall that $P_{\leqslant k}$ denotes the distorted Littlewood-Paley projector
with symbol $\varphi_{\leqslant k} = \sum_{k' \leqslant k} \varphi_{k'}$; see \S\ref{secnotation}.
Using the Fourier inversion (Proposition \ref{propdFT}) as done before, we can write 
the above expression in distorted Fourier space as 
\begin{align}\label{secF0}
\begin{split}
& (2\pi)^{9/2} F_\ast = \sum_{\epss \in \{+,-\}} 
  \eps_1\eps_2 F_{\ast,\epss}\big(f_{\eps_1},f_{\eps_2}\big),
  \qquad f_+=f, \,\, f_- = \bar{f}, \qquad \ast \in \{L,H\},
\end{split}
\end{align}
where
\begin{align}\label{secF0'}
\begin{split}
& F_{\ast,\epss}(a,b) := \int_0^t \int_{\R^6} \frac{e^{-is \Phi_\epss (\xi,\eta,\sigma)}}{\jeta \jsig} 
    \wt{a}(s,\eta) \wt{b}(s,\sigma) \mu_\ast (\xi,\eta,\sigma) d\eta d\sigma \, ds, 
    \\
    & \qquad \Phi_\epss (\xi,\eta,\sigma):= \jxi -\eps_1\jeta -\eps_2\jsig,
\end{split}
\end{align}
with
\begin{align}\label{secF0mu}
\begin{split}
\mu_L(\xi,\eta,\s) := \mu(\xi,\eta,\s) \varphi_{\leqslant M_0}(\jxi)\varphi_{\leqslant M_0}(\jeta)\varphi_{\leqslant M_0}(\jsig),
\qquad \mu_H := \mu - \mu_L,
\end{split}
\end{align}
recall \eqref{muDuh}.
Note that we are suppressing the dependence of $F_{\ast,\epss}(a,b)$ from its independent variables $(t,\xi)$;
we may sometimes indicate this dependence but will usually omit it, as it will be clear from the context.
Similarly, we may omit the dependence on the inputs $(f_{\eps_1},f_{\eps_2})$.
We remark that in what follows the signs $\epss$ will not play an essential role, 
since the phases $\Phi_{\epss}$
will have similar properties (in particular we will show they are suitably lower bounded 
on a certain region of frequency space, see Lemma \ref{lemphases}),
and the conjugation on $f$ (or $h$ or $g$) is irrelevant in view of Lemma \ref{lemconj}.

The estimate for the high frequency part is the following:

\begin{lemma}\label{lemmaH}
Under the a priori assumptions \eqref{bootstrap-A-A'}-\eqref{boot},
there exists $\delta'\in (\delta,\beta/2)$ such that
\begin{align}\label{FH1}
\Vert \partial_{\xi} F_{H}(t) \Vert_{L^2} \lesssim \e^{\beta} t \rho^{1-\beta+\delta'}(t), 
\end{align}
and $a\in(0,1)$ such that
\begin{align}\label{FH2}
  {\big\| e^{itL}  \wtF^{-1}(F_H)(t) \big\|}_{L^\infty} \lesssim \frac{\rho_0^{a}}{\jt^{1+a}}.
\end{align}
\end{lemma}
The proof of this lemma is postponed to Subsection \ref{SsecH}.

\subsubsection{Low frequencies}
In view of Lemma \ref{lemmaH}, from \eqref{FHL}-\eqref{secF0mu},
we see that for \eqref{mainquadbound} it suffices to show that for all $\epss = \{+,-\}$
we have 
\begin{align}\label{mqbeps}
{\big\| \partial_{\xi} F_{L,\epss}(t) \|}_{L^2} \lesssim \e^{\beta} t \rho^{1-\beta+\delta'}(t);
\end{align}
similarly, \eqref{quaddecaybound} would follow from
\begin{align}\label{qdbeps}
{\big\| e^{itL} \wtF^{-1}\left(F_{L,\epss} \right)(t) \big\|}_{L^\infty} \lesssim \frac{\rho_0^{a}}{\jt^{1+a}}.
\end{align}

The next result, whose proof is postponed to Subsection \ref{Ssecverylowfreq},
reduces matters to profiles that are supported at frequencies that are not too small;
more precisely, we have:

\begin{lemma} \label{verylowfreq}
Under the a priori assumptions \eqref{bootstrap-A-A'}-\eqref{boot},
there exists $\delta' \in (0,\beta/2)$ and $a\in(0,1)$ such that
\begin{align*}
&\Vert \partial_{\xi} F_{L,\epss}(A_{\eps_1},B_{\eps_2}) (t) \Vert_{L^2} 
  \lesssim \varepsilon^{\beta} t \rho^{1-\beta+\delta'}(t), 
\\
& {\big\| e^{itL} \wtF^{-1}\left(F_{L,\epss}(A_{\eps_1},B_{\eps_2}) \right)(t) \big\|}_{L^\infty} 
  \lesssim \frac{\rho_0^{a}}{\jt^{1+a}},
\\ 
&
  \mbox{for all} \quad (A,B) \in \big\{ (f_{\leqslant \js^{-5}}, f_{\leqslant \js^{-5}}), 
  \, (f_{\leqslant \js^{-5}} , f_{\geqslant \js^{-5}}),
  \, (f_{\geqslant \js^{-5}} , f_{\leqslant \js^{-5}})  \big\}.
\end{align*}
Here $f_{\ast} := P_{\ast}f$, where $P$ is the distorted Littlewood-Paley projection (see \S\ref{secnotation}).
\end{lemma}

As a consequence of Lemma \ref{verylowfreq} the profiles we consider from now on 
carry a localization to frequencies larger than $\js^{-5}.$ 
Furthermore, since
\begin{align*}
{\big\| \partial_{\xi} \big( \wt{P_{\geqslant \js^{-5}} f } \big) \big\|}_{L^2} \lesssim 
  {\| \partial_{\xi} \wt{f} \|}_{L^2} ,
\end{align*}
which is obtained as a direct consequence of Hardy's inequality, 
we can safely suppress the notation $P_{\geqslant \js^{-5}}$ on the profiles,
and still denote by $F_{L,\epss}(t)$ the bilinear operator as in \eqref{secF0'}-\eqref{secF0mu}
with this additional restriction to $|\eta|,|\s| \gtrsim \js^{-5}.$

\subsubsection{Types of quadratic interactions and organization of the proof}\label{types}
Since $f=h-g$, we need to estimate the quadratic terms
\begin{align}\label{FsecF}
\begin{split}
F_{L,\epss}(a,b) 
    \qquad \mbox{with} \qquad a \in\{h_{\eps_1}, g_{\eps_1}\}, \,\, b \in\{h_{\eps_2}, g_{\eps_2}\}.
\end{split}
\end{align}
We then have three main kinds of interactions to analyze
which we denote as follows: 

\begin{itemize}

\smallskip
\item {\it Purely Fermi-type interactions} when $a,b = g_\pm$;

\smallskip
\item {\it Fermi-continuous interactions} when $\{a,b\} = \{h_\pm, g_\pm\}$;

\smallskip
\item {\it Purely continuous interactions} when $a,b = h_\pm$.

\end{itemize}

\smallskip
These three types of interactions require different treatments and we dedicate to each of them separate 
subsections within Sections \ref{secFS} and \ref{secFR}.
One important aspect of our analysis will be the splitting of the bilinear terms \eqref{FsecF}
- more precisely of their leading order contribution, corresponding to the leading order contribution
of $\mu$ -
into a {\it singular} and a {\it regular} part; see Proposition \ref{mudecomp} below,
and the consequent definition of a singular part $F^S$ and a regular part $F^R$ in \eqref{FS+FR}.
The singular, resp. regular, parts are treated in Section \ref{secFS}, resp. \ref{secFR}.

For the singular part of the bilinear interactions we can perform a normal form transformation.
The (time)-boundary terms are estimates in Subsection \ref{FSbdry},
and the three types of interactions described above are dealt with 
in Subsections \ref{secSgg}, \ref{secShg} and \ref{secShh}.
For the regular part, the three different kinds of interactions are 
estimated in Subsections \ref{secRgg}, \ref{secRhg} and \ref{secRhh}.

\medskip
\subsection{Decomposition of the NSD}\label{SecFmu}
In the previous subsection we have reduced the problem to the mid-low frequencies interaction term $F_L$;
in particular, all the bilinear operators that we are going to estimate will be supported (before time integration) 
in a region where all input and output frequencies are less than $2^{M_0} \approx \js^{\delta_N}$,
with an additional restriction on the inputs being projected to not-very-low frequencies via
implicit $P_{\geqslant  \langle s \rangle^{-5} }$ operators. 
Then, our goal is to prove \eqref{mqbeps} and \eqref{qdbeps} for $F_L$,
which, together with Lemmas \ref{lemmaH} and \ref{verylowfreq}, 
will imply Propositions \ref{mainquad} and \ref{mainquaddecay}.

To prove our main results 
we first need a precise description of the interaction 
of three distorted complex exponentials, encoded by the distribution $\mu(\xi,\eta,\sigma)$,
and a splitting of it into a singular and regular part, plus a suitable remainder.
This is contained in the next proposition, which is mostly based on results from
\cite{PS} (see Proposition \ref{decomp-meas-PS} in Appendix \ref{Appmu}):

\begin{proposition}\label{mudecomp}
Let $\mu$ be the distribution defined in \eqref{muDuh}, 
with $\psi$ satisfying \eqref{psieq}
for a generic potential\footnote{As already mentioned, a bound on sufficiently
many Schwartz semi-norms of $V$ is enough. 
For example, we can assume that $V$ is in a weighted Sobolev space, say $\jx^N V \in H^N$, where $N$
is the Sobolev regularity parameter in \eqref{mteps}; see also Theorem \ref{Wobd}.} 
$V \in \mathcal{S}$. 
Let\footnote{In our application the parameter $M_0 \geqslant 1$ that appears here 
will coincide with the one in \eqref{HLsplit}.} 
$M_0$ be a fixed parameter, and let\footnote{The parameter $N_2$
can be fixed large enough as a fraction of $N$.} $N_2$ be a sufficiently large integer.

We can decompose $\mu$ into a singular part (apex $S$), a regular part (apex $R$) 
and a remainder part (apex $Re$) 
\begin{align}\label{mudecomp00}
\mu = \mu^S + \mu^R + \mu^{Re}.
\end{align}
These three components can be further decomposed as follows, 
up to irrelevant multiplicative constants: 
\begin{align}\label{mudecomp0}
\begin{split}
& \mu^S(\xi,\eta,\sigma) := \delta_0(\xi-\eta-\sigma) + \mu_1^S(\xi,\eta,\sigma) 
  + \mu_2^S(\xi,\eta,\sigma) + \mu_3^S(\xi,\eta,\sigma),
\\
& \mu^R(\xi,\eta,\sigma) :=  \mu_1^R(\xi,\eta,\sigma), 
\end{split}
\end{align}
where the right-hand sides are defined as follows:

\smallskip
\setlength{\leftmargini}{1.5em}
\begin{itemize}

\smallskip
\item $\mu_1^S$ and $\mu_1^R$ are given by the formulas
\begin{align}\label{mu1SR}
\mu_1^{\ast}(\xi,\eta,\sigma) 
  & := \nu_1^{\ast} (-\xi +\eta,\s) +  \nu_1^{\ast} (-\xi +\s,\eta) +  \overline{\nu_1^{\ast} (-\eta-\s,\xi)}, 
  \qquad \ast \in \lbrace S,R \rbrace
\end{align}
with
\begin{align}\label{nu1S}
\begin{split}
\nu_1^S(p,q) & := \varphi_{\leqslant -M_0-5}(\vert p \vert - \vert q \vert) \Big[ 
  \nu_0 (p,q)
  \\ & + \frac{1}{\vert p \vert} \sum_{a=1}^{N_2} \sum_{J \in \Z} 
  b_{a,J}(p,q) \cdot 2^J K_a \big(2^J (\vert p \vert - \vert q \vert) \big) \Big]
\end{split}
\end{align}
with 
\begin{align}\label{nu0}
\nu_0(p,q) :=  \frac{b_0(p,q)}{|p|} \Big[ i\pi \, \delta(|p|-|q|) + \pv \frac{1}{|p|-|q|} \Big],
\end{align}
where $b_0$ is a symbol satisfying \eqref{Propnu+1.1},
$K_a$ are Schwartz functions, $b_{a,J}$ are symbols satisfying the bounds stated in \eqref{Propnu+2.1},
and $\nu_1^R$ satisfies, for all $|p|\approx 2^P$, $|q|\approx 2^Q$, with $P,Q \leqslant M_0$
and for all $|a|+|b| \leqslant N_2$,
\begin{align} \label{nuRest}
\begin{split}
& \vert \nabla_{p}^a \nabla_{q}^b \nu_1^R (p,q) \vert
\\
& \lesssim \begin{cases}
2^{-2 P} \, 2^{-(|a|+ |b|)P} \cdot 2^{(2 + |a| + |b|)5M_0} & \textrm{if}  
  \quad |P-Q| < 5 
 \\
2^{-2 (P \vee Q)} \cdot 2^{-|a| (P \vee Q)} \max \lbrace 1 , 2^{(1-\vert b \vert) Q^-}  \rbrace
  \cdot 2^{(|a| + |b| + 2)5M_0} &  \textrm{if} \quad |P-Q|  \geqslant 5.
\end{cases}
\end{split}
\end{align}

\smallskip
\item $\mu_2^S$ 
is given by the formulas
\begin{align}\label{mu2SR}
\mu_2^{S}(\xi,\eta,\sigma) 
  & := \nu_{2}^{S,1} (\xi,\eta,\s) + \nu_{2}^{S,2} (\xi,\eta,\s) + \nu_{2}^{S,2} (\xi,\s,\eta), 
\end{align}
with
\begin{align}\label{nu2S1}
\begin{split}
\nu_{2}^{S,1}(\xi,\eta,\s) := \frac{1}{|\xi|} \sum_{i=1}^{N_2} 
 \sum_{J \in \Z 
 } \varphi_{\leqslant -M_0-5} (|\xi|-|\eta|-|\s|)
 \\
 \times b_{i,J}(\xi,\eta,\s) \cdot K_i\big( 2^J (|\xi|-|\eta|-|\s|) \big),  
\end{split}
\end{align}
and 
\begin{align}\label{nu2S2}
\begin{split}
\nu_{2}^{S,2}(\xi,\eta,\s) = \frac{1}{\vert \eta \vert} \sum_{\epsilon \in \lbrace 1, -1 \rbrace} 
  \sum_{i=1}^{N_2} \sum_{J \in \Z 
  } \varphi_{\leqslant -M_0-5} (|\xi|+\epsilon|\eta|-|\s|) 
  \\ 
  \times b_{i,J}^{\epsilon}(\xi,\eta,\s) \cdot 
  K_i \big(2^J (|\xi|+\epsilon|\eta|-|\s|) \big), 
\end{split}
\end{align}
where $K_i$ are Schwartz functions and $b_{i,J}$ and $b_{i,J}^\epsilon$  
are symbols satisfying the bounds stated in \eqref{Propnu212} and \eqref{Propnu222}.


\smallskip
\item $\mu_3^S$ 
is given by the formula
\begin{align}\label{mu3S}
\begin{split}
\mu_{3}^S (\xi,\eta,\s) := \sum_{i=0}^{N_2} \sum_{J \in \Z 
  } 
  \varphi_{\leqslant -M_0-5}(|\xi|-|\eta|-|\s|) 
  \\ \times b_{i,J}(\xi,\eta,\s) \cdot K_i \big(2^J(|\xi|-|\eta|-|\s|) \big), 
\end{split}
\end{align}
where $K_i$ are Schwartz functions
and $b_{i,J}$ are symbols satisfying the bounds stated in \eqref{Propnu+2.1}.

\smallskip
\item $\mu^{Re}$ satisfies the following estimate:
for all $|\xi| \approx 2^{k_1}$, $|\eta| \approx 2^{k_2}$ and $|\s| \approx 2^{k_3}$ with $k_1,k_2,k_3 \leqslant M_0$
and for all $|a|+|\alpha|+|\beta| \leqslant N_2$,
\begin{align}\label{nuReest}
\vert \nabla_{\xi}^a \nabla_{\eta}^{\alpha} \nabla_{\s}^{\beta} \mu^{Re} (\xi,\eta,\s) \vert 
  & \lesssim 2^{-2 \max \lbrace k_1,k_2,k_3 \rbrace} 
  \cdot 2^{-(\vert \alpha_{a} \vert + \vert \alpha_{e} \vert) \max \lbrace k_1, k_2,k_3 \rbrace}  
  \\ \notag & \cdot \max \big(1 , 2^{(1-\vert \alpha_{i} \vert) \min \lbrace k_1, k_2 ,k_3 \rbrace} \big) 
  \cdot 2^{5M_0(|a| + |\alpha| + |\beta| +1)},
\end{align}
where $\alpha_{i}$ denotes the order of differentiation of the variable that is located at the smallest frequency, 
$\alpha_{e}$ is similarly defined for the middle frequency and $\alpha_{a}$ for the largest one.
\end{itemize}

\end{proposition}

\smallskip
The above proposition essentially follows from Proposition \ref{decomp-meas-PS},
 whose content is taken from \cite{PS}, and which we rewrote in the Appendix for the convenience of the reader.
 Proposition \ref{decomp-meas-PS} gives a decomposition of $\mu(\xi,\eta,\s)$ in terms of the singularities
 in $|\xi-\eta| - |\s|$, $|\xi-\s| - |\eta|$ and so on.
 Understanding the precise structure of these singularities is one of the keys to
 extracting decay in distorted frequency space from the integrals \eqref{secF0'}.
 
 Here are a few comments on the statement of Proposition \ref{mudecomp}:

\begin{itemize}

 \smallskip
 \item Proposition \ref{mudecomp} goes a step further compared to Proposition \ref{decomp-meas-PS},
 and adapts $\mu$ and its various components to the nonlinear structure of the Klein-Gordon flow.
 In particular, the various pieces are decomposed into singular and regular parts.
 There is also an even more regular remainder part which we omit for the sake of this discussion.
 
 The distinction between singular and regular parts of $\mu$,
 which also gives an analogous splitting for the nonlinear terms (see \eqref{FS+FR})
 is made so that, on the support of the singular parts, we will be able to lower bound the 
 nonlinear Klein-Gordon phase $\Phi_{\epss}$ (see \eqref{secF0'});
 this is done in Lemma \ref{lemphases}. 


 \smallskip
 \item The estimate \eqref{nuRest}, implies that, up to small losses for large frequencies
 (which will be innocuous for our arguments), the regular measure $\nu_1^R(p,q)$ behaves
 almost like a Coifman-Meyer multiplier times a factor of $(|p|\vee|q|)^{-2}$,
 and better (upon localization) than a multiplier in $\whF^{-1}_{(x,y)\rightarrow(p,q)} L^1_{x,y}$
 times a factor of $(|p|\vee|q|)^{-2}$.
 More importantly, it is not very singular so that we can also use these estimates pointwise.

 \smallskip
 \item The distributions $\nu_1^S$ are the most singular and arise from the interaction (in the definition
 of $\mu$) of two plane waves and one `perturbed' wave $\psi(x,\zeta) - e^{ix\cdot \zeta}$; see \eqref{psieq}.
 The distributions (measures, in fact) $\nu_2^{S,1}$ and  $\nu_2^{S,2}$  
 are a little more regular since they arise from the interaction of at least two perturbed waves.
 Estimates on the bilinear operators associated to these distributions are given in  Theorem \ref{bilinearmeas}.
 The statement also contains bilinear estimates for the operators associated to
 $\mu^R$ and $\mu^{Re}$. 
 
\end{itemize}

\begin{proof}

We look at Proposition \ref{decomp-meas-PS} 
and extract all the distributions in the statement from the formulas there.

We first inspect the formula for $\mu_1$ in \eqref{mu1}, which gives 
\begin{align*}
\mu_1(\xi,\eta,\sigma) & = \nu_1(-\xi+\eta,\sigma) + \nu_1(-\xi+\sigma,\eta) + \overline{\nu_1(-\eta-\sigma,\xi)}
\end{align*}
with the decomposition of $\nu_1$ in \eqref{Propnu+1}, with \eqref{Propnu+1.0} and \eqref{Propnu+1.1},
\eqref{Propnu+2} and \eqref{Propnu+2.1}, and the bounds \eqref{Propnu+3}-\eqref{Propnu+3imp}.
Then, with the definition \eqref{nu1S}-\eqref{nu0} for $\nu_1^S$
and letting 
\begin{align}\label{nu1R}
\begin{split}
\nu_1^R(p,q) & := \varphi_{> -M_0-5}(\vert p \vert - \vert q \vert) \Big[ \nu_0 (p,q) 
  \\ & + \frac{1}{\vert p \vert} \sum_{a=1}^{N_2} \sum_{J \in \Z} 
  b_{a,J}(p,q) \cdot 2^J K_a \big(2^J (\vert p \vert - \vert q \vert) \big) \Big] + \nu_R (p,q),
\end{split}
\end{align}
where $\nu_R (p,q)$ is the symbol satisfying \eqref{Propnu+3}-\eqref{Propnu+3imp},
we have that the sum of \eqref{nu1S} and \eqref{nu1R} gives us the term $\nu_1$ in \eqref{Propnu+1},
and therefore
\begin{align}\label{prmu1}
\mu_1(\xi,\eta,\sigma) = \mu_1^S(\xi,\eta,\sigma) + \mu_1^R(\xi,\eta,\sigma).
\end{align}
The bound \eqref{nuRest} for the expression in \eqref{nu1R}
holds as a result of \eqref{Propnu+3}-\eqref{Propnu+3imp} (which is the desired bound for $\nu_R$)
and since the same can be verified directly for the first expression on the right-hand side of \eqref{nu1R}.



To verify the properties in the second bullet,
we start by inspecting the formulas \eqref{Propnu210}-\eqref{Propnu211} for $\nu_2^1$,
with the bounds \eqref{Propnu212} and \eqref{Propnu213},
and the formulas \eqref{Propnu220}-\eqref{Propnu221}  for $\nu_2^2$, with the bounds \eqref{Propnu222}-\eqref{Propnu223}.
We then set
\begin{align}
\label{nu2R1}
\begin{split}
\nu_{2}^{R,1}(\xi,\eta,\s) := \frac{1}{\vert \xi \vert} \sum_{i=1}^{N_2} \sum_{J \in \Z 
  } \varphi_{>-M_0-5} (|\xi|-|\eta|-|\s|) 
  \\
  \times b_{i,J}(\xi,\eta,\s) \cdot K_i\big( 2^J (\vert \xi \vert - \vert \eta \vert - \vert \s \vert) \big),
\end{split}
\\
\label{nu2R2}
\begin{split}
\nu_{2}^{R,2}(\xi,\eta,\s) := \frac{1}{\vert \eta \vert} \sum_{\epsilon \in \lbrace 1, -1 \rbrace} 
  \sum_{i=1}^{N_2} \sum_{J \in \Z 
  } \varphi_{>-M_0-5} (|\xi|+\epsilon|\eta|-|\s|) 
  \\ \times b_{i,J}^{\epsilon}(\xi,\eta,\s) \cdot 
  K_i \big(2^J (\vert \xi \vert + \epsilon \vert \eta \vert - \vert \s \vert ) \big),
\end{split}
\end{align}
and see that we can write
\begin{align}
\nu_2^1(\xi,\eta,\sigma) = \nu_2^{S,1}(\xi,\eta,\sigma) + \nu_2^{R,1}(\xi,\eta,\sigma) + \nu_{2,R}^1(\xi,\eta,\sigma)
\end{align}
where (note the different placing of the indexes) $\nu_{2,R}^1$ is the remainder in \eqref{Propnu210}
satisfying the estimate \eqref{Propnu213}.
Similarly, 
\begin{align}
\nu_2^2(\xi,\eta,\sigma) = \nu_2^{S,2}(\xi,\eta,\sigma) + \nu_2^{R,2}(\xi,\eta,\sigma) + \nu_{2,R}^2(\xi,\eta,\sigma)
\end{align}
where (note again the different placing of the indexes) $\nu_{2,R}^2$ is the remainder 
appearing at the end of the formula \eqref{Propnu220} which satisfies the estimates \eqref{Propnu223}.
From \eqref{mu2} and the definitions in \eqref{mu2SR},
we then deduce that the $\mu_2$ component in \eqref{mu2} is given by 
\begin{align}\label{prmu2}
\begin{split}
\mu_2(\xi,\eta,\sigma) & = \mu_2^S(\xi,\eta,\sigma) 
 \\
 & + \nu_2^{R,1}(\xi,\eta,\sigma) + \nu_2^{R,2}(\xi,\eta,\sigma) + \nu_2^{R,2}(\xi,\sigma,\eta)
 \\
 & + \nu_{2,R}^1(\xi,\eta,\sigma) + \nu_{2,R}^2(\xi,\eta,\sigma) + \nu_{2,R}^2(\xi,\sigma,\eta).
\end{split}
\end{align}
Similarly, by letting 
\begin{align*}
\mu_{3}^R (\xi,\eta,\s) :=   \sum_{i=0}^{N_2} \sum_{J \in \Z 
  } \varphi_{>-M_0-5} (|\xi|+|\eta|-|\s|)
  b_{i,J}(\xi,\eta,\s) 
  \cdot K_i \big(2^J (\vert \xi \vert - \vert \eta \vert - \vert \s \vert) \big), 
\end{align*}
we can see that the term $\mu_3$ in \eqref{Propmu30} can be written as
\begin{align}\label{prmu3}
\begin{split}
\mu_3(\xi,\eta,\sigma) & = \mu_3^S(\xi,\eta,\sigma) + \mu_3^R(\xi,\eta,\sigma) + \mu_{3,R}(\xi,\eta,\sigma),
\end{split}
\end{align}
where $\mu_{3,R}$ (note again the different placing of the indexes)
is the last term in \eqref{Propmu30} and satisfies \eqref{Propmu3R}.

We then let
\begin{align}\label{prmuRe}
\begin{split}
\mu^{Re}(\xi,\eta,\sigma) & := \nu_{2,R}^1(\xi,\eta,\sigma) 
  + \nu_{2,R}^2(\xi,\eta,\sigma) +  \nu_{2,R}^2(\xi,\sigma,\eta) 
  \\
  & + \nu_2^{R,1}(\xi,\eta,\sigma) 
  + \nu_2^{R,2}(\xi,\eta,\sigma) + \nu_2^{R,2}(\xi,\sigma,\eta)
  \\
  & + \mu_3^R(\xi,\eta,\sigma) + \mu_{3,R}(\xi,\eta,\sigma)
\end{split}
\end{align}
so that the identities \eqref{mudecomp00}-\eqref{mudecomp0} hold true in view of 
\eqref{mudecomp-0} and \eqref{prmu1}, \eqref{prmu2} and \eqref{prmu3}.

To conclude we need to verify the bound \eqref{nuReest} for the remainder \eqref{prmuRe}.
Notice first that  \eqref{Propmu3R} gives a stronger bound for $\mu_{3,R}$.
Using the formula \eqref{prmu3} we can directly verify the bound for $\mu_3^R$.
We can then inspect the formulas \eqref{nu2R1} and \eqref{nu2R2} to check the bounds for the terms in
the second line of \eqref{prmuRe}.
Finally, for the terms on the first line of the right-hand side of \eqref{prmuRe}
we invoke \eqref{Propnu213} and \eqref{Propnu223}.
\end{proof}


\medskip
\subsection{Bilinear operators: identities and bounds}\label{ssecid}
Recall that our main term to estimate is $F_L$ 
from \eqref{secF0}-\eqref{secF0'}.
Below we give some useful identities and bounds for 
bilinear terms as in \eqref{secF0'} that we are going to use in the rest of the arguments.

First, in view of Proposition \ref{mudecomp},
we define several bilinear operators, which correspond to the various pieces in which $\mu$ is decomposed.
For a general symbol $b=b(\xi,\eta,\s)$ let us denote
\begin{subequations}\label{idop}
\begin{align}
\label{idop0}
T_0[b](g,h)(x) & := \whF^{-1}_{\xi\rightarrow x} \iint_{\R^3\times\R^3} g(\eta) h(\s) 
  \,b(\xi,\eta,\s) \, \delta_0(\xi-\eta-\s) \, d\eta d\s,
\\
\label{idop11}
T_1^{\ast,1}[b](g,h)(x) & := \whF^{-1}_{\xi\rightarrow x} \iint_{\R^3\times\R^3} g(\xi-\eta) h(\sigma) 
  \,b(\xi,\eta,\sigma)\, \nu_1^{\ast}(\eta,\sigma) \, d\eta d\sigma, \quad \ast \in\{S,R\},
\\
\label{idop12}
T_1^{\ast,2}[b](g,h)(x) &:= \whF^{-1}_{\xi\rightarrow x} \iint_{\R^3\times\R^3} g(-\eta-\sigma) h(\sigma) 
  \,b(\xi,\eta,\sigma)\, \overline{\nu_1^{\ast}(\eta,\xi)} \, d\eta d\sigma, \quad \ast \in\{S,R\},
\\
\label{idop2}
T_2^{S} [b](g,h)(x) & := \whF^{-1}_{\xi\rightarrow x} \iint_{\R^3\times\R^3} g(\eta) h(\sigma) 
  \,b(\xi,\eta,\sigma)\, \mu_2^{S}(\xi,\eta,\sigma) \, d\eta d\sigma,
\\
\label{idop3}
 T_3^{S} [b](g,h)(x) & := \whF^{-1}_{\xi\rightarrow x} \iint_{\R^3\times\R^3} g(\eta) h(\sigma) 
  \,b(\xi,\eta,\sigma)\, \mu_3^{S}(\xi,\eta,\sigma) \, d\eta d\sigma,
\\
\label{theomuRe}
T^{Re}[b](g,h)(x) & := \whF^{-1}_{\xi\rightarrow x} \iint_{\R^3\times\R^3} g(\eta) h(\sigma) 
  \,b(\xi,\eta,\sigma)\, \mu^{Re}(\xi,\eta,\sigma) \, d\eta d\sigma.
\end{align}
\end{subequations}
The distributions appearing above are defined in Proposition \ref{mudecomp} and in its proof;
see \eqref{nu1S}, \eqref{mu2SR} and \eqref{mu3S}, \eqref{nu1R} and \eqref{prmuRe}.

Second, for a general symbol $b$, let us adopt the following notation for $b$ evaluated at
different inputs:
\begin{align}\label{idsym}
\begin{split}
b_1(\xi,\eta,\sigma) & = b(\xi,\xi-\eta,\sigma),
\\
b_1'(\xi,\eta,\sigma) & = b (\xi,\s,\xi-\eta),
\\
b_2(\xi,\eta,\sigma) & = b(\xi,-\eta-\sigma,\sigma).
\end{split}
\end{align}
With the definitions \eqref{idop} and the notation \eqref{idsym}, 
using the decomposition from Proposition \ref{mudecomp} and applying changes of variable,
we see that we can write bilinear operators of the type \eqref{secF0}-\eqref{secF0'}
as follows:
with $\wt{A_{\eps_1}} := e^{is\eps_1\jxi} \wt{a}$ and $\wt{B_{\eps_2}}:=  e^{is\eps_2\jxi} \wt{b}$ we have
\begin{align}\label{idbil}
\begin{split}
& e^{is \jnab} \what{\mathcal{F}}_{\xi\mapsto x}^{-1}  
  \int_{\R^6} e^{-is \Phi_\epss (\xi,\eta,\sigma)}
  \wt{a}(s,\eta) \wt{b}(s,\sigma) \mu^S (\xi,\eta,\sigma) \, m(\xi,\eta,\sigma) \, d\eta d\sigma 
\\
& = T_0[m] \big(\wt{A_{\eps_1}}, \wt{B_{\eps_{2}}}\big)(x)  
\\
& + T_1^{S,1}[m_1] \big(\wt{A_{\eps_1}}, \wt{B_{\eps_{2}}}\big)(x)
  + T_1^{S,1}[m_1'] \big( \wt{B_{\eps_{ 2 }}},\wt{A_{\eps_1}}\big)(x)
+ T_1^{S,2}[m_2] \big(\wt{A_{\eps_1}}, \wt{B_{\eps_{ 2 }}}\big)(x)
\\
& + T_2^{S}[m] \big(\wt{A_{\eps_1}}, \wt{B_{\eps_{2}}}\big)(x)
+ T_3^{S}[m] \big(\wt{A_{\eps_1}}, \wt{B_{\eps_{ 2}}}\big)(x),
\end{split}
\end{align}
and
\begin{align}\label{idbilR}
\begin{split}
& e^{is \jnab} \what{\mathcal{F}}_{\xi\mapsto x}^{-1}  
  \int_{\R^6} \frac{e^{-is \Phi_\epss (\xi,\eta,\sigma)}}{\jeta \jsig} 
  \wt{a}(s,\eta) \wt{b}(s,\sigma) \mu^R (\xi,\eta,\sigma) \, m(\xi,\eta,\sigma) \, d\eta d\sigma 
\\
& = 
T_1^{R,1}[m_1] \big(\wt{A_{\eps_1}}, \wt{B_{\eps_{ 2 }}}\big)(x)
  + T_1^{R,1}[m_1'] \big( \wt{B_{\eps_{2}}},\wt{A_{\eps_1}}\big)(x)
  + T_1^{R,2}[m_2] \big(\wt{A_{\eps_1}}, \wt{B_{\eps_{2 }}}\big)(x).
\end{split}
\end{align}

In particular, we can write $F_{\epss}(a,b)$
as the time integral of the sum of the expressions in \eqref{idbil} and \eqref{idbilR}
with $m(\xi,\eta,\sigma) = \langle \eta \rangle^{-1} \langle \sigma \rangle^{-1}$,
with similar formulas for $F_{L,\epss}$ (inserting the cutoffs $\varphi_{\leqslant M_0}$).


\medskip
We now state the main bilinear estimates satisfied by the operators in \eqref{idop}:

\begin{theorem}\label{bilinearmeas}
Let a symbol $b=b(\xi,\eta,\s)$ be given such that
\begin{align*}
\textrm{supp}(b) \subseteq \big\{ (\xi,\eta,\sigma) \in \mathbb{R}^9\,:\, |\xi|+|\eta|+|\sigma| \leqslant 2^A
  \big\},
\end{align*}
for some $A \geqslant 1$ and such that,
for $\vert \xi \vert \approx 2^K, \vert \eta \vert \approx 2^{L}, \vert \s \vert \approx 2^M$,
\begin{align}\label{bilmeasbest}
\big \vert \nabla_{\xi}^a \nabla_{\eta}^{\alpha} \nabla_{\s}^{\beta} 
  b(\xi,\eta,\s)  \big \vert \lesssim 2^{-\vert a \vert K - \vert \alpha \vert L - \vert \beta \vert M} 
  \cdot 2^{(\vert a \vert + \vert \alpha \vert + \vert \beta \vert)A},
  \qquad \vert a \vert , \vert \alpha \vert , \vert \beta \vert \leqslant 4.
\end{align}
Let $p,q,r \in (1,\infty),$ with 
$$\frac{1}{p} + \frac{1}{q} > \frac{1}{r},$$
and assume there is $10A \leqslant D \leqslant 2^{A/10}$ such that 
\begin{align}\label{bilmeasD}
\mathcal{D}(g,h) := \Vert g \Vert_{L^2} \Vert h \Vert_{L^2} 
  + \min \big( \Vert \partial_{\xi} g \Vert_{L^2} \Vert h \Vert_{L^2},  
  \Vert g \Vert_{L^2} \Vert \partial_{\xi} h \Vert_{L^2} \big) \leqslant 2^D.
\end{align}

Then, there exists an absolute constant $C_0$ (for example, $C_0:=65$ is suitable), such that 
the following bilinear bounds hold for the operators defined in \eqref{idop}:
\begin{subequations}\label{mainbilinest}
\begin{align}
\label{mainbilin0} 
& \big \Vert T_0[b](g,h) \big \Vert_{L^r} 
  \lesssim \Vert \widehat{g} \Vert_{L^p} \Vert \widehat{h} \Vert_{L^q} \cdot 2^{C_0 A} ,  
\\
\label{mainbilin1} 
& \big \Vert T_1^{S,1}[b](g,h) \big \Vert_{L^r} 
  + \big \Vert P_{K} T_1^{S,2} [b](g,h) \big \Vert_{L^r}
  \lesssim \Vert \widehat{g} \Vert_{L^p} \Vert \widehat{h} \Vert_{L^q} \cdot 2^{C_0 A} + 2^{-D} \mathcal{D} (g,h),  
\\
\label{mainbilinR}
& \big \Vert T_1^{R,1}[b](g,h) \big \Vert_{L^r} 
  + \big \Vert P_{K} T_1^{R,2} [b](g,h) \big \Vert_{L^r} 
  \lesssim \Vert \widehat{g} \Vert_{L^p} \Vert \widehat{h} \Vert_{L^q} \cdot 2^{C_0 A}
\\
& \label{mainbilin2}
\big \Vert P_K T_i^{S}[b](g,h) \big \Vert_{L^r} 
  \lesssim \Vert \widehat{g} \Vert_{L^p} \Vert \widehat{h} \Vert_{L^q} \cdot 2^{C_0 A}, 
  \qquad i=2,3,  
\\
\label{mainbilinRe}
& \big \Vert P_K T^{Re}[b](g,h) \big \Vert_{L^r} 
  \lesssim \Vert \widehat{g} \Vert_{L^p} \Vert \widehat{h} \Vert_{L^q} \cdot 2^{C_0 A}.
\end{align}
\end{subequations}

\end{theorem}

\begin{proof}
For the convenience of the reader we include a short explanation 
on how these estimates follow from (the proof of) Theorem \ref{theomu1}, which is in \cite{PS}. 

Consider first the operators $T_1^i$, $i=1,2$, as they appear in \eqref{theomu11bb}-\eqref{theomu12bb}. 
The proof of the estimates \eqref{theomu1conc}-\eqref{theomu1concT12} 
is based on a decomposition made according to the size/distance from the singularity
(e.g. depending on the size of $|\eta|-|\s|$ in the case of $T_1^1$). 
The only difference with the operators $T_1^{S,i}$, $i=1,2$, is the presence of an additional cut-off
$\varphi_{\leqslant -M_0 - 5}$ in \eqref{nu1S};
compare with 
the definition of $\nu_1$ in \eqref{Propnu+1}.
These slight modifications do not change the estimate for the corresponding 
bilinear operators, and only restrict some indices of summation in the proof of Theorem \ref{theomu1}.
The same holds true for the (complementary) operators $T_1^{R,i}$ associated to $\nu_1^R$;
see \eqref{mu1SR} and \eqref{nu1R}.

The proofs for the other operators $T_2^S$ and $T_3^S$ are also identical 
to those for the similar operators associated to \eqref{mu2} and \eqref{Propmu31}, 
since the extra cutoffs are again harmless. 

Finally, note that the $\mathcal{D}$ factor is only introduced to deal with the principal value singularity 
and is absent in \eqref{mainbilinR}, \eqref{mainbilin2} and \eqref{mainbilinRe}.
\end{proof}

\begin{remark}\label{remuse}
Using Theorem \ref{bilinearmeas} in combination with the identity \eqref{idbil},
and the boundedness of the wave operator $\whF^{-1}\wtF$, see Theorem \ref{Wobd},
will let us estimate the bilinear expressions that arise from $\partial_\xi F_{L}$,
and establish the desired bound \eqref{mqbeps}, as well as \eqref{qdbeps}.
\end{remark}

In our argument we will also need to deal with derivatives of the components of the measure $\mu$
and use the following remark.

\begin{remark}[Derivative of $\mu_1^S$]\label{remnuS}
When estimating $\partial_\xi$ of the expressions in \eqref{secF0} we also need to 
take care of the differentiation of the measure $\mu$, and, in particular, of the leading order 
singular component $\mu_1^S$, see \eqref{mu1SR}. 
Notice that derivatives of the regular component are not so harmful, essentially by definition;
see, for example, \eqref{nuRest}.

Note first that, for a function $b$, we can write 
\begin{align*}
\nabla_\xi b(|-\xi+\s| - |\eta|) & = \frac{\s-\xi}{|\s-\xi|} \, \frac{\eta}{|\eta|} \cdot \nabla_\eta b(|-\xi+\s| - |\eta|),
\\
\nabla_\xi b(|\eta+\s| - |\xi|) & = - \frac{\xi}{|\xi|} \, \frac{\eta+\s}{|\eta+\s|} \cdot \nabla_\eta b(|\eta+\s| - |\xi|).
\end{align*}

Second, from \eqref{nu1S} and \eqref{nu0} we see that $\mu_1^S$ is a linear combination 
of terms of the form 
\begin{align*}
& \nu_S(-\xi+\eta, \s) = \frac{1}{|-\xi+\eta|}\nu(|-\xi+\eta|-|\sigma|) m_0(-\xi+\eta,\s), 
\\
& \nu_S(-\xi+\s, \eta) = \frac{1}{|-\xi+\sigma|}\nu(|-\xi+\s| - |\eta|) m_0(-\xi+\sigma,\eta),
\\
& \nu_S(-\eta-\s, \xi) = \frac{1}{|-\eta-\sigma|}\nu(|-\eta-\sigma| - |\xi|) m_0(-\eta-\sigma,\xi),
\end{align*}
where $\nu$ is some distribution (either $\delta$ or $\pv$) or a Schwartz function,
and $m_0$ denote suitable symbols, as described in Proposition \ref{mudecomp} above;
for example, $m_0$ may be the coefficient/symbol $b_0$ in \eqref{nu0} satisfying \eqref{Propnu+1.1},
or any of the other rather harmless coefficients/symbols appearing in the statement.

Then, we have the following distributional identities 
for the derivatives of these singular expressions: 
\begin{align}\label{dxinuS}
\begin{split}
\nabla_\xi \nu_S(-\xi+\eta,\s) & = -\nabla_\eta \nu_S(-\xi+\eta,\s),
\\
\nabla_\xi \nu_S(-\xi+\s,\eta) & = \frac{\xi-\s}{|\xi-\s|} \, \frac{\eta}{|\eta|} \cdot 
  \nabla_\eta \Big[ \nu(|-\xi+\s| - |\eta|) \Big] \frac{m_0(-\xi+\s,\eta)}{|-\xi+\s|} 
  \\ & + \nu(|-\xi+\s| - |\eta|) \nabla_\xi \Big[\frac{m_0(-\xi+\s,\eta)}{|-\xi+\s|} \Big],
  \\
\nabla_\xi \nu_S(-\eta-\s,\xi) & = - \frac{\xi}{|\xi|} \, \frac{\eta+\s}{|\eta+\s|} \cdot \nabla_\eta \Big[ 
  \nu(|\eta+\s| - |\xi|) \Big] \, \frac{m_0(-\eta-\s,\xi)}{|\eta+\s|}
  \\ & + \nu(|\eta-\s| - |\xi|) \frac{\nabla_\xi m_0(-\eta-\s,\xi)}{|\eta+\s|}.
\end{split}
\end{align}
The first formula above can be used directly to integrate by parts,
thus removing the $\nabla_\xi$ on the distribution by transforming it into a $\nabla_\eta$ derivative.

Next, let us inspect the second formula in \eqref{dxinuS}; 
the analysis for the third one is identical. 
We have 
\begin{align}
\label{dxinuS1}
\nabla_\xi \nu_S(-\xi+\s,\eta) & = \frac{\xi-\s}{|\xi-\s|} \frac{\eta}{|\eta|} \cdot \nabla_\eta \nu_S(-\xi+\s,\eta) 
  + \nu_S'(\xi,\eta,\sigma)
\end{align}
where
\begin{align}\label{dxinuS2}
\begin{split}
\nu_S'(\xi,\eta,\sigma) := & - \frac{\xi-\s}{|\xi-\s|} \, \nu(|-\xi+\s| - |\eta|) \, \frac{\eta}{|\eta|} 
  \cdot \frac{\nabla_\eta m_0(-\xi+\s,\eta)}{|-\xi+\s|}
  \\ & + \nu(|-\xi+\s| - |\eta|) \nabla_\xi \Big[\frac{m_0(-\xi+\s,\eta)}{|-\xi+\s|} \Big]
\end{split}
\end{align}
In other words, 
$\nabla_\xi \nu_S(-\xi+\s,\eta)$ can be written as the sum of two terms:

\smallskip
\noindent
1. $\nabla_\eta \nu_S(-\xi+\s,\eta)$ times harmless factors involving symbols of Riesz transforms, and

\smallskip
\noindent
2. A distribution, $\nu'_S$, that looks exactly like $\nu_S$ itself 
(it has the same singular structure) up to possible extra singular factor of $1/|\eta|$ or $1/|-\xi+\sigma|$
(these have the same size on the support of the $\nu_S$-type distributions under consideration).

\smallskip
The first term will be handled by integrating by parts in $\eta$, and we will detail
how to bound the resulting contributions. The presence of the symbols of the Riesz transforms
is harmless since we will always be in a situation where the respective variables 
are localized at fixed dyadic scales. 
Also, the differentiation of the factor $\eta/|\eta|$ upon integration by parts 
can be handled via Hardy's inequality;
in particular it has the same effect of differentiating the input profile,
which is already included in our treatment 
in view of the integration by parts in $\eta$ just described above.

The operators associated to the distribution $\nu_S'$ satisfy the same estimates  as that of the operators associated to $\nu_S$ up possible extra singular factors 
of $2^{-K},$ where $2^K$ is the size of the input 
variables $|-\xi+\sigma|\approx |\eta|$.
Once again, this extra factor can be handled through Hardy's inequality.
Therefore, we can systematically disregard the remainder-type distributions $\nu'_S$.
The above considerations show that we can effectively utilize the identity
\begin{align}\label{dximu1S}
\begin{split}
& \nabla_\xi \mu_1^S(\xi,\eta,\s) \approx \nabla_\eta \mu_1^S(\xi,\eta,\s) \approx 
  \nabla_\s \mu_1^S(\xi,\eta,\s),
\end{split}
\end{align}
where $\approx$ is up to harmless terms as described above.

Finally, since we can write
\begin{align*}
\nabla_{\xi} b (\vert \xi \vert + \epsilon_1 \vert \eta \vert + \epsilon_2 \vert \s \vert ) &
  =\epsilon_1 \frac{\xi}{\vert \xi \vert} \frac{\eta}{\vert \eta \vert} \cdot \nabla_{\eta} 
  b(\vert \xi \vert + \epsilon_1 \vert \eta \vert + \epsilon_2 \vert \s \vert ) 
  \\
& = \epsilon_2 \frac{\xi}{\vert \xi \vert} \frac{\s}{\vert \s \vert} \cdot \nabla_{\s} 
  b(\vert \xi \vert + \epsilon_1 \vert \eta \vert + \epsilon_2 \vert \s \vert ),
\end{align*}
we deduce, in the same sense of \eqref{dximu1S}, that
\begin{align}\label{dximuiS}
\begin{split}
& \nabla_\xi \mu_i^S(\xi,\eta,\s) \approx \nabla_\eta \mu_i^S(\xi,\eta,\s)
  \approx \nabla_\s \mu_i^S(\xi,\eta,\s) , \qquad i=2,3.
\end{split}
\end{align}
\end{remark}

\medskip
\subsection{Singular vs. Regular decomposition of nonlinear terms}\label{secFSRsplit}
Consistently with the decomposition in Proposition \ref{mudecomp}, we now split the nonlinear terms $F_L$
(which we will often just call $F$ for better legibility). Recalling \eqref{secF0}-\eqref{secF0mu}, and recall the mid-low frequency localizations,
\begin{align}\label{FS+FR}
\begin{split}
& F_L = F^S + F^R +F^{Re}, 
  \qquad F^\ast = \sum_{\epss \in \lbrace \pm \rbrace } \eps_1\eps_2 F_{\epss}^\ast\big(f_{\eps_1},f_{\eps_2}\big),
\\
& F_{\epss}^\ast(a,b) := \int_0^t \int_{\R^6} \frac{e^{-is \Phi_\epss (\xi,\eta,\sigma)}}{\jeta \jsig} 
    \wt{a}(s,\eta) \wt{b}(s,\sigma) 
    \\ & \qquad \qquad \qquad \times 
    \varphi_{\leqslant M_0}(\jxi)\varphi_{\leqslant M_0}(\jeta)\varphi_{\leqslant M_0}(\jsig)
    \, \mu^\ast(\xi,\eta,\sigma) \, d\eta d\sigma ds,  \qquad \ast \in \{S,R, Re\}.
\end{split}
\end{align}
The next lemma shows that, on the support of the singular distribution $\mu^S$ 
the phases of the operators \eqref{FsecF} are lower bounded, up to a small loss depending
on the high-low frequencies cutoff at scale $2^{M_0}$.

\begin{lemma}\label{lemphases}
On the support of the integral defining $F^S$,
see \eqref{FS+FR} and the definition of $\mu^S$ in Proposition \ref{mudecomp}, 
we have (recall also the notation \eqref{secF0'})
\begin{align}\label{no-t-res}
\big| \Phi_\epss(\xi,\eta,\sigma) \big| \geqslant \frac{1}{8}
  \frac{1}{\langle \xi \rangle + \langle \eta \rangle + \langle \sigma \rangle},
\end{align}
for all $\eps_1,\eps_2 \in \{+,-\}$.
Moreover we have the (weak) symbol type bounds
\begin{align} \label{sym-Phi} 
\bigg \vert \varphi_{\leqslant M_0}(\jxi) 
  \varphi_{\leqslant M_0}(\jeta) \varphi_{\leqslant M_0}(\jsigma) \nabla_{\xi}^{a} \nabla_{\eta}^{\alpha} \nabla_{\sigma}^{\beta} 
  \frac{1}{\Phi_\epss(\xi,\eta,\sigma)} \bigg \vert \lesssim 
  2^{M_0(|a|+ |\alpha| + |\beta| + 1)}.
\end{align}
\end{lemma}

\begin{remark}
The bound \eqref{sym-Phi} is far from optimal since we allow growth for large frequencies, 
but it suffices for our purposes since frequencies are not very large compared to the time variable.
Moreover, we need to allow this type of weak symbol bounds in any case
when we differentiate the coefficients appearing in the distribution $\mu^S$.
\end{remark}

\begin{proof}
The estimate \eqref{no-t-res} in the case $(\epss) = (-,-)$ is obvious.
To prove \eqref{no-t-res} on the support of $\delta_0(\xi-\eta-\s)$ for the case $(\epss) = (-,+)$ we write 
\begin{align*}
\Phi_{-,+}(\xi,\eta,\xi-\eta) & = \langle \xi \rangle + \langle \xi -\eta \rangle - \langle \eta \rangle 
  \geqslant \frac{1+2\big( - \vert \xi \vert \vert \xi-\eta \vert + \langle \xi-\eta \rangle \langle \xi \rangle \big)
  }{\langle \xi \rangle + \langle \xi-\eta \rangle + \langle \eta \rangle} 
\\
= & \frac{1}{\langle \xi \rangle + \langle \xi-\eta \rangle + \langle \eta \rangle} 
  \bigg(1 + 2 \frac{1+\vert \xi-\eta \vert^2 + \vert \xi \vert^2
  }{\vert \xi \vert \vert \xi-\eta \vert + \langle \xi \rangle  \langle \xi-\eta \rangle} \bigg) 
\\
\geqslant & \frac{1}{\langle \xi \rangle + \langle \xi-\eta \rangle + \langle \eta \rangle}.
\end{align*}
\eqref{no-t-res} for the case $(\epss) = (+,-)$ follows by exchanging the roles of $\eta$ and $\xi-\eta$,
and for the case $(\epss) = (+,+)$ by exchanging the role of $\xi$ and $\eta$.

Let us now look at the case of the principal value distribution $\mu_1^S$, 
see \eqref{mu1SR}-\eqref{nu0}.
We begin with the estimate on the support of $\nu_1^S(-\xi+\eta,\sigma)$.
In this region we have 
\begin{align*}
|\langle \xi-\eta \rangle - \jsig | 
  \leqslant ||\eta-\xi| - |\s|| \leqslant c 2^{-M_0} \leqslant c \langle \max(|\xi|,|\eta|,|\s|) \rangle^{-1},
\end{align*}
for some $c \leqslant 1/8$, so that
\begin{align*}
| \jxi + \langle \xi - \eta \rangle + \jeta | \leqslant 
  | \jxi + \jeta + \jsig | + c\langle \max(|\xi|,|\eta|,|\s|) \rangle^{-1} \leqslant 
  2| \jxi + \jeta + \jsig |.
\end{align*}
For the case $(\epss) = (-,+)$, using the inequalities above, we have
\begin{align*}
\big| \Phi_{-,+}(\xi,\eta,\sigma) \big|
  & \geqslant \big| \jxi  + \jeta - \langle \xi-\eta \rangle \big| - \big| \langle \xi-\eta \rangle - \jsig \big|
  \\
  & \geqslant \frac{1}{\jxi + \langle \xi-\eta \rangle + \jeta} 
  - c \langle \max(|\xi|,|\eta|,|\s|) \rangle^{-1}
\geqslant \frac{1}{4}\frac{1}{\jxi + \jeta + \jsig}.
\end{align*}
Since $\Phi_{+,+}(\xi,\eta,\s) = -\Phi_{+,-}(\eta,\xi,\s)$,
we are left with the case $(\epss) = (+,+)$. 
We can use an argument similar to the one above
by writing $\Phi_{+,+}(\xi,\eta,\s)  = \Phi_{+,+}(\xi,\eta,\xi-\eta) - \jsig + \langle \xi-\eta \rangle$,
and using the lower bound \eqref{no-t-res} with $\s = \xi-\eta$ just proved above.

The proof of \eqref{no-t-res} on the support of $\nu_S(-\xi+\s,\eta)$ follows similarly by writing 
\begin{align*}
\big| \Phi_{-,+}(\xi,\eta,\sigma) \big|
  & \geqslant \big| \Phi_{-,+}(\xi,\xi-\s,\s) 
  \big| - \big| \jeta - \langle \xi-\s \rangle \big|
\end{align*}
and using the bound \eqref{no-t-res} with $\eta=\xi-\s$ 
and $\big| \jeta - \langle \xi-\s \rangle \big| \leqslant c\langle \max(|\xi|,|\eta|,|\s|) \rangle^{-1}$
on the support of the distribution.
Exchanging the roles of $\xi$ and $\s$ proves the desired lower bound for 
$\Phi_{+,+}(\xi,\eta,\sigma)  = - \Phi_{-,+}(\s,\eta,\xi)$.
The remaining case of $\Phi_{+,-}$ can be treated relying again on
the lower bound for $\Phi_{+,-}(\xi,\xi-\s,\s)$.
Also on the support of $\nu_S(-\eta-\s,\xi)$ analogous arguments apply.

Now we prove \eqref{no-t-res} on the support of $\mu_i^S, i=2,3$; see \eqref{mu2SR} and \eqref{mu3S}. 
We start with the basic inequality 
\begin{align*}
\jeta + \jsigma - \langle |\eta| +|\s| \rangle 
  & = \jsigma - \frac{\vert \s \vert^2 + 2 |\eta||\s|}{\jeta + \langle |\eta| +|\s| \rangle} 
  \\
& = \frac{1}{\jeta + \langle  |\eta| +|\s| \rangle} 
  \Bigg( \frac{1 +|\s|^2 + |\eta|^2}{\jsigma \jeta +|\s| |\eta|} 
  + \frac{1 +|\s|^2 + \vert |\eta| +|\s| \vert^2}{\jsigma \langle |\eta| +|\s| \rangle  + |\s|(|\s|+|\eta|)} \Bigg) 
\\
& \geqslant \frac{1}{4} \frac{1}{\jeta + \langle \s \rangle +  \langle  |\eta| +|\s| \rangle}.
\end{align*}

Next, note that on the support of $\mu_i^S$ 
there exist $\epsilon_1' , \epsilon_2' \in \lbrace +,- \rbrace, (\epsilon_1' , \epsilon_2') \neq (+,+)$
such that $\big \vert \vert \xi \vert + \epsilon_1' |\eta| + \epsilon_2'|\s| \big \vert \leqslant c 2^{-M_0}$
for some small constant $c$.
Starting with the case $\epsilon_1'=\epsilon_2'=-$, we write that 
\begin{align*}
\vert \Phi_{+,+} (\xi,\eta,\s) \vert & \geqslant \big \vert \langle |\eta| +|\s| \rangle 
  + \epsilon_1 \jeta + \epsilon_2 \jsigma \big \vert - \big \vert \big \langle \xi \rangle 
  -  \langle |\eta| +|\s| \rangle \big \vert  
\\
& \geqslant \frac{1}{4} \frac{1}{\jeta + \jsigma + \langle |\eta| +|\s| \rangle} - c 2^{-M_0}
\geqslant \frac{1}{8} \frac{1}{\langle \xi \rangle + \jeta + \jsigma}.
\end{align*}
The other cases are deduced from this computation after permutation of the variables.
To treat the case $(\eps_1,\eps_2)=(-,+)$ we exchange the roles of variables $\xi$ and $\s$ and write, 
using the previous inequality,  
\begin{align*}
\vert \Phi_{-,+} (\xi,\eta,\s) \vert 
=  \vert \Phi_{+,+}(\s,\eta,\xi) \vert \geqslant \frac{1}{8} \frac{1}{\langle \xi \rangle + \jeta + \jsigma}.
\end{align*}
To treat the case $(\eps_1,\eps_2)=(+,-)$ we write $\vert \Phi_{+-} (\xi,\eta,\s) \vert  
=  \vert \Phi_{++}(\eta,\xi,\s) \vert $ 
and obtain the result exchanging the roles of $\xi$ and $\eta.$
The other combinations of signs  $\eps_1',\eps_2'$ are obtained similarly.

The weak symbol bound \eqref{sym-Phi} directly follows from \eqref{no-t-res}, 
simply by noting the basic fact that $\nabla_{x}^{a} \jx$ is uniformly bounded for $a \geqslant 1.$
\end{proof}

\medskip
\section{Singular quadratic terms}\label{secFS}
The goal of this section is to estimate as in \eqref{mqbeps} and \eqref{qdbeps} 
the singular terms $F^S$ defined in \eqref{FS+FR}.

\subsection{Set-up}
We begin by decomposing the integrals inserting a time localization
(see the notation in \S\ref{secnotation}):
\begin{align}\label{defop1}
\begin{split}
F^S_{\epss} (G,H) & = \sum_{m=0}^{L+1} F^S_{\epss,m} (G,H) 
\\
F^S_{\epss,m} (G,H) & := 
  \int_0^t \tau_m(s) \int_{\R^6} \frac{e^{-is \Phi_\epss(\xi,\eta,\sigma)}}{\jeta \jsig} 
  \wt{G}(s,\eta) \wt{H}(s,\sigma) 
  \\
  & \qquad \times  \varphi_{\leqslant M_0}(\jeta) \varphi_{\leqslant M_0}(\jsigma)\varphi_{\leqslant M_0}(\jxi) 
  \mu^S(\xi,\eta,\sigma) d\eta d\sigma ds,
\end{split}
\end{align}
recall \eqref{HLsplit}. 

\begin{remark}[Notation simplification]
We may sometime dispense of some of the indexes, such as the `$m$' indicating the dyadic time scale
in our notation for the quadratic terms \eqref{defop1}, and for similar ones, in order to improve the legibility.
Also, since all bounds are obvious for $m=0$ (or finite fixed $m$), we may assume without loss of generality
that $m\geqslant 1$.

Since the signs $\epss$ also play a minor role we may sometimes omit them too.
We are also often omitting the dependence on $(t,\xi)$ when this is obvious, for example in the case of expressions like the one on the left-hand side of \eqref{defop1}.
We may adopt other similar simplifications in the course of our argument to lighten up the notation.
\end{remark}

\medskip
We learn from the inequalities in Lemma \ref{lemphases} 
that the phases in the operators $F^S_{\epss,m}$ are lower bounded;
more precisely they are lower bounded by $C 2^{-M_0}$.
Therefore, we can integrate by parts in $s$ to see that
\begin{align}\label{IBPs}
F^S_{\epss,m} (G,H) = T_{m}^{\epss,(1)}(G,H) + T_{m}^{\epss,(2)}(G,H) + \mbox{`symmetric and easier'},
\end{align}
where we defined
\begin{align}\label{Tij1}
T^{\epss,(1)}_{m}(G,H) & := \int_0^t \tau_m(s) \int_{\R^6} 
  \frac{e^{-is \Phi_\epss(\xi,\eta,\sigma)}}{i \Phi_\epss(\xi,\eta,\sigma) \jeta \jsig} 
  \varphi_{\leqslant M_0}(\jeta) \varphi_{\leqslant M_0}(\jsigma) 
  \\
\notag & \qquad \times \partial_s \wt{G}(s,\eta) \, \wt{H}(s,\sigma)
  \, \varphi_{\leqslant M_0}(\jxi)
  \, \mu^S (\xi,\eta,\sigma) d\eta d\sigma ds,
\end{align}
`symmetric' denotes an analogous term where $\partial_s$ hits $\wt{H}$ instead of $\wt{G}$,
and 
\begin{align}\label{Tij2}
T^{\epss,(2)}_{m} (G,H) & := \tau_m(t) 
  \int_{\R^6} \frac{e^{-it \Phi_\epss(\xi,\eta,\sigma)}}{
  -i \Phi_\epss(\xi,\eta,\sigma) \jeta \jsig} 
  \varphi_{\leqslant M_0}(\jeta) \varphi_{\leqslant M_0}(\jsigma)  
  \\
\notag & \qquad \times \wt{G}(t,\eta) \wt{H}(t,\sigma)  \, \varphi_{\leqslant M_0}(\jxi) 
  \, \mu^S(\xi,\eta,\sigma) d\eta d\sigma;
\end{align}
the `easier' term is the one where $\partial_s$ hits the cutoff $\tau_m$.
The first operator is a bulk term (an integral over space and time), 
while the second is a (time-)boundary term.
We can obviously disregard the `symmetric and easier' terms without loss of generality.

To obtain the desired bounds in Propositions \ref{mainquad} and \ref{mainquaddecay}
for the singular bilinear terms $F^S$ in \eqref{FS+FR},
also
according to \eqref{defop1} and 
\eqref{IBPs}-\eqref{Tij2}, it then suffices to prove

\begin{proposition}[Estimates for the Singular part]\label{proFS}
Under the a priori assumptions \eqref{boot}, 
for all $m=0,1,\dots, L+1$, for all $\eps_1,\eps_2 \in \{+,-\}$, and for $j=1,2$, 
we have
\begin{align}\label{proFSest1}
 {\big\| \partial_\xi T^{\epss,(j)}_{m} (G_{\eps_1}, H_{\eps_2}) \big\|}_{L^2} 
  \lesssim \e^{\beta} 2^{m} \rho^{1-\beta+\delta'}(2^m),
  \qquad G,H \in \{g,h\},
\end{align}
for some $\delta'\in(0,\beta/2).$

Moreover, there exists $0<a<1$ 
such that, for $j=1,2$,
\begin{align}\label{proFSest2}
{\left\|  e^{itL} \wtF^{-1}_{\xi\mapsto x} \big[
  T^{\epss,(j)}_{m} (G_{\eps_1}, H_{\eps_2}) \big] \right\|}_{L^{\infty}_x} 
  \lesssim \frac{\rho_0^{a}}{\jt^{1+a+\delta'}}, \qquad G,H \in \{g,h\}.
\end{align}

\end{proposition}


Let us introduce the following convenient quantity:
\begin{align}\label{Z}
Z = Z_\beta(\e,m) := \big(2^m \rho(2^m)\big)^{1-\beta} \big(2^m \varepsilon \big)^{\beta}. 
\end{align}
Recall that $\rho(2^m) \approx \min(2^{-m},\e^2)$
so that we have, for $2^m \approx t$,
\begin{align*}
Z \approx \left\{ 
\begin{array}{ll}
\e^\beta t^\beta, 
	& \quad t \gtrsim \e^{-2}, 
\\ 
\\
\e^{2-\beta} t,  & \quad t \lesssim \e^{-2}.
\end{array} 
\right.
\end{align*}

\begin{remark}[About \eqref{Z}, \eqref{proFSest1} and ``acceptable'' upperbounds]\label{remm}
First note that the a priori assumption \eqref{boot0} can be written as follows:
for all $t \approx 2^m$, with $2^m \gtrsim \rho_0^{-1/2} \approx \e^{-1}$  
(i.e., past the trivial local-existence time) 
\begin{align}\label{bootZ}
{\big\| \nabla_{\xi} \wt{h}(t) \big\|}_{L^2} 
  \leqslant C \e^{\beta} 2^m \rho^{1-\beta}(2^m) = C Z_\beta(\e,m).
\end{align}

For $t \approx 2^m$, the right-hand side of \eqref{proFSest1}
in terms of $Z$ reads 
\begin{align}\label{Z1}
\e^{\beta} t \rho^{1-\beta+\delta'}(t) 
\approx Z \cdot \rho^{\delta'}(t).
\end{align}
In the sequel we will encounter two types of
``acceptable'' upperbounds for the left-hand side of \eqref{proFSest1}:

\begin{itemize}

\smallskip
\item Bounds by $2^{-am} Z$ for some $a>0,$ 
which are stronger than the right-hand side of \eqref{mqbeps} for $a > 2\delta'$, see \eqref{Z1},
and since we will only consider times past $\rho_0^{-1/2}$; 
see Remark \ref{localT} which deals with the local existence times $t\lesssim \rho_0^{-1/2} \approx \e^{-1}$.

\smallskip
\item Bounds like $2^{-am} Z^2$ 
which are again stronger than the right-hand side of \eqref{proFSest1} 
provided $a > \beta + \delta'$, since $Z \lesssim 2^{\beta m} \e^{\beta}$,
and we can choose $2\delta' < \beta$.

\end{itemize}

Notice that we can pick $\beta$ and $\delta'$ arbitrarily small,
and, in particular, small enough depending on the exponent `$a$' appearing in the 
two types of bounds described above, which we will obtain from the nonlinear estimates.
A convenient choice is to let $2\delta' < \beta < a/2$.

\end{remark}

To prove \eqref{proFSest2} in the case of $j=1$, that is, for the bulk operator \eqref{Tij1}, 
we will use the following result: 

\begin{lemma}\label{basicL8}
Assume that for some $a \ll 1$ such that $a > 5\delta_N$ we have
\begin{align}\label{L8as}
{\left\| \whF^{-1}_{\xi\mapsto x} I(s) 
  \right\|}_{L^{8/7}_x} \lesssim \frac{\rho_0^{a}}{\js^{1+2a}}, 
\end{align}
then 
\begin{align}\label{L8conc}
{\Big\| e^{itL} \wtF^{-1}_{\xi\mapsto x} \int_0^t \varphi_{\leqslant \delta_N m + 5}(\xi)
  e^{-is\jxi} I(s,\xi) \, \tau_m(s) ds
  \Big\|}_{L^{\infty}_x} \lesssim \frac{\rho_0^{a}}{\jt^{1+a+\delta'}}, 
\end{align}
for $\delta'<\delta_N$.
\end{lemma}

We will apply this Lemma with
\begin{align}\label{L8appl}
I(t) & =  e^{it\jxi} \partial_t T^{\epss,(1)}_{m} (G_{\eps_1}, H_{\eps_2})(t),
\end{align}
so that, after verifying the bound \eqref{L8as} for this choice of $I$,
and using the fact that for $s \approx 2^m$ the frequency integral in \eqref{Tij1} is already localized
in $|\xi| \leqslant 2^{\delta_N m + 3}$, i.e.
$$ \partial_t T^{\epss,(1)}_{m}(s) 
  = \varphi_{\leqslant \delta_N m + 5} 
  (\xi)
  \partial_t T^{\epss,(1)}_{m}(s),$$ 
we will obtain the desired 
\eqref{proFSest2}.



\begin{proof}[Proof of Lemma \ref{basicL8}]
Using the (distorted) Bernstein's inequality (with $|\xi| \lesssim \jt^{\delta_N}$), 
the $L^8$ decay estimate from Lemma \ref{decay},
and Bernstein again, we find that 
\begin{align*}
& {\Big\| e^{itL}
  \int_0^t \wtF_{\xi \mapsto x}^{-1} \big(\varphi_{\leqslant \delta_N m+5}(\xi)
  e^{-is\jxi} I(s,\xi) \big)  \, \tau_m(s) ds \Big\|}_{L^{\infty}_x} 
  \\
  & \lesssim \jt^{3\delta_N/8} 
  {\Big\| \int_0 ^t e^{i(t-s)L} P_{\leqslant \delta_N m + 5} 
    \wtF_{\xi \mapsto x}^{-1} \big( I(\xi,s) \big) \, \tau_m(s) ds \Big\|}_{L^8_x} 
  \\
  & \lesssim \jt^{2\delta_N} \int_0^t \frac{1}{\langle t-s \rangle^{9/8}} \js^{2\delta_N}
  {\big\| \wtF_{\xi \mapsto x}^{-1} I(\xi,s)  \big\|}_{L^{8/7}_x} \, \tau_m(s) ds.
\end{align*}
Notice that in the inequality above we have also used Bernstein to deal with the integral
from $t-1$ to $t$, passing from $L^8_x$ to $L_x^2$, using Minkowski, the unitarity of $e^{itL}$
and then Bernstein again from $L^2_x$ to $L^{8/7}_x$. 
Then we conclude using the boundedness of wave operators, the assumption \eqref{L8as}
and the fact that 
\begin{align*}
\jt^{2\delta_N} \int_0^t \frac{1}{\langle t-s \rangle^{9/8}} \js^{2\delta_N}
  \frac{1}{\langle s \rangle^{1+2a}} ds \lesssim \frac{1}{\jt^{1+2a-4\delta_N}} \lesssim \frac{1}{\jt^{1+a+\delta'}},
\end{align*}
for $\delta'+4\delta_N < a$.
\end{proof}

Lemma \ref{basicL8} allows us to systematically reduce the proof of \eqref{proFSest2} 
to $L^{8/7}_x$ bounds for the inverse transform of the right-hand side of \eqref{L8appl}, 
that is: 
\begin{align}\label{L8appl'}
\begin{split}
\tau_m(s)
   \whF^{-1} & \int_{\R^6} 
   \frac{\varphi_{\leqslant M_0}(\jxi) \varphi_{\leqslant M_0}(\jeta)
   \varphi_{\leqslant M_0}(\jsigma)}{\Phi_\epss(\xi,\eta,\sigma) \jeta\jsig}
   \, e^{i\eps_1s \jeta} e^{i\eps_2s\jsig} \\ & \times \partial_s \wt{G_{\eps_1}}(s,\eta) 
   \, \wt{H_{\eps_2}}(s,\s) 
   \mu^S(\xi,\eta,\s) d\eta d\s, 
   \qquad G,H \in \{g,h\}.
\end{split}
\end{align}
Note that this expression gives bilinear terms like those appearing in \eqref{idop}, 
where 
the inputs are (up to complex conjugation) of the form 
$(\wtF(e^{-is L} \partial_s G), \wtF(e^{-i s L}H))$.
Therefore, since $e^{-is L} \partial_s g$ and $e^{-is L} \partial_s h$  
are quadratic in $A$ and $f$, and in view of the product estimates of Theorem \ref{bilinearmeas} 
the terms \eqref{L8appl'} should be  essentially thought of as cubic terms in $A$ and $f$.



\smallskip
A similar lemma that we will use to prove \eqref{proFSest2} 
for the boundary-type operators \eqref{Tij2} is the following:
\begin{lemma}\label{basicL8'}
Assume that 
\begin{align}\label{L8'as}
{\left\| \whF^{-1}_{\xi\mapsto x} J(t) 
  \right\|}_{L^{8/7}_x} \lesssim \rho_0^{a}, 
\end{align}
then, for $\delta_N$ small enough, one has 
\begin{align}\label{L8'conc}
{\Big\| e^{itL} \wtF^{-1}_{\xi\mapsto x} \big[ \varphi_{\leqslant \delta_N\log_2\jt +10}(\xi) J(t,\xi) \big]
  \Big\|}_{L^{\infty}_x} \lesssim \frac{\rho_0^{a}}{\jt^{10/9}}.
\end{align}
\end{lemma}

The proof follows the same steps as in the proof of Lemma \ref{basicL8} so we skip it.

\medskip
In the rest of this subsection, the goal is to prove Proposition \ref{proFS}, which will then imply the main bounds
\eqref{mainquadbound} and \eqref{quaddecaybound} for the terms $F^S$ through \eqref{IBPs}. 

\medskip
\subsection{Estimates for the boundary terms}\label{FSbdry}
We start by treating the easier terms that involve the operators \eqref{Tij2}. 

\begin{lemma}\label{bdry-no-t-res}
Under the assumptions of Proposition \ref{proFS}, we have 
\begin{align}\label{lembdary1}
  {\big\| \partial_{\xi} T^{\epss,(2)}_{m} (G_{\eps_1},H_{\eps_2}) \big\|}_{L^2_{\xi}}
  \lesssim \e^{\beta} 2^{m} \rho^{1-\beta +\delta'}(2^m), \qquad G,H\in\{g,h\},
\end{align}
for some $0<\delta'<\beta/2$.

Moreover, there exists $a>\delta'$ such that
\begin{align}\label{lembdary2}
{\left\| e^{itL} \wtF^{-1}_{\xi\mapsto x}
  T^{\epss,(2)}_{m} (G_{\eps_1}, H_{\eps_2})\right\|}_{L^{\infty}_x} 
  \lesssim \frac{\rho_0^{a}}{\jt^{1+a+\delta'}}, 
  \qquad G,H \in \{h,g\}.
\end{align}
\end{lemma}

\begin{proof}
We first prove the more difficult estimate \eqref{lembdary1}, and then the simpler \eqref{lembdary2}.

\smallskip
\noindent
{\it Proof of \eqref{lembdary1}}.
We subdivide the proof into three steps.

\smallskip
{\it Step 1}.
When applying $\partial_\xi$ to $T^{\epss,(2)}$, 
we claim that is suffices to focus on the worst case 
when the derivative hits the exponential.  
In fact, the terms where the frequency derivative hits the symbol $1/\Phi$ or the cutoff $\varphi_{\leqslant M_0}$ 
are clearly much easier to treat, also in view of \eqref{no-t-res}. 
When instead $\partial_\xi$ hits the measure $\mu$ we can use Remark \ref{remnuS} to argue as follows:

\smallskip
1. For the $\delta$ contribution we integrate by parts using the distributional identity 
$\partial_\xi\delta(\xi-\eta-\s) = -\partial_\eta\delta(\xi-\eta-\s)$. 
This generates:

(a) a term where $\partial_\eta$ hits the exponential, 
which is essentially identical to the leading order term that we will treat, 

(b) some other easier terms where the symbols and cutoff are differentiated,
and 

(c) a term where $\partial_\eta$ hits $\wt{f}$.
This latter is the only non-trivial term, although it is still easier to estimate than
the term where the differentiation falls on the exponential; 
we will treat this term in details together with the main one.

\smallskip
2. For the $\mu_1^S$ parts, we can use the identities \eqref{dxinuS}
to do similar integration by parts and obtain terms like those described above.
The analogues of (a) and (b) above are all substantially easier to bound 
than the main term (where the derivative hits the phase) and the analogue of the term in (c);
below we are going to treat in full details these two types of terms.

\smallskip
According to the discussion above, 
the main terms to estimate from $\partial_{\xi} T^{\epss,(2)}_{m}(G_{\eps_1},H_{\eps_2})$ are 
\begin{align}\label{prbdary9}
\begin{split}
B_1(t,\xi) & := t \, \tau_m(t) \int_{\R^6} e^{-it \Phi_\epss(\xi,\eta,\sigma)}
  \wt{G}(t,\eta) \wt{H}(t,\sigma)  \, 
  \, b(\xi,\eta,\s) \, \mu^S(\xi,\eta,\sigma) d\eta d\sigma,
\\
& b(\xi,\eta,\sigma) := 
  \frac{\xi}{\jxi} 
  \frac{\varphi_{\leqslant M_0}(\jeta) \varphi_{\leqslant M_0}(\jsigma)
  \varphi_{\leqslant M_0}(\jxi)}{\jeta \jsigma \Phi_\epss(\xi,\eta,\s)},
\end{split}
\end{align}
and
\begin{align}\label{prbdary9'}
\begin{split}
B_2(t,\xi) & := \tau_m(t) \int_{\R^6} e^{-it \Phi_\epss(\xi,\eta,\sigma)}
  \partial_\eta \wt{G}(t,\eta) \wt{H}(t,\sigma)  \, 
  \, b'(\xi,\eta,\s) \, \mu^S(\xi,\eta,\sigma) d\eta d\sigma,
\end{split}
\end{align}
where $b'$ denotes any of the symbols that are
obtained when performing the integration by parts manipulation described above; 
see Remark \ref{remnuS} and \eqref{dximu1S},
and the exact definitions in \eqref{prbdarysym} below.


\smallskip
{\it Step 2}.
Next, we compare the formula in \eqref{prbdary9} and the identity \eqref{idbil} 
with the definitions of the operators in \eqref{idop} 
and insert additional localization in $|\eta| \approx 2^{k_1}$ and $|\s| \approx 2^{k_2}$
in the integrals in their definitions; we can then estimate
\begin{align}\label{prbdary10}
\begin{split}
& \big \Vert \whF^{-1}_{\xi\mapsto x}\big( e^{it\jxi}  B_1(t) \big) \big \Vert_{L^2} 
\lesssim \sum_{k_1,k_2 \leqslant M_0} \big \Vert t \, T_0[b_0]
  \big( e^{\eps_1 it\jxi}\wt{G_{\eps_1}}, e^{\eps_2 it\jxi}\wt{H_{\eps_2}} \big) \big \Vert_{L^2}  
  \\
  & + \sum_{\vert k_1-k_2 \vert <5, k_2 \leqslant M_0 +5} \Big( \big\Vert t \, 
  T_1^{S,1}[b_1]\big( e^{\eps_1 it\jxi}\wt{G_{\eps_1}}, e^{\eps_2 it\jxi}\wt{H_{\eps_2}} \big) \big \Vert_{L^2}
  + \big \Vert t \, T_1^{S,1}[b_1']\big( e^{\eps_2 it\jxi}\wt{H_{\eps_2}}, e^{\eps_1 it\jxi}\wt{G_{\eps_1}} \big) 
  \big \Vert_{L^2} \Big) 
  \\
  & + \sum_{k_1,k_2 \leqslant M_0} \Big( \big \Vert t \, T_1^{S,2}[b_2]
  \big( e^{\eps_1 it\jxi}\wt{G_{\eps_1}}, e^{\eps_2 it\jxi}\wt{H_{\eps_2}} \big) 
  \big \Vert_{L^2}
  \\
  & + \big \Vert t \, T_2^S[b_0] \big( e^{\eps_1 it\jxi}\wt{G_{\eps_1}}, e^{\eps_2 it\jxi}\wt{H_{\eps_2}} \big) \big \Vert_{L^2} 
  + \big \Vert t \, T_3^S[b_0]\big( e^{\eps_1 it\jxi}\wt{G_{\eps_1}}, e^{\eps_2 it\jxi}\wt{H_{\eps_2}} \big) \big \Vert_{L^2}
  \Big)
\end{split}
\end{align}
where, using a notation similar to that of \eqref{idsym} (but with the additional frequency localization) we let
\begin{align}\label{prbdary11}
\begin{split}
b_0(\xi,\eta,\s) & := b(\xi,\eta,\s) \varphi_{k_1}(\eta) \varphi_{k_2}(\s),
\\
b_1(\xi,\eta,\s) & := b(\xi,\xi-\eta,\s) \varphi_{k_1}(\eta) \varphi_{k_2}(\s) ,
\\
b_1'(\xi,\eta,\s) & := b(\xi,\s,\xi-\eta) \varphi_{k_1}(\eta) \varphi_{k_2}(\s),  
\\
b_2(\xi,\eta,\s) & := b(\xi,-\eta-\s,\s) \varphi_{k_1}(\eta) \varphi_{k_2}(\s).
\end{split}
\end{align}
Note the following:

\begin{itemize}

\item[1.] We are not indicating explicitly the dependence of the symbols in \eqref{prbdary11} on $k_1,k_2$,
but we are dealing with integrals in $d\eta d\sigma$ (as in \eqref{idop}) with localization 
to $|\eta| \approx 2^{k_1}$ and $|\s| \approx 2^{k_2}$.

\item[2.]  The restriction $\vert k_1 - k_2 \vert <5$ 
for the $T_1^{S,1}$ operators comes from the formula \eqref{nu1S}, and the indicator functions 
in the bounds \eqref{Propnu+1.1} and  \eqref{Propnu+2.1}.

\item[3.] The restriction $k_2 \leqslant M_0$ (and, similarly, the restriction $k_1 \leqslant M_0$ for the other terms) 
comes from the restriction to inputs and outputs in the symbol of \eqref{prbdary9}.
Note that, to be precise, we should write $M_0 + 3$ instead of $M_0$ in some cases, 
but we do not make such a distinction as it is clearly not relevant.

\end{itemize}



\medskip
Using \eqref{sym-Phi}, 
we see that $b$ as in \eqref{prbdary9}
satisfies \eqref{theomu1asb2} in Theorem \ref{theomu1} with the choice of $A=M_0$,
and so do all the symbols \eqref{prbdary11}. 
In particular, we can apply Theorem \ref{bilinearmeas} to the terms in \eqref{prbdary10}.

More precisely, using \eqref{mainbilin1} we can write, for $|t|\approx 2^m$, 
\begin{align}\label{prbdary12}
\begin{split}
& t \, {\big\| T_1^{S,1}[b_1]\big( e^{\eps_1 it\jxi}\wt{G_{\eps_1}}, e^{\eps_2 it\jxi}\wt{H_{\eps_2}} \big) \big\|}_{L^2} 
  \\ 
  & \lesssim 2^m
  \cdot {\big\| e^{itL} G_{ \eps_1 } \big\|}_{L^3} {\| e^{itL} P_{k_2}H_{\eps_2 } \|}_{L^{6-}} \cdot 2^{C_0M_0} 
  +  2^{-D} \cdot \mathcal{D}\big(e^{\eps_1 it\jxi}\wt{G_{\eps_1}}, e^{\eps_2 it\jxi} \wt{H_{\eps_2}}\big),
\end{split}
\end{align}
having used the boundedness of wave operators as well. 
To estimate the first term on the right-hand side of \eqref{prbdary12} for any $G,H \in \{g,h\}$,
we use the bounds \eqref{dispbootf} from Lemma \ref{dispersive-bootstrap} to get, for $|t|\approx 2^m$,
\begin{align}\label{prbdary12'}
\begin{split}
{\big\| e^{itL} G \big\|}_{L^3} {\| e^{itL} P_{k_2}H \|}_{L^{6-}} 
 & \lesssim \rho(2^m)^{1-\beta} 2^{m/2 + 2\delta_N m} \e^{\beta}
 \cdot \rho(2^m)^{1-\beta} 2^{2\delta_N m} \e^{\beta}
\\
& \lesssim \rho(2^m)^{2-2\beta} 2^{3m/4} \varepsilon^{2\beta} 
  = 2^{-5m/4} Z^2,
\end{split}
\end{align}
where $Z$ is defined in \eqref{Z}.
Note that this bound is consistent with a bound by $2^{-a m} Z^2$ for the term in \eqref{prbdary12}, with $a=1/4$,
which is an acceptable bound (see Remark \ref{remm}).

\begin{remark}\label{finitesum}
On the support of $T^{S,1}_1[b_1]$ we have that either $k_2 \sim 1$ when $H=g$,
or that $-5m \leqslant k_2 \leqslant M_0$ when $H=h$ (see Lemma \ref{verylowfreq}). 
In all cases the sum has at most $O(m)$ terms since $\vert k_1 - k_2 \vert < 5.$ 
An acceptable bound like those in Remark \ref{remm} for each fixed pair $k_1,k_2$
will then yield the desired estimates. 
The same reasoning holds for $T_1^{S,1}[b_1']$ with $k_1$ instead of $k_2$ and $G$ instead of $H.$
\end{remark}

For the $\mathcal{D}$ term in \eqref{prbdary12'}, recall the definition \eqref{bilmeasD}, 
we use the a priori bound on the weighted norm of $h$ from \eqref{boot},
and the larger bound on the weighted norm of $g$ from \eqref{growth}, as well as \eqref{inftyfreqg} and the 
$L^2$ bound on $h$ in \eqref{boot},
to see that, for all $G,H\in\{g,h\}$ and $|t|\approx 2^m$,
\begin{align}\label{prbdary15}
\begin{split}
\mathcal{D}\big(e^{\eps_1 it\jxi} P_{k_1}\wt{G_{\eps_1}}, e^{\eps_2 it\jxi} P_{k_2}\wt{H_{\eps_2}}\big) 
  & \lesssim \Big( {\big\| \partial_\xi \big(\varphi_{k_1}\wt{G_{\eps_1}} \big) \big\|}_{L^2} 
  + {\| \varphi_{k_1}\wt{G_{\eps_1}} \|}_{L^2} \Big) {\big\| \varphi_{k_2}\wt{H_{\eps_2}} \big\|}_{L^2}
  \\
  & \lesssim m^2 2^{3m/2} \rho(2^m) \cdot 
  \big[ m 2^{m} \rho(2^m) + \e^\beta 2^m \rho(2^m)^{1-\beta} \big] 
\\
  & \lesssim \rho(2^m)^{2-2\beta} 2^{5m/2+} \varepsilon^{2\beta} \lesssim 2^{m/2+} Z^2. 
\end{split}
\end{align}
This bound is consistent with the assumption on $D$ (see above \eqref{bilmeasD})
if we choose, for example, $D=5m$; indeed, our current choice is $A = M_0$, 
where $2^{M_0} \approx 2^{\delta_N m}$ (see \eqref{HLsplit}) 
with $\delta_N \ll 1$ that can be chosen sufficiently small
so to ensure $20M_0 \leqslant  5m \leqslant 2^{M_0/5}$ which suffices.
The second term on the right-hand side of \eqref{prbdary12} is then bounded by $2^{-4m} Z^2$.

The same argument used for $T_1^{S,1}[b_1]$ above applies verbatim to $T_1^{S,1}[b_1']$,
and to the simpler case of $T_0[b_0]$.
A similar argument can be applied to the operators $T_i^S[b_0], i=2,3$ as well.
The only difference is that there is no additional restriction on indices in the sum for these terms,
except for $-5m \leqslant k_1, k_2 \leqslant M_0$;
see the formulas \eqref{idop2}-\eqref{idop3}, the definition of $b_0$ in \eqref{prbdary11}
and Lemmas \ref{lemmaH} and \ref{verylowfreq}.
We can then use \eqref{mainbilin2} for each fixed pair $k_1$, $k_2$ and 
sum over these indexes at the expense of a harmless $O(m^2)$ factor.
Note that there are no $\mathcal{D}$ terms in \eqref{mainbilin2}.

For the last term $T_{1}^{S,2}$, we can use an almost identical argument 
based on the bilinear bound \eqref{mainbilin1}; 
the only minor difference is the need to additionally localize in $|\xi| \approx 2^K$ 
to be able to use \eqref{mainbilin1}. 
But this can also be easily taken care of applying Bernstein for $K\leqslant 0$ and estimating the $L^{2-}$ norm
of $P_K T_2[b_2]$, as well as using that $K \leqslant M_0$;
the (arbitrarily) small loss of $2^{M_0}$ resulting from this 
can be absorbed by the bounds \eqref{prbdary12'} and \eqref{prbdary15}.

In conclusion, 
putting together \eqref{prbdary12'} and \eqref{prbdary15}, and the similar bounds for all the other terms
on the right-hand side of \eqref{prbdary10}, we have obtained
\begin{align*}
\begin{split}
 {\| B_1(t) \|}_{L^2} 
&  \lesssim \big( 2^m 2^{C_0M_0} \cdot 2^{-5m/4} Z^2 + 2^{-4m} Z^2 \big) 2^{(0+)m}
\lesssim 2^{-m/5} Z^2,
\end{split}
\end{align*}
provided $C_0M_0$ is small enough (that is, $\delta_N$ is small enough), which is sufficient.

\smallskip
{\it Step 3.}
To estimate the $L^2$ norm of \eqref{prbdary9'} we can use very similar arguments.
Note that this terms does not carry a factor of $t$ in front, and is therefore easier to handle.
We begin with the analogue of \eqref{prbdary10}, that is,
\begin{align}\label{prbdary10'}
\begin{split}
& \big \Vert \whF^{-1}_{\xi\mapsto x}\big( e^{it\jxi}  B_2(t) \big) \big \Vert_{L^2} 
  \lesssim
  \sum_{k_1,k_2 \leqslant M_0} \big \Vert \, T_0[b_0]\big( e^{\eps_1 it\jxi} \partial_\xi \wt{G_{\eps_1}}, 
  e^{\eps_2 it\jxi}\wt{H_{\eps_2}} \big) \big \Vert_{L^2}  
  \\
  & + \big \Vert T_1^{S,2}[m_2^j]\big( e^{\eps_1 it\jxi} \partial_{\xi_j} \wt{G_{\eps_1}}, 
  e^{\eps_2 it\jxi} \wt{H_{\eps_2}} \big) \big \Vert_{L^2},
  \\
& + \big \Vert T_2^S[m_0^j] \big( e^{\eps_1 it\jxi} \partial_{\xi_j} \wt{G_{\eps_1}}, e^{\eps_2 it\jxi}\wt{H_{\eps_2}} \big) 
  \big \Vert_{L^2} 
+ \big \Vert T_3^S[m_0^j]\big( e^{\eps_1 it\jxi} \partial_{\xi_j} \wt{G_{\eps_1}}, e^{\eps_2 it\jxi}\wt{H_{\eps_2}} \big) 
  \big \Vert_{L^2}
  \\
& + \sum_{\vert k_1-k_2 \vert <5, k_2 \leqslant M_0 } \Big( 
  \big \Vert T_1^{S,1}[b_1]\big( e^{\eps_1 it\jxi} \partial_\xi \wt{G_{\eps_1}}, 
  e^{\eps_2 it\jxi}\wt{H_{\eps_2}} \big) \big \Vert_{L^2}
  \\
  & + \big \Vert\, T_1^{S,1}[m_1^j]\big( e^{\eps_2 it\jxi}\wt{H_{\eps_2}}, e^{\eps_1 it\jxi} 
  \partial_{\xi_j} \wt{G_{\eps_1}} \big) \big \Vert_{L^2} \Big)
\end{split}
\end{align}
where we are implicitly summing over $j=1,2,3$ in the last three lines, 
and the symbols are
\begin{align}\label{prbdarysym}
\begin{split}
m_0^{j}(\xi,\eta,\s) & = \frac{\xi}{\vert \xi \vert} \frac{\eta_j}{\vert \eta \vert} b_0 (\xi,\eta,\s),
\qquad 
m_1^j(\xi,\eta,\s) = \frac{\eta}{|\eta|} \frac{\s_j}{\vert \s \vert} b_1'(\xi,\eta,\s), 
\\ 
m_2^{j}(\xi,\eta,\s) & = 
  \frac{\xi}{|\xi|} \frac{\eta_j}{\vert \eta \vert} 
  b_2 (\xi,\eta,\s);
\end{split}
\end{align}
see Remark \ref{remnuS} and the definitions \eqref{prbdary11}.

To estimate the terms in \eqref{prbdary10'} we use the bilinear bounds \eqref{mainbilinest} 
proceeding as follows:
For the case when $G=h$ in \eqref{prbdary10'} we do an $L^2 \times L^6$ estimate
using the a priori bound on $\partial_\xi\wt{h}$ from \eqref{boot} 
and the $L^6$ decay bounds in Lemma \ref{dispersive-bootstrap};
for the case when $G=g$ in \eqref{prbdary10'} we can use instead \eqref{growth} 
and again the $L^6$ estimates in Lemma \ref{dispersive-bootstrap}.
The term $\mathcal{D}$ that appears on the right-hand side of \eqref{mainbilin1}
is a fast decaying remainder that
can be handled similarly to \eqref{prbdary15}; exploiting the minimum in the definition
\eqref{bilmeasD} we only differentiate the input that does not already carry the weight,
so to never have the appearance of $\partial_\xi^2\wt{h}$.

\medskip
{\it Proof of \eqref{lembdary2}}.
In view of Lemma \ref{basicL8'} it suffices to prove, see the formula \eqref{Tij2}, that
\begin{align}\label{lembdary2pr1}
{\big\| \whF^{-1}_{\xi\mapsto x} J(t) \big\|}_{L^{8/7}_x} \lesssim \rho_0^a,
\end{align}
where
\begin{align}\label{lembdary2pr2}
\begin{split}
J(t) & = \tau_m(t) 
  \int_{\R^6} \frac{
   \varphi_{\leqslant M_0}(\jxi) \varphi_{\leqslant M_0}(\jeta) \varphi_{\leqslant M_0}(\jsigma)}{
  \Phi_\epss(\xi,\eta,\sigma) \jeta \jsig} 
  \\
  & \qquad \times e^{i\eps_1 t\jeta}\wt{G}(t,\eta) e^{i\eps_1 t\jsig}\wt{H}(t,\sigma) 
  \, \mu^S(\xi,\eta,\sigma) d\eta d\sigma.
\end{split}
\end{align}

Choosing $b$ as in \eqref{prbdary9}, but without the $\xi/\jxi$ factor, 
and with the same definitions \eqref{prbdary11}, we have identities analogous to \eqref{prbdary10}.
It follows that
\begin{align}\label{prbdary10''}
\begin{split}
& \big \Vert \whF^{-1}_{\xi\mapsto x} J(t) \big \Vert_{L^{8/7}_x} \lesssim 
 \\
&  \sum_{k_1,k_2 \leqslant M_0} \big \Vert T_0[b_0]\big( e^{\eps_1 it\jxi}\wt{G_{\eps_1}}, 
 e^{\eps_2 it\jxi}\wt{H_{\eps_2}} \big) \big \Vert_{L^{8/7}_x}
+ \big \Vert T_2^S[b_0] \big( e^{\eps_1 it\jxi}\wt{G_{\eps_1}}, 
 e^{\eps_2 it\jxi}\wt{H_{\eps_2}} \big) \big \Vert_{L^{8/7}_x}
\\
& + \big \Vert T_3^S[b_0]\big( e^{\eps_1 it\jxi}\wt{G_{\eps_1}}, 
  e^{\eps_2 it\jxi}\wt{H_{\eps_2}} \big) \big \Vert_{L^{8/7}_x} 
+ \big \Vert T_1^{S,2}[b_2]\big( e^{\eps_1 it\jxi}\wt{G_{\eps_1}}, 
  e^{\eps_2 it\jxi}\wt{H_{\eps_2}} \big) \big \Vert_{L^{8/7}_x}
  \\
& 
+ \sum_{\vert k_1-k_2 \vert <5, k_2 \leqslant M_0 } \big \Vert T_1^{S,1}[b_1]\big( e^{\eps_1 it \jxi} \wt{G_{\eps_1}}, 
  e^{\eps_2 it\jxi}\wt{H_{\eps_2}} \big) \big \Vert_{L^{8/7}_x}
  + \big \Vert T_1^{S,1}[b_1']\big( e^{\eps_2 it\jxi}\wt{H_{\eps_2}}, 
  e^{\eps_1 it\jxi}\wt{G_{\eps_1}} \big) \big \Vert_{L^{8/7}_x}. 
\end{split}
\end{align}

The estimate \eqref{mainbilin1} gives us
\begin{align}\label{lembdary2pr3}
\begin{split}
&  \big \Vert T_1^{S,1}[b_1]\big( e^{\eps_1 it \jxi} \wt{G_{\eps_1}}, e^{\eps_2 it\jxi}\wt{H_{\eps_2}} \big) \big \Vert_{L^{8/7}_x}
  \\
  & \lesssim {\big\| e^{itL} G \big\|}_{L^{8/3-}} {\| e^{itL} P_{k_2}H \|}_{L^{2}} \cdot 2^{(C_0+1)M_0} 
  +  2^{-D} \cdot \mathcal{D}\big(e^{\eps_1 it\jxi}\wt{G_{\eps_1}}, e^{\eps_2 it\jxi} \varphi_{k_2} \wt{H_{\eps_2}}\big). 
\end{split}
\end{align}
Then we can use the $L^q$ dispersive estimate in Lemma \ref{dispersive-bootstrap}
to obtain (recall Remark \ref{finitesum}), for all $t\approx 2^m$,
\begin{align*}
 {\big\| e^{itL} G \big\|}_{L^{8/3-}} {\| e^{itL} P_{k_2}H \|}_{L^{2}}  
  & \lesssim \rho(2^m)^{1-\beta} 2^{5m/8 + 10\delta_N m} \e^\beta \cdot \rho(2^m)^{1-\beta} 2^m \e^\beta
  \\
  & \lesssim \rho(2^m)^{2-2\beta} 2^{7m/4} \e^{2\beta} \lesssim \e^{2a},
\end{align*} 
for some $a > 2\beta$, which is consistent with \eqref{lembdary2pr1}.
All the other terms in \eqref{prbdary10''} can be estimated similarly as above, 
using \eqref{mainbilin0}-\eqref{mainbilin2}.
The $\mathcal{D}$ factor is estimated by \eqref{prbdary15}.
These bounds and \eqref{lembdary2pr3} 
give us \eqref{lembdary2pr1} for $a$ small enough, provided $M_0$ is small enough.
\end{proof}

\medskip
In the remainder of this section we estimate the integrated terms \eqref{Tij1} 
by distinguishing the three different types of interactions described in \S\ref{types}.
Recall the definition of the bulk operator \eqref{Tij1},
and that we are aiming to prove the $L^2$ weighted bound \eqref{proFSest1}
and the decay bound \eqref{proFSest2} (via Lemma \ref{basicL8}).



\medskip
\subsection{Fermi interactions}\label{secSgg}
We first look at the interactions that involve only the Fermi component $g$.

\begin{lemma}\label{gg-no-t-res}
Under the assumptions of Proposition \ref{proFS}, for $\epss \in \{+,-\}$, 
we have
\begin{align}
\label{ggbd1}
& \big \Vert \partial_{\xi} T^{\epss,(1)}_{m 
  }(g,g) \big \Vert_{L^2} 
  \lesssim 2^{m} \rho^{1-\beta + \delta'}(2^m) \e^{\beta},
\\
\label{ggbd2}
&  {\big\| \mathcal{F}^{-1}_{\xi \mapsto x} e^{it\jxi} 
  \partial_t T^{\epss,(1)}_{m 
  }(g,g) \big\|}_{L^{8/7}_x} 
  \lesssim \frac{\rho_0^{1/2}}{\langle t \rangle^{11/8}},
\end{align}
for some $\delta' >0$.
\end{lemma}

\begin{proof}

\medskip
{\it Proof of \eqref{ggbd1}}.
Arguing as at the beginning of the proof of Lemma \ref{bdry-no-t-res}
we can reduce to bounding the $L^2$-norm of the two main terms,
that is those where $\partial_\xi$ hits the exponential in the formula for $T^{\epss,(1)}_{m 
}(g,g)$, see \eqref{Tij1},
and where the derivatives hits $\wt{g}$.
Let us denote these two types of terms by
\begin{align}\label{ggbdpr0.5}
C_1(t,\xi) := C_{1,\epss}[g_{\eps_1},g_{\eps_2}](t,\xi), \qquad C_{2}(t,\xi) := C_{2,\epss}[g_{\eps_1},g_{\eps_2}](t,\xi),
\end{align}
where we define
\begin{align}\label{ggbdpr1}
\begin{split}
C_{1,\epss}[G,H](t,\xi) := \int_0^t s \, \tau_m(s) \int_{\R^6} e^{-is \Phi_\epss(\xi,\eta,\sigma)}
  \partial_s \wt{G}(s,\eta) \, \wt{H}(s,\sigma)  \, 
  \\
  \times \, b(\xi,\eta,\s) \, \mu^S(\xi,\eta,\sigma) d\eta d\sigma \, ds ,
\\
C_{2,\epss}[G,H](t,\xi) := \int_0^t \tau_m(s) \int_{\R^6} e^{-is \Phi_\epss(\xi,\eta,\sigma)}
  \partial_s \wt{G}(s,\eta) \, \partial_\s \wt{H}(s,\sigma)  \, 
  \\
  \times \, b'(\xi,\eta,\s) \, \mu^S(\xi,\eta,\sigma) d\eta d\sigma \, ds,
\end{split}
\end{align}
and $b,b'$ are defined as in \eqref{prbdary9}, \eqref{prbdary9'} (see also
the more precise \eqref{prbdarysym}).
Notice that, as done before, when reducing to the terms \eqref{ggbdpr0.5}-\eqref{ggbdpr1} we have
taken into account the case where $\partial_\xi$ hits the distribution $\mu^S$; 
in this case the derivative can be converted
into $\partial_\s$ through formulas like \eqref{dxinuS}-\eqref{dximu1S} with the roles of $\eta$ and $\sigma$ exchanged,
and an integration by parts in $\s$ can be performed without the need to differentiate $\partial_s \wt{g}$ in $\eta$.

\medskip
{\it Estimates for the $C_1$ terms}.
We can proceed similarly to the proof of Lemma \ref{bdry-no-t-res}.
In particular, we let $b$ be the same symbol defined in \eqref{prbdary9}
and use the same notation in \eqref{prbdary11}
so that we can estimate similarly to \eqref{prbdary10}.
As in \eqref{prbdary11}, we introduce frequency localizations 
in $\eta, \s$.
Notice that since $g$ is localized at frequency $\sim 1$ (see \eqref{def-h}),
the sums over $k_1$ and $k_2$ are finite for all the terms except $T_1^{S,2}.$ 
We have 
\begin{align}\label{ggbdpr3}
\begin{split}
& {\Big\|\whF^{-1}_{\xi\mapsto x} C_1(t) \Big\|}_{L^2}
  \lesssim 2^{2m} \sup_{k_1,k_2 \approx 1} \sup_{s\approx 2^m} 
  \Big[ {\big\| T_0[b_0]\big( e^{\eps_1 is\jxi}\partial_s\wt{g_{\eps_1}}, 
    e^{\eps_2 is\jxi}\wt{g_{\eps_2}} \big) \big\|}_{L^2} 
  \\
  & + \sum_{i=2,3} {\big\| T_i^S[b_0]\big( e^{\eps_1 is\jxi}\partial_s\wt{g_{\eps_1}}, 
    e^{\eps_2 is\jxi}\wt{g_{\eps_2}} \big) \big\|}_{L^2} 
  \\
  & + {\big\| T_1^{S,1}[b_1]\big( e^{\eps_1 is\jxi}\partial_s\wt{g_{\eps_1}}, 
    e^{\eps_2 is\jxi}\wt{g_{\eps_2}} \big) \big\|}_{L^2}  + {\big\| T_1^{S,1}[b_1']\big( e^{\eps_2 is\jxi}\wt{g_{\eps_2}},  e^{\eps_1 is\jxi}\partial_s \wt{g_{\eps_1}} \big) \big\|}_{L^2} 
\\   
    &+ 2^{2m} \sup_{s \approx 2^m}  \sum_{k_1 <5, k_2 \approx 1} 
    {\big\| T_1^{S,2}[b_0]\big( e^{\eps_1 is\jxi} \partial_s \wt{g_{\eps_1}}, 
    e^{\eps_2 is\jxi}\wt{g_{\eps_2}} \big) \big\|}_{L^2} \Big].
\end{split}
\end{align}

Let us look at the $T_1^{S,1}$ term on the right-hand side above first. Using \eqref{mainbilin1} we have
\begin{align}\label{ggbdpr4}
\begin{split}
{\big\| T_1^{S,1}[b_1]\big( e^{\eps_1 is\jxi}\partial_s\wt{g_{\eps_1}}, e^{\eps_2 is\jxi}\wt{g_{\eps_2}} \big) \big\|}_{L^2} 
  & \lesssim {\big\| e^{isL} P_{k_1}\partial_s g \big\|}_{L^{3-}} 
    {\| e^{isL} P_{k_2}g \|}_{L^6} \cdot 2^{C_0M_0} 
  \\
  & +2^{-D} \cdot \mathcal{D}\big(e^{\eps_1 is\jxi}\varphi_{k_1}\partial_s\wt{g_{\eps_1}}, 
    e^{\eps_2 is\jxi} \varphi_{k_2}\wt{g_{\eps_2}}\big).
\end{split}
\end{align}
From the dispersive estimates \eqref{dispbootg} and \eqref{estdsg}, we see that, for all $s \approx 2^m$,
\begin{align*}
\begin{split}
{\big\| e^{isL} \partial_s g \big\|}_{L^{3-}} 
  {\| e^{isL} g \|}_{L^6} \cdot 2^{C_0M_0} 
   & \lesssim \rho(2^m)^2 \cdot m \cdot 2^{C_0 \delta_N m} 
   \\
   & \lesssim 2^{-2m} \cdot \rho(2^m)^{\beta/2} m 2^{C_0 \delta_N m} 
   \cdot 2^m \rho(2^m) \cdot Z,
\end{split}
\end{align*}
having used $\e^{-\beta} \rho(2^m)^{\beta/2} \lesssim 1$.
Multiplied by the factor of $2^{2m}$ which appears on the right-hand side of \eqref{ggbdpr3},
the above bound is acceptable (see Remark \ref{remm}) provided $\delta_N$ is small enough. 

To estimate the second line of \eqref{ggbdpr4} we use \eqref{growth} 
to bound $\partial_s \wt{g}$ in $L^2$, 
and use again \eqref{estdsg} to see that, recall the definition \eqref{bilmeasD}, 
\begin{align}
\mathcal{D}\big(e^{\eps_1 is\jxi}\varphi_{k_1}\partial_s\wt{g_{\eps_1}}, 
    e^{\eps_2 is\jxi} \varphi_{k_2}\wt{g_{\eps_2}}\big) \lesssim \rho(2^m)^2 \cdot 2^{3m/2}m^2.
\end{align}
We can then pick, say, $D = 5m$ (here $A = M_0 = m\delta_N$) so that the second line of \eqref{ggbdpr4}
is a fast decaying remainder.

Bounds like those just proved for 
$T_1^{S,1}[b_1]\big( e^{\eps_1 is\jxi}\partial_s\wt{g_{\eps_1}}, e^{\eps_2 is\jxi}\wt{g_{\eps_2}} \big)$
can be obtained identically for all the other terms in \eqref{ggbdpr3}, which concludes the estimate for $C_1$.

\medskip
{\it Estimates for the $C_2$ terms}.
For the second integral in \eqref{ggbdpr1}, using the same notation as above (see \eqref{prbdarysym}), 
we see that
\begin{align}\label{ggbdpr10}
\begin{split}
& {\Big\|\whF^{-1}_{\xi\mapsto x} C_2(t) \Big\|}_{L^2}
  \lesssim 2^{m} \sup_{k_1,k_2 \approx 1} \sup_{s\approx 2^m} 
  \Big[ \big \Vert T_0[b_0]\big( e^{\eps_1 is\jxi} \partial_s \wt{g_{\eps_1}}, 
  e^{\eps_2 is\jxi} \partial_\xi \wt{g_{\eps_2}} \big) \big \Vert_{L^2}
  \\
& + \big \Vert T_2^S[m_0^j]\big( e^{\eps_1 is\jxi} \partial_s \wt{g_{\eps_1}},
  e^{\eps_2 is\jxi} \partial_{\xi_j} \wt{g_{\eps_2}} \big) \big \Vert_{L^2} 
  + \big \Vert T_3^S[m_0^{j}] 
  \big( e^{\eps_1 is\jxi} \partial_s \wt{g_{\eps_1}}, 
  e^{\eps_2 is\jxi} \partial_{\xi_j} \wt{g_{\eps_2}} \big) \big \Vert_{L^2}
  \\
  & +\big \Vert T_1^{S,1}[b_1]\big( e^{\eps_1 is\jxi} \partial_s \wt{g_{\eps_1}}, 
  e^{\eps_2 is\jxi} \partial_\xi \wt{g_{\eps_2}} \big) \big \Vert_{L^2}
  + 
  \big \Vert T_1^{S,1}[m_1^j] \big( e^{\eps_2 is\jxi}\partial_{\xi_j} \wt{g_{\eps_2}}, 
  e^{\eps_1 is\jxi}\partial_s \wt{g_{\eps_1}} \big) \big \Vert_{L^2} \Big]
  \\
  & +2^m \sup_{s \approx 2^m} \sum_{k_1 <5, k_2 \approx 1}
  \big \Vert T_1^{S,2}[m_2^j]\big( e^{\eps_1 is\jxi} \partial_s \wt{g_{\eps_1}}, 
  e^{\eps_2 is\jxi} \partial_{\xi_j} \wt{g_{\eps_2}} \big) \big \Vert_{L^2}.
  \end{split}
\end{align}
Once again it suffices to look at the $T_1^{S,1}$ term above, since the others are identical;
the bilinear estimate \eqref{mainbilin1} and the boundedness of wave operators give
\begin{align}\label{ggbdpr4'}
\begin{split}
2^m {\| T_1^{S,1}[b_1]\big( e^{\eps_1 is\jxi}\partial_s\wt{g_{\eps_1}}, 
  e^{\eps_2 is\jxi} \partial_\xi \wt{g_{\eps_2}} \big) \|}_{L^2} 
 & \lesssim 2^m {\big\| e^{isL} \partial_s g \big\|}_{L^{2-}} 
    {\| \whF^{-1} e^{\eps_2is\jxi} (\partial_\xi \wt{g}) \|}_{L^\infty} \cdot 2^{C_0M_0} 
  \\
  & + 2^m  2^{-D} \cdot 
    \mathcal{D}\big(e^{\eps_1 is\jxi} \partial_s\wt{g_{\eps_1}}, 
    e^{\eps_2 is\jxi} \partial_\xi\wt{g_{\eps_2}}\big).
\end{split}
\end{align}
The first term above can be bounded using \eqref{estdsg} and \eqref{dxiwtgest}:
\begin{align*}
2^m {\big\| e^{isL} \partial_s g \big\|}_{L^{2-}} 
    {\| e^{\eps_2isL} \whF^{-1} (\partial_\xi \wt{g}) \|}_{L^\infty} \cdot 2^{C_0M_0} 
    & \lesssim 2^m \rho(2^m)  \cdot {\| \partial_\xi \wt{g} \|}_{L^1_\xi} \cdot 2^{C_0 \delta_N m} 
    \\
    & \lesssim \big( 2^m \rho(2^m) \big)^2 2^{(C_0+1) \delta_N m}
    \lesssim 2^{(C_0+1) \delta_N m} \rho(2^m)^{\beta/2} Z,
\end{align*}
which suffices for our purposes (see \eqref{Z1} and Remark \ref{remm}). 

Next we estimate the remainder. We can pick $D = 5m$ as before and see that the last term on
the right-hand side of \eqref{ggbdpr4'} is fast decaying:
using \eqref{dxi2g} and \eqref{estdsg} we have
\begin{align}\label{Dgg}
\begin{split}
2^{-D} \mathcal{D}\big(e^{\eps_1 is\jxi} \partial_s\wt{g_{\eps_1}}, 
    e^{\eps_2 is\jxi} \partial_\xi \wt{g_{\eps_2}}\big) 
    & \lesssim 2^{-5m} \cdot {\| \partial_s\wt{g_{\eps_1}} \|}_{L^2} \cdot {\| \partial_\xi^2\wt{g_{\eps_1}} \|}_{L^2} 
    \\
    & \lesssim \rho(2^m)^2 2^{-2m} \lesssim 2^{-4m} Z.
\end{split}
\end{align}
This completes the estimate \eqref{ggbd1}.

\medskip
{\it Proof of \eqref{ggbd2}}.
For the second estimate, using the same notation for $b$ as in \eqref{prbdary9} 
and \eqref{prbdary11}, but without the $\xi/\jxi$ factor, 
we can write
\begin{align*}
&\sup_{t\approx 2^m} {\Big\| \whF^{-1}_{\xi \mapsto x} e^{it\jxi} 
  \partial_t T^{\epss,(1)}_{m 
  } (g,g) \Big\|}_{L^{8/7}_x} \lesssim
  \sup_{t \approx 2^m} \sup_{k_1,k_2 \approx 1} 
  \Big[ \sum_{i=0,2,3} {\big\| T_i^S[b_0]\big( e^{\eps_1 it \jxi}\partial_t\wt{g_{\eps_1}}, 
  e^{\eps_2 it\jxi}\wt{g_{\eps_2}} \big) \big\|}_{L^{8/7}_x} 
  \\
  & + {\big\| T_1^{S,1}[b_1]\big( e^{\eps_1 it\jxi}\partial_t\wt{g_{\eps_1}}, 
  e^{\eps_2 it\jxi}\wt{g_{\eps_2}} \big) \big\|}_{L^{8/7}_x}  
  + {\big\| T_1^{S,1}[b_1']\big( e^{\eps_2 it\jxi}\wt{g_{\eps_2}},
  e^{\eps_1 it\jxi}\partial_t \wt{g_{\eps_1}} \big) \big\|}_{L^{8/7}_x} \Big]
\\   
  & + \sup_{t \approx 2^m} \sum_{k_1 <5, k_2 \approx 1}
  {\big\| T_1^{S,2}[b_0]\big( e^{\eps_1 it\jxi} \partial_t \wt{g_{\eps_1}}, 
  e^{\eps_2 it\jxi}\wt{g_{\eps_2}} \big) \big\|}_{L^{8/7}_x}.
\end{align*}
We focus again on the $T_1^{S,1}[b_1]$ term since the others can be treated identically.
Using \eqref{mainbilin1}, and \eqref{Dgg},
\begin{align*}
\sup_{t\approx 2^m} {\big\| T_1^{S,1}[b_1]\big( e^{\eps_1 it\jxi}\partial_t\wt{g_{\eps_1}}, 
    e^{\eps_2 it\jxi}\wt{g_{\eps_2}} \big) \big\|}_{L^{8/7}_x}
  & \lesssim 2^{C_0\delta_N m} 
  \sup_{t\approx 2^m} {\big\| e^{itL} \partial_t g \big\|}_{L^{24/17-}_x} {\big\| e^{itL} g \big\|}_{L^6_x}
  \\
  & \lesssim
2^{C_0\delta_N m} \cdot \rho(2^m) \cdot \rho(2^m)m,
\end{align*}
having used \eqref{estdsg} and \eqref{dispbootg} in the last inequality.
This is sufficient to obtain a bound as in \eqref{ggbd2}.
\end{proof}


\medskip
\subsection{Fermi-continuous interaction}\label{secShg}
The goal of this subsection is to prove the following lemma
which deals with the bilinear (and trilinear) interaction of the good (continuous) and bad (Fermi) components:

\begin{lemma}\label{gh-no-t-res}
Under the assumptions of Proposition \ref{proFS}, for all $\epss \in \{+,-\}$, we have
\begin{align}\label{ghbd1}
\big \Vert \partial_{\xi} T^{\epss,(1)}_{m}(g,h) \big \Vert_{L^2_x}
  + \big \Vert \partial_{\xi} T^{\epss,(1)}_{m}(h,g) \big \Vert_{L^2_x}
  \lesssim 2^{m} \rho^{1-\beta+\delta'}(2^m) \e^{\beta},
\end{align}
for some $\delta'>0$,
and
\begin{align}\label{ghbd2}
{\big\| \mathcal{F}^{-1}_{\xi \mapsto x} e^{it\jxi}
  \partial_t T^{\epss,(1)}_{m}(g,h) \big\|}_{L^{8/7}_x}
  + {\big\| \mathcal{F}^{-1}_{\xi \mapsto x} e^{it\jxi}
  \partial_t T^{\epss,(1)}_{m}(h,g) \big\|}_{L^{8/7}_x} 
  \lesssim 2^{3m/4} \rho^{2-3\beta}(2^m). 
\end{align}
\end{lemma}

\begin{proof}
We adopt similar notation to those in Lemmas \ref{bdry-no-t-res} and \ref{gg-no-t-res} above.

\medskip
{\it Proof of \eqref{ghbd1}}.
Applying again the same reduction from the beginning of Lemma \ref{bdry-no-t-res},
and using the notation from \eqref{ggbdpr1}, we can 
focus on estimating the $L^2$-norms of the terms
\begin{align}\label{ghbdpr1}
\begin{split}
C_{1,\epss}[g_{\eps_1},h_{\eps_2}], \qquad C_{1,\epss}[h_{\eps_1},g_{\eps_2}],
\\
C_{2,\epss}[g_{\eps_1},h_{\eps_2}], \qquad C_{2,\epss}[h_{\eps_1},g_{\eps_2}].
\end{split}
\end{align}

\smallskip
{\it Estimates for the $C_1$ terms}.
We have analogues of the estimate \eqref{ggbdpr3}, with the definitions \eqref{idop} and \eqref{prbdary11}, 
for all the terms in \eqref{ghbdpr1}.
In particular, we see that the first term is bounded as follows: 
\begin{align}\label{ghbdpr3}
\begin{split}
& {\Big\| \whF^{-1}_{\xi\mapsto x} C_{1,\epss}[g_{\eps_1},h_{\eps_2}](t) \Big\|}_{L^2_x}
  \lesssim  
  \Big\| \int_0^t s \tau_m(s) e^{-isL} \Big[ 
  \sum_{k_1,k_2 \leqslant M_0} T_0[b_0]\big( e^{\eps_1 is\jxi}\partial_s\wt{g_{\eps_1}}, 
    e^{\eps_2 is\jxi}\varphi_{k_2}\wt{h_{\eps_2}} \big)
  \\
  & + T_1^{S,2}[b_0]\big( e^{\eps_1 is\jxi} \partial_s \wt{g_{\eps_1}}, 
    e^{\eps_2 is\jxi}\varphi_{k_2}\wt{h_{\eps_2}} \big)
    + \sum_{i=2,3} T_i^S[b_0]\big( e^{\eps_1 is\jxi}\partial_s\wt{g_{\eps_1}}, 
    e^{\eps_2 is\jxi}\varphi_{k_2}\wt{h_{\eps_2}} \big)
  \\
  & + \sum_{\vert k_1 - k_2 \vert<5,k_2 \leqslant M_0} 
  \Big( T_1^{S,1}[b_1]\big( e^{\eps_1 is\jxi}\partial_s\wt{g_{\eps_1}}, 
    e^{\eps_2 is\jxi}\varphi_{k_2}\wt{h_{\eps_2}} \big) 
  + T_1^{S,1}[b_1']
    \big( e^{\eps_2 is\jxi} 
    \wt{h_{\eps_2}}, e^{\eps_1 is\jxi}\partial_s \wt{g_{\eps_1}} \big) \Big)  \Big]\, ds \Big\|_{L^2}.
\end{split}
\end{align}

As before it suffices to treat the contribution from $T_1^{S,1}[b_1]$ since the others are almost identical.
An analogue of Remark \ref{finitesum} holds in this case and 
therefore we may reduce to sum over $k_1,k_2$ with only $O(m)$ terms.
Note also that a direct application of the bilinear bounds of Theorem \ref{bilinearmeas},
as done before in Lemma \ref{gg-no-t-res},  would not suffice in this case.
We first apply the Strichartz estimate from Lemma \ref{Strichartz}
with the indices $(\wt{q},\wt{r},\gamma)=(2+,\infty-,0+),$ 
and then use \eqref{mainbilin1} to obtain: 
\begin{align}\label{ghbdpr4}
\begin{split}
& {\Big\| \int_0^t s \tau_m(s) e^{-isL} T_1^{S,1}[b_1]\big( e^{\eps_1 is\jxi}\partial_s\wt{g_{\eps_1}}, 
    e^{\eps_2 is\jxi} \varphi_{k_2}\wt{h_{\eps_2}} \big) \, ds \Big\|}_{L^2_x}
  \\ & \lesssim 2^{3m/2+} \cdot  \sup_{s\approx 2^m} 
  {\big\| T_1^{S,1}[b_1]\big( e^{\eps_1 is\jxi}\partial_s\wt{g_{\eps_1}},
    e^{\eps_2 is\jxi} \varphi_{k_2}\wt{h_{\eps_2}} \big) \big\|}_{L^{1+}_x}
  \\
  & \lesssim  2^{3m/2+} \sup_{s\approx 2^m} \Big[ {\big\| e^{isL} \partial_s g \big\|}_{L^{6/5-}} 
    {\| e^{isL} P_{k_2} h \|}_{L^6} \cdot 2^{C_0M_0} 
  + 2^{-D} \cdot 
    \mathcal{D}\big(e^{\eps_1 is\jxi} \partial_s\wt{g_{\eps_1}}, 
    e^{\eps_2 is\jxi} \varphi_{k_2}\wt{h_{\eps_2}}\big) \Big].
\end{split}
\end{align}
Using \eqref{estdsg} and \eqref{dispbootf} 
we obtain an estimate for the first term on the last line of \eqref{ghbdpr4} by
\begin{align*}
C 2^{3m/2} \cdot 2^{(C_0+1)\delta_N m} \cdot \rho(2^m) \cdot \rho(2^m)^{1-\beta} \e^{\beta}  
\end{align*}
which is sufficient for \eqref{ghbd1}.
To estimate the contribution from $\mathcal{D}$ (recall \eqref{bilmeasD}) 
we can use once again the admissible choice $D=5m$,
and the a priori bound \eqref{boot} on $\|\partial_\xi \wt{h}\|_{L^2}$ and \eqref{estdsg}.

To treat the term $ C_{1,\epss}[h_{\eps_1},g_{\eps_2}]$ we instead
use first the Strichartz estimates with indices $(\widetilde{q},\widetilde{r},\gamma) = (6,3,2/3)$:
\begin{align}\label{ghbdpr5}
\begin{split}
& {\Big\| \int_0^t s \tau_m(s) e^{isL} T_1^{S,1}[b_1]\big( e^{\eps_1 is\jxi}\partial_s \wt{h_{\eps_1}}, 
    e^{\eps_2 is\jxi} 
    \wt{g_{\eps_2}} \big) \, ds \Big\|}_{L^2_x}
  \\ 
  & \lesssim 2^{11m/6 + \delta_N m} \cdot \sup_{s\approx 2^m} 
  {\big\| T_1^{S,1}[b_1]\big( e^{\eps_1 is\jxi} \partial_s \wt{h_{\eps_1}},
    e^{\eps_2 is\jxi} 
    \wt{g_{\eps_2}} \big) \big\|}_{L^{3/2}_x}
  \\
  & \lesssim 2^{11m/6 + \delta_N m} \cdot \sup_{s\approx 2^m} \Big[ 
    {\| e^{isL} P_{k_2} \partial_s h \|}_{L^{2+}} {\big\| e^{isL} g \big\|}_{L^{6-}} \cdot 2^{C_0M_0} 
  \\ 
  & \qquad + 2^{-D} \cdot 
    \mathcal{D}\big(e^{\eps_1 is\jxi} \wt{g_{\eps_1}}, 
    e^{\eps_2 is\jxi} \varphi_{k_2} \partial_s\wt{h_{\eps_2}}\big) \Big].
\end{split}
\end{align}
Using \eqref{estdsh}, together with Bernstein and the fact that the frequencies of $h$ are bounded by $C2^{M_0}$,
and \eqref{dispbootg},
\begin{align*}
\begin{split}
{\big\| e^{isL} 
  \partial_s h \big\|}_{L^{2+}} 
  {\big\| e^{isL}g \big\|}_{L^{6-}} \lesssim \rho(2^m) \cdot \rho(2^m) \cdot 2^{2\delta_N m} 
  \lesssim 2^{-2m + 2\delta_N m} \rho(2^m)^{\beta/2} Z ,
\end{split}
\end{align*}
which is sufficient for the desired bound, as per the usual Remark \ref{remm}.
Using also \eqref{growth} we can easily handle the remainder term $\mathcal{D}$ in \eqref{ghbdpr5}
by making the usual choice $D = 5m$.

\smallskip
{\it Estimates for the $C_2$ terms}.
Let us now look at the $C_2$ terms from \eqref{ghbdpr1}; recall the definition from \eqref{ggbdpr1}. 
We can proceed similarly to the proof above. 
First, with the usual notation, we write the analogue of \eqref{ghbdpr3} 
(see also \eqref{ggbdpr10}) for the first term: 
\begin{align}\label{ghbdpr10}
\begin{split}
& {\Big\|\whF^{-1}_{\xi\mapsto x} C_{2,\epss}[g_{\eps_1},h_{\eps_2}](t) \Big\|}_{L^2}
  \lesssim
  \Big\| \int_0^t \tau_m(s) e^{-is L} \Big[ \sum_{k_1,k_2 \leqslant M_0} 
  T_0[b_0]\big( e^{\eps_1 is\jxi}\partial_s\wt{g_{\eps_1}}, 
    e^{\eps_2 is\jxi} \varphi_{k_2}\partial_\xi \wt{h_{\eps_2}} \big)
  \\
  & + \sum_{i=2,3} T_i^S[m_0^j]\big( e^{\eps_1 is\jxi}\partial_s\wt{g_{\eps_1}}, 
    e^{\eps_2 is\jxi} \varphi_{k_2}\partial_{\xi_j} \wt{h_{\eps_2}} \big) 
    + T_1^{S,2}[m_2^j]\big( e^{\eps_1 is\jxi} \partial_s \wt{g_{\eps_1}}, 
    e^{\eps_2 is\jxi} \varphi_{k_2}\partial_{\xi_j} \wt{h_{\eps_2}} \big) 
  \\
  & + \sum_{\vert k_1-k_2 \vert <5, k_2 \leqslant M_0}
  \Big( T_1^{S,1}[b_1]\big( e^{\eps_1 is\jxi}\partial_s\wt{g_{\eps_1}}, 
    e^{\eps_2 is\jxi} \varphi_{k_2}\partial_\xi \wt{h_{\eps_2}} \big)
  \\
  & + T_1^{S,1}[m_1^j]\big( e^{\eps_2 is\jxi}\varphi_{k_2}\partial_{\xi_j} \wt{h_{\eps_2}}, 
    e^{\eps_1 is\jxi}\partial_s \wt{g_{\eps_1}} \big) \Big)
  \Big] \, ds \Big\|_{L^2}.
\end{split}
\end{align}

Focusing on the usual main contribution,
an application of the Strichartz estimate with $(\wt{q},\wt{r},\gamma)=(6,3,2/3)$
and \eqref{mainbilin1} give us 
\begin{align}\label{ghbdpr11}
\begin{split}
& {\Big\| \int_0^t \tau_m(s) T_1^{S,1}[b_1]\big( e^{\eps_1 is\jxi}\partial_s\wt{g_{\eps_1}}, 
  e^{\eps_2 is\jxi} \partial_\xi \wt{h_{\eps_2}} \big) \, ds \Big\|}_{L^2_x} 
  \\
  & \lesssim 2^{5m/6} \cdot 2^{M_0} \cdot \sup_{s\approx 2^m} 
  {\big\| T_1[b_1]\big( e^{\eps_1 is\jxi} \partial_s\wt{g_{\eps_1}},
    e^{\eps_2 is\jxi} \varphi_{k_2} \partial_\xi \wt{h_{\eps_2}} \big) \big\|}_{L^{3/2}_x}
  \\
  & \lesssim 2^{5m/6} \cdot 2^{M_0} \cdot \sup_{s\approx 2^m} \Big[ {\big\| e^{isL} \partial_s g \big\|}_{L^{6-}} 
    {\big\| \whF^{-1} e^{\eps_2 is\jxi}(\varphi_{k_2} \partial_\xi \wt{h_{\eps_2}} ) \big\|}_{L^{2+}} \cdot 2^{C_0M_0} 
  \\
  & \qquad \qquad \qquad \qquad + 2^{-D} \cdot 
    \mathcal{D}\big(e^{\eps_1 is\jxi} \partial_s \wt{g_{\eps_1}}, 
    e^{\eps_2 is\jxi} \varphi_{k_2} \partial_\xi\wt{h_{\eps_2}}\big) \Big].
\end{split}
\end{align}
Using \eqref{estdsg} and \eqref{boot} we can bound by an absolute constant
the sum over $k_2$ of first term in the last expression above; 
using \eqref{estdsg} and the bound 
\begin{align}\label{estdsdxig}
{\| \partial_\xi \partial_s \wt{g} \|}_{L^2} \lesssim 2^m \rho(2^m)
\end{align}
we can bound the remainder term in \eqref{ghbdpr11} with the usual choice $D=5m$.

For the last term $C_{2,\epss}[h_{\eps_1},g_{\eps_2}]$ from \eqref{ghbdpr1}
(recall the notation \eqref{ggbdpr1}) the roles of $h$ and $g$ are reversed
compared to the term just treated above. 
Starting with the analogue of \eqref{ghbdpr10} matters are reduced to bounding
\begin{align}\label{ghbdpr12}
\begin{split}
& {\Big\| \int_0^t \tau_m(s) e^{isL} T_1^{S,1}[b_1]\big( e^{\eps_1 is\jxi}\varphi_{k_1}\partial_s\wt{h_{\eps_1}}, 
  e^{\eps_2 is\jxi} \partial_\xi \wt{g_{\eps_2}} \big) \, ds \Big\|}_{L^2_x}.
\end{split}
\end{align}
Using \eqref{mainbilin1} we upperbound the above term by
\begin{align*}
\begin{split}
C 2^m \cdot \sup_{s\approx 2^m} \Big[ {\big\| e^{isL} P_{k_1} \partial_s h \big\|}_{L^2_x} 
    {\| \partial_\xi \wt{g} \|}_{L^{1+}_{\xi}} \cdot 2^{C_0M_0} 
  +2^{-D} \cdot 
    \mathcal{D}\big(e^{\eps_1 is\jxi}\varphi_{k_1} \partial_s \wt{h_{\eps_1}}, 
    e^{\eps_2 is\jxi} \partial_\xi\wt{g_{\eps_2}}\big) \Big].  
\end{split}
\end{align*}
Then we can use 
\eqref{estdsh}, and interpolate \eqref{growth} with \eqref{dxiwtgest} to estimate 
the sum over $k_1$ of the first term above by
\begin{align*}
 C 2^m \cdot 2^{(C_0+) \delta_N m} \cdot \rho(2^m) \cdot 2^m\rho(2^m) &  
 \lesssim  2^m \rho(2^m) \cdot 2^{(C_0+) \delta_N m} \rho(2^m)^{\beta/2} \cdot Z 
 \\ & \lesssim 2^{(C_0+) \delta_N m} \rho(2^m)^{\beta/2} \cdot Z, 
\end{align*}
where we used that $2^m \rho(2^m) \leqslant Z \rho(2^m)^{\beta/2}$ for the first line, and $2^m \rho(2^m) \lesssim 1$ for the second. Given our choices of parameters, this bound suffices by Remark \ref{remm}. 

Using \eqref{dxi2g} we can take care of the remainder in the usual way.
This concludes the proof of \eqref{ghbd1}.

\medskip
{\it Proof of \eqref{ghbd2}}.
Moving on to the second estimate, we adopt the same notation above
and proceed similarly to the proof of \eqref{ggbd2}.
For the first term on the left-hand side of \eqref{ghbd2} we can write up to irrelevant constants
\begin{align*}
&\whF^{-1}_{\xi \mapsto x} e^{it\jxi} \partial_t T^{\epss,(1)}_{m} (g,h)
   = \sum_{k_1,k_2 \leqslant M_0} \Big[ T_0[b_0]\big( e^{\eps_1 it\jxi} \partial_t\wt{g_{\eps_1}}, 
   e^{\eps_2 it\jxi} \varphi_{k_2}\wt{h_{\eps_2}}\big) 
   \\
   & + \sum_{i=2,3}
   T_i^S[b_0]\big( e^{\eps_1 it\jxi} \partial_t\wt{g_{\eps_1}}, e^{\eps_2 it\jxi} \varphi_{k_2}\wt{h_{\eps_2}}\big)
   + T_1^{S,2}[b_0]\big( e^{\eps_1 it\jxi} \partial_t\wt{g_{\eps_1}}, 
  e^{\eps_2 it\jxi} \varphi_{k_2}\wt{h_{\eps_2}} \big) \Big]
   \\
& + \sum_{\vert k_1- k_2 \vert <5, k_2 \leqslant M_0} 
  \Big[ T_1^{S,1}[b_1]\big( e^{\eps_1 it\jxi} \partial_t\wt{g_{\eps_1}}, e^{\eps_2 it\jxi} \varphi_{k_2}\wt{h_{\eps_2}} \big)
  \\
  & + T_1^{S,1}[b_1']\big( e^{\eps_2 it\jxi} \varphi_{k_2}\wt{h_{\eps_2}}, 
  e^{\eps_1 it\jxi} \partial_t\wt{g_{\eps_1}} \big) \Big]
  .
\end{align*}
We focus again on the $T_1^{S,1}[b_1]$ term since the others can be treated identically, 
using Bernstein's inequality to ensure convergence of the sums.
Using \eqref{mainbilin1} 
with the choice of $D=5m$, we have, up to a faster decaying term,
\begin{align*}
\sup_{t \approx 2^m} {\Big\| \whF^{-1}_{\xi \mapsto x} e^{it\jxi} 
  \partial_t T^{\epss,(1)}_{m} (g,h) \Big\|}_{L^{8/7}_x} 
  & \lesssim 2^{C_0\delta_N m}
  \sup_{t\approx 2^m} {\big\| e^{itL} \partial_t g \big\|}_{L^{24/17}_x} {\big\| e^{itL} h \big\|}_{L^{6-}_x}
  \\
 & \lesssim  2^{C_0 \delta_N m} \cdot \rho(2^m) \cdot \rho(2^m)^{1-\beta} 2^{10 \delta_N m} \varepsilon^{\beta},
\end{align*}
having used \eqref{estdsg} and \eqref{dispbootf}.
This suffices for \eqref{ghbd2}.

A similar estimate can be used to get the same bound when the roles of $g$ and $h$ are exchanged:
using Theorem \ref{bilinearmeas} followed by \eqref{estdsh} and \eqref{dispbootgbis}, we have
\begin{align*}
\sup_{t\approx 2^m} \sum_{k_1}  {\Big\| \whF^{-1}_{\xi \mapsto x} e^{it\jxi} 
  \partial_t T^{\epss,(1)}_{m,k_1,k_2} (h,g) \Big\|}_{L^{8/7}_x} 
  & \lesssim 2^{C_0\delta_N m}
  \sup_{t\approx 2^m} \sum_{k_1}  {\big\| \partial_t P_{k_1} h \big\|}_{L^{2+}_x} {\big\| e^{itL} g \big\|}_{L^{8/3-}_x}
  \\
  & \lesssim 2^{(C_0+2)\delta_N m} \cdot \rho(2^m) \cdot \rho(2^m) 2^{m} \cdot 2^{-3m/8+} .
\end{align*}
This suffices to bound the second term on the left-hand side  
of \eqref{ghbd2} as desired, and concludes the proof of the lemma.
\end{proof}

\medskip
\subsection{Continuous interactions}\label{secShh}
We now deal with terms that are quadratic in $h.$ 
Note that these are 
easier than the ones treated before since $h$ 
essentially satisfies better estimates than $g$ does.

\begin{lemma}\label{hh-no-t-res}
For all $\epss \in \{+,-\}$ we have
\begin{align}\label{hhbd1}
  \big \Vert \partial_{\xi} T^{\epss,(1)}_{m}(h,h) \big \Vert_{L^2_x}
  \lesssim \e^{\beta} 2^{m} \rho^{1-\beta+\delta'}(2^m),
\end{align}
for some $\delta'>0,$ and
\begin{align}\label{hhbd2}
  {\big\| \mathcal{F}^{-1}_{\xi \mapsto x} e^{it\jxi}
  \partial_t T^{\epss,(1)}_{m}(h,h) \big\|}_{L^{8/7}_x}
  \lesssim \frac{\rho_0^{a}}{\jt^{5/4}}
\end{align}
for some $a>0$. 
\end{lemma}

\begin{proof}
The proofs can be done similarly to those in Lemma \ref{gh-no-t-res}, and are in fact easier.
We provide some details for completeness.
%
Following the same steps in the proof of \eqref{ghbd1} 
we can write an estimate like \eqref{ghbdpr3} and reduce to estimating a term like the one in \eqref{ghbdpr4}
with $\wt{g_{\eps_1}}$ replaced by $\varphi_{k_1}\wt{h_{\eps_1}}$;
for such a term, using Strichartz estimates with $(\wt{q},\wt{r},\gamma)=(6,3,2/3)$,
and Theorem \ref{bilinearmeas} we have
\begin{align}\label{hhbdpr4}
\begin{split}
& {\Big\| \int_0^t s \tau_m(s) e^{-isL} T_1^{S,1}[b_1]\big( e^{\eps_1 is\jxi}\partial_s\wt{h_{\eps_1}}, 
    e^{\eps_2 is\jxi} \varphi_{k_2}\wt{h_{\eps_2}} \big) \, ds \Big\|}_{L^2_x}
  \\ & \lesssim 2^{11m/6 + \delta_Nm} \cdot  \sup_{s\approx 2^m} 
  {\big\| T_1[b_1]\big( e^{\eps_1 is\jxi}\varphi_{k_1}\partial_s\wt{h_{\eps_1}},
    e^{\eps_2 is\jxi} \varphi_{k_2}\wt{h_{\eps_2}} \big) \big\|}_{L^{3/2}_x}
  \\
  & \lesssim  2^{11m/6 + \delta_Nm}  
    \sup_{s\approx 2^m} \Big[ {\big\| e^{isL} P_{k_1} \partial_s h \big\|}_{L^{2+}} 
    {\| e^{isL} P_{k_2} h \|}_{L^{6-}} \cdot 2^{C_0M_0} 
  \\
  & \qquad \qquad \qquad \qquad + 2^{-D} \cdot 
    \mathcal{D}\big(e^{\eps_1 is\jxi} \varphi_{k_1} \partial_s\wt{h_{\eps_1}}, 
    e^{\eps_2 is\jxi} \varphi_{k_2} \wt{h_{\eps_2}}\big) \Big].
\end{split}
\end{align}
From this we can deduce the desired bound, using the decay from \eqref{estdsh} and \eqref{dispbootf},
and handling the remainder in the usual way.


The estimate on the $C_2$ type terms follows by estimating
\begin{align}\label{hhbdpr12}
\begin{split}
& \sum_{\vert k_1-k_2 \vert <5, k_2 \leqslant M_0} {\Big\| \int_0^t \tau_m(s) e^{isL} T_1^{S,1}[b_1]\big( e^{\eps_1 is\jxi}\varphi_{k_1}\partial_s\wt{h_{\eps_1}}, 
  e^{\eps_2 is\jxi} \partial_\xi \wt{h_{\eps_2}} \big) \, ds \Big\|}_{L^2_x}.
\end{split}
\end{align}
This term is analogous to \eqref{ghbdpr12}, but here the bound
follows directly after using the Strichartz estimate with indices $(\widetilde{q},\widetilde{r},\gamma) = (6,3,2/3)$,
followed by an $L^{6-}\times L^{2+}$ product estimate via \eqref{mainbilin1}, 
Bernstein,  \eqref{estdsh} and \eqref{boot}. 

For \eqref{hhbd2}, using Theorem \ref{bilinearmeas} we can bound the term on the left-hand side,
up to a faster decaying remainder, by
\begin{align*}
C \sum_{k_1,k_2 \leqslant M_0} 
  \sup_{\approx 2^m} {\| \partial_t P_{k_1} h(t) \|}_{L^{2+}} {\| e^{itL} P_{k_2} h(t) \|}_{L^{8/3-}} \cdot 2^{C_0\delta_N m} 
  \\
 \lesssim 2^{(C_0+1)\delta_N m} \cdot \rho(2^m) \cdot \rho(2^m)^{1-\beta} 2^{5/8m + 10 \delta_N m} \varepsilon^{\beta},
\end{align*}
having used Bernstein and the bounds \eqref{estdsh} and \eqref{dispbootf} for the last inequality.
This suffices for \eqref{hhbd2} and concludes the proof of the Lemma.
\end{proof}

\smallskip
With the proofs of Lemmas \ref{bdry-no-t-res}, \ref{gg-no-t-res}, \ref{gh-no-t-res} and \ref{hh-no-t-res} 
(see also Lemmas \ref{basicL8} and \ref{basicL8'}),
the proof of Proposition \ref{proFS} is complete.
To finish the proof of the main Propositions \ref{mainquad} and \ref{mainquadbound},
we are left with estimating the regular terms, see \eqref{FS+FR}.
We proceed to do this in the next section.


\medskip
\section{Regular quadratic terms}\label{secFR}


In this section we estimate the bilinear terms corresponding to the regular parts of the distribution. 
Looking at \eqref{FS+FR} we define
\begin{align} \label{defopFR}
\begin{split}
T_m^{\epss,(3)} (G,H) & := 
\int_0 ^t \tau_m(s) \int_{\R^6} \frac{e^{-is \Phi_{\epss}(\xi,\eta,\sigma)}}{
  \jeta \jsig} \wt{G}(s,\eta) \wt{H}(s,\sigma) 
  \\
& \times  \varphi_{\leqslant M_0}(\jeta) \varphi_{\leqslant M_0} (\jsigma)  
  \varphi_{\leqslant M_0}(\jxi)  \mu^R(\xi,\eta,\sigma) 
  \, d\eta d\sigma ds,
\\
T_m^{\epss,(4)} (G,H) & := 
\int_0 ^t \tau_m(s) \int_{\R^6} \frac{e^{-is \Phi_{\epss}(\xi,\eta,\sigma)}}{
  \jeta \jsig} \wt{G}(s,\eta) \wt{H}(s,\sigma)  
  \\
& \times \varphi_{\leqslant M_0}(\jeta) \varphi_{\leqslant M_0} (\jsigma) \varphi_{\leqslant M_0}(\jxi)  \mu^{Re}(\xi,\eta,\sigma) 
  \, d\eta d\sigma ds,
\end{split}
\end{align}
where the distributions are defined in Proposition \ref{mudecomp}
and the phase is as in \eqref{secF0'}.
The goal of this section is to prove the analogue of Proposition \ref{proFS} for the regular terms defined above:

\begin{proposition}\label{proFR}
Under the assumptions \eqref{boot}, for all $m=0,1,...,L+1,$ 
for all $\epss \in \lbrace +,- \rbrace,$ we have for some $0<\delta<\beta$, and $j=3,4$:
\begin{align} \label{proFRest1}
\Vert \nabla_{\xi} T_{m}^{\epss,(j)} (G_{\epsilon_1}, H_{\epsilon_2}) \Vert_{L^2} 
  \lesssim \rho(2^m)^{1-\beta+\delta} 2^m \e^{\beta} , \qquad G,H \in \lbrace g,h \rbrace.
\end{align}
Moreover, there exist $0<a<1$ and $\delta'>0$ such that 
\begin{align}\label{proFRest2}
\bigg \Vert e^{itL} \wt{\mathcal{F}}^{-1}_{\xi \mapsto x} \big[ T_{m}^{\epss,(j)} 
  (G_{\epsilon_1},H_{\epsilon_2})   \big] \bigg \Vert_{L^{\infty}_x} 
  \lesssim \frac{\rho_0^{a}}{\jt^{1+a+\delta'}}, \qquad G,H \in \lbrace g,h \rbrace.
\end{align} 
\end{proposition}

With Proposition \ref{proFS} and Lemmas \ref{lemmaH} and \ref{verylowfreq},
this will complete the proof of Propositions \ref{mainquad} and \ref{mainquaddecay}.

\medskip
\subsection{Preliminaries}
We first reduce the proof of \eqref{proFRest1} and \eqref{proFRest2} to some 
relatively easier weighted $L^2$ bounds.
We start by reducing \eqref{proFRest2} to $L^p$ estimates for some $p>6/5$,
through the following analogue of Lemma \ref{basicL8}:

\begin{lemma} \label{L8-space}
Assume that there exist $0<a < b <1/10$, such that
\begin{align}\label{L8-spaceas}
\big \Vert \widehat{\mathcal{F}}^{-1}_{\xi \mapsto x} I(s) \big \Vert_{L^{6/(5+6b)}_x}  
  \lesssim \frac{\rho_0^a}{\langle s \rangle^{1+2a}},
\end{align}
where
\begin{align} \label{defI}
I(s) := e^{is \jxi} \partial_s T^{\epss,(j)}_{m} (G_{\epsilon_1},H_{\epsilon_2}), \qquad j=3,4.
\end{align}
Then 
\begin{align*}
{\Big\| e^{itL} \wt{\mathcal{F}}^{-1}_{\xi \mapsto x} \big[T_{m}^{\epss,(j)}(G_{\epsilon_1},H_{\epsilon_2})  
  \big] \Big\|}_{L^{\infty}_x} \lesssim \frac{\rho_0^a}{\langle t \rangle^{1+a+\delta'}}, \qquad j=3,4.
\end{align*}
\end{lemma}

\begin{proof}
The proof is almost identical to the one of Lemma \ref{basicL8}. 
In this case we first use Bernstein to go from $L^\infty$ to $L^{q}$ with $q:=(1/6-b)^{-1}$.
Notice that $q' = (5/6+b)^{-1} = 6/(5+6b)$, consistently with the assumption \eqref{L8-spaceas}.
Using that the linear decay rate for the propagator $e^{itL}$ in $L^q$ is 
$\jt^{-(3/2)(1-2/q)} = \jt^{-1-3b}$ and that $b>a$, we obtain the desired conclusion.
\end{proof}

We will also need a bilinear lemma for operators with symbol $\partial_{\xi} \nu_1^R$.
Recall the definition \eqref{idop11}-\eqref{idop12};
in what follows, whenever needed, we will write explicitly the dependence on the measures 
by denoting 
\begin{align}
\label{idop11R}
\begin{split}
T_1^{R,1}[\nu;b](g,h)(x) & := \whF^{-1}_{\xi\rightarrow x} \iint_{\R^3\times\R^3} g(\xi-\eta) h(\sigma) 
  \,b(\xi,\eta,\sigma)\, \nu(\eta,\sigma) \, d\eta d\sigma,
\\
T_1^{R,2}[\nu;b](g,h)(x) &:= \whF^{-1}_{\xi\rightarrow x} \iint_{\R^3\times\R^3} g(-\eta-\sigma) h(\sigma) 
  \,b(\xi,\eta,\sigma)\, \overline{\nu(\eta,\xi)} \, d\eta d\sigma,
\\
T[\mu,b](g,h)(x) & := \whF^{-1}_{\xi\rightarrow x} \iint_{\R^3\times\R^3} g(\eta) h(\sigma) 
  \, b(\xi,\eta,\sigma)\, \mu(\xi,\eta,\sigma) \, d\eta d\sigma.
\end{split}
\end{align}
We then have the following analogue of Theorem \ref{bilinearmeas}:

\begin{lemma} \label{bilindxiR}
With the definitions \eqref{idop11R}, for $b$ (and $A$) as in the statement of Theorem \ref{bilinearmeas},
and for all $p,q \geqslant 1, r>1$ with $1/p+1/q>1/r$, we have
\begin{align}\label{bilindxiR1}
{\big\| T_1^{R,1}[\partial_1 \nu^R_1;b] (G,H) \big\|}_{L^r}
+ {\big\| T_1^{R,2}[\partial_2\nu^R_1;b] (G,H) \big\|}_{L^r} 
  \lesssim 2^{C_0A} \Vert \widehat{G} \Vert_{L^p} \Vert \widehat{H} \Vert_{L^q},
\end{align}
for some absolute constant $C_0$.
Moreover, under the same assumptions, we have
\begin{align}
\label{bilindxiR2} 
{\big\| T[\partial_1 \mu^{Re};b] (G,H) \big\|}_{L^r}  \lesssim 2^{C_0A} 
 \Vert \widehat{G} \Vert_{L^p} \Vert \widehat{H} \Vert_{L^q}.
\end{align}
\end{lemma}

\begin{proof}
\eqref{bilindxiR1} follows exactly as in the proof of \eqref{mainbilinR},
that is, from the standard bilinear Lemma 6.3 in \cite{PS} 
and the estimate \eqref{nuRest} for the derivative of $\nu^R_1$.
\eqref{bilindxiR2} follows similarly using instead the estimate \eqref{nuReest}.
\end{proof}

Next, let us introduce the operators
which will appear when applying $\nabla_\xi$ to \eqref{defopFR}. Let 
\begin{align}\label{ggbdpr1R}
\begin{split}
D_{1,m,\epss}(G,H)
  & := \int_0^t s \, \tau_m(s) \int_{\R^6} e^{-is \Phi_\epss(\xi,\eta,\sigma)}
  \wt{G}(s,\eta) \, \wt{H}(s,\sigma)  \, 
  \, b(\xi,\eta,\s) \, \mu^R(\xi,\eta,\sigma) d\eta d\sigma \, ds,
\\
D_{2,m,\epss}(G,H) 
  & := \int_0^t \tau_m(s) \int_{\R^6} e^{-is \Phi_\epss(\xi,\eta,\sigma)}
 \wt{G}(s,\eta) \, \wt{H}(s,\sigma)  \, 
  \, b'(\xi,\eta,\s) \, \nabla_{\xi} \mu^R(\xi,\eta,\sigma) d\eta d\sigma \, ds, 
\end{split}
\end{align}
where
\begin{align}\label{ggbdpr1Rsym}
b'(\xi,\eta,\s) & := \frac{\varphi_{\leqslant M_0}(\jeta) \varphi_{\leqslant M_0}(\jsigma) 
  \varphi_{\leqslant M_0}(\jxi)}{\jeta \jsigma}, \qquad  b(\xi,\eta,\s) := \frac{\xi}{\jxi} b'(\xi,\eta,\s),
\end{align}
and, for any $k_1,k_2 \leqslant M_0$
\begin{align}\label{ggbdpr1R'}
\begin{split}
E_{1,k_1,k_2,m,\epss}(G,H)
  & := \int_0^t s \, \tau_m(s) \int_{\R^6} e^{-is \Phi_\epss(\xi,\eta,\sigma)}
  \wt{G}(s,\eta) \, \wt{H}(s,\sigma)  \, 
  \\
  & \qquad \qquad \times \, b_{Re}(\xi,\eta,\s) \, \mu^{Re}(\xi,\eta,\sigma) d\eta d\sigma \, ds,
\\
E_{2,k_1,k_2,m,\epss}(G,H)
  & := \int_0^t \tau_m(s) \int_{\R^6} e^{-is \Phi_\epss(\xi,\eta,\sigma)}
 \wt{G}(s,\eta) \, \wt{H}(s,\sigma)  \, 
  \\
  & \qquad \qquad \times \, b'_{Re}(\xi,\eta,\s) \, \nabla_{\xi} \mu^{Re}(\xi,\eta,\sigma) d\eta d\sigma \, ds,
\end{split}
\end{align}
where 
\begin{align*}
b'_{Re}(\xi,\eta,s) & := \frac{\varphi_{k_1}(\eta) \varphi_{k_2}(\sigma)
  \varphi_{\leqslant M_0}(\jxi)}{\jeta \jsigma}, \qquad b_{Re}(\xi,\eta,\s) := \frac{\xi}{\jxi} b'_{Re}(\xi,\eta,s).
\end{align*}
Note that we are 
omitting the variables $(t,\xi)$ from the above expressions, as this should cause no confusion.
We are also omitting the $k_1,k_2$ dependence of the last two symbols.

By definition we have
\begin{align}\label{iddxiFR}
\begin{split}
\nabla_\xi T_m^{\epss,(3)} (G,H) & = D_{1,m,\epss}[G,H](t,\xi) + D_{2,m,\epss}[G,H](t,\xi) 
\\
\nabla_\xi T_m^{\epss,(4)} (G,H) & = \sum_{k_1,k_2\leqslant M_0} 
  \Big( E_{1,k_1,k_2,m,\epss}(G,H) + E_{2,k_1,k_2,m,\epss}(G,H) \Big).
\end{split}
\end{align}
We can obviously bound $D_{2,m,\epss}$ and $E_{2,k_1,k_2,m,\epss}$, respectively, by the quantities
\begin{align}\label{D2'}
D_{2,m,\epss}'(G,H)
  := 2^m \sup_{s \approx 2^m}
  \Big| \int_{\R^6} e^{is(\eps_1\jeta + \eps_2 \jsig)}
 \wt{G}(s,\eta) \, \wt{H}(s,\sigma)  \, 
  \, b'(\xi,\eta,\s) \, \nabla_{\xi} \mu^R(\xi,\eta,\sigma) d\eta d\sigma \Big|
\end{align}
and
\begin{align}\label{E2'}
\begin{split}
E_{2,k_1,k_2,m,\epss}'(G,H)
  := 2^m \sup_{s \approx 2^m}
  \Big| \int_{\R^6} e^{is(\eps_1\jeta + \eps_2 \jsig)}
  \wt{G}(s,\eta) \, \wt{H}(s,\sigma)  \, 
  \\ \times \, b'_{Re}(\xi,\eta,\s) \, \nabla_{\xi} \mu^{Re}(\xi,\eta,\sigma) d\eta d\sigma 
  \Big|. 
\end{split}
\end{align}


We are now ready to reduce the proof of Proposition \ref{proFR}
to proving bounds on the operators defined above.

\begin{lemma}\label{lemgoal}
With the definitions \eqref{ggbdpr1R}, \eqref{ggbdpr1R'} and \eqref{D2'}-\eqref{E2'},
assume that for all $m =0, 1,...,L+1$ 
and all $G,H \in \{ g,h \}$, the following estimates hold: 
\begin{align} 
\label{goalR} 
\Vert D_{1,m,\epss}(G_{\eps_1},H_{\eps_2})
  \Vert_{L^2}, 
  \sum_{k_1,k_2 \leqslant M_0} \Vert E_{1,k_1,k_2,m,\epss}(G_{\eps_1},H_{\eps_2}) \Vert_{L^2} & 
  \lesssim \rho(2^m)^{1-\beta+\delta'} 2^m \varepsilon^{\beta},
  \\
\label{goalRbis} 
\Vert D_{2,m,\epss}'(G_{\eps_1},H_{\eps_2})
  \Vert_{L^2}, 
  \sum_{k_1,k_2 \leqslant M_0} \Vert E_{2,k_1,k_2,m,\epss}'(G_{\eps_1},H_{\eps_2}) \Vert_{L^2} & 
  \lesssim 
 2^{-m/10} \cdot \rho(2^m) 
  2^m \varepsilon^{\beta},
\end{align}
for some $\delta' >\delta$.
Then, \eqref{proFRest1} and \eqref{proFRest2} hold.
\end{lemma}

Notice the extra factor of $2^{-m/10}$ in \eqref{goalRbis}.

\begin{proof}
It is clear from \eqref{iddxiFR} and the definitions \eqref{D2'} and \eqref{E2'} 
that \eqref{goalR}-\eqref{goalRbis} imply \eqref{proFRest1}.
To show that \eqref{goalR}-\eqref{goalRbis} imply \eqref{proFRest2}
we use Lemma \ref{L8-space} and an interpolation argument. 
First, we observe that 
\begin{align*}
{\| f \|}_{L^{\frac{6}{5+6b}}_x} 
  \lesssim {\| \jx f \|}_{L^2_x}^{\alpha(b)} {\| \jx^2 f \|}_{L^2_x}^{1-\alpha(b)},
\end{align*}
where $0<\alpha(b) \rightarrow 1$ as $b\rightarrow 0$.
Then, using Lemma \ref{L8-space} we see that \eqref{proFRest2} would follow from
(we disregard the $L^2$-norms without derivatives which satisfy stronger bounds)
\begin{align}\label{goalpr1}
  {\| \nabla_\xi I \|}_{L^2_x}^{\alpha(b)} \cdot {\| \nabla_\xi^2 I \|}_{L^2_x}^{1-\alpha(b)}
  \lesssim \frac{\rho_0^a}{\langle s \rangle^{1+2a}}, \qquad a\in(0,b),
\end{align}
where
\begin{align} \label{I} 
I(s) := e^{is \jxi} \partial_s T^{\epss,(j)}_{m} (G_{\epsilon_1},H_{\epsilon_2}), \qquad j=3,4.
\end{align}
Let us just concentrate on the case $j=3$, as the case $j=4$ is identical thanks 
to the estimate satisfied by $\mu^{Re}$; see \eqref{nuReest}.
In particular, the definition \eqref{D2'} and the assumption \eqref{goalRbis} imply that, for all $s\approx 2^m$,
\begin{align}\label{goalpr3} 
{\| \nabla_\xi I(s) \|}_{L^2} 
  \lesssim 
  2^{-m/10} \cdot \rho(2^m) 
  \varepsilon^{\beta}.
\end{align}
Therefore, since $\alpha(b)$ can be made sufficiently close to $1$ 
by choosing $b$ small (so that $a$ is a also sufficiently small)
in order to obtain \eqref{goalpr1} it would suffice 
to show that
\begin{align}\label{goalprmain}
\sup_{s\approx 2^m}{\| \nabla_\xi^2 I(s) \|}_{L^2_x} \lesssim 2^{3m}.
\end{align}
This is of course far from optimal but suffices for our purpose since any growth in the above right-hand side
can be offset by an appropriate choice of small enough $b$.


To prove \eqref{goalprmain} we start by differentiating \eqref{defI}. 
We can disregard the easier terms where $\nabla_\xi$ hits the symbol $b'$.
Using \eqref{mudecomp0}-\eqref{mu1SR} to expand $\mu^R=\mu^R_1$, 
exactly as we did for $\mu_1^S$ in \eqref{ggbdpr3}, we find that
\begin{align}\label{goalpr4}
\begin{split}
& \Bigg \Vert \int_{\R^6} e^{is(\eps_1\jeta + \eps_2 \jsig)} 
 \wt{G}(s,\eta) \, \wt{H}(s,\sigma)  \, 
  \, b'(\xi,\eta,\s) \, \nabla_{\xi}^2 \mu^R(\xi,\eta,\sigma) d\eta d\sigma \Bigg \Vert_{L^2}
\\
& 
 \lesssim \sup_{s \approx 2^m} \sum_{k_1,k_2 \leqslant M_0} 
 \big \Vert T^{R,1}_1[ \partial_1^2 \nu_1^R,b_1] (e^{is \eps_1 \jxi} \wt{G_{\eps_1}}, 
 e^{is \epsilon_2 \jxi} \wt{H_{\eps_2}}) \big \Vert_{L^2} 
 \\
 & + \big \Vert T^{R,1}_1 [ \partial_1^2 \nu_1^R;b_1'] (e^{is \epsilon_2 \jxi} \wt{H_{\eps_2}},
 e^{is \eps_1 \jxi} \wt{G_{\eps_1}}) \big \Vert_{L^2}
 \\
& + \big \Vert T_1^{R,2}[ \partial_2^2 \nu_1^R,b_2](e^{is \epsilon_1 \jxi} \wt{G_{\eps_1}},
  e^{is \epsilon_2 \jxi} \wt{H_{\eps_2}}) \big \Vert_{L^2},
\end{split}
\end{align}
where we are adopting the notation \eqref{idop11R},
using $\partial_j \nu_1^R$ to denote differentiation in the $j$-th variable,
and where the symbols are defined, as in \eqref{prbdary11}, by

\begin{align} \label{symbolnuR}
\begin{split}
& b_1(\xi,\eta,\s) = b'(\xi,\xi-\eta,\s)\varphi_{k_1}(\eta) \varphi_{k_2}(\s),
\\
& b_1'(\xi,\eta,\s) = b'(\xi,\xi-\s,\eta) \varphi_{k_1}(\eta) \varphi_{k_2}(\s),
\\ 
& b_2(\xi,\eta,\s) = b'(\xi,-\eta-\s,\s) \varphi_{k_1}(\eta) \varphi_{k_2}(\s). 
\end{split}
\end{align}
We bound the expressions in \eqref{goalpr4} using
the pointwise bound \eqref{nuRest} which, in particular, implies, for all $|p|\approx 2^P$
and $|q|\approx 2^Q$ with $P,Q \leqslant M_0 + 10$ that
\begin{align}\label{nuRestsimp}
\begin{split}
& \vert \nabla_{p}^2 \nu_1^R (p,q) \vert \lesssim 2^{-4 (P \vee Q)} \cdot 2^{22M_0},
\\
& \vert \nabla_{q}^2 \nu_1^R (p,q) \vert \lesssim 2^{-2 (P \vee Q)} \cdot 2^{-Q} \cdot 2^{22M_0}.
\end{split}
\end{align}

Recalling the definition \eqref{idop11R},
using H\"older in $\xi$, and bounding ${\|b'\|}_{L^\infty}\lesssim 1$, we find that, for $s \approx 2^m$,
\begin{align}
\nonumber
& 
\big \Vert T^{R,1}_1[ \partial_1^2 \nu_1^R;b_1] (e^{is \epsilon_1 \jxi} \wt{G_{\eps_1}}, 
  e^{is \epsilon_2 \jxi} \wt{H_{\eps_2}}) \big \Vert_{L^2} 
  \\ \nonumber
& \lesssim 
  2^{3M_0/2}
  \cdot \sup_{\vert \xi \vert \lesssim 2^{M_0}} \bigg \vert \int_{\mathbb{R}^6} 
  e^{is(\eps_1 \langle \xi-\eta \rangle + \eps_2 \jsig)} b_1(\xi,\eta,\s) \wt{G_{\eps_1}}(s,\xi-\eta) 
  \wt{H_{\eps_2}}(s,\s) \, \nabla^2_{\eta} \nu_1^R(\eta,\s) \, d\eta d\s \bigg \vert 
  \\ \nonumber
  & \lesssim 2^{3M_0/2} \sup_{\vert \xi \vert \lesssim 2^{M_0}}
  \Big( \int_{\R^3} |\wt{G_{\eps_1}}(\xi-\eta)| \varphi_{k_1}(\eta)\, d\eta \Big)
  \Big( \int_{\R^3} |\wt{H_{\eps_2}}(\s)|\varphi_{k_2}(\s) d\s \Big) \cdot 2^{-4 (k_1 \vee k_2)} 2^{22M_0}
  \\ \nonumber
& \lesssim 2^{25M_0} \cdot 2^{-4(k_1 \vee k_2)} \cdot
  2^{5k_1/2} \Vert \wt{G_{\eps_1}} \Vert_{L^6} \cdot 2^{5k_2/2} \Vert \wt{H_{\eps_2}} \Vert_{L^6} 
  \\
& \lesssim 2^{25M_0} \cdot 2^{(5/2)(k_1+k_2)}  \cdot 2^{-4 (k_1 \vee k_2)} 
  \cdot \rho(2^m)^2 \cdot 2^{3m}m^4,
\label{goalpr5}
\end{align}
where we used Sobolev's embedding to get the crude bound
\begin{align*}
\Vert \wt{F}(\xi) \Vert_{L^6} 
  \lesssim 
  \Vert \nabla_{\xi} \wt{F} \Vert_{L^2}
  \lesssim 
  \rho(2^m) 2^{3m/2} m^2, \qquad F \in \lbrace g,h \rbrace,
\end{align*}
see \eqref{boot} and \eqref{growth}.

Similarly, we also have 
\begin{align}\label{goalpr6}
\begin{split}
& \big \Vert T^{R,1}_1 [\partial_1^2 \nu_1^R;b_1'] 
  (e^{is \epsilon_2 \jxi} \wt{H},e^{is \epsilon_1 \jxi} \wt{G}) \big \Vert_{L^2} 
  \\
  & \lesssim 
  2^{(5/2)(k_1 + k_2)} \cdot 2^{-4(k_1 \vee k_2)} 2^{25M_0} \cdot \rho(2^m)^2 \cdot 2^{3m} m^4.
\end{split}
\end{align}
The estimates \eqref{goalpr5}-\eqref{goalpr6} can then be summed over $k_1,k_2 \leqslant M_0$ to obtain
a bound consistent with \eqref{goalprmain}.

Finally, recalling the definition \eqref{idop11R} and the second estimate in \eqref{nuRestsimp} 
we get, for $k\leqslant M_0$, $s\approx 2^m$,
\begin{align*}
& \big \Vert P_k T_1^{R,2}[ \partial_2^2 \nu_1^R;b_2](e^{is \epsilon_1 \jxi} \wt{G_{\eps_1}}, 
  e^{is \epsilon_2 \jxi} \wt{H_{\eps_2}}) \big \Vert_{L^2} 
  \\
  & \lesssim  2^{3k/2}
  \cdot \sup_{\vert \xi \vert \approx 2^k} \bigg \vert \int_{\mathbb{R}^6} 
  |\wt{G_{\eps_1}}(s,-\eta-\s) | \,
  |\wt{H_{\eps_2}}(s,\s)| \, |\nabla^2_{\xi} \nu_1^R(\eta,\xi)| \, \varphi_{k_1}(\eta)\varphi_{k_2}(\sigma) 
  \, d\eta d\s \bigg \vert
  \\
  & \lesssim 
  2^{3k/2} 
  \cdot 2^{5k_1/2} \Vert \wt{G_{\eps_1}} \Vert_{L^6} 
  \cdot 2^{5k_2/2} \Vert \wt{H_{\eps_2}} \Vert_{L^6} 
  \cdot 2^{-2 (k \vee k_2)} 2^{-k} 2^{22M_0}
  \\
& \lesssim 2^{k/2} 2^{22M_0} \cdot 2^{5k_1/2} 2^{k_2/2} 
  \cdot \rho(2^m)^2 \cdot 2^{3m}m^4.
\end{align*}
This can be easily summed over the indexes $k,k_1,k_2$ to obtain a bound as in \eqref{goalprmain}.
We have then obtained \eqref{goalprmain} and concluded the proof of the lemma.
\end{proof}

\medskip
In view of Lemma \ref{lemgoal}, in order to prove Proposition \ref{proFR}
it suffices to prove the estimate \eqref{goalR} and \eqref{goalRbis} for the operators in \eqref{ggbdpr1R}-\eqref{ggbdpr1R'}
and \eqref{D2'}-\eqref{E2'}.
We first prove the bounds \eqref{goalR} and \eqref{goalRbis} for the $(g,g)$ interaction in Subsection \ref{secRgg}.
The bound \eqref{goalRbis} is proved in a unified manner for all remaining cases in Lemma \ref{lemgoalbis}.
We then prove \eqref{goalR} for the $(g,h)$ interaction in Subsection \ref{secRhg},
and for the $(h,h)$ interaction in Subsection \ref{secRhh}.

\medskip
\subsection{Fermi interactions}\label{secRgg}
The next lemma treats the discrete component self-interaction.

\begin{lemma} \label{ggnuR}
Under the assumptions of Proposition \ref{proFR}, the estimates \eqref{goalR} and \eqref{goalRbis} 
hold when $(G,H) = (g,g)$.
\end{lemma}

\begin{proof}
Let us use the short-hand
\begin{align*}
D_1(t,\xi) := D_{1,m,\epss}[g_{\epsilon_1},g_{\epsilon_2}](t,\xi), 
  \qquad D_2'(\xi) := D_{2,m,\epss}'[g_{\epsilon_1},g_{\epsilon_2}](\xi)
\end{align*}
for the operators in \eqref{ggbdpr1R} and \eqref{D2'}.

\medskip
\noindent
{\it Estimate of $D_{1}$.}
From the identity \eqref{idbilR} (recall also the notation \eqref{idsym}), we have that
\begin{align}\label{ggnuRpr0}
\begin{split}
\Vert \widehat{\mathcal{F}}^{-1}_{\xi \mapsto x} D_{1}(t,\xi) \Vert_{L^2} 
  & \lesssim 2^{2m} \sum_{k_1,k_2 \leqslant M_0} \sup_{s \approx 2^m} 
  \bigg[\big \Vert T_1^{R,1}[b_1 ](e^{\epsilon_1 is \jxi } \wt{g_{\epsilon_1}}, 
  e^{\epsilon_2 is \jxi } \wt{g_{\epsilon_2}}) \big \Vert_{L^2_{\xi}}   
  \\
& + \big \Vert T_1^{R,1}[b_1' ]( e^{\epsilon_2 is \jxi } \wt{g_{\epsilon_2}},
  e^{\epsilon_1 is \jxi } \wt{g_{\epsilon_1}}) \big \Vert_{L^2_{\xi}}
  + \big \Vert T_1^{R,2}[b_2 ](e^{\epsilon_1 is \jxi } \wt{g_{\epsilon_1}}, 
  e^{\epsilon_2 is \jxi } \wt{g_{\epsilon_2}}) \big \Vert_{L^2_{\xi}}  \bigg],
\end{split}
\end{align}
where the symbols are defined as in \eqref{symbolnuR}.

Let $\ell_0:= \lfloor -m + \delta m \rfloor.$ 
We insert cut-offs $\varphi_{\ell_1}^{(\ell_0)}(\jeta -2\lambda)$ 
and $\varphi_{\ell_2}^{(\ell_0)} (\jsigma -2 \lambda)$ and write, 
using H\"older and the volume restriction $\vert \xi \vert \lesssim 2^{M_0},$
\begin{align}\label{ggnuRpr1}
\begin{split}
& 2^{2m} \big \Vert T_1^{R,1}[b_1](e^{\epsilon_1 is \jxi } \wt{g_{\epsilon_1}}, 
  e^{\epsilon_2 is \jxi } \wt{g_{\epsilon_2}}) \big \Vert_{L^2_{\xi}}  
  \\
  & \lesssim 2^{2m} \cdot 2^{3M_0/2} \cdot \sum_{k_1,k_2} \sum_{\ell_1, \ell_2 \geqslant \ell_0} 
  \sup_{\vert \xi \vert \leqslant 2^{M_0}} \sup_{s \approx 2^m} 
  \bigg \vert \int_{\R^6}
  e^{is \epsilon_{1 } \langle \eta \rangle} \big(\varphi_{\sim 0} \wt{g_{\epsilon_1}} \big)(s,\eta) 
  \, e^{is \epsilon_{2} \langle \sigma \rangle} \big(\varphi_{\sim 0} \wt{g_{\epsilon_2}} \big)(s,\sigma) 
  \\
  & \times 
  \varphi_{\ell_1}^{(\ell_0)}(\langle \eta \rangle - 2 \lambda) 
  \varphi_{\ell_2}^{(\ell_0)}(\langle \sigma \rangle - 2 \lambda) 
  \, b_1(\xi,\xi-\eta,\s) 
  \nu^R_1(\xi-\eta,\sigma) d\eta d\sigma \bigg \vert,
\end{split}
\end{align}
where we recall that, by our notation, 
$b_1(\xi,\xi-\eta,\s) := b(\xi,\eta,\s)\varphi_{k_1}(\xi-\eta)\varphi_{k_2}(\s)$,
see \eqref{prbdary11} and \eqref{ggbdpr1Rsym}.
We can then disregard the sum over $k_2$, since $|\s| \approx 1$ on the support of $\wt{g}(\s)$.
As for the sum over $k_1$, we may assume that $k_1\geqslant -10m$ since 
the complementary case can be treated directly integrating in $\eta$, and using the 
crude bound \eqref{dxi2g} to control ${\| \wt{g} \|}_{L^\infty}$,
and the bound from \eqref{nuRest}
\begin{align}\label{ggnuRpr2}
| \varphi_P(p) \varphi_Q(q) \nu_1^R (p,q) | \lesssim 2^{-2 (P \vee Q)} 2^{10M_0}.
\end{align}
Notice that in applying this bound to \eqref{ggnuRpr1} we have $2^Q = 2^{k_2} \gtrsim 1$.
We may then bound \eqref{ggnuRpr1} by a factor of $O(m)$ times the bound for
the integral expression with fixed $k_1,k_2$. 

In the following argument we may assume without loss of generality that 
$\ell_1 \geqslant \ell_2$ since the complementary case is analogous (although not exactly symmetric).

If $\ell_1=\ell_2=\ell_0$, we use \eqref{ggnuRpr2}
and \eqref{inftyfreqg} to integrate directly in \eqref{ggnuRpr1} and obtain the bound 
\begin{align}\label{ggnuRpr2'}
\begin{split}
& C 2^{2m} \cdot 2^{2\ell_0}
  \cdot 2^{3M_0/2} \cdot 2^{3(k_1+k_2)} \cdot 2^{-2(k_1\vee k_2)} 2^{10M_0} \cdot  
  {\| \wt{g} \|}_{L^{\infty}_{\xi}}^2
  \\
  & \lesssim 2^{2\delta m} 2^{16M_0} \big( m 2^m \rho(2^m) )^2
  \lesssim 2^{2\delta m} 2^{16M_0} 
  2^m \rho(2^m) \cdot Z \rho(2^m)^{\beta/2} m^2;
\end{split}
\end{align}
having used $2^m\rho(2^m) \leqslant Z \rho^{\beta/2}(2^m)$ (recall the definition \eqref{Z} and Remark \ref{remm});
the bound obtained above suffices for \eqref{goalR} since $\beta > 2\delta'$.

If $\ell_2 >\ell_0$ we integrate by parts in both $\eta$ and $\sigma$.
Also in view of \eqref{nuReest} which guarantees that 
differentiating $\nu_1^R$ is analogous to differentiating innocuous cutoffs,
we see that the worst term corresponds to the case where both derivatives hit the profiles,
which is bounded as follows, using \eqref{ggnuRpr2} and \eqref{dxiwtgest}:
\begin{align}\label{ggnuRpr3}
\begin{split}
C 2^{2m} \cdot 2^{3M_0/2} \cdot \sup_{s \approx 2^m} 
  \sup_{|\xi| \leqslant 2^{M_0}} \sum_{\ell_1, \ell_2 > \ell_0}
  \int_{\mathbb{R}^6} \frac{1}{s}\vert \big(\varphi_{\sim 0} \nabla_{\sigma} \wt{g} \big)(s,\sigma) \vert  
  \frac{1}{s} \vert \big(\varphi_{\sim 0} \nabla_{\eta} \wt{g} \big)(s,\eta) \vert
  \\
\times  
  \varphi_{\ell_1}(\langle \eta \rangle - 2 \lambda) 
  \varphi_{\ell_2}(\langle \eta \rangle - 2 \lambda) \big| \nu^R_1(\xi-\eta,\sigma) \big| d\eta d\sigma 
  \\
\lesssim 2^{23M_0/2} m^2
  \Vert \varphi_{\ell_1} \nabla_{\eta} \wt{g} \Vert_{L^1_{\eta}} 
  \Vert \varphi_{\ell_2} \nabla_{\sigma} \wt{g} \Vert_{L^1_{\s}} 
  \\
  \lesssim 2^{23M_0/2} m^4 \rho(2^m)^{\beta/2} \cdot 2^m \rho(2^m) \cdot Z.
\end{split}
\end{align}
As above, this last upperbound is sufficient.

The final case $\ell_0=\ell_2 < \ell_1$ is treated similarly, 
integrating by parts in $\eta$ and integrating directly in $\sigma.$
This will give a bound like \eqref{ggnuRpr2'} with an extra factor of $2^{\delta m}$,
which is acceptable for $\delta < \beta/2$.

The estimate for the second term in \eqref{ggnuRpr0} is similar by symmetry.
Also the term $T_{1}^{R,2}$ can be treated similarly with small modifications; 
the only difference is that for this term we need to additionally localize in $|\xi|\approx 2^k$
but can then also rely on the $\xi/\jxi$ factor in $b$ (see \eqref{ggbdpr1Rsym}).
We first insert the cutoffs 
in the expression for $T_1^{R,2}$ (see \eqref{idop2}) as in \eqref{ggnuRpr1}
which leads to the analogous estimate
\begin{align}\label{ggnuRpr1'}
\begin{split}
& 2^{2m} \big \Vert P_k T_1^{R,2}[b_2](e^{\epsilon_1 is \jxi } \wt{g_{\epsilon_1}}, 
  e^{\epsilon_2 is \jxi } \wt{g_{\epsilon_2}}) \big \Vert_{L^2_{\xi}}  
  \\
  & \lesssim 2^{2m} \cdot 2^{3k/2} \cdot \sum_{k_1,k_2} \sum_{\ell_1, \ell_2 \geqslant \ell_0} 
  \sup_{|\xi|\approx 2^k} \sup_{s \approx 2^m} 
  \bigg \vert \int_{\R^6} \frac{\xi}{\jxi}
  e^{is \epsilon_{1} \langle \eta \rangle} \big(\varphi_{\sim 0} \wt{g_{\epsilon_1}} \big)(s,\eta) 
  \, e^{is \epsilon_{2} \langle \sigma \rangle} \big(\varphi_{\sim 0} \wt{g_{\epsilon_2}} \big)(s,\sigma) 
  \\
  & \times 
  \varphi_{\ell_1}^{(\ell_0)}(\langle \eta \rangle - 2 \lambda) 
  \varphi_{\ell_2}^{(\ell_0)}(\langle \sigma \rangle - 2 \lambda) 
  \, b_2(\xi,-\eta-\s,\s) 
  \overline{\nu^R_1(-\eta-\sigma,\xi)} d\eta d\sigma \bigg \vert,
\end{split}
\end{align}
where we recall our notation
$b_2(\xi,-\eta-\s,\s) := b(\xi,\eta,\s)\varphi_{k_1}(-\eta-\s)\varphi_{k_2}(\s)$,
see \eqref{prbdary11} and \eqref{ggbdpr1Rsym}.
As before, we can dispense of the $k_2$ index and the associated summation,
and of the $k_1$ summation as well, up to an $O(m)$ factor.

In the case $\ell_1,\ell_2 > \ell_0$ we integrate by parts in $\eta$ and $\sigma$
and see that the main contribution to \eqref{ggnuRpr1'} in this case is bounded by
\begin{align}\label{ggnuRpr5}
\begin{split}
C 2^{2m} \cdot 2^{3k/2} 2^{\min(k,0)} \sum_{\ell_1, \ell_2 > \ell_0} 
  \sup_{\vert \xi \vert \approx 2^k} \sup_{s \approx 2^m} 
  \int_{\R^6} \frac{1}{s} | \varphi_{\sim 0} \nabla_\eta \wt{g_{\epsilon_1}}(s,\eta) |
  \, \frac{1}{s} |\varphi_{\sim 0} \nabla_\sigma\wt{g_{\epsilon_2}}(s,\sigma) |
  \\
  \times 
  \varphi_{\ell_1}(\langle \eta \rangle - 2 \lambda) 
  \varphi_{\ell_2}(\langle \sigma \rangle - 2 \lambda) 
  \, |\nu^R_1(-\eta-\s,\xi)| \, d\eta d\sigma
  \\
  \lesssim 2^{3k/2} 2^{\min(k,0)} \cdot 
  {\| \varphi_{\ell_1} \nabla_{\eta} \wt{g} \|}_{L^1_{\eta}} 
  {\| \varphi_{\ell_2} \nabla_{\s} \wt{g} \|}_{L^1_{\s}}
  \cdot 2^{-2k} 2^{10 M_0},
\end{split}
\end{align}
having used \eqref{ggnuRpr2} to bound $|\nu^R_1(-\eta-\s,\xi)| \lesssim 2^{-2k}$;
we can then conclude as in \eqref{ggnuRpr3} using \eqref{dxiwtgest}.

The remaining cases $\ell_1 > \ell_2 = \ell_0$ and $\ell_1=\ell_2=\ell_0$ can be handled 
as explained before, with the additional exploitation of the extra $2^{k^-}$ factor.


\medskip
\noindent
{\it Estimate of $D_2'(t,\xi)$.} 
We now show how to obtain \eqref{goalRbis} for the operator in \eqref{D2'} with \eqref{ggbdpr1Rsym}.
Expanding $\mu^R$ as above, see for example \eqref{goalpr4} 
where there are two $\xi$ derivatives instead of just one, as in the present case, we have
\begin{align}\label{goalpr4'}
\begin{split}
{\| \widehat{\mathcal{F}}^{-1}_{\xi \mapsto x}  D_{2}'(\xi) \|}_{L^2}
  & \leqslant 2^m \sup_{s \approx 2^m} 
  {\Big\| \int_{\R^6} e^{is(\eps_1\jeta + \eps_2 \jsig)} 
  \wt{g_{\eps_1}}(t,\eta) \, \wt{g_{\eps_2}}(t,\sigma)  \, 
  \, b'(\xi,\eta,\s) \, \nabla_{\xi} \mu^R(\xi,\eta,\sigma) d\eta d\sigma \Big\|}_{L^2}
\\ 
& \lesssim 2^m 
  \sup_{s \approx 2^m} \sum_{k_1,k_2 \leqslant M_0} 
 \big \Vert T^{R,1}_1[\partial_1 \nu_1^R,b_1] (e^{is \eps_1 \jxi} \wt{g_{\eps_1}}, 
 e^{is \epsilon_2 \jxi} \wt{g_{\eps_2}}) \big \Vert_{L^2} 
 \\
 & + \big \Vert T^{R,1}_1 [\partial_1 \nu_1^R;b_1'] (e^{is \epsilon_2 \jxi} \wt{g_{\eps_2}},
 e^{is \eps_1 \jxi} \wt{g_{\eps_1}}) \big \Vert_{L^2}
 \\
& + \big \Vert T_1^{R,2}[\partial_2 \nu_1^R,b_2](e^{is \epsilon_1 \jxi} \wt{g_{\eps_1}},
  e^{is \epsilon_2 \jxi} \wt{g_{\eps_2}}) \big \Vert_{L^2}.
\end{split}
\end{align}
The symbols are as in \eqref{symbolnuR}, and we are adopting the notation \eqref{idop11R}.

We note that in \eqref{goalpr4'} the terms involving the operator $T_{1}^{R,1}$, see \eqref{idop11},
are such that $|\s| \approx 1$ on their support.
Then,  in view of \eqref{nuRest}, the corresponding measure $\partial_1 \nu^R_1(\eta,\sigma)$ satisfies the same 
pointwise estimate as $\nu^R_1(\eta,\sigma)$.
Therefore the exact same estimates as those established for $D_1$ apply.
Comparing with \eqref{ggnuRpr0}, we see that in this case we have one less $2^{m}$ factor
hence the estimates already established give a bound of the form $2^{-m}$ times the right-hand side of \eqref{goalR} 
which is sufficient for \eqref{goalRbis}.

Next, we look at the operator $T_1^{R,2}[\partial_2 \nu_1^R,b_2]$ on the last line of \eqref{goalpr4'}.
This term can be treated similarly to the one in \eqref{ggnuRpr1'}.
The only difference is that we now have one additional advantageous factor of $2^{-m}$ 
but also an extra $\xi$ derivative on $\nu_1^R$,
which gives a $2^{-(k\vee k_1)}$ factor (see \eqref{ggnuRpr2})
compared with the bound in \eqref{ggnuRpr5}.
Proceeding as before, with the same notation used in \eqref{ggnuRpr1'}, we first estimate
\begin{align}\label{ggnuRpr1''}
\begin{split}
& 2^{m} \big \Vert P_k T_1^{R,2}[b_2](e^{\epsilon_1 is \jxi } \wt{g_{\epsilon_1}}, 
  e^{\epsilon_2 is \jxi } \wt{g_{\epsilon_2}}) \big \Vert_{L^2_{\xi}}  
  \\
  & \lesssim 2^{2m} \cdot 2^{3k/2} \cdot 
  \sum_{\ell_1, \ell_2 \geqslant \ell_0} 
  \sup_{|\xi|\approx 2^k} \sup_{s \approx 2^m} 
  \bigg \vert \int_{\R^6} \frac{\xi}{\jxi}
  e^{is \epsilon_{1} \langle \eta \rangle} \big(\varphi_{\sim 0} \wt{g_{\epsilon_1}} \big)(s,\eta) 
  \, e^{is \epsilon_{2} \langle \sigma \rangle} \big(\varphi_{\sim 0} \wt{g_{\epsilon_2}} \big)(s,\sigma) 
  \\
  & \times 
  \varphi_{\ell_1}^{(\ell_0)}(\langle \eta \rangle - 2 \lambda) 
  \varphi_{\ell_2}^{(\ell_0)}(\langle \sigma \rangle - 2 \lambda) 
  \, b_2(\xi,-\eta-\s,\s) 
  \partial_\xi \overline{\nu^R_1(-\eta-\sigma,\xi)} d\eta d\sigma \bigg \vert.
\end{split}
\end{align}
Note that we have disregarded the innocuous sum over $k_1,k_2$ as before.
For $\ell_1,\ell_2 > \ell_0$ we integrate by parts and obtain the bound (compare with \eqref{ggnuRpr5}) 
\begin{align}\label{ggnuRpr5'}
\begin{split}
C 2^{m} \cdot 2^{3k/2} 2^{\min(k,0)} \sum_{\ell_1, \ell_2 > \ell_0} 
  \sup_{\vert \xi \vert \approx 2^k} \sup_{s \approx 2^m} 
  \int_{\R^6} \frac{1}{s} | \varphi_{\sim 0} \nabla_\eta \wt{g_{\epsilon_1}}(s,\eta) |
  \, \frac{1}{s} |\varphi_{\sim 0} \nabla_\sigma\wt{g_{\epsilon_2}}(s,\sigma) |
  \\
  \times 
  \varphi_{\ell_1}(\langle \eta \rangle - 2 \lambda) 
  \varphi_{\ell_2}(\langle \sigma \rangle - 2 \lambda) 
  \, |\partial_\xi \nu^R_1(-\eta-\s,\xi)| \, d\eta d\sigma
  \\
  \lesssim 2^{-m} \cdot 2^{3k/2} 2^{\min(k,0)} \cdot 
   \sum_{\ell_1, \ell_2 > \ell_0} {\| \varphi_{\ell_1} \nabla_{\eta} \wt{g} \|}_{L^1_{\eta}} 
  {\| \varphi_{\ell_2} \nabla_{\s} \wt{g} \|}_{L^1_{\s}}
  \cdot 2^{-3(k \vee k_1)} 2^{10 M_0}
  \\
  \lesssim  2^{-m+11M_0} \cdot 2^{-(1/2)(k\vee k_1)} \cdot (m2^m\rho(2^m))^2,
\end{split}
\end{align}
having used \eqref{ggnuRpr2} to bound, on the support of the integral, 
$|\partial_\xi \nu^R_1(-\eta-\s,\xi)| \lesssim 2^{-3(k\vee k_1)}$ and \eqref{dxiwtgest}.
The bound \eqref{ggnuRpr5'} is more than sufficient 
for \eqref{goalRbis} provided, for example, that $k\vee k_1 \geqslant -m/4$.

In the remaining case $k\vee k_1 < -m/4$, we only integrate by parts in $\s$ but not in $\eta$
and can bound as follows:
\begin{align*}
\begin{split}
C 2^{m} \cdot 2^{3k/2} 2^{\min(k,0)} \sum_{\ell_1, \ell_2 > \ell_0} 
  \sup_{\vert \xi \vert \approx 2^k} \sup_{s \approx 2^m} 
  \int_{\R^6} | \varphi_{\sim 0} \wt{g_{\epsilon_1}}(s,\eta) |
  \, \frac{1}{s} |\varphi_{\sim 0} \nabla_\sigma\wt{g_{\epsilon_2}}(s,\sigma) |
  \\
  \times 
  \varphi_{\ell_1}(\langle \eta \rangle - 2 \lambda) 
  \varphi_{\ell_2}(\langle \sigma \rangle - 2 \lambda) 
  \, |\partial_\xi \nu^R_1(-\eta-\s,\xi)| \, d\eta d\sigma
  \\
  \lesssim 2^{3k/2} 2^{\min(k,0)} \cdot 
  \sum_{\ell_1,\ell_2} 2^{3k_1} {\| \varphi_{\ell_1} \wt{g} \|}_{L^\infty_{\eta}} 
  {\| \varphi_{\ell_2} \nabla_{\s} \wt{g} \|}_{L^1_{\s}}
  \cdot 2^{-3(k \vee k_1)} 2^{10 M_0}
  \\
  \lesssim 2^{11M_0} \cdot  2^{(5/2)k_1} 
  \cdot (m2^m\rho(2^m))^2,
\end{split}
\end{align*}
having used \eqref{inftyfreqg} and \eqref{dxiwtgest}.
Since $k_1 \leqslant -m/4$ this is enough for \eqref{goalRbis}.

\smallskip
The estimates for the operators $E$ with the $\mu^{Re}$ distribution, see \eqref{ggbdpr1R'} and \eqref{E2'},
can be done very similarly to $T^{R,1}_1$, using the estimate \eqref{nuReest}, 
therefore we can skip the details.
\end{proof}

\medskip
Next, we prove \eqref{goalRbis} in a unified manner for all remaining cases.

\begin{lemma}\label{lemgoalbis}
Under the assumptions of Proposition \ref{proFR} the bound \eqref{goalRbis} holds for $(G,H) = (g,h),(h,g), (h,h).$
\end{lemma}

\begin{proof}
Since $h$ has better localization than $g,$ we can integrate by parts and
rely on the H\"older-type estimate associated to the operators 
$\partial_1\nu_1^R$,  $\partial_2\nu_1^R$ from Lemma \ref{bilindxiR}.
We use similar notation to the one in the proof of Lemma \ref{ggnuR},
and seek to estimate ${\| \widehat{\mathcal{F}}^{-1}_{\xi \mapsto x}  D_{2}'(\xi) \|}_{L^2}$,
see the definitions \eqref{D2'} and \eqref{ggbdpr1Rsym}, by the right-hand side of \eqref{goalRbis}.
For this it suffices to control, for $G\in\{g,h\}$, 
\begin{align}\label{goalbispr1}
\begin{split}
  & 2^m \sup_{s \approx 2^m} 
  {\Big\| \int_{\R^6} e^{is(\eps_1\jeta + \eps_2 \jsig)} 
  \wt{G_{\eps_1}}(t,\eta) \, \wt{h_{\eps_2}}(t,\sigma)  \, 
  \, b'(\xi,\eta,\s) \, \nabla_{\xi} \mu^R(\xi,\eta,\sigma) d\eta d\sigma \Big\|}_{L^2}
  \\
  & \lesssim 2^m \sup_{s \approx 2^m} \sum_{k_1,k_2 \leqslant M_0} 
  {\big\| T^{R,1}_1[\partial_1 \nu_1^R,b_1] (e^{is \eps_1 \jxi} \wt{G_{\eps_1}}, 
  e^{is \epsilon_2 \jxi} \varphi_{\sim k_2} \wt{h_{\eps_2}}) \big\|}_{L^2} 
\\
  & + {\big\| T^{R,1}_1 [\partial_1 \nu_1^R;b_1'] (e^{is \epsilon_2 \jxi} \varphi_{\sim k_2} \wt{h_{\eps_2}},
  e^{is \eps_1 \jxi} \wt{G_{\eps_1}}) \big\|}_{L^2} 
\\
& + {\big\| T_1^{R,2}[\partial_2 \nu_1^R,b_2](e^{is \epsilon_1 \jxi} \wt{G_{\eps_1}},
  e^{is \epsilon_2 \jxi} \varphi_{\sim k_2} \wt{h_{\eps_2}}) \big\|}_{L^2} ;
\end{split}
\end{align}
see the analogous \eqref{goalpr4'} (and recall \eqref{symbolnuR} and \eqref{idop11R})
and notice that we have explicitly written out the localization at frequencies $\approx 2^{k_2}$ for the input $h$. 

Let us look at the first term on the right-hand side of \eqref{goalbispr1}.
Recalling the definition \eqref{idop11R} we integrate by parts in $\s$, the frequency variable for $h$,
when $k_2 \geqslant -m/2+\delta m$, for some small $\delta>0$.
The main term that we obtain is the one where $\nabla_\s$ hits the profile, 
that is (once again we can disregard the summation over $k_1,k_2$)
\begin{align} \label{integgh}
\begin{split}
2^m \sup_{s \approx 2^m} 
  {\big\| T^{R,1}_1[\partial_1 \nu_1^R,b_1] 
  \big( e^{is \eps_1 \jxi} \wt{G_{\eps_1}}, \,
  s^{-1} \, (\jxi \xi)/|\xi|^2 \, e^{is \epsilon_2 \jxi} \varphi_{k_2} \nabla_\xi \wt{h_{\eps_2}} \big) \big\|}_{L^2} .
\end{split}
\end{align}
Using Lemma \ref{bilindxiR},
the decay estimates \eqref{dispbootgbis}-\eqref{dispbootf}
which imply
\begin{align*}
\Vert e^{isL} G \Vert_{L^{8/3}} \lesssim \rho(2^m)^{1-\beta} 2^{5m/8+10 \delta_N m} 
 , \qquad G \in \lbrace g,h \rbrace,
\end{align*}
Bernstein's inequality, and the usual a priori bound \eqref{boot}, we control \eqref{integgh} by
\begin{align*}
  & C \Vert e^{isL} G \Vert_{L^{8/3}} 
  \cdot 2^{-k_2} \Vert P_{k_2} \wt{\mathcal{F}}^{-1} \nabla_{\xi} \wt{h} \Vert_{L^8} 
  \\
  & \lesssim 2^{(C_0+10)\delta_Nm} \cdot \rho(2^m)^{1-\beta} 2^{5m/8} \cdot 2^{-k_2} \cdot 2^{9k_2/8} 
  \rho(2^m)^{1-\beta} 2^m \e^\beta;
\end{align*}
this is bounded by the right-hand side of \eqref{goalRbis} 
provided, as usual, that $\delta_N$ and $\beta$ are small enough.

In the case when $k_2 \leqslant -m/2 + \delta m$, we cannot integrate by parts; we instead 
use directly \ref{bilindxiR} followed by Bernstein's and Hardy's inequality 
to estimate (once again we disregard the $k_1$ summation up to an $O(m)$ factor)
\begin{align*}
\begin{split}
& 2^m \sup_{s \approx 2^m} \sum_{\substack{k_1 \leqslant M_0 \\ k_2 \leqslant -m/2+\delta m}} 
  {\big\| T^{R,1}_1[\partial_1 \nu_1^R,b_1] (e^{is \eps_1 \jxi} \wt{G_{\eps_1}}, 
  e^{is \epsilon_2 \jxi} \varphi_{\sim k_2} \wt{h_{\eps_2}}) \big\|}_{L^2} 
  \\
  & \lesssim 2^m \cdot 2^{(C_0+1)M_0}
  \cdot {\| G \|}_{L^2} \cdot \sum_{k_2\leqslant  -m/2+\delta m} {\| P_{k_2} h \|}_{L^{\infty-}}
  \\
  & \lesssim 2^m \cdot 2^{(C_0+2)M_0} 
  \cdot \sum_{k_2\leqslant -m/2+\delta m} 2^{(\frac{5}{2}-)k_2} {\| \nabla_\xi \wt{h} \|}_{L^2}
  \\
  & \lesssim 2^{(C_0+3)M_0} \cdot 2^{-m/4} Z,
\end{split}
\end{align*}
which suffices for \eqref{goalRbis}.

The second term on the right-hand side of \eqref{goalbispr1} is symmetric so we can skip it.
The third term can also be treated similarly integrating by parts in $\partial_{\eta-\sigma}$
(or changing variables $\eta \mapsto \eta-\sigma$, and then integrating by parts in $\s$ as above),
and using Lemma \ref{bilindxiR}.

We can use the same arguments to estimate the $E'$ term as well (see the definition \eqref{E2'})
since $T[\partial_1 \mu^{Re};b_{Re}']$ (see the definition \eqref{idop11}),
satisfies similar or better estimates than the operators just treated;
see the estimates \eqref{bilindxiR2} in Lemma \ref{bilindxiR}, and 
the pointwise bound \eqref{nuReest}. 
\end{proof}

\medskip
\subsection{Fermi-continuous interaction}\label{secRhg}

We now treat the regular terms with Fermi-continuous type interactions. We will use 
several notation and arguments similar to those in Subsection \ref{secRgg} above.

\begin{lemma} \label{ghnuR}
With the assumptions and notation of Proposition \ref{proFR}, 
the bound \eqref{goalR} holds for $(G,H) = (g,h),(h,g).$
\end{lemma} 


\begin{proof}
Recall the definition of the terms $D_1(G,H):= D_{1,m,\epss}(G,H)$ and $E_1(G,H):=E_{1,k_1,k_2,m,\epss}(G,H)$ 
from \eqref{ggbdpr1R} and \eqref{ggbdpr1R'}.
We will mostly concentrate on $D_1(g,h)$;
the treatment of the terms in $D_1(h,g)$ is similar (or simpler in many cases) 
and we will only provide some indications on how to handle them. 
The starting point is the estimate
\begin{align}\label{ghnuRpr0'}
\begin{split}
\Vert \widehat{\mathcal{F}}^{-1}_{\xi \mapsto x} D_{1}(t,\xi) \Vert_{L^2} 
  & \lesssim \sum_{k_1,k_2 \leqslant M_0} \sup_{s \approx 2^m} 
  \int_0^t s \tau_m(s)\bigg[\big \Vert T_1^{R,1}[b_1 ](e^{\epsilon_1 is \jxi } \wt{g_{\epsilon_1}}, 
  e^{\epsilon_2 is \jxi } \wt{h_{\epsilon_2}}) \big \Vert_{L^2_{\xi}}   
  \\
& + \big \Vert T_1^{R,1}[b_1' ]( e^{\epsilon_2 is \jxi } \wt{h_{\epsilon_2}},
  e^{\epsilon_1 is \jxi } \wt{g_{\epsilon_1}}) \big \Vert_{L^2_{\xi}}
  \\
  & + \big \Vert T_1^{R,2}[b_2 ](e^{\epsilon_1 is \jxi } \wt{g_{\epsilon_1}}, 
  e^{\epsilon_2 is \jxi } \wt{h_{\epsilon_2}}) \big \Vert_{L^2_{\xi}}  \bigg] \, ds,
\end{split}
\end{align}
which is a consequence of the identity \eqref{idbilR} (see also the analogous \eqref{ggnuRpr0});
as before, the symbols are defined as in \eqref{symbolnuR} but with $b$ instead of $b'$
on the right-hand side, and with $b$ given by \eqref{ggbdpr1Rsym}; 
once again we left the $k_1,k_2$ dependence implicit.

\medskip
{\it Estimate of the $T_1^{R,1}[b_1]$ contribution}.
We begin by looking at the main contribution $T_1^{R,1}[b_1] = T_1^{R,1}[\nu_1^R;b_1]$.
The proof will proceed in several cases and subcases depending on the size of the 
frequencies of the inputs, of the output and of the integration variables.
According to our usual notation \eqref{idop11R}, we localize the variables in 
$T_1^{R,1}[b_1]$ by defining the following operator:
for fixed small $\delta$ and $\underline{k} := (k,k_1,k_2,\ell)$ let
\begin{align}\label{ghnuRop}
\begin{split}
T_{\underline{k}}(t,\xi) := \int_0^t s \, \tau_m(s) \Big[ \int_{\R^6} e^{-is \Phi_\epss(\xi,\xi-\eta,\sigma)}
 \wt{g_{\epsilon_1}}(s,\xi-\eta) \wt{h_{\epsilon_2}}(s,\sigma)  \, 
  \\
  \times \varphi_{\underline{k}}(\xi,\eta,\s) \, b_1(\xi,\eta,\s) \, \nu_1^R(\eta,\sigma) \, d\eta d\sigma \Big] \, ds,
\end{split}
\end{align}
where $\Phi_\epss(\xi,\xi-\eta,\s) = - \jxi + \eps_1 \langle \xi-\eta\rangle + \eps_2 \jsig$ as 
per our usual notation, and
\begin{align}\label{ghnuRop1}
\begin{split}
& \varphi_{\underline{k}}(\xi,\eta,\s) :=  \varphi_{\sim k_1}(\eta)\varphi_{\sim k_2}(\s) 
   \varphi_k(\xi) \varphi_{\ell}^{(\ell_0)}( \langle \xi-\eta \rangle - 2\lambda), \qquad \ell_0 := -m + \delta m,
\\
& b_1(\xi,\eta,\s) = \frac{\xi}{\jxi} b'(\xi,\xi-\eta,\s) \varphi_{k_1}(\eta)\varphi_{k_2}(\s),
\end{split}
\end{align}
with $b'$ as in \eqref{ggbdpr1Rsym}.
The variables in \eqref{ghnuRop} are then localized as follows:
\begin{align}\label{ghnuRloc}
\begin{split}
& |\xi| \approx 2^k \qquad |\eta|\approx 2^{k_1}, \qquad |\s|\approx 2^{k_2}, 
  \qquad |\langle \xi-\eta \rangle - 2\lambda| \approx 2^{\ell}.
\\
& \mbox{with} \quad -5m < k,k_1,k_2 \leqslant M_0, \qquad \ell_0 < \ell < 0,
\end{split}
\end{align}
with $|\langle \xi-\eta \rangle - 2\lambda| \lesssim 2^{-m+\delta m}$ in the case $\ell=\ell_0$.

Our goal is to estimate the sum over $k,k_1,k_2,\ell$ of the $L^2$ 
norm of $T_{\underline{k}}$ by the right-hand side of \eqref{goalR}.
Since the sums are over $O(m^4)$ parameters, it also suffices to estimate 
$\|T_{\underline{k}}\|_{L^2}$ for a single 4-tuple, 
provided we can gain an arbitrary $2^{-(0+)m}$ factor.

\medskip
{\bf Case 1: $k_1,k_2 \leqslant -10$}.\label{Case1} 
In this case $|\eta|,|\s|\leqslant 2^{-5}$ and the phase is lower bounded uniformly:
\begin{align}\label{ghnuRphase}
|\Phi_\epss(\xi,\xi-\eta,\s)| \gtrsim 1 .
\end{align}
Indeed if $(\epss) = (-,-)$ then the result is obvious. 
The cases $(\epss) = (+,-), (+,+)$ can be treated similarly by noticing that
\begin{align*}
\vert \Phi_{+,\epsilon_2} (\xi,\xi-\eta,\s) \vert & \geqslant 
  1 - \big| - \jxi + \langle \xi-\eta\rangle | - |\jsig-1|
  \\ 
  & \geqslant
  1 - 2|\eta| - |\s| \gtrsim 1.
\end{align*}
Finally, the case $(\epss) = (-,+)$ can be handled by noticing that 
\begin{align*}
\vert \Phi_{-,+} (\xi,\xi-\eta,\s) \vert \geqslant \jxi + \langle \xi-\eta \rangle - 1 - |\jsig - 1|
  \gtrsim 1 - |\s|.
\end{align*}
Then we can proceed as in the proof of \eqref{proFSest1} in Proposition \ref{proFS}.
More precisely, we can integrate by parts in $s$ without any loss, 
and reduce everything to estimating the same operators \eqref{Tij1} and \eqref{Tij2}
treated before, but with $\mu^R$ instead of $\mu^S$.
Since for the operators associated to the regular measure $\nu_1^R$ 
we have similar (in fact, better) bilinear bounds than for those associated to $\nu_1^S$
(compare \eqref{mainbilin1}-\eqref{mainbilin2} with \eqref{mainbilinR}),
the proof carries out verbatim as in Subsection \ref{secFS};
see the proof of the estimate \eqref{lembdary1} in Lemma \ref{bdry-no-t-res}
and of the estimate \eqref{ghbd1} in Lemma \ref{gh-no-t-res}.

\smallskip
From now on we may assume $k_1\wedge k_2 > -10$.
Our next goal is to try to integrate by parts in $\s$, but in order to do this,
we first need to deal with very small $k_2$. Let us fix a small $\delta>0$.

\medskip
{\bf Case 2: $k_2 \leqslant -m/2+\delta m$}.\label{Case2}
Notice that, in view of Case 1, we may assume $|\eta|\gtrsim 1$. 
In particular, from \eqref{nuRest} we 
have that $\nu_1^R(\eta,\s)$ is bounded by $2^{10M_0}$.
For $\ell > \ell_0$ we integrate by parts in $\eta$, using that $\partial_\eta \langle \xi-\eta \rangle \approx 1$
on the support of the integral \eqref{ghnuRop}.
We then obtain a main contribution when $\partial_\eta$ hits $\wt{g}$, which is bounded by
\begin{align}\label{ghnupr5}
\begin{split}
\int_0^t \tau_m(s) \int_{\R^6} 
 | \nabla_\eta \wt{g_{\epsilon_1}}(s,\xi-\eta)| \,|\wt{h_{\epsilon_2}}(s,\sigma)|
  \varphi_{\underline{k}}(\xi,\eta,\s) \, b_1(\xi,\eta,\s) \, |\nu_1^R(\eta,\sigma)| \, d\eta d\sigma  \, ds.
\end{split}
\end{align}
The other terms where $\partial_\eta$ hits $\nu_1^R$ or the cutoff (in which case we can 
repeat the integration by parts) are better. We then estimate
\begin{align}\label{ghnupr5.5}
\begin{split}
{\| \eqref{ghnupr5} \|}_{L^2} \lesssim 2^m 2^{3k/2} \sup_{s\approx 2^m}
 {\| \varphi_{\ell} \nabla_\eta \wt{g_{\epsilon_1}}(s) \|}_{L^1} 
 \cdot {\| \varphi_{k_2} \wt{h_{\epsilon_2}}(s) \|}_{L^1} \cdot 2^{10M_0}
 \\
 \lesssim 2^m 2^{12M_0} \cdot 2^m\rho(2^m)m \cdot 2^{5k_2/2} Z,
\end{split}
\end{align}
having used \eqref{dxiwtgest}, H\"older, and the a priori assumptions.
Since $k_2 \leqslant -m/2 +\delta m$ the above suffices.

When $\ell = \ell_0$ we do not need to integrate by parts in $\eta$, but can directly 
do the integration in \eqref{ghnuRop}-\eqref{ghnuRop1}, using \eqref{inftyfreqg}, 
and obtain the same bound (provided $\delta \leqslant \delta_N/2$):
\begin{align}\label{ghnupr5'}
\begin{split}
{\| T_{\underline{k}} \|}_{L^2} \lesssim 2^{2m} 2^{3k/2} \sup_{s\approx 2^m}
 {\| \varphi_{\leqslant \ell_0} \wt{g_{\epsilon_1}}(s) \|}_{L^1} 
 \cdot {\| \varphi_{k_2} \wt{h_{\epsilon_2}}(s) \|}_{L^1} \cdot 2^{10M_0}
 \\
 \lesssim 2^{12M_0} \cdot 2^m\rho(2^m)m \cdot 2^{5k_2/2} Z.
\end{split}
\end{align}

In what follows we may assume $k_2 \geqslant -m/2 + \delta m$,
and will always have the option of integrating by parts in $\s$.
We then distinguish two more cases depending on whether $k_2$ is relatively close to zero or not.

\medskip
{\bf Case 3: $k_2 \geqslant -m/10$}.\label{Case3} 
In this case $k_2$ is not small 
so that we can efficiently use integration by parts in $\s$, then Strichartz estimates,
our usual bilinear estimate, and exploit the decay of $e^{itL}g$ to conclude.
More precisely, integrating by parts in \eqref{ghnuRop} gives a main term of the form
\begin{align}\label{ghnupr6}
\begin{split}
\int_0^t \tau_m(s) \Big[ \int_{\R^6} e^{-is \Phi_\epss(\xi,\xi-\eta,\sigma)}
 \wt{g_{\epsilon_1}}(s,\xi-\eta)
 \, \frac{\sigma \jsig}{|\sigma|^2} \cdot \nabla_\s\wt{h_{\epsilon_2}}(s,\sigma) 
 \, 
  \\
  \times \varphi_{\underline{k}}(\xi,\eta,\s) \, b_1(\xi,\eta,\s) \, \nu_1^R(\eta,\sigma) \, d\eta d\sigma \Big] \, ds.
\end{split}
\end{align}
For the other terms where the derivatives hit the cutoff $\varphi_{k_2}$ or,
equivalently, the factor $1/|\s|$ we have a net gain and can repeat the integration by parts;
if the derivatives hit $\nu_1^R$ we have a similar or better contribution, in view of \eqref{nuRest}.
We then write
\begin{align}\label{ghnupr6'} 
\begin{split}
& {\| \eqref{ghnupr6} \|}_{L^2}
 \\
 & \lesssim {\Big\| \int_0^t \tau_m(s) e^{-isL} P_k T_1^{R,1}[b_1]
  \Big( e^{is\epsilon_1 \jxi} \varphi_\ell^{[\ell_0,0]}(\jxi-2\lambda)\wt{g_{\eps_1}}, 
  e^{is\epsilon_2 \jxi} \varphi_{k_2}(\xi)\big(\frac{\xi \jxi}{|\xi|^2} \cdot \nabla \big) \wt{h_{\eps_2}} \Big) 
  ds \Big\|}_{L^2_x}
\end{split}
\end{align}
and estimate this using 
the Strichartz estimate (Lemma \ref{Strichartz}) with parameters $(q,r,\gamma) = (6,3,2/3)$,
i.e., landing in $L^{6/5}_s \langle D \rangle^{-2/3} L^{3/2}_x$,
and the bilinear estimate \eqref{mainbilinR}:
\begin{align}\label{ghnupr6.5} 
\begin{split}
{\| \eqref{ghnupr6} \|}_{L^2}
& \lesssim 2^{2k/3} {\Big\|P_k T_1^{R,1}[b_1]
  \Big( e^{is\epsilon_1 \jxi} \varphi_\ell^{[\ell_0,0]}(\jxi-2\lambda)\wt{g}, 
  e^{is\epsilon_2 \jxi} \varphi_{k_2}(\xi)\big(\frac{\xi \jxi}{|\xi|^2} \cdot \nabla \big) \wt{h} \Big) 
  \Big\|}_{L^{6/5}_{s\approx 2^m} L^{3/2}_x}
\\
& \lesssim 2^{(C_0+1) M_0} 
  \cdot \Vert e^{isL} g \Vert_{L^{6/5}_{s \approx 2^m} L^{6-}_x} 
  \cdot 2^{-k_2} \Vert \nabla \wt{h} \Vert_{L^{\infty}_s L^2_x} 
  \\
& \lesssim  2^{(C_0+1) M_0} \cdot \big(\rho^{6/5}(2^m)\, 2^{(0+)m} \, 2^m \big)^{5/6} \cdot 
  2^{-k_2} \cdot Z 
\end{split}
\end{align}
having used \eqref{dispbootgbis}.
This gives the desired estimate since $k_2 \geqslant -m/10$.

\medskip
{\bf Case 4: $-m/2+\delta m < k_2 < -m/10$}.\label{Case4}
In this case, $k_2$ is relatively small, but we can still 
integrate by parts in the frequencies 
and, in addition, exploit the fact 
that we must have $|\eta|\gtrsim 1$ so that the measure $\nu_1^R$ does not create any harm.

We skip the case $\ell=\ell_0$ as it is similar to what we will do below, 
and the estimates can be done without integrating by parts in $\eta$.
For $\ell > \ell_0$ we integrate by parts in $\s$ in the formula \eqref{ghnuRop} 
obtaining a term as in \eqref{ghnupr6}, and then also integrate by parts in $\eta$, 
as done before in order to arrive at \eqref{ghnupr5}.
We then obtain the main contribution when both derivatives hit the profiles, and this is bounded by
\begin{align}\label{ghnupr8}
\begin{split}
\int_0^t \int_{\mathbb{R}^6} s \, \tau_m(s) 
 \frac{1}{s}| \nabla_\eta \wt{g_{\epsilon_1}}(s,\xi-\eta)| \, 
 \frac{1}{s} \frac{1}{|\s|} |\nabla_\s \wt{h_{\epsilon_2}}(s,\sigma)|  \, 
 \, \varphi_{\underline{k}}(\xi,\eta,\s) \, |\nu_1^R(\eta,\sigma)| \, d\eta d\sigma \, ds.
\end{split}
\end{align}

Using H\"older's inequality,
the pointwise bounds on the measure \eqref{nuRest}, \eqref{dxiwtgest} and \eqref{boot},
we can estimate 
\begin{align}\label{ghnupr8.5}
\begin{split}
{\| \eqref{ghnupr8} \|}_{L^2} 
& \lesssim 
  2^{3k^+/2} \cdot 2^{-k_2} \Vert \varphi_{k_2} \nabla_{\s} \wt{h_{\eps_2}} 
  \Vert_{L^1_{\s}} \cdot \Vert \varphi_{\ell} \nabla_{\eta} \wt{g_{\eps_1}} \Vert_{L^1_{\eta}}
\\
& \lesssim 
  2^{3M_0/2} \cdot 2^{k_2/2} Z 
  \cdot \rho(2^m) 2^m m, 
\end{split}
\end{align}
which suffices since $k_2 < -m/10$.


\medskip
{\it Estimate of the $T_1^{R,1}[b_1']$ contribution}.
The estimates in this case are similar to the previous ones, but not symmetric;
in fact they are easier since the second input is $g$ which is localized around frequencies of size $\approx 1$,
and therefore the measure $\nu_1^R$ is always bounded by $2^{10M_0}$.
We provide some details for the convenience of the reader.

Similarly to \eqref{ghnuRop}-\eqref{ghnuRop1} it suffices to look at the operator
\begin{align}\label{ghnuRop'}
\begin{split}
T_{\underline{k}'}'(t,\xi) := \int_0^t s \, \tau_m(s) \Big[ \int_{\R^6} e^{-is \Phi_{\eps_2,\eps_1}(\xi,\xi-\eta,\sigma)}
 \wt{h_{\epsilon_2}}(s,\xi-\eta) \wt{g_{\epsilon_1}}(s,\sigma)  \, 
  \\
  \times \psi_{\underline{k}'}(\xi,\eta,\s) \, b_1'(\xi,\eta,\s) \, \nu_1^R(\eta,\sigma) \, d\eta d\sigma \Big] \, ds,
\end{split}
\end{align}
where $\underline{k}' = (k,k_1,k_3,\ell)$ and
\begin{align}\label{ghnuRop1'}
\begin{split}
& \psi_{\underline{k}'}(\xi,\eta,\s) :=  \varphi_{\sim k_1}(\eta)
  \varphi_{k_3}(\xi-\eta) 
  \varphi_k(\xi) \varphi_{\ell}^{(\ell_0)}( \jsig - 2\lambda), \qquad \ell_0 := \lfloor -m + \delta m \rfloor, 
\\
& b_1'(\xi,\eta,\s) = \frac{\xi}{\jxi} b'(\xi,\xi-\s,\eta) \varphi_{k_1}(\eta)\varphi_{k_2}(\s).
\end{split}
\end{align}
As before, we may easily reduce to at most $O(m^4)$ terms in the sum over the indexes,
and do the estimates for a fixed $4$-tuple.
We fix $\delta>0$ small, and look at three main cases depending on the size of $k_3$.

\medskip
{\bf Case 1b: $k_3 \leqslant -m/2 + \delta m$}.
In this case we cannot integrate by parts in $\eta$, but we integrate by parts in $\s$ and 
can proceed identically to Case 2 above, and obtain the same bounds as in \eqref{ghnupr5.5} and \eqref{ghnupr5'}
with $k_3$ instead of $k_2$; this suffices.

\medskip
{\bf Case 2b: $-m/2 + \delta m < k_3 < -m/10$}.
This case is analogous to Case 4 above: we integrate by parts in $\s$ and in $\eta$,
obtain a main term similar to \eqref{ghnupr8} with the roles of $g$ and $h$ exchanged,
a factor of $1/|\xi-\eta|$ instead of $1/|\s|$, and $\psi_{\underline{k}'}$ instead of $\varphi_{\underline{k}}$.
We can then use H\"older's inequality to obtain the same bound as in \eqref{ghnupr8.5}
with $k_3$ instead of $k_2$, which suffices.

\medskip
{\bf Case 3b: $k_3 \geqslant -m/10$}.
This case can be treated similarly to the previous Case 3. We integrate by parts only in $\eta$ and
obtain a main contribution similar to \eqref{ghnupr6}, which, similarly to \eqref{ghnupr6'}, 
is bounded in $L^2$ by
\begin{align}\label{ghnupr10} 
\begin{split}
C {\Big\| \int_0^t \tau_m(s) e^{-isL} P_k T_1^{R,1}[b_1']
  \Big( e^{is\epsilon_2 \jxi} \varphi_{k_3}(\xi)\big(\frac{\xi \jxi}{|\xi|^2} \cdot \nabla \big) \wt{h_{\eps_2}},
  e^{is\epsilon_1 \jxi} \varphi_\ell^{(\ell_0)}(\jxi-2\lambda)\wt{g_{\eps_1}} 
  \Big) ds \Big\|}_{L^2_x}.
\end{split}
\end{align}
We can then estimate using the Strichartz estimate exactly as in \eqref{ghnupr6.5},
and obtain the same bound with $k_3$ instead of $k_2$.
Since $k_3 \geqslant -m/10$ this is enough to obtain 
a bound by the right-hand side of \eqref{goalR}, as desired.


\medskip
Note that the estimates obtained so far also take care of the $T_1^{R,1}$ operators
when the roles of $h$ and $g$ are reversed.
To conclude the proof of the lemma we then need to estimate the $T_1^{R,2}$ contribution 
in the last line of \eqref{ghnuRpr0'}.

\medskip
{\it Estimate of the $T_1^{R,2}$ contribution}.
The estimates for this term are similar but not completely analogous to the previous ones, 
so we provide the details.
With our notation from \eqref{idop11R} and \eqref{ggbdpr1Rsym} we look at 
$T_1^{R,2}[b_2 ](e^{\epsilon_1 is \jxi } \wt{g_{\epsilon_1}}, e^{\epsilon_2 is \jxi } \wt{h_{\epsilon_2}})$ 
and begin by defining a suitable localized version of this operator,
as done similarly before in \eqref{ghnuRop} and \eqref{ghnuRop'}.

For $\underline{k}'' := (k,k_1,k_2,\ell)$ let 
\begin{align}\label{ghnuRopR2}
\begin{split}
T''_{\underline{k}''}(t,\xi) := \int_0^t s \, \tau_m(s) 
 \Big[ \int_{\R^6} e^{-is \Phi_\epss(\xi,-\eta-\sigma,\sigma)}
 \wt{g_{\epsilon_1}}(s,-\eta-\s) \wt{h_{\epsilon_2}}(s,\sigma)  \, 
  \\
  \times \rho_{\underline{k}''}(\xi,\eta,\s) \, b_2(\xi,\eta,\s) \, \overline{\nu_1^R(\eta,\xi)} \, d\eta d\sigma \Big] \, ds,
\end{split}
\end{align}
where $\Phi_\epss(\xi,-\eta-\s,\s) = - \jxi + \eps_1 \langle -\eta -\s\rangle + \eps_2 \jsig$,
\begin{align}\label{ghnuRop1R2}
\begin{split}
& \rho_{\underline{k}''}(\xi,\eta,\s) :=  \varphi_{\sim k_1}(\eta)\varphi_{\sim k_2}(\s) 
   \varphi_k(\xi) \varphi_{\ell}^{(\ell_0)}( \langle \eta+\s \rangle - 2\lambda), \qquad \ell_0 := \lfloor -m + \delta m \rfloor,
\\
& b_2(\xi,\eta,\s) = \frac{\xi}{\jxi} b'(\xi,-\eta-\s,\s) \varphi_{k_1}(\eta)\varphi_{k_2}(\s),
\end{split}
\end{align}
with $b'$ as in \eqref{ggbdpr1Rsym}.
As before, we can reduce to obtaining a slightly better bound than the claimed one for fixed $\underline{k}''$.
Also, we may assume that $\ell > \ell_0$, since the estimates for $\ell=\ell_0$
can be done without integrating by parts in $\eta$.

Notice that in contrast with the previous estimates, in what follows we rely on the 
$\xi/\jxi$ factor in the symbol $b_2$ in order to cancel
the mild singularity from the measure $\nu_1^R(\eta,\xi)$ of the form $(|\eta|\vee|\xi|)^{-2}$, 
see \eqref{nuRest}.

\medskip
{\bf Case 1c: $k_2 \leqslant -m/2 + \delta m$}.
This case is analogous to Case 2 and Case 1b above. 
For $\ell>\ell_0$ we integrate by parts in $\eta$ in \eqref{ghnuRopR2}
obtaining a contribution bounded pointwise, similarly to \eqref{ghnupr5}, by
\begin{align}\label{ghnupr5''}
\begin{split}
\int_0^t \tau_m(s) \int_{\R^6} 
 | \nabla_\eta \wt{g_{\epsilon_1}}(s,-\eta-\s)| \,|\wt{h_{\epsilon_2}}(s,\sigma)|
  \rho_{\underline{k}''}(\xi,\eta,\s) \, b_2(\xi,\eta,\s) \, |\nu_1^R(\eta,\xi)| \, d\eta d\sigma \, ds.
\end{split}
\end{align}
Taking the $L^2$ norm, using H\"older, the bound $|b_2| |\nu_1^R| \lesssim 2^k \cdot 2^{-2(k\vee k_1)}2^{10M_0}$, 
and the usual a priori bounds \eqref{dxiwtgest} and \eqref{boot}, we get, similarly to \eqref{ghnupr5.5},
\begin{align}\label{ghnupr5''.5}
\begin{split}
{\| \eqref{ghnupr5''} \|}_{L^2} \lesssim 2^m 2^{3k/2} \sup_{s\approx 2^m}
 {\| \varphi_{\ell} \nabla_\eta \wt{g_{\epsilon_1}}(s)\|}_{L^1} 
 \cdot {\| \varphi_{k_2} \wt{h_{\epsilon_2}}(s) \|}_{L^1} \cdot 2^k 2^{-2(k\vee k_1)} 2^{10M_0}
 \\
 \lesssim 2^m 2^{12M_0} \cdot 2^m\rho(2^m)m \cdot 2^{5k_2/2} Z,
\end{split}
\end{align}
which suffices.

\medskip
{\bf Case 2c: $-m/2 + \delta m < k_2 < -m/50$}.
In this case we can integrate by parts in $\s-\eta$ as well, using that 
$e^{-is \Phi_\epss(\xi,-\eta-\sigma,\sigma)} 
= (-is\eps_2\sigma/\jsig)^{-1}  (\partial_\s - \partial_\eta) e^{-is \Phi_\epss(\xi,-\eta-\sigma,\sigma)}$.
We can then proceed as in Case 4 above, see \eqref{ghnupr8} and \eqref{ghnupr8.5}.
More precisely, taking also advantage of the extra $2^k$ factor, we get a bound by
\begin{align*}
  C 2^{3k/2} \cdot {\| \varphi_{\ell} \nabla_{\eta} \wt{g_{\eps_1}} \|}_{L^1_{\eta}}
  \cdot 2^{-k_2} {\| \varphi_{k_2} \nabla_{\s} \wt{h_{\eps_2}} \|}_{L^1_{\s}}
  \cdot 2^k 2^{-2(k\vee k_1)} 2^{10M_0}
\\
\lesssim 2^{11M_0} \cdot 2^{k_2/2} Z 
  \cdot \rho(2^m) 2^m m,
\end{align*}
which is stronger than the right-hand side of \eqref{goalR}.

\medskip
{\bf Case 3c: $k_2 \geqslant -m/50$}.
This case is similar to Cases 3 and 3b above. We integrate by parts only in $\s$,
obtaining the main contribution (compare with \eqref{ghnupr6'} and \eqref{ghnupr10})
\begin{align*}
\begin{split}
{\Big\| \int_0^t \tau_m(s) e^{-isL} P_k T_1^{R,2}[b_2]
  \Big(e^{is\epsilon_1 \jxi} \varphi_\ell(\jxi-2\lambda)\wt{g_{\eps_1}}, 
  e^{is\epsilon_2 \jxi} \varphi_{k_2}(\xi)\big(\frac{\xi \jxi}{|\xi|^2} \cdot \nabla \big) \wt{h_{\eps_2}}
  \Big) ds \Big\|}_{L^2_x},
\end{split}
\end{align*}
which, using Lemma \ref{Strichartz} with $(q,r,\gamma) = (6,3,2/3)$,
the bilinear estimate \eqref{mainbilinR}, \eqref{dispbootgbis} and \eqref{boot}, 
can be estimated (exactly as in \eqref{ghnupr6.5}) by
\begin{align*}
\begin{split}
  2^{(C_0+1) M_0} \cdot \big(\rho^{6/5}(2^m)\, 2^{(0+)m} \, 2^m \big)^{5/6} \cdot 2^{-k_2} \cdot Z
  \lesssim 
  2^{m(-1/6 + (C_0+2)\delta_N + \beta)} \cdot 2^{-k_2} \cdot \rho(2^m) 2^m \cdot \e^\beta.
\end{split}
\end{align*}
This suffices for \eqref{goalR} provided $\delta_N$ and $\beta$ are small enough.

\medskip
The same arguments above can be used to treat
$T_1^{R,2}[b_2 ](e^{\epsilon_1 is \jxi } \wt{h_{\epsilon_1}}, e^{\epsilon_2 is \jxi } \wt{g_{\epsilon_2}})$,
that is, the case when the roles of $g$ and $h$ are reversed,
by inserting a cutoff in the size of the frequency of $h$, 
letting $|\eta+\s| \approx 2^{k_3}$, and distinguishing the same three case depending on 
the size of $k_3$. 

\medskip
The estimates for the terms $E_1(G,H):=E_{1,k_1,k_2,m,\epss}(G,H)$, $G,H \in \{(h,g), (g,h)\}$
are easier. 
In fact, the operator associated to the measure $\mu^{Re}$ 
satisfies the same bilinear bounds used above (see \eqref{mainbilinR} and \eqref{mainbilinRe}),
and the measure $\mu^{Re}$ satisfies better pointwise bounds
than the measures $\nu_1^R$ that appeared above,
as we can see by comparing \eqref{nuRest} and \eqref{nuReest},
and noticing that, since the input $g$ is localized at frequency of size $\approx 1$,
on the support of $E_1$ we have 
(see the notation \eqref{ggbdpr1R'}) $\max \{ \vert \xi \vert ,\vert \eta \vert ,|\s| \} \gtrsim 1$.
We therefore obtain the same bounds for the $E_1$ terms and the proof of the lemma is complete.
\end{proof}

\medskip
\subsection{Continuous interactions}\label{secRhh}
To complete the proof of Proposition \ref{proFR} 
we are left with showing the following: 

\begin{lemma} \label{hhnuR}
Under the assumptions of Proposition \ref{proFR}, the bound \eqref{goalR} holds $(G,H) = (h,h).$
\end{lemma}  

\begin{proof}
We will use similar notation and arguments to those in the proof of Lemma \ref{ghnuR}.
We concentrate only on the $D_1$ terms as the $E_1$ terms can be treated in the same way
(using \eqref{mainbilinRe} instead of \eqref{mainbilinR}, and \eqref{nuReest} instead of \eqref{nuRest}).
Recall the notation \eqref{idop11}-\eqref{idop12} and the identity \eqref{idbilR};
using also the symmetry of the inputs, we can write (compare with \eqref{ghnuRpr0'}), 
\begin{align*}
& 
{\big\| \whF^{-1}_{\xi \mapsto x} D_1 (t,\xi) \big\|}_{L^2} 
  \lesssim 
  \sum_{k_1,k_2 \leqslant M_0} {\| I^{(1)}(t) \|}_{L^2} + {\| I^{(2)}(t) \|}_{L^2}, 
  \\
& I^{(j)}(t,\xi) := \int_0^t s \, \tau_m(s) e^{-isL} T_1^{R,j}
  [b_j]\big(e^{i\epsilon_1 s \jxi} \wt{h_{\epsilon_1}} , e^{i\epsilon_2 s \jxi} \wt{h_{\epsilon_2}}\big) \, ds.
\end{align*}
As before, the symbols $b_1,b_2$ (that depend on $k_1,k_2$)
are defined as in \eqref{prbdary11} and \eqref{ggbdpr1Rsym}, and we suppressed the dependence on
$m,\eps_1,\eps_2$ for lighter notation.
We proceed in parallel to the proof of Lemma \ref{ghnuR} by looking first at $I^{(1)}$ and then at $I^{(2)}$.

\medskip
{\it Estimate of $I^{(1)}$}.
Similarly to \eqref{ghnuRpr0'}, we define a localized version of $I^{(1)}$ as follows:
for $\underline{k}:=(k,k_1,k_2,k_3)$, let
\begin{align}\label{hhnuRop}
\begin{split}
I_{\underline{k}}(t,\xi) := \int_0^t s \, \tau_m(s) \Big[ \int_{\R^6} e^{-is \Phi_\epss(\xi,\xi-\eta,\sigma)}
 \wt{h_{\epsilon_1}}(s,\xi-\eta) \wt{h_{\epsilon_2}}(s,\sigma)  \, 
  \\
  \times \varphi_{\underline{k}}(\xi,\eta,\s) \, b_1(\xi,\eta,\s) \, \nu_1^R(\eta,\sigma) \, d\eta d\sigma \Big] \, ds,
  \\
  \varphi_{\underline{k}}(\xi,\eta,\s) :=  \varphi_{\sim k_1}(\eta)\varphi_{\sim k_2}(\s) 
  \varphi_{k_3}(\xi-\eta) \varphi_k(\xi),
\end{split}
\end{align}
where $\Phi_\epss(\xi,\xi-\eta,\s) = - \jxi + \eps_1 \langle \xi-\eta\rangle + \eps_2 \jsig$,
and $b_1(\xi,\eta,\s) = \frac{\xi}{\jxi} b'(\xi,\xi-\eta,\s) \varphi_{k_1}(\eta)\varphi_{k_2}(\s)$ as before.
The variables in \eqref{ghnuRop} are localized as follows:
\begin{align}\label{hhnuRloc}
\begin{split}
& |\xi| \approx 2^k, \qquad |\eta|\approx 2^{k_1}, \qquad |\s|\approx 2^{k_2}, 
  \qquad |\xi-\eta| \approx 2^{k_3},
\\
& \mbox{with} \quad -5m < k,k_1,k_2,k_3 \leqslant M_0.
\end{split}
\end{align}
The lower bound on $k_1$ is not a consequence of our earlier reductions (see Lemma \ref{verylowfreq}),
but can be assumed by treating the case of $k_1 < -5m$ just using H\"older and a straightforward 
integration in \eqref{hhnuRop}.

We proceed to estimate ${\| I_{\underline{k}} \|}_{L^2}$ for fixed $\underline{k}$
showing that a bound by the right-hand side of \eqref{goalR} times a factor of $2^{-(0+)m}$ holds true.
We look at three main cases.


\medskip
{\bf Case 1: $k_1,k_2\leqslant -10$}. As in Case 1 in the proof of Lemma \ref{ghnuR}, we observe that if $|\eta|,|\sigma| \leqslant 2^{-8}$ the phase is uniformly lower bounded,
and can then proceed by integration by parts in $s$.
In what follows we may assume $k_1 \vee k_2 > -10$ so that, in particular, see \eqref{nuRest},
\begin{align}\label{hhnuRnu}
| \partial_\eta^a \partial_\s^b \nu_1^R (\eta,\s) | \lesssim 
  2^{(|a|+|b|+ 2)5M_0} \max\{ 1, 2^{(1-|b|)k_2^-} \}.
\end{align}

We fix $\delta>0$ small and distinguish cases depending on the sizes of the input frequencies.

\medskip
{\bf Case 2: $k_2 \vee k_3 \leqslant -m/2 + \delta m $}.
In this case, we cannot integrate by parts in either $\s$ or $\eta$.
However, we have $|\nu_1^R(\eta,\s)|\lesssim 2^{10M_0}$ (since $|\eta| \gtrsim 1$)
and direct H\"older estimates give us
\begin{align}\label{hhnuR1}
\begin{split}
{\| I_{\underline{k}}(t)\|}_{L^2} \lesssim 2^{2m} \cdot 2^{3k^+/2} 2^{k^-} \sup_{s\approx 2^m}
 {\| \varphi_{k_3} \wt{h_{\epsilon_1}}(s)\|}_{L^1} 
 \cdot {\| \varphi_{k_2} \wt{h_{\epsilon_2}}(s) \|}_{L^1} \cdot 2^{10M_0}
 \\
 \lesssim 2^{2m} 2^{12M_0} \cdot 2^{5k_3/2} Z \cdot 2^{5k_2/2} Z
\end{split}
\end{align}
which largely suffices in the current scenario.

\medskip
{\bf Case 3: $k_2 \wedge k_3 \leqslant -m/2 + \delta m$ and $k_2 \vee k_3 \geqslant -m/2 + \delta m$}.
In this case we can integrate by parts at least in one of the variables $\s$ or $\eta$. 
Let us assume without loss of generality that $k_3 \leqslant -m/2 + \delta m \leqslant k_2$.
We then integrate by parts in $\s$ in the formula \eqref{hhnuRop}; the leading order contribution 
is the one where $\partial_\s$ hits the profile, which is estimated using H\"older,
followed by \eqref{boot}, by
\begin{align}\label{hhnuR2}
\begin{split}
C 2^{2m} \cdot 2^{3k/2} \sup_{s\approx 2^m}
 {\| \varphi_{k_3} \wt{h_{\epsilon_1}}(s)\|}_{L^1} 
 \cdot 2^{-m-k_2} {\| \varphi_{k_2} \nabla_\s \wt{h_{\epsilon_2}}(s) \|}_{L^1} \cdot 2^{10M_0}
 \\
 \lesssim 2^m 2^{12M_0} \cdot 2^{5k_3/2} Z \cdot 2^{k_2/2} Z.
\end{split}
\end{align}
Since $k_3 \leqslant -m/2 + \delta m$ this bound suffices.

In what follows we may assume $k_2,k_3 \geqslant -m/2+\delta m$,
so that we can integrate by parts in both $\eta$ and $\s$.

\medskip
{\bf Case 4: $k_2 \wedge k_3 \geqslant -m/2 + \delta m $}. 
We look at two sub-cases:

\medskip
{\it Sub-case 4.1: $k_2 \wedge k_3 \geqslant -m/10 $}.
We integrate by parts in $\eta$ and $\s$ in the formula \eqref{hhnuRop}, 
and obtain the following leading order contribution when both derivatives hit the profiles
(the terms where the cutoffs or the measure are hit all have better estimates):
\begin{align}\label{hhnuR5}
\int_0^t s^{-1} \, \tau_m(s) 
  e^{isL} T_1^{R,1} [b_1] \Big(e^{i\epsilon_1 s \jxi}\varphi_{k_3} \big(\frac{\jxi\xi}{|\xi|^2} 
  \cdot \nabla \big) \wt{h_{\epsilon_1}}, 
  e^{i\epsilon_2 s \jxi} \varphi_{k_2} \big(\frac{\jxi\xi}{|\xi|^2} 
  \cdot \nabla \big) \wt{h_{\epsilon_2}} \Big) ds.
\end{align}
We then use Strichartz estimates with parameters $(2+,\infty-,0+)$,
followed by an $L^2\times L^2 \rightarrow L^{1+}$ bilinear estimate \eqref{mainbilinR} and \eqref{boot}, 
to bound
\begin{align}\label{hhnuR6}
\begin{split}
& {\| \eqref{hhnuR5} \|}_{L^2} 
\\
& \lesssim
2^{(0+)M_0}{\Big\| s^{-1} \cdot T_1^{R,1} [b_1]
  \Big( e^{i\epsilon_1 s \jxi} \varphi_{k_3} \big(\frac{\jxi\xi}{\vert \xi \vert^2} \cdot \nabla \big) 
  \wt{h_{\epsilon_1}}, 
  e^{i\epsilon_2 s \jxi} \varphi_{k_2}\big(\frac{\jxi\xi}{\vert \xi \vert^2} 
  \cdot \nabla \big) \wt{h_{\epsilon_2}} \Big) \Big\|}_{L^{2-}_{s \approx 2^m} L^{1+}_x} 
\\
& \lesssim 2^{(C_0+1)M_0} \cdot 2^{-(m/2)+} \cdot 2^{-k_2} Z \cdot 2^{-k_3} Z.
\end{split}
\end{align}
Using $-k_2-k_3 \leqslant -m/5,$ this gives a sufficient bound.


\medskip
{\it Sub-case 4.2: $k_2 \wedge k_3 < -m/10$}. 
In this case we can use H\"older and \eqref{hhnuRnu} to write, 
for $s\approx 2^m$, 
\begin{align}\label{hhnuR7}
\begin{split}
{\| \eqref{hhnuR5} \|}_{L^2} 
& \lesssim 2^{3k/2} \cdot 2^{10 M_0} \cdot 2^{-k_2-k_3} 
  {\| \varphi_{k_3} \nabla_\eta \wt{h} \|}_{L^1} {\| \varphi_{k_2} \nabla_\s \wt{h} \|}_{L^1}
\\
& \lesssim 2^{12 M_0} \cdot 2^{(k_2+k_3)/2}  Z^2.
\end{split}
\end{align}
This is enough to conclude.

\medskip
{\it Estimate of $I^{(2)}$}. 
For this second and last term we can use similar arguments.
Recalling the notation \eqref{idop12} with \eqref{prbdary11} and \eqref{ggbdpr1Rsym}, and arguing as before, 
we reduce to estimating the $L^2$ norm of the localized operator 
\begin{align}\label{hhnuRT2}
\int_0^t s \, \tau_m(s) e^{-isL} P_k T_1^{R,2}
  [b_2]\big(e^{i\epsilon_1 s \jxi} \varphi_{k_3}\wt{h_{\epsilon_1}} , 
  e^{i\epsilon_2 s \jxi} \varphi_{\sim k_2}\wt{h_{\epsilon_2}}\big) \, ds, 
\end{align}
for $m=1,2,\dots$ and fixed $k,k_1,k_2,k_3 \in [-5m,M_0]$.

We can proceed in parallel to the Cases 1-4 from the previous part of the proof of this lemma.
%
We first observe that if $k \wedge k_1 \leqslant -10$ (which means $|\xi|,|\eta|\leqslant 2^{-5}$
with the variables as in \eqref{idop12}) then the phase of the operator in Fourier space is uniformly lower bounded:
$|- \jxi +\eps_1\langle \eta+\s\rangle +\eps_2 \jsig| \gtrsim 1$. 
Sufficient bounds can then be obtained by integration by parts in $s$;
see also Case 1 on page \pageref{Case1}.
In the complementary case when $k \vee k_1 \leqslant -10$ the measure $\nu_1^R(\eta,\xi)$
is uniformly bounded (up to innocuous $2^{M_0}$ factors) with all its derivatives on the support of \eqref{hhnuRT2}.
We then argue as follows:

\noindent
- For $k_2 \vee k_3 \leqslant -m/2 + \delta m$, for some small enough $\delta>0$, 
we estimate as in Case 2 from earlier in the proof of this lemma, and obtain the same bound \eqref{hhnuR1};

\noindent
- For $k_3 \leqslant -m/2 + \delta m \leqslant k_2$ we integrate by parts in $\eta-\s$ as in Case 3
from earlier in this proof, and estimate exactly as in \eqref{hhnuR2}.
The case $k_2 \leqslant -m/2 + \delta m \leqslant k_3$ can be handled in the same way integrating by parts in $\eta$;

\noindent
- For $k_3\wedge k_2 \geqslant -m/2 + \delta m$ we integrate by parts both in 
$\eta-\s$ and $\s$ as in Case 4, obtaining a term like \eqref{hhnuR6}.
We then estimate as in \eqref{hhnuR7} when $k_3\wedge k_2 \leqslant -m/10$,
while for $k_3\wedge k_2 \geqslant -m/10$ we use the same Strichartz estimate,
with the bilinear bound \eqref{mainbilinR}, to obtain the same final upper bound in \eqref{hhnuR6}.

The proof of the lemma and of the main Proposition \ref{proFR} is concluded. 
\end{proof}

\section{Mixed terms}\label{secmixed}
In this section we deal with the mixed terms identified in \eqref{M} and \eqref{RM}.
We will first concentrate on the leading order terms \eqref{M1} and \eqref{M2}
which, recall, are given by
\begin{align}
\label{mixedM1}
M_1(t,\xi) & := -i\int_0 ^t B(s) \int_{\R^3} e^{-is(\jxi - \jeta - \lambda)} 
  \jeta^{-1} \wt{h}(s,\eta) \nu(\xi,\eta) d\eta \, ds,
\\ 
\label{mixedM2}
M_2(t,\xi) & := i \int_0 ^t B(s) \int_{\R^3} e^{-is(\jxi - \jeta - \lambda)} 
  \jeta^{-1} \wt{g}(s,\eta) \nu(\xi,\eta) d\eta \, ds.
\end{align}
We aim to show the following main proposition:

\begin{proposition}\label{proM12}
Under the a priori assumptions of Proposition \ref{boot} we have
\begin{align}\label{mixedest}
{\| \partial_\xi M_1(t) \|}_{L^2} + {\| \partial_\xi M_2(t) \|}_{L^2} 
  \lesssim \rho(t) ^{1-\beta + \delta} t \varepsilon^{\beta},
\end{align}
for some $\delta >0$.
\end{proposition}

Consistently with Remark \ref{localT}, the estimate \eqref{mixedest} 
suffices to obtain the desired result, as in \eqref{bootconc}, for these terms.
In Subsection \ref{ssecRM} we will discuss how to handle the remainder terms in \eqref{RM}.

\subsection{Preliminaries} \label{prelimMixmain}
We first recall some basic facts that will be useful in the rest of the section,
and then we set up the proof of Proposition \ref{proM12}. 

Recall that  we have defined
\begin{align}\label{Mnudef}
\nu(\xi,\eta) = \int_{\R^3} \overline{\psi(x,\xi)}\psi(x,\eta) \phi(x)\, dx.
\end{align}
In view of 
the estimates
\begin{align}\label{psiLinfty}
& | \partial_x^\alpha \partial_\xi^\beta \psi(x,\xi) | 
  \lesssim_{\alpha,\beta} ( \jxi^{|\alpha|} + (|\xi|/\jxi)^{1-|\beta|} ) \jx^{|\beta|},
  \qquad 0 \leqslant |\alpha|,|\beta| \leqslant N,
\end{align}
for $V$ with sufficiently many bounded Schwartz semi-norms (see Lemma 3.2 in \cite{PS} for a precise statement),
we have
\begin{align}\label{Mnuest}
\begin{split}
& 
  \big| \varphi_{k}(\xi) \varphi_{k_1}(\eta) \nabla_\xi^a \nabla^b_\eta \nu(\xi,\eta) \big| 
  \lesssim  1 \qquad |a|,|b|\leqslant 1,
\\
& \big| \varphi_{k}(\xi) \varphi_{k_1}(\eta)
  \nabla_\xi^a \nabla_\eta^b \nu(\xi,\eta) \big| \lesssim_{a,b} 
  2^{-(|a|-1)k^-} 2^{-(|b|-1)k_1^-}, \qquad 1 \leqslant |a|,|b| \leqslant N.
\end{split}
 \end{align}
In particular, $\nu$ does not behave worse than cutoff functions, so we will not need to worry much about
derivatives in $\eta$ falling on it during our integration by parts procedures.

We also recall the following lemma that will be useful in this section:
\begin{lemma} \label{Kernel-Sch}
Let $F:\mathbb{R}^3 \rightarrow \mathbb{R}$ be smooth in a ball 
$B_R(z) \subset \mathbb{R}^3, z \in \mathbb{R}^3, R>0.$ Then 
\begin{align*}
\int_{B_R(z)} \varphi_{\leqslant \lambda} \big( F(x) \big) 
\varphi_{\geqslant \mu} \big( \nabla F (x) \big) dx \leqslant 2^{-\mu} 2^{\lambda} R^2.
\end{align*} 
\end{lemma}

\begin{proof}
The proof is a straightforward adaptation of Lemma 6.2 part (1) in \cite{PW}.
\end{proof}

We start the analysis of $M_1$ and $M_2$; differentiating and introducing
localizations in $s,\xi,\eta$, and write, for $i=1,2$ (up to irrelevant constants)
\begin{align}\label{M100}
\begin{split}
\partial_{\xi} M_i &=  \sum_{m,k,k_1} I_{m,k,k_1}(G_i) + II_{m,k,k_1}(G_i), 
  \qquad G_1 := h, \quad G_2:=g,
\\
I_{m,k,k_1}(G_i) &:= \varphi_{k}(\xi) \frac{\xi}{\jxi} \int_0^t s B(s) \int_{\R^3} e^{-is(\jxi - \jeta - \lambda)} \varphi_{k_1}(\eta) \wt{G_i}(s,\eta)
  \nu(\xi,\eta) d\eta \, \tau_m(s) ds,  \\
II_{m,k,k_1}(G_i) &:=  \varphi_k(\xi) \int_0^t B(s) \int_{\R^3} e^{-is(\jxi - \jeta - \lambda)}  \varphi_{k_1}(\eta) \wt{G_i}(s,\eta)
 \partial_{\xi} \nu(\xi,\eta) d\eta \, \tau_m(s) ds.
\end{split}
\end{align}

The terms $I_{m,k,k_1}(g)$ and $I_{m,k,k_1}(h)$ need quite different treatments, 
with the latter being more challenging. This will be carried out in Subsections \ref{ssecM2} and \ref{ssecM1}-\ref{lemmaA} respectively.
The terms $II_{m,k,k_1}(g)$ and $II_{m,k,k_1}(h)$ can essentially be handled in the same way. 
This will be done in Subsection \ref{ssecII}.
In the rest of the section we will drop the dependence of $I_{m,k,k_1}$ and $II_{m,k,k_1}$ 
on the profiles when there is no confusion.

\medskip
\subsection{Estimate of $I_{m,k,k_1}(g)$}\label{ssecM2}
We begin with $I_{m,k,k_1}(g)$ which is easier to handle than $I_{m,k,k_1}(h)$. 
The main result of this subsection is

\begin{lemma}
For $m=0,1,\dots$ we have for some $\delta>0$
\begin{align}\label{M2.1}
 \Big \Vert \sum_{k} I_{m,k,k_1}(g) \Big \Vert_{L^2_\xi} \lesssim \rho(2^m)^{1-\beta +\delta} 2^m\e^{\beta}.
\end{align} 
\end{lemma}

Note that since $g$ is localized at frequencies $\sim 1,$ the sum over the index $k_1$ is finite.

\begin{proof}
In this case we recombine the sum on $k,$ and denote $I_{m,k_1}:=\sum_{k} I_{m,k,k_1}$.
First, let us observe that if we restrict the frequency integral to 
the region where $|\jeta - 2\lambda| \leqslant 2^{-10m}$, which has approximate size $2^{-10m}$,
then \eqref{M2.1} can be obtained easily using $|\nu| \lesssim 1$ and the bound \eqref{inftyfreqg} 
which gives ${\| \wt{g}(s) \|}_{L^\infty} \lesssim m$ for $s\approx 2^m$.
Therefore, in the rest of the proof we may assume that $|\jeta - 2\lambda| \geqslant 2^{-10m}$;
in particular, 
integrating by parts in the formula for $\wt{g}$ we can slightly abuse notation inserting the corresponding cutoff,
and write
\begin{align}\label{g1g2}
\wt{g}(s,\eta) & = i \varphi_{[-10m,-C]}(\jeta - 2 \lambda) \big( \wt{g}_1(s,\eta) + \wt{g}_2(s,\eta) \big),
\\
\label{g1}
\wt{g}_1(s,\eta) & := \frac{e^{-is(\jeta - 2 \lambda)}}{\jeta - 2\lambda}  B^2(s) \wt{\theta}(\eta) 
 -  \frac{B^2(0)}{\jeta - 2 \lambda} \wt{\theta}(\eta), \qquad \theta := {\bf{P_c}} \phi^2,
\\
\label{g2}
\wt{g}_2(s,\eta) & := \int_0^s \dot{B}(\tau) B(\tau) e^{-i\tau(\jeta - 2 \lambda)}
  \, \frac{\wt{\theta}(\eta)}{\jeta - 2 \lambda} \,d\tau. 
\end{align}
For convenience, we may omit the cutoff $\varphi_{[-10m,-C]}$.
Plugging the contribution from \eqref{g1} into the integral \eqref{M100} gives the terms
\begin{align*}
I_{m,k_1}^{(11)}(g) := & \int_{0}^t s B^3(s) e^{-is(\jxi - 3 \lambda)} \int_{\R^3} 
  \frac{\wt{\theta}(\eta)}{\jeta - 2 \lambda} \nu(\xi,\eta) d\eta \,\tau_m(s) ds,
\\
I_{m,k_1}^{(12)}(g) := & B(0)^2 \int_{0}^t s B(s) \int_{\R^3} e^{-is(\jxi - \jeta - \lambda)} 
  \frac{\wt{\theta}(\eta)}{\jeta - 2 \lambda} \nu(\xi,\eta) d\eta \,\tau_m(s) ds.
\end{align*}
We bound the first term using Plancherel and the inhomogeneous Strichartz estimate in Lemma \ref{Strichartz} 
with $(q,r,\gamma) = (\infty,2,0)$ and $(\widetilde{q},\widetilde{r},\gamma)=(2,\infty,0)$: 
\begin{align*}
{\| I_{m,k_1}^{(11)}(g) \|}_{L^2_\xi} & =
  {\Big\| \int_{0}^t s B^3(s) e^{-is(L - 3 \lambda)} 
  \phi \, \frac{1}{L - 2 \lambda} \theta \, \tau_m(s) ds \Big\|}_{L^2_x} 
\\
&
\lesssim {\Big\| s B^3(s) \phi \, \frac{1}{L - 2 \lambda} \theta \, \tau_m(s) \Big\|}_{L^2_s L^1_x}
\\
& \lesssim \Vert s B^3(s) \Vert_{L^2_{s \approx 2^m}}
  {\Big\| \phi \frac{1}{L - 2\lambda} \theta \Big\|}_{L^1} \lesssim \rho(2^m)^{3/2} 2^{3m/2}    ,
\end{align*}
having used the boundedness of the resolvent on weighted $L^2$ spaces. Recalling Remark \ref{remm}, this is sufficient.

For the term $I_{m,k_1}^{(12)}(g)$ above
we can use again Lemma \ref{Strichartz} followed by the local decay estimates from Lemma \ref{localdecay}
to make up for the lack of a $B^2(s)$ factor compared to the term $I_{m,k_1}^{(11)}$:
\begin{align*}
{\| I_{m,k_1}^{(12)}(g) \|}_{L^2_\xi} 
  & \lesssim \varepsilon_0 \Vert s B(s) \Vert_{L^2_{s\approx 2^m}} 
  \sup_{s\approx 2^m} 
  {\Big\|\phi \, e^{isL} 
  \, \frac{1}{L - 2\lambda} \theta \Big\|}_{L^1_x} 
  \\
  & \lesssim \varepsilon_0 \Vert s B(s) \Vert_{L^2_{s\approx 2^m}} 
  \cdot 2^{-(3/2)m}
  \lesssim \varepsilon_0 \rho(2^m)^{1/2} 2^{3m/2} \cdot 2^{-(3/2)m},
\end{align*}
which is sufficient.

Plugging the contribution from \eqref{g2} into the integral in \eqref{M100},
and summing over $k$ as above, gives the term
\begin{align*}
& I_{m,k_1}^{(2)}(g) := \int_0^t e^{-is(\jxi-\lambda)} s B(s)  \int_{\R^3} \wt{G}(s,\eta)\, \nu(\xi,\eta) d\eta \, \tau_m(s) ds,
  \\
& \wt{G}(s,\eta) := \int_0^s  e^{2i\lambda\tau} \dot{B}(\tau) B(\tau) e^{i(s-\tau)\jeta}
  \, \frac{\wt{\theta}(\eta)}{\jeta - 2 \lambda} \,d\tau.
\end{align*}
Using Lemmas \ref{localdecay} and \ref{renorm-A}, for $s \approx 2^m$, we see that
\begin{align*}
{\| \jx^{-10} G(s) \|}_{L^1_x} \lesssim \int_0^s |\dot{B}(\tau)||B(\tau)|
  {\Big\| \jx^{-10} e^{i(s-\tau)L} \frac{\theta}{L- 2 \lambda}  \Big\|}_{L^1_x} \,d\tau
  \\ \lesssim \int_0^s \big( \rho^{2}(\tau) + \rho^{3/2-a}(\tau) \langle \tau \rangle^{-1}  \big) \langle s-\tau \rangle^{-3/2} d\tau 
  \lesssim \rho(2^m)^{3/2-}.
\end{align*}
Applying the same Strichartz estimate used for $II$ above gives
\begin{align*}
{\| II_{m,k_1}^{(2)}(g) \|}_{L^2_\xi}
  \lesssim {\| s B(s) \|}_{L^2_{s\approx 2^m}} 
  \sup_{s\approx 2^m} 
  {\| \phi \, G(s) \|}_{L^1} \lesssim 2^{3m/2} \cdot \rho(2^m)^{2-},
\end{align*}
which suffices and concludes the proof of the lemma.
\end{proof}

\medskip
\subsection{Estimate of $I_{m,k,k_1}(h)$}\label{ssecM1}
This term is harder to treat than the previous one and we need a more refined analysis.
We estimate the terms $I_{m,k,k_1}$ in the remainder of this subsection and the next,
and treat the simpler term $II_{m,k,k_1}$ in Subsection \ref{ssecII}.

The main goal of this subsection is then to prove the following:

\begin{proposition}\label{propM1}
Under the a priori assumptions of Proposition \ref{propboot},
and with the definition \eqref{M100}, 
we have, for all $m=0,1,\dots$, 
\begin{align}\label{propM1bd}
\sum_{k,k_1\in\Z}\Vert I_{m,k,k_1} (h) \Vert \lesssim \rho(2^m)^{1-\beta+\delta} 2^{m} \varepsilon^{\beta} ,
\end{align}
for some $\delta \in (0,\beta/2)$. 
\end{proposition}
In view of \eqref{M100}, the above result is consistent with the desired bound in the main Proposition \ref{proM12}.
Proposition \ref{propM1} will follow from Lemmas \ref{LemmaM1A} and \ref{LemM1k1} and 
Proposition \ref{propM1'}.

Let us first deal with very small or large frequencies. 
\begin{lemma}\label{LemmaM1A}
Under the assumptions of Proposition \ref{propM1}, we have 
\begin{align}\label{M1Abd1}
\sum_{\min(k,k_1) < -3m} 
  \Vert I_{m,k,k_1} (h) \Vert \lesssim \rho(2^m)^{1-\beta+\delta} 2^{m} \varepsilon^{\beta},
\end{align}
and
\begin{align}\label{M1Abd2}
\sum_{\max(k,k_1) > \beta m} 
  \Vert I_{m,k,k_1} (h) \Vert \lesssim \rho(2^m)^{1-\beta+\delta} 2^{m} \varepsilon^{\beta}.
\end{align}
\end{lemma}

\begin{proof}
Using H\"older, Sobolev's embedding, and the a priori estimates, we have 
\begin{align}\label{Pk1h}
{\| P_{k_1} h(t) \|}_{L^2} & \lesssim \min\big( 2^{-Nk_1^+} {\| h(t) \|}_{H^N}, 
  \, 2^{k_1} {\big\| \wt{h}(t) \big\|}_{H^1_\xi} \big)
  \lesssim \big( 2^{-Nk_1^+} \wedge 2^{k_1} \big) \rho(t)^{1-\beta} t \e^{\beta}.
\end{align}
Then we can estimate using Plancherel and \eqref{Pk1h}
\begin{align} \label{reuse}
{\| I_{m,k,k_1}(h) \|}_{L^2} & \lesssim 2^{k^-} \int_0^t s |B(s)| 
  {\big\| P_k ( \phi P_{k_1} L^{-1}e^{isL}h(s) ) \big\|}_{L^2} \tau_m(s) ds
\\
\nonumber
& \lesssim 2^{2m} \rho(2^m) \cdot 2^{k^-} 
  \sup_{s\approx 2^m} 2^{-(N/2)k^+} {\big\| P_{k_1} L^{-1}e^{isL}h(s) \big\|}_{H^{N/2}}
\\
\nonumber
& \lesssim 2^{k^-} 2^{-(N/2) k^+} \cdot \big( 2^{-(N/2)k_1^+} \wedge 2^{k_1} \big) 
  \cdot 2^{2m} \rho(2^m)^{2-\beta} 2^m \e^{\beta} .
\end{align}
This gives simultaneously 
\eqref{M1Abd1} as well as \eqref{M1Abd2}, since, by our choice \eqref{mteps}, we have $\beta N/2 > 1$.
\end{proof}

Note that, in view of the above lemma, there are only $O(m^2)$ terms left 
to estimate in the sums in \eqref{propM1bd}; therefore, to obtain Proposition \ref{propM1} 
it suffices to show that for some $\delta' > \delta$
\begin{align}\label{M1est}
\begin{split}
{\| I_{m,k,k_1}(h) \|}_{L^2_\xi} \lesssim \rho(2^m)^{1+\delta^\prime-\beta} 2^m \e^{\beta}, 
  \qquad  -3m \leqslant k, k_1 \leqslant \beta m.
\end{split}
\end{align}

We make another reduction in the following lemma
by treating the case of relatively small $|\eta|$ or $|\xi|$.

\begin{lemma}\label{LemM1k1}
Let $\min(k,k_1) \geqslant -3m$ and $\max(k,k_1) \leqslant \beta m$,
and assume in addition that $\min(k,k_1) \leqslant - 5\beta m$. Then we have
\begin{align*}
{\| I_{m,k,k_1}(h) \|}_{L^2} \lesssim \rho(2^m)2^m.
\end{align*}
\end{lemma}

\begin{proof}
We apply Plancherel and the Strichartz estimate from Lemma \ref{Strichartz} to estimate
\begin{align}\label{LemM1k10}
\begin{split}
{\|  I_{m,k,k_1}(h) \|}_{L^2_\xi} & = c {\Big\| P_k \frac{R |H|^{1/2}}{L}\int_0^t s B(s) e^{isL} 
  \big( \phi \, e^{isL} L^{-1} h \big) \, \tau_m(s) ds \Big\|}_{L^2}
  \\ 
  & \lesssim 2^{k^-} \cdot {\big\| s B(s) \big\|}_{L^2_{s\approx 2^m}} 
  \sup_{s\approx 2^m} {\big\|  \phi \, \big( L^{-1} e^{isL} P_{k_1} h \big) \big\|}_{L^1}
  \\ 
  & \lesssim 2^{k^-} \cdot 2^{3m/2} \rho(2^m)^{1/2} 
  \cdot \sup_{s\approx 2^m} {\big\| \jx^2 \phi \, ( L^{-1} e^{isL} P_{k_1} h ) \big\|}_{L^2}.
\end{split}
\end{align}
Here $R$ is the distorted Riesz transform $\wt{Rf}(\xi) = (\xi/|\xi|) \wt{f}(\xi)$
and $H=-\Delta+V$ the Schr\"odinger operator, $\wtF (|H|^\alpha f)(\xi) = |\xi|^\alpha\wt{f}(\xi)$.
Let us define
\begin{align}\label{LemM1k12}
H_{k_1}(s,\xi) = \widetilde{\mathcal{F}} \Big( \jx^2 \phi \, \big( L^{-1} e^{isL} P_{k_1} h\big) \Big)
 = 
 \int_{\R^3} e^{is\jeta} \varphi_{k_1}(\eta) \, \jeta^{-1}\wt{h}(s,\eta) \nu'(\xi,\eta) d\eta,
\end{align} 
where 
\begin{align*}
\nu'(\xi,\eta) := \int_{\mathbb{R}^3} \overline{\psi(x,\xi)} \psi(x,\eta) \langle x \rangle^2 \phi(x) \,dx.  
\end{align*}
Notice that $\nu'$ satisfies the estimates \eqref{Mnuest} as well.
By \eqref{LemM1k10} and Plancherel it 
suffices to show that, for $s\approx 2^m$, 
we have
\begin{align}\label{LemM1k11}
2^{k^-}{\big\| H_{k_1}(s) \big\|}_{L^2_\xi} \lesssim  2^{-m}, \quad \mbox{when} \quad \min(k,k_1) \leqslant -5\beta m.
\end{align}


Since we can estimate, using Bernstein and \eqref{Pk1h},
\begin{align*}
{\| H_{k_1}(s) \|}_{L^2} \lesssim {\| e^{isL} P_{k_1} h \|}_{L^\infty} 
 \lesssim 2^{5k_1/2} \cdot Z, 
\end{align*}
we immediately get \eqref{LemM1k11} provided 
$k_1 \leqslant -m/2 + \delta m$, for $\delta$ small enough (recall $2^m \geqslant \rho_{0}^{-1/2}$).

When instead $k_1 \geqslant -m/2 + \delta m$ we can integrate by parts and write
\begin{align*}
H_{k_1}(s,\xi) &:= H_{k_1}^{(1)}(s,\xi) + H_{k_1}^{(2)}(s,\xi) + \, \{\mbox{similar terms}\},
\\
H_{k_1}^{(1)}(s,\xi) & := \frac{i}{s} 
  \int_{\R^3} e^{is\jeta} \frac{\eta}{|\eta|^2} \cdot \nabla_{ \eta }\wt{h}(s,\eta) 
  \, \varphi_{k_1}(\eta) \nu'(\xi,\eta) d\eta,
  \\
H_{k_1}^{(2)}(s,\xi) & := 
  \frac{i}{s} \int_{\R^3} e^{is\jeta} \wt{h}(s,\eta)  \varphi_{k_1}(\eta) \, 
  \frac{\eta}{|\eta|^2} \cdot \nabla_\eta \nu'(\xi,\eta) d\eta,
\end{align*} 
where ``similar terms'' are those with $\nabla_\eta$ hitting the cutoff or $\eta /\vert \eta \vert^2$; it will be
clear how such terms can be handled like the other ones. 

Then we bound, for $s\approx 2^m$,
\begin{align}\label{LemM1k13}
\begin{split}
{\| H_{k_1}^{(1)}(s) \|}_{L^2} & \lesssim 
2^{-m}  {\big\| \wt{\mathcal{F}}^{-1} \big( \varphi_{k_1} e^{is \langle \eta \rangle} 
  \eta/|\eta|^{-2} \cdot \nabla_{\eta } \wt{h} \big) \big\|}_{L^\infty} 
\lesssim 
2^{-m} 2^{k_1/2} \cdot Z. 
\end{split}
\end{align}

Similarly, we can handle the remaining term using Plancherel and $|\nabla_\eta \psi| \lesssim 1$: 
\begin{align}\label{LemM1k14}
\begin{split}
{\| H_{k_1}^{(2)}(s) \|}_{L^2} & \lesssim 
  2^{-m} { \Big\| \int_{\mathbb{R}^3} \overline{\psi(x,\xi)} \jx^2 \phi(x)
    \Big[ \int_{\R^3} e^{is\jeta} \wt{h}(s,\eta)  \varphi_{k_1}(\eta) \, 
  \frac{\eta}{|\eta|^2} \cdot \nabla_\eta \psi(x,\eta) d\eta \Big] \, dx \Big\|}_{L^2_\xi}
\\
& \lesssim 2^{-m} { \Big\| \int_{\R^3} e^{is\jeta} \wt{h}(s,\eta)  \varphi_{k_1}(\eta) \, 
  \frac{\eta}{|\eta|^2} \cdot \nabla_\eta \psi(x,\eta) d\eta\ \Big\|}_{L^\infty_x}
\\
& \lesssim 2^{-m} 2^{-k_1} {\| \varphi_{k_1}\wt{h} \|}_{L^1}
 \lesssim 2^{-m} 2^{k_1/2} \cdot Z. 
\end{split}
\end{align}
In view of \eqref{LemM1k13} and \eqref{LemM1k14}, 
and since $2^{k^-} 2^{-m} 2^{k_1/2} \cdot \rho(2^m)^{1-\beta} 2^m \e^{\beta}
  \lesssim 2^{-m}$ when $\min(k,k_1) \leqslant -5\beta m$,
we obtain \eqref{LemM1k11} and conclude the proof.
\end{proof}


With Lemmas \ref{LemmaM1A} and \ref{LemM1k1} we have reduced the proof Proposition \ref{propM1}
to the proof of the following proposition, which will occupy the rest of this subsection.

\begin{proposition}\label{propM1'}
With the assumptions and notation of Proposition \ref{propM1},
assume that $-5\beta m \leqslant k,k_1 \leqslant \beta m$.
Then
\begin{align}\label{propM1'bd}
\Vert I_{m,k,k_1}(h) \Vert_{L^2} \lesssim \rho(2^m)^{1-\beta+\delta'} 2^m \e^{\beta},
\end{align}
for some $\delta' >\delta$.
\end{proposition}

\begin{proof}
We introduce an additional localization in $\Phi.$ More precisely let us define, for $m = 1,2,\dots$
\begin{align}\label{M1I}
\begin{split}
I_{m,k,k_1,p}(t,\xi) := 
  \frac{\xi}{\jxi} \varphi_k(\xi) \int_0^t s B(s) \int_{\R^3} e^{-is\Phi(\xi,\eta) } 
  \varphi_p^{(p_0)}(\Phi(\xi,\eta)) \varphi_{k_1}(\eta) \, \jeta^{-1}\wt{h}(s,\eta) \\
  \times  \nu(\xi,\eta) d\eta \, \tau_m(s) ds,
\end{split} 
\\
\nonumber
  & \Phi(\xi,\eta) := \jxi -\jeta - \lambda, \qquad p_0 := \lfloor -3m/5 \rfloor.
\end{align}
We divide the proof in various cases and subcases depending on the sizes of the parameters
$p$, $2^p \approx |\Phi|$, and $k_1$, $2^{k_1} \approx |\eta|$.

\medskip
\noindent
{\bf Case 1: $p = p_0$.}
In this case the phase is small and there are few oscillations in $s$.
%
%
We integrate by parts in frequency in \eqref{M1I} and write (up to irrelevant constants)
\begin{align}\label{M15}
\begin{split}
& I_{m,k,k_1,p_0} = J_{m,k,k_1,p} \, + \, \{\mbox{easier terms}\},
\\
& J_{m,k,k_1,p} := \frac{\xi}{\jxi} \int_{0}^t B(s) \int_{\R^3} e^{-is\Phi(\xi,\eta)}
  \frac{\eta}{|\eta|^2} \cdot \nabla_{\eta} \wt{h}(s,\eta) 
   \varphi_{k}(\xi) \varphi_{k_1}(\eta) \varphi_{p_0}(\Phi) \nu(\xi,\eta) \,d\eta \,\tau_m(s)ds.
\end{split}
\end{align}
The ``easier terms'' are those where $\nabla_\eta$ hits the cutoff $\varphi_{k_1}$ (or the factor of $\eta/|\eta|^2$),
or the cutoff $\varphi_p(\Phi)$, or the measure $\nu$; 
in all these cases there is a net gain in negative powers of $s$.
For example, we see that $\partial_\eta\varphi_p(\Phi) = \varphi_{\sim p}(\Phi) 2^{-p} \eta/\jeta$;
then, we can repeat the integration by parts in $\eta$ gaining an $s^{-1}$ factor, while we have only lost 
$2^{-p} \lesssim 2^{3m/5}$, and apply the same arguments given below.

Denote $$K_2(\xi,\eta) := 2^{-k_1} \varphi_k(\xi) \varphi_{k_1}(\eta) \varphi_{p_0} (\Phi) \nu(\xi,\eta).$$ 
Then, using Lemma \ref{Kernel-Sch},
\begin{align*}
 \int_{\R^3} \vert K_2(\xi,\eta) \vert d\eta & \lesssim 2^{p_0-k_1 -k_1^-} 2^{2 \beta m} 
 	\sup_{\vert \xi \vert \approx 2^k, \vert \eta \vert \approx 2^{k_1}} \vert \nu(\xi,\eta) \vert, 
 \\ 
 \int_{\R^3} \vert K_2(\xi,\eta) \vert d\xi & \lesssim 2^{p_0-k_1} 2^{-k^-} 2^{2 \beta m} 
	\sup_{\vert \xi \vert \approx 2^k, \vert \eta \vert \approx 2^{k_1}} \vert \nu(\xi,\eta) \vert.
\end{align*}
In particular, we see that (recall the notation \eqref{defSchur})
\begin{align*}
{\| K_2 \|}_{Sch} \lesssim 2^{p_0 - k_1 -k_1^{-}/2- k^{-}/2 } 2^{2 \beta m} 
	\lesssim 2^{-11m/20},
\end{align*}
provided $\beta$ is small enough.
As a consequence of Schur's test we get
\begin{align}\label{M15bound}
{\big\| J_{m,k,k_1,p} \big\|}_{L^2} & 
  \lesssim 2^m \sup_{s \approx 2^m} 
  {\Big\| |B(s)| \int_{\R^3} \big| K_2(\xi,\eta) \big| \, |\nabla_{\eta} \wt{h}(s,\eta)| d\eta \Big\|}_{L^2} 
  \lesssim 2^m \rho(2^m)^{1/2} {\| K_2 \|}_{Sch} \cdot Z,
\end{align}
which is bounded by the right-hand side of \eqref{propM1'bd}.

\medskip
\noindent
{\bf Case 2: $p > p_0$}.
In this case we can also take advantage of integration by parts in time,
%
writing
\begin{align}
\nonumber
&\vert I_{m,k,k_1,p} 
  \vert \leqslant \big\vert  I^{(1)}_{m,k,k_1,p} \big\vert 
 +  \big\vert I^{(2)}_{m,k,k_1,p}  \big \vert, 
\\
\label{M16}
& I^{(1)}_{m,k,k_1,p}(t,\xi) := \tau_m(t) \, t \frac{\xi}{\jxi} B(t) \varphi_k(\xi) \int_{\R^3} e^{-it\Phi} 
  \frac{\varphi_{p_0}(\Phi) }{\Phi} \varphi_{k_1}(\eta) \wt{h}(t,\eta) \jeta^{-1} \nu(\xi,\eta) d\eta,
\\
\label{M17}
& I^{(2)}_{m,k,k_1,p}(t,\xi) := \int_0^t \varphi_k(\xi) \frac{\xi}{\jxi} \int_{\R^3} e^{-is\Phi}
  \frac{\varphi_p(\Phi)}{\Phi} 
  \partial_s \big( s \, B(s) \tau_m(s) \wt{h}(s,\eta) \big) \varphi_{k_1}(\eta) \jeta^{-1} \nu(\xi,\eta) d\eta ds.
\end{align}
Denoting 
\begin{align}\label{K_3}
K_{3;p,k,k_1}(\xi,\eta) = K_3(\xi,\eta) := \frac{\varphi_{p}\big(\Phi(\xi,\eta)\big)}{\Phi(\xi,\eta)} 
  \varphi_{[k_1-2,k_1+2]}(\eta)\varphi_k(\xi) \jeta^{-1} \nu(\xi,\eta), 
\end{align}
we have
\begin{align*}
 \int_{\R^3} \vert K_3(\xi,\eta) \vert d\eta & \lesssim \min\big( 2^{-k_1^+ -k_1^-} 2^{2 \beta m}, 2^{-p + 3k_1} \big) 
 	\sup_{\vert \xi \vert \approx 2^k, \vert \eta \vert \approx 2^{k_1}} \vert \nu(\xi,\eta) \vert,
 \\ 
 \int_{\R^3} \vert K_3(\xi,\eta) \vert d\xi & \lesssim \min\big( 2^{-k_1^+} 2^{-k^-} 2^{2 \beta m}, 2^{-p + 3k} \big)
	\sup_{\vert \xi \vert \approx 2^k, \vert \eta \vert \approx 2^{k_1}} \vert \nu(\xi,\eta) \vert.
\end{align*}
Then, using \eqref{Mnuest},
we get (recall the definition \eqref{defSchur})
\begin{align}\label{K_3bound}
{\| K_3 \|}_{Sch} \lesssim \min \big( 2^{(1/2)(-2p + 3k_1 + 3k)},
  2^{-(1/2)(k_1^- + k^-)} \big) 2^{2 \beta m} \lesssim 2^{7\beta m}, 
\end{align}
where the last inequality follows from $k_1,k \geqslant -5\beta m$.



\medskip
\noindent
{\it Estimate of \eqref{M16}}.
To treat $I^{(1)}_{m,k,k_1,p}$ in \eqref{M16} we integrate by parts in frequency 
and find that
\begin{align*}
I^{(1)}_{m,k,k_1,p} = \tau_m(t) \frac{\xi}{\jxi} B(t) \varphi_k(\xi)
  \int_{\R^3} e^{-it\Phi} \frac{\varphi_{p}(\Phi) }{\Phi}  
  \, \frac{\eta}{|\eta|^2} \cdot \nabla_{\eta} \wt{h}(t,\eta) 
  \varphi_{k_1}(\eta) \nu(\xi,\eta) d\eta \, + \, \{\mbox{easier terms}\}.
\end{align*}
The ``easier terms'' are similar to those in \eqref{M15}, 
and can be dealt with arguing as in the paragraph that follows that expression.
Denoting $K_4(\xi,\eta) := \varphi_{p}(\Phi) \Phi^{-1} \varphi_{k_1}(\eta) \, \eta |\eta|^{-2} \, \nu(\xi,\eta)$, 
we have, by the same reasoning as in \eqref{K_3bound},
\begin{align}\label{M1K4}
{\| K_4 \|}_{Sch} \lesssim \min \big( 2^{-k_1 -k_1^{-}/2-k^{-}/2},  2^{-p+k_1/2+3k/2 } \big) 2^{2 \beta m} 
	\lesssim 2^{12 \beta m}.
\end{align}
Using this fact,  we can bound, for $|t| \approx 2^m$
\begin{align}\label{M115}
{\big\| I^{(1)}_{m,k,k_1,p}(t) \big\|}_{L^2} & \lesssim \rho^{1/2}(2^m) \cdot {\| K_4 \|}_{Sch}
	\cdot {\| \nabla_\eta \wt{h}(t) \|}_{L^2}
	\lesssim 2^{-m/2} \cdot 2^{12\beta m} \cdot Z ,
\end{align}
which is sufficient.

\medskip
\noindent
{\it Estimate of \eqref{M17}}.
Moving on to $I^{(2)}_{m,k,k_1,p},$ we start with the easier case where the time derivative falls 
on the amplitude $B(s)$. We denote the corresponding term
\begin{align*}
I^{(21)}_{m,k,k_1,p}(t,\xi) :=
  \int_0^t \frac{\xi}{\jxi} \varphi_k(\xi) \int_{\R^3} e^{-is\Phi} \frac{\varphi_p(\Phi)}{\Phi}\,
  \frac{d}{ds} \big(s B(s)\big) \, \wt{h}(s,\eta) \varphi_{k_1}(\eta) \jeta^{-1} \nu(\xi,\eta) d\eta \, \tau_m(s) ds.
\end{align*}
Notice that this is similar to the boundary term just treated and, integrating by parts in $\eta$ and using Lemma \ref{renorm-A}, we find 
\begin{align*}
{\big\| I_{m,k,k_1,p}^{(21)} (t,\xi) \big\|}_{L^2} \lesssim 2^{m} 
  \cdot \big(\rho(2^m)^{3/2} + \rho(2^m)^{3/2-a} 2^{-m} + \rho(2^m)^{1/2} 2^{-m} \big) \cdot Z.
\end{align*}
This yields the desired result.

Next, we treat the case where the $s$-derivative in \eqref{M17} falls on the profile $\wt{h}$, that is,
\begin{align}\label{M172}
& I^{(22)}_{m,k,k_1,p}(t,\xi) := \int_0^t s \frac{\xi}{\jxi} \varphi_k(\xi) \int_{\R^3} e^{-is\Phi}
  \frac{\varphi_p(\Phi) }{\Phi} B(s) \tau_m(s) \, \partial_s \wt{h}(s,\eta) \varphi_{k_1}(\eta) \jeta^{-1} \nu(\xi,\eta) d\eta ds.
\end{align}
Let $\psi(x) = \varphi(x)/x \in \mathcal{S}$, we can write
(recall that $R$ is the distorted Riesz transform 
and $H=-\Delta+V$ the Schr\"odinger operator) 
\begin{align}
\nonumber
& I^{(22)}_{m,k,k_1,p}(t,\xi) 
  \\
  \nonumber
  & = \int_0^t  s \frac{\xi}{\jxi} \varphi_k(\xi) \int_\R 2^{-p} \check{\psi}(z) \Big[ \int_{\R^3} e^{-i(s+z2^{-p})\Phi}
  B(s) \tau_m(s) \,  \jeta^{-1}  \partial_s \wt{h}(s,\eta) \varphi_{k_1}(\eta) \nu(\xi,\eta) d\eta 
  \Big] dz ds 
  \\
  \nonumber
  & =  \int_\R 2^{-p} \check{\psi}(z) \, \wt{\mathcal{F}}^{-1} \Big[ P_k \frac{R |H|^{1/2}}{L} \int_0^t s \, e^{-i(s+z2^{-p})L} 
  \big( e^{-i(s+z2^{-p})\lambda} B(s) 
  \\ & \qquad \qquad \qquad \qquad \qquad \qquad \times \phi \,  e^{i(s+z2^{-p})L}  P_{k_1} L^{-1} \partial_s h(s)
  \big) \tau_m(s) ds \Big] dz . 
\label{M172'}
\end{align}
Therefore, we can write that for any $q \in [2,\infty],$
\begin{align}
\nonumber
& {\big\| I^{(22)}_{m,k,k_1,p}(t,\xi) \big\|}_{L^2} 
  \\
  \nonumber
  & \lesssim \int_\R 2^{-p} \vert \check{\psi}(z) \vert
  {\big\| sB(s) \cdot \phi \, e^{i(s+z2^{-p})L} P_{k_1} \partial_s h(s) \big\|}_{L^2_{s\approx 2^m} L^1_x} \, dz
  \\
  & \lesssim \int_\R 2^{-p} \vert \check{\psi}(z) \vert
  \cdot 2^{3m/2} \rho(2^m)^{1/2} 
  \cdot \sup_{s\approx 2^m} {\big\| e^{i(s+z2^{-p})L} P_{k_1} \partial_s h(s) \big\|}_{L^q_x} \, dz .
\label{I22bound}
\end{align}

Next, recall that (see \eqref{def-h} and \eqref{Duhamelf} or, equivalently, \eqref{Duhamelw0})
\begin{align}\label{dswth}
\begin{split}
\partial_s \wt{h}(s) &= e^{-is \jeta} \Big[ \wt{\mathcal{F}}(a^2(s) \theta) 
	+ \wt{\mathcal{F}}\big(a(s) \phi \, L^{-1} \Im w(s) \big) 
	+ \wt{\mathcal{F}}\big( \big( L^{-1}\Im w(s) \big)^2 \big) \Big]
	- \partial_s \wt{g}(s).
\end{split}
\end{align}
In particular, from the definition of $g$, we see that
\begin{align}\label{dswth'}
\begin{split}
\partial_s h(s) - e^{-is L} \big( L^{-1}\Im w \big)^2 
  = e^{-is L} O(\rho(s)) \theta + e^{-is L} \big( a(s) \phi \cdot  L^{-1}\Im w(s) \big).
\end{split}
\end{align}
We can then use \eqref{I22bound} 
to estimate the contributions of the two terms on the right-hand side above to \eqref{M172'}.
For the first term we use the estimate, for all $s\approx 2^m$,
\begin{align*}
{\| e^{i z 2^{-p}L} \rho(s) \theta \|}_{L^\infty_x} \lesssim (1 + z 2^{-p})^{-3/2} \rho(2^m),
\end{align*}
which, when plugged in \eqref{I22bound} with $q=\infty$, gives the upperbound
\begin{align*}
 \int_\R 2^{-p} \vert  \check{\psi}(z) \vert \cdot 2^{3m/2} \rho(2^m)^{3/2} (1+z2^{-p})^{-3/2} dz \lesssim 2^{(0+) m},
\end{align*} 
(recall that $\psi$ is Schwartz).
For the second term on the right-hand side of \eqref{dswth'} we can proceed similarly, using instead the estimate
\begin{align*}
{\big\| e^{i z 2^{-p}L}  a(s) \phi \cdot L^{-1} \Im w(s) \phi \big\|}_{L^\infty_x} & \lesssim 
	(1+z2^{-p})^{-3/2} {\big\| a(s) \phi \cdot L^{-1} \Im w(s) \big\|}_{W^{3,1}\cap H^2} 
	\\
	& \lesssim (1+z2^{-p})^{-3/2} \cdot \rho(2^m)^{1/2} \cdot \rho(2^m)^{1-\beta} 2^{5\delta_N m} \e^{\beta}
\end{align*}
relying on \eqref{decayp} (with $p=6$) and
Sobolev's embedding, the $H^N$ a priori bound and Sobolev-Gagliardo-Nirenberg interpolation
to deal with the couple extra derivatives.

To complete the estimate of $I^{(22)}_{m,k,k_1,p}(t,\xi)$ it then suffices to estimate
by the right-hand side of \eqref{propM1'bd} the remaining contribution to \eqref{M172} given by
\begin{align}
\label{M123}
& A_{m,k,k_1,p}(t,\xi) :=  \frac{\xi}{\jxi} \int_0^t s \int_{\R^3} e^{-is(\jxi-\lambda)} 
  B(s)\, \wtF \big( \big(L^{-1} \Im w \big)^2 \big)(s,\eta) K_3(\xi,\eta) d\eta \, \tau_m(s) ds,
\end{align}
where $K_3(\xi,\eta)$ is the kernel defined in \eqref{K_3} which satisfies \eqref{K_3bound}.



%

We can treat right away the case of relatively large $p$ using \eqref{I22bound}. 
Indeed, using $L^6-L^3$ estimates from Lemma \ref{dispersive-bootstrap} we see that
\begin{align*}
{ \big\| e^{-i z 2^{-p}L}  \big(L^{-1} \Im w \big)^2(s) \big\|}_{L^2_x} \lesssim 
\rho(2^m)^{2-2\beta} \cdot 2^{m/2+2\delta_N m} \e^{2\beta},
\end{align*}
which plugged into \eqref{I22bound} (with $q=2$), gives an acceptable bound if, for example, $p\geqslant-m/3$.

To summarize, we see that in order to complete the estimate of $I^{(22)}_{m,k,k_1,p}(t,\xi)$, hence the proof of Proposition \ref{propM1'},
it will then suffice to show the following:

\begin{lemma}\label{lemA}
With the assumptions and notation of Proposition \ref{propM1'}, with the definition \eqref{M123},
for all $m=1,2,\dots$, and $-3m/5 < p < -m/3$ we have 
\begin{align}\label{lemAbd}
{\big\| A_{m,k,k_1,p} \big\|}_{L^2} \lesssim \rho(2^m)^{1-\beta+\delta'} 2^m \e^{\beta},
\end{align}
for some $\delta'>\delta$.
\end{lemma}
The proof of this lemma occupies the next subsection.

\medskip
\subsection{Proof of Lemma \ref{lemA}} \label{lemmaA}
The last main term to treat for the proof of Proposition \ref{propM1} is 
\eqref{M123}.
Expanding out $ \wtF(L^{-1} \Im w)^2$, disregarding irrelevant constants, we can write
\begin{align}\label{M123a}
\begin{split}
& A = A_{m,k,k_1,p} 
  = \sum_{\epsilon_1,\epsilon_2 \in \lbrace + , - \rbrace} A^{\epsilon_1 \epsilon_2}, 
  \qquad  A^{\epsilon_1 \epsilon_2} = 
  A_{S}^{\epsilon_1 \epsilon_2} + A_{R}^{\epsilon_1 \epsilon_2}
  + A_{Re}^{\epsilon_1 \epsilon_2}
\\
& A_{\ast}^{\epsilon_1 \epsilon_2}(t,\xi) =\int_0^t s \, B(s)\, \frac{\xi}{\jxi}\varphi_k(\xi) \int_{\R^9} 
  e^{-is\Phi^{\epsilon_1 \epsilon_2} (\eta,\s,\rho)} \jsigma^{-1}
  \wt{f_{\epsilon_1}}(s,\s) \, \jrho^{-1} \wt{f_{\epsilon_2}}(s,\rho) 
  \\
& \qquad \qquad \qquad \qquad \times  \mu^\ast(\eta,\s,\rho) 
  \, K_3(\xi,\eta) d\s d\rho \, d\eta  \, \tau_m(s) ds,  \qquad \ast \in \{S,R,Re\},
  \\ 
& f_{+} = f, \quad f_{-} = \overline{f}, \quad  
\\
& \Phi^{\epsilon_1 \epsilon_2} (\eta,\s,\rho) := \jxi- \lambda - \epsilon_1 \jsig - \epsilon_2 \langle \rho \rangle.
\end{split}
\end{align}
Note that we used Proposition \ref{mudecomp} to decompose the quadratic terms into 
the piece $A_S^{\epsilon_1 \epsilon_2}$ corresponding to the singular part of the measure,
the piece $A_R^{\epsilon_1 \epsilon_2}$ corresponding to the regular part of the measure,
and remainders $A_{Re}^{\epsilon_1 \epsilon_2}$.
We treat these terms in separate sections below.

\medskip
\subsubsection{Estimate of $A_S^{\epsilon_1 \epsilon_2}$.}
We use \eqref{mudecomp0} to further decompose (again we disregard irrelevant constants)
\begin{align}\label{ASsplit}
A_{S}^{\epsilon_1 \epsilon_2} = A_{0}^{\epsilon_1 \epsilon_2} + A_{1}^{\epsilon_1 \epsilon_2},
\end{align}
where
\begin{align}\label{M123del}
\begin{split}
A_{0}^{\epsilon_1 \epsilon_2}(t,\xi) := \sum_{k_2,k_3 \in \Z} \int_0^t s \frac{\xi}{\jxi} 
  \int_{\R^6} e^{-is \Phi_{0}^{\epsilon_1 \epsilon_2}(\xi,\eta,\s)} 
  B(s)\, \jsigma^{-1} \big(\varphi_{k_2}\wt{f_{\epsilon_1}}\big)(\s)  
  \\
\qquad \qquad \times  \langle \eta - \s \rangle^{-1} 
 \big(\varphi_{k_3}\wt{f_{\epsilon_2}}\big)(s,\eta-\s) \, d\s \, K_3(\xi,\eta) d\eta  \, \tau_m(s) ds,
\end{split}
\\ 
\nonumber
& \Phi_{0}^{\epsilon_1 \epsilon_2}(\xi,\eta,\s) := \jxi-\lambda - \epsilon_1 \jsig - \epsilon_2 \langle \eta-\s \rangle,
\end{align}
with $K_3 = K_{3;p,k,k_1}$ defined in \eqref{K_3},
and where $A_{1}^{\epsilon_1 \epsilon_2}$ contains the terms corresponding to the $\nu_1$ measures; 
these latter are explicitly written out in \eqref{M123pv} and then estimated after that.
Note that we have inserted additional localization in 
$|\sigma| \approx 2^{k_2}$ and $|\eta-\sigma| \approx 2^{k_3}$.
We may also restrict to
\begin{align}\label{M123freq}
k_2,k_3 \in [-3m,\beta m] \cap \Z,
\end{align}
since very low and high frequencies are easier to estimate, as it will be clear to the reader.
Since the sums over $k_2$ and $k_3$ will then cost at most an $O(m^2)$ factor,
we will abuse notation slightly and discard them from \eqref{M123del}.
To improve legibility, we may sometime drop the indices $\epsilon_1, \epsilon_2$ 
in the estimates of $A_{0}^{\eps_1,\eps_2}=A_0$, 
since all the cases can be treated identically.
Recall that we have also restricted our estimates to $p < -m/3$ already.

\medskip
\noindent
{\it Estimate of $A_{0}$.}
We notice that the phase in \eqref{M123del} satisfies 
\begin{align}\label{M123delph0}
|\Phi_{0}^{\epsilon_1 \epsilon_2}| \gtrsim |\jeta -\epsilon_1 \jsig - \epsilon_2 \langle \eta-\s \rangle \big|
  - \big| \jxi-\lambda -\jeta | 
  \gtrsim 2^{-\beta m} + O(2^p) \gtrsim 2^{-\beta m},
\end{align}
in view of 
Lemma \ref{lemphases}, inequality \eqref{no-t-res}, 
and the restrictions $|\eta|,|\s| \lesssim 2^{\beta m}$, as well as $p < -m/3.$
We can then integrate by parts in $s$ in \eqref{M123del} at very little cost. 
Moreover, we have
\begin{align}\label{M123delph}
\big| \partial_\eta \Phi_{0}^{\epsilon_1 \epsilon_2} \big| \approx 2^{k_3}, 
\qquad \big| \partial_{\eta+\s} \Phi_{0}^{\epsilon_1 \epsilon_2} \big| \approx 2^{k_2}, 
\end{align}
and, therefore we can also integrate by parts in $\eta$ or $\eta+\s$ unless $k_3$ or $k_2$ are too small.

Integrating by parts in $s$ gives two terms that are similar (depending on which one of the two profiles 
gets differentiated) plus easier ones (when $B$ or $\tau_m$ are differentiated);
therefore, we can write 
\begin{align*}
\begin{split}
\Vert A_{0} \Vert_{L^2} & \lesssim 
\Bigg \Vert \int_0^t s B(s) \frac{\xi}{\jxi} \int_{\R^6} e^{-is \Phi_{0}} \jsigma^{-1}
  \, \partial_s \wt{f}(s,\s) \,  \langle \eta - \s \rangle^{-1} \wt{f}(s,\eta-\s) \, 
  \\
& \times  \frac{\varphi_{k_2}(\s)\varphi_{k_3}(\eta-\s)}{\Phi_{0}(\xi,\eta,\s)} 
  K_3(\xi,\eta) d\s d\eta  \, \tau_m(s) ds  \Bigg \Vert_{L^2} + \{ \textrm{similar or easier terms} \}.
\end{split}
\end{align*}
For the leading order term above we have the upper bound
\begin{align}\label{M123IBPs}
\begin{split}
& 
  C 2^{2m} \rho(2^m)^{1/2} \cdot \sup_{s\approx 2^m} 
  {\Big\| \int_{\R^3}
  K^f(\xi,\s) \,  e^{is\jsig} \partial_s \wt{f}(s,\s) d\s \Big\|}_{L^2_\xi},
  \\
  & K^f(\xi,\s) : =  \jsigma^{-1} \int_{\R^3} e^{is\langle \eta -\s \rangle} 
  \langle \eta - \s \rangle^{-1} \wt{f}(s,\eta-\s) 
  \, \frac{\varphi_{k_2}(\s)\varphi_{k_3}(\eta-\s)}{\Phi_{0}(\xi,\eta,\s)} K_3(\xi,\eta) d\eta.
\end{split}
\end{align}

Starting from the above estimate we distinguish a few cases to eventually obtain a uniform bound 
on the $L^2$-norm of $A_{0}$. 
Since ${\| \partial_s f \|}_{L^2} \lesssim \rho(2^m)$, see Lemma \ref{decay-der}, 
applying Schur's test to \eqref{M123IBPs}, we see that it suffices to prove that
\begin{align}\label{M123Kest}
{\| K^f \|}_{Sch} \lesssim 2^{-m/2-10 \beta m}.
\end{align}

To prove \eqref{M123Kest} we split $f=h+g$, denoting $K^h$ and $K^g$ the corresponding kernels. 
To deal with $K^h$ we integrate by parts in frequency when $k_3 \geqslant -m/2 + \delta m$
for some small $\delta >0$, and obtain (again we disregard lower order terms)
\begin{align}\label{M123Kest1}
{\|K^h(\xi,\s)\|}_{Sch}  &  \lesssim 2^{\beta m} \cdot 2^{-m} 2^{-k_3} {\Big\| 
  \int_{\R^3} | \nabla \wt{h}(s,\eta-\s)|\, 
  \varphi_{k_3}(\eta-\s) |K_3(\xi,\eta)| d\eta \Big\|}_{Sch}  
  \\
\notag & \lesssim 2^{\beta m -m} \cdot 2^{-k_3} \Vert K_3 \Vert_{Sch} 
  \Vert \varphi_{k_3} \nabla \wt{h} \Vert_{L^1} \lesssim 2^{11 \beta m- m} \cdot 2^{k_3/2} \cdot Z,
\end{align}
having used \eqref{K_3bound} and \eqref{M123delph0}. This yields \eqref{M123Kest}. 
If instead $k_3 \leqslant -m/2 + \delta m$, we just estimate
\begin{align*}
{\|K^h(\xi,\s)\|}_{Sch} 
  & \lesssim 2^{\beta m} \cdot \Vert K_3 \Vert_{Sch} 
  \Vert \varphi_{k_3} \wt{h} \Vert_{L^1} \lesssim 2^{\beta m} \cdot 2^{5k_3/2} \cdot Z,
\end{align*}
which again is more than sufficient.

For the $K^g$ kernel, we introduce a cut-off $\varphi_{\ell}^{(\ell_0)}(\langle \eta - \s \rangle - 2 \lambda), 
\ell \geqslant \ell_0 :=\lfloor -m + \delta m \rfloor$. 
We integrate directly in frequency when $\ell=\ell_0$ and integrate by parts in frequency for the other terms:
\begin{align}\label{K^g}
\begin{split}
\Vert K^{g}(\xi,\s) \Vert_{Sch} & \lesssim 2^{\beta m} 
  {\Big \Vert \int_{\mathbb{R}^3} \vert \wt{g}(s,\eta - \s)  \vert \varphi_{\ell_0} 
  \langle \eta - \s \rangle - 2 \lambda) |K_3(\xi,\eta)| d\eta \Big\|}_{Sch} 
  \\
& + 2^{\beta m} \cdot 2^{-m} \sum_{\ell_0<\ell<-10} {\Big\| 
  \int_{\R^3} |\nabla \wt{g}(s,\eta-\s)|\, 
 \vert \varphi_{\ell} (\langle \eta - \s \rangle - 2 \lambda) |K_3(\xi,\eta)| d\eta \Big\|}_{Sch} 
 \\
 & \lesssim 2^{11 \beta m} \cdot 2^{\ell_0} \cdot 2^m \rho(2^m) m + 2^{11 \beta m -m} 
 \cdot 2^m \rho(2^m) m ,  
\end{split}
\end{align}
where we used \eqref{K_3bound} and the bounds \eqref{inftyfreqg} and \eqref{dxiwtgest} 
for the last line. This yields \eqref{M123Kest} and concludes the estimate of $A_{0}$.

\medskip
\subsubsection{Estimate of $A_{1}^{\epsilon_1 \epsilon_2}$} 
We now look at the other terms in \eqref{ASsplit}. 
Inserting cutoffs $\varphi_{k_2}(\s)$ and $\varphi_{k_3}(\rho)$, 
similarly to the case of $A_{0}$ above, we can write 
(using the symmetry in $(\rho,\s)$, with a slight abuse of notation)
\begin{align}\label{M123pv}
\begin{split}
& \Vert A_{1}^{\epsilon_1 \epsilon_2} \Vert_{L^2} 
  \lesssim \Vert A_{1,1}^{\epsilon_1 \epsilon_2} \Vert_{L^2} 
  + \Vert A_{1,2}^{\epsilon_1 \epsilon_2} \Vert_{L^2}, 
  \\
  \\
\begin{split}
A_{1,1}^{\epsilon_1 \epsilon_2}(t,\xi) := \int_0^t s B(s) \frac{\xi}{\jxi} 
  \int_{\R^9} e^{-is \Phi_{1}^{\epsilon_1 \epsilon_2}(\xi,\s,\rho)} 
  \, \jsigma^{-1} \wt{f}(s,\s) \langle \rho \rangle^{-1} \wt{f}(s,\rho)  
  \\ 
\varphi_{k_2}(\s) \varphi_{k_3}(\rho) \nu_1^S(-\eta+\s,\rho) d\s d\rho \, K_{3;k,k_1,p}(\xi,\eta) d\eta  \, \tau_m(s) ds,  
\end{split}
\\
\begin{split}
A_{1,2}^{\epsilon_1 \epsilon_2}(t,\xi) := \int_0^t s B(s) \frac{\xi}{\jxi} 
  \int_{\R^9} e^{-is \Phi_{1}^{\epsilon_1 \epsilon_2}(\xi,\s,\rho)} 
  \, \jsigma^{-1} \wt{f}(s,\s) \langle \rho \rangle^{-1} \wt{f}(s,\rho)  
  \\ 
  \varphi_{k_2}(\s) \varphi_{k_3}(\rho) \overline{\nu_1^S(-\s-\rho,\eta)} d\s d\rho 
  \, K_{3;k,k_1,p}(\xi,\eta) d\eta  \, \tau_m(s) ds,
\end{split}
  \\
  \\
  & \Phi_{1}^{\epsilon_1 \epsilon_2}(\xi,\s,\rho) := \jxi-\lambda - \epsilon_1 \jsig - \epsilon_2 \langle \rho \rangle.
\end{split}
\end{align}
As before, we may assume, 
\begin{align}\label{M123pvfreq}
-3 m \leqslant k_2 \leqslant k_3 \leqslant \beta m, \qquad p \leqslant -m/3.
\end{align}
In what follows $A_{1,1}^{\epsilon_1 \epsilon_2}$ and $A_{1,2}^{\epsilon_1 \epsilon_2}$
will be treated very similarly.
Therefore, we will mostly focus on the first term, and then indicate how the argument 
can be adapted for $A_{1,2}^{\epsilon_1 \epsilon_2}$.

Recall the definition of $\nu_1^S$ in \eqref{nu1S}-\eqref{nu0}.
As it will be clear from the argument below, it suffices to look at the 
leading order contribution from 
the first term in \eqref{nu1S} containing the $\pv$ distribution; that is,
it suffices to look at \eqref{M123pv} with $\nu_1^S(p,q)$ replaced by 
\begin{align}
\varphi_{\leqslant -M_0-5}(|p|-|q|)\frac{b_0(p,q)}{|p|} 
  \pv \frac{1}{|p|-|q|}. 
\end{align}
The $\delta(|p|-|q|)$ part can be handled similarly (see also the proof of Theorem 6.1 from \cite{PS}).
Without loss of generality we may also replace $b_0$ by $1$.
One can easily handle the remaining parts from \eqref{nu1S} 
using Lemma \ref{PSvol}, along with the fast decay of the $K_a$ functions appearing in \eqref{nu1S}, 
and the summation properties of $b_{a,J}$.

We insert a localization in the singularity of the measure above,
by decomposing

\begin{align*}
\varphi_{\leqslant -M_0-5}(\vert \rho \vert - \vert -\eta+\s  \vert)
  = \sum_{C \leqslant -M_0-5} \varphi_{C} (\vert \rho \vert - \vert   -\eta + \s \vert).
\end{align*}
We denote $A_{1,1;C}^{\epsilon_1 \epsilon_2}$ 
the terms correspond to the $\varphi_{C}$ cutoff, which is given by
(changing variable $\s \mapsto \s + \eta$ in \eqref{M123pv})
\begin{align}\label{A11B}
\begin{split}
A_{1,1,C}^{\epsilon_1 \epsilon_2}(t,\xi) := \int_0^t s B(s) \frac{\xi}{\jxi} 
  \int_{\R^3} e^{-is (\jxi - \lambda)} 
  R_{k_2,k_3,C} (f_1, f_2) (\eta) K_{3;k,k_1,p}(\xi,\eta) d\eta  \, \tau_m(s) ds,  
\end{split}
\end{align}
where
\begin{align*}
& R_{k_2,k_3,C} (f_1,f_2) (s,\eta) := \int_{\mathbb{R}^6}  
  \wt{f_1}(s,\s+\eta) \wt{f_2}(s,\rho) \frac{1}{\vert \s \vert}
  \frac{\varphi_{C} 
  (|\s| - |\rho|)}{|\s|-|\rho|} 
  d\s d\rho, 
  \\ 
  & \wt{f_1}(s,\s+\eta) := \langle \s +\eta \rangle^{-1} \varphi_{k_2}(\s+\eta) 
  e^{-is\eps_1\langle \s +\eta\rangle} \wt{f_{\eps_1}}(s,\s+\eta),
  \\
  & \wt{f_2}(s,\rho) := \langle \rho \rangle^{-1} \varphi_{k_3}(\rho) e^{-is\eps_2 \langle \rho \rangle} \wt{f_{\eps_2}}(s,\rho).
\end{align*}
We look at two cases:

\medskip
{\it 
Case 1: $C \leqslant C_0 := -500 m$}.
We first decompose in physical space
\begin{align*}
f_1 := \varphi_{\geqslant X} f_1 + \varphi_{<X} f_1  =: f_{11} + f_{12}, 
  \qquad X:=10 m-\frac{1}{100} C,
\end{align*}
and then treat the two pieces above separately.

For the term with $f_{11}$ we bound, using the product estimate from Lemma \ref{PSvol},
(recall that we have $k_1,k_2,k_3 \leqslant \beta m$) and the boundedness of wave operators
on weighted spaces \eqref{Wobdx}:
\begin{align*}
\sum_{C \leqslant C_0} \Vert R_{k_2,k_3,C} (f_{11},f_2) \Vert_{L^2} 
  & \lesssim \sum_{C \leqslant C_0} 
  2^{10\beta m} 
  \cdot {\big\| \wt{f_{11}} \big\|}_{L^2} {\big\| \wt{f_2} \big\|}_{L^2} 
  \\
  & \lesssim
   \sum_{C \leqslant C_0} 
  2^{10\beta m} 
  \cdot 2^{-X} \cdot \Vert \jx f_{11} \Vert_{L^2} \Vert f_2 \Vert_{L^2} 
  \\
  & \lesssim
   \sum_{C \leqslant C_0} 
  2^{10\beta m} 
  \cdot 2^{-10m + \frac{1}{100}C} \cdot \Vert \nabla \wt{f_{1}} \Vert_{L^2} 
  \Vert f_2 \Vert_{L^2} 
  \lesssim 2^{-10 m} Z,
\end{align*}
which is sufficient.

For the term involving $f_{12}$ we write 
\begin{align*}
& R_{k_2,k_3,C} (f_{12},f_2) = R^{(1)} 
  (f_{12},f_2) + R^{(2)} (f_{12},f_2) + R^{(3)} (f_{12},f_2),
\\
& R^{(1)} (f_{12},f_2) := \int_{\R^6} \big[ \wt{f_{12}}(s,\eta+\s) 
- \wt{f_{12}}(s,\eta + \s \vert \rho \vert / \vert \s \vert) \big] \wt{f_2}(s,\rho) \frac{1}{\vert \s \vert}
\frac{\varphi_{C} \big( \vert \s \vert - \vert \rho \vert \big)}{\vert \s \vert  - \vert \rho \vert} d\s d\rho,  
\\
& R^{(2)} (f_{12},f_2) := \int_{\R^6}  \wt{f_{12}}(s,\eta+\s) \wt{f_2}(s,\rho) 
  \frac{1}{\vert \s \vert^2} \varphi_{C} \big( \vert \s \vert - \vert \rho \vert \big) d\s d\rho,
\\
& R^{(3)} (f_{12},f_2) = \int_{\R^6}  \wt{f_{12}}(s,\eta + \s |\rho|/|\s|) \wt{f_2}(s,\rho) 
  \frac{|\rho|}{|\s|^2} \frac{\varphi_{C} \big( \vert \s \vert - \vert \rho \vert \big)}{\vert \eta \vert - \vert \s \vert} d\s d\rho.
\end{align*}
Notice that the last term equals zero, as it can be seen passing to polar coordinates in $\s$.
For the second term, using again Lemma \ref{PSvol}, we find that
\begin{align*}
\Vert  R^{(2)} (f_{12},f_2) \Vert_{L^2} 
  \lesssim 2^{10\beta m} \cdot 2^{C} \cdot 
  \Vert f_{12} \Vert_{L^2} \Vert f_2 \Vert_{L^2}, 
\end{align*}
which is enough to conclude.
For the remaining term we first write that 
\begin{align*}
\wt{f_{12}}(\eta + \s) - \wt{f_{12}}(\eta + \theta \vert \rho \vert) = \int_0^1 \nabla 
  \wt{f_{12}} (\eta + t \s + (1-t) \vert \rho \vert \theta ) \cdot \theta (\vert \s \vert - \vert \rho \vert) ,
\end{align*}
with $\theta:=\s/\vert \s \vert.$ Using the symbol bounds on this expression ensuing from 
$| \nabla^\alpha \wt{f_{12}} | \lesssim 2^{|\alpha|X}$ in combination with Lemma \ref{PSvol},
we can obtain
\begin{align*}
\sum_{C \leqslant C_0} {\Vert R^{(1)} (f_{12},f_2)  \Vert}_{L^2} 
  & \lesssim 2^{10 \beta m} \sum_{C \leqslant C_0} 2^{15 X} 
  \cdot 2^{C} \cdot \Vert \nabla \wt{f_{12}} \Vert_{L^2} \Vert f_2 \Vert_{L^2} 
  \lesssim 2^{-10m} Z
\end{align*}
in view of our choices of $X$ and $C_0$.
We have then taken care of the region very close to the singularity of the $\pv$

\medskip
{\it Case 2: $C_0 < 
C \leqslant -M_0-5$
}.
For this case we use that the phase, see \eqref{M123pv}, is uniformly lower bounded:
\begin{align}\label{M123pvph}
|\Phi_{1}^{\epsilon_1 \epsilon_2} (\xi,\s,\rho) | \gtrsim |\jeta - \epsilon_1 \jsig - \epsilon_2 \jrho \big| - \big| \jxi-\lambda -\jeta| 
  \gtrsim 2^{-\beta m} + O(2^p) \gtrsim 2^{-\beta m};
\end{align}
see Lemma \ref{lemphases}.
Once again we will drop the signs $\epsilon_1,\epsilon_2$ 
since all cases can be treated identically.
Thanks to the lower bound on the phase, 
we can then proceed very similarly to the case of $A_{0}$ above; see the argument after \eqref{M123delph0}.

More precisely, we integrate by parts in time in \eqref{A11B} and obtain
\begin{align*}
\Vert A_{1,1;C} \Vert_{L^2} \lesssim \Big \Vert
  \int_0^t s B(s) \frac{\xi}{\jxi} \int_{\mathbb{R}^9} e^{-is\Phi_{1}^{\epss}} \jsigma^{-1} 
  \partial_s \wt{f_{\epsilon_1}}(s,\s) \langle \rho \rangle^{-1} \wt{f_{\epsilon_2}}(s,\rho) 
  \frac{\varphi_{k_2}(\s) \varphi_{k_3}(\rho)}{\Phi_{1}(\xi,\sigma,\rho)}
 \\
\times \varphi_{C}(\vert \eta - \s \vert - \vert \rho \vert) \nu_1^S(-\eta + \s,\rho) d\s d\rho 
  K_{3;k,k_1,p}(\xi,\eta) d\eta \, \tau_m(s) ds \Big \Vert_{L^2}  
 \\ 
  + \lbrace \textrm{similar or easier terms} \rbrace.
\end{align*}
The ``similar or easier terms" include the boundary terms, which can be treated in the same way as the main term,
and the terms where $\partial_s$ hits $sB(s)$ that are completely analogous.
We can also drop the dependence on the index $C$ since in this case there are $O(m)$ terms in the sum,
therefore a gain of a factor $2^{-\delta m}$ for some $\delta>0$ will ensure convergence.

We write, as in the case of $A_0$, see \eqref{M123IBPs}, that the main contribution to 
${\big\|  A_{1,1}(t,\cdot) \big\|}_{L^2}$ is bounded by a numerical constant times
\begin{align}\label{M123IBPs'}
\begin{split}
& 2^{2m} \rho(2^m)^{1/2} \cdot \sup_{s\approx 2^m} 
  {\Big\| \int_{\R^{ 3}} K^f(\xi,\s) \,  e^{is\jsig} \partial_s \wt{f}(s,\s) d\s \Big\|}_{L^2_\xi},
  \\
\begin{split}
K^f(\xi,\s) : =  \jsigma^{-1} \int_{\R^{ 6 }} e^{is\langle \rho \rangle} \langle \rho \rangle^{-1} \wt{f}(s,\rho) 
  \, \frac{\varphi_{k_2}(\s)\varphi_{k_3}(\rho)}{\Phi_{1}(\xi,\s,\rho)} K_3(\xi,\eta) 
  \\
  \times  \varphi_{C}(\vert \eta - \s \vert - \vert \rho 
\vert) \nu_1^S(-\eta+ \s , \rho) d\rho d\eta .
\end{split}
\end{split}
\end{align}
Note that we can insert a cut-off $\varphi_{\sim k_3}(-\eta + \s)$ in $K^f$ given the support property of $\nu_1^S$,
see \eqref{nu1S} and \eqref{Propnu+1.1}.
Once again, in view of Lemma \ref{decay-der}, it suffices to show 
\begin{align}\label{M123Kest'}
{\| K^f \|}_{Sch} \lesssim 2^{-m/2-10 \beta m}.
\end{align}

We split again $f = g + h$ and denote $K^g$ and $K^h$ the corresponding kernels from \eqref{M123IBPs'}.
For $K^h,$ if $k_3 \leqslant -2m/5,$ we can estimate, using \eqref{K_3bound},
\begin{align} \label{Kh1}
\Vert K^h \Vert_{Sch} \lesssim 
  2^{\beta m} \cdot {\| K_3 \|}_{Sch} {\| \varphi_{k_3} \wt{h} \|}_{L^1} \cdot 2^{-k_3} 
  \lesssim 2^{8\beta m} \cdot 2^{-k_3} \cdot 2^{5k_3/2} Z,
\end{align}
which is acceptable. 
Note the factor of $2^{-k_3}$ coming from the $1/|\eta-\s|$ factor in 
the leading order expression for $\nu_1^S(-\eta+ \s , \rho)$, see \eqref{nu1S}.

If $k_3 > -2m/5,$ we integrate by parts in the formula for $K^h$ using the identity,
for a differentiable function $F$,
\begin{align} \label{idIPP1}
\bigg(\frac{\eta - \s}{\vert \eta - \s \vert} \cdot \nabla_{\eta} 
  + \frac{\rho}{\vert \rho \vert} \cdot \nabla_{\rho} \bigg) 
  F\big( \vert \eta - \s \vert - \vert \rho \vert \big) 
  = 0,
\end{align}
and 
\begin{align} \label{idIPP2}
\bigg(\frac{\eta - \s}{\vert \eta - \s \vert} \cdot \nabla_{\eta} + \frac{\rho}{\vert \rho \vert} \cdot \nabla_{\rho} \bigg) e^{is\langle \rho \rangle} = is \frac{\vert \rho \vert}{\langle \rho \rangle} e^{is \langle \rho \rangle}.
\end{align}
We can then estimate (disregarding faster decaying contributions)
\begin{align} \label{Kh2}
\Vert K^h \Vert_{Sch} \lesssim 2^{\beta m}2^{-k_3} \cdot 2^{-m} 2^{-k_3} 
  \cdot \Vert K_3 \Vert_{Sch} \cdot 2^{3k_3/2} \Vert \nabla \wt{h} \Vert_{L^2}
  \lesssim 2^{-3m/5} Z.
\end{align}

\medskip
\noindent
\begin{remark}[About the term $A_{1,2}$]
Note that the case of the term $A_{1,2}$ in \eqref{M123pv} 
is simpler since on its support we have $\vert \eta \vert \gtrsim 2^{-5 \beta m}$ and,
therefore,
$\vert \s + \rho \vert \geqslant \vert \eta \vert - \big \vert \vert \s + \rho \vert - \vert \eta \vert \big \vert 
\gtrsim 2^{-5 \beta m}$, whenever we are close to the singularity,
say $| |\s+\rho|-|\eta| | \approx 2^{C}$ with $C < -10\beta m$.
Moreover, the factors of $2^{-k_3}$ in the bounds corresponding to \eqref{Kh1} and \eqref{Kh2} 
can be replaced by a harmless $2^{5 \beta m}$ factor.
\end{remark}

To estimate the kernel $K^g$ we first note that we have $k_3 \sim 0$,
and we can also insert a cut-off $ \varphi_{\sim 0} (\eta - \s).$
We then localize further introducing a cut-off 
$\varphi_{\ell}^{(\ell_0)}(\langle \rho \rangle - 2 \lambda)$,
with $\ell \geqslant \ell_0 := \lfloor -m + \delta m \rfloor$, $\delta>0$ small.
We denote the corresponding kernel by $K^g_{\ell}$ and then estimate similarly to \eqref{K^g}:
for $\ell = \ell_0$ we have
\begin{align*}
\Vert K^g_{\ell_0} \Vert_{Sch} & \lesssim 2^{\beta m-k_3} \Vert K_3 \Vert_{Sch} 
  \Vert 
  \varphi_{\leqslant \ell_0}(\langle \rho \rangle - 2\lambda) \wt{g} \Vert_{L^1_{\rho}} 
  \\
  & \lesssim 2^{11 \beta m} 
  \cdot 2^{\ell_0} \cdot 2^m \rho(2^m) 
  \lesssim 2^{-4m/5} 2^m \rho(2^m),
\end{align*}
which is sufficient for \eqref{M123Kest'};
if instead $\ell > \ell_0$, 
we integrate by parts using identities \eqref{idIPP1} and \eqref{idIPP2} and obtain the same estimate.

\medskip
\subsubsection{Estimate of $A^{\epss}_{R}$}\label{ssecAR}
Recall the definition in \eqref{M123a}.
For the term $A^{\epss}_{R}$ the measure $\mu^R$ is smooth (up to small losses);
see \eqref{mu1SR} and the pointwise estimates for $\nu_1^R$ in \eqref{nuRest}. 
Then, we can directly exploit integration by parts in both $\s$ and $\rho$ unless their size is too small.
Since the estimates will be independent of the signs $\epsilon_1, \epsilon_2$ we will drop these indices.

The arguments that follow are similar to those in Section \ref{secFR}, 
see in particular the integration by parts arguments 
in the proofs of Lemmas \ref{ggnuR}, \ref{ghnuR} and \ref{hhnuR} (and recall that $K_3$ satisfies \eqref{K_3bound}).

We split again $f= g + h $ and denote $A_{R}^{gg},  A_{R}^{gh}, A_{R}^{hg}$ and $A_{R}^{hh}$ 
the bilinear terms of the form \eqref{M123a} with arguments $(g,g)$, $(g,h)$, $(h,g)$ and $(h,h)$ respectively.
Also, it suffices to look at the case where $\mu^R(\eta,\s,\rho)$ is replaced by $\nu^R_1(-\eta+\s,\rho)$,
(see \eqref{mu1SR}) since the other cases are completely analogous or easier.
Therefore, with the notation in \eqref{M123a} we identify, for $G,H\in \{g,h\}$,
\begin{align}\label{M123a'}
\begin{split}
& A_{R}^{GH}(t,\xi) 
  = \int_0^t s \, B(s)\, \frac{\xi}{\jxi}\varphi_k(\xi) e^{-is (\jxi - \lambda)} \int_{\R^3}
  D_{k_2,k_3,k_4}^{GH}(s,\eta) \, K_{3;k,k_1,p}(\xi,\eta) d\eta \, \tau_m(s) ds,  
  \\ 
\begin{split}
D_{k_2,k_3,k_4}^{GH}(s,\eta) :=
  \int_{\R^6}  \jsigma^{-1} e^{is\jsig}
  \varphi_{k_2}(\s) \wt{G_{\epsilon_1}}(s,\s) \, 
  \varphi_{k_3}(\rho) \jrho^{-1} e^{is\langle \rho \rangle} \wt{H_{\epsilon_2}}(s,\rho)
  \\ \times \varphi_{k_4}(\eta-\s) \nu_1^R(-\eta+\s,\rho)  d\s d\rho.
\end{split}
\end{split}
\end{align}
Once again we have inserted additional localization in frequency 
to $|\s|\approx 2^{k_2}, |\rho|\approx 2^{k_3}$ and $|\eta-\s| \approx 2^{k_4}$ with $k_2,k_3,k_4 \in [-3m,\beta m]$, 
and omitted the summation up to a negligible loss of $O(m^3)$.
Note that we have 
\begin{align}\label{AR0}
\begin{split}
{\big\| A_{R}^{GH}(t) \big\|}_{L^2_\xi} & \lesssim  2^{2m} \rho^{1/2}(2^m)  \cdot 2^{k^-} \cdot {\| K_{3;k,k_1,p} \|}_{Sch}
	\sup_{s\approx 2^m} {\big\|   D_{k_2,k_3,k_4 }^{GH}(s) \big\|}_{L^2_\eta}
	\\
	& \lesssim 2^{3m/2 + 10\beta m} 
	\sup_{s\approx 2^m} {\big\|   D_{k_2,k_3,k_4}^{GH}(s) \big\|}_{L^2_\eta},
\end{split}
\end{align}
and that we can bound
\begin{align}
\label{AR0'}
& {\big\|  D_{k_2,k_3,k_4}^{GH}(s) \big\|}_{L^2_\eta} \lesssim {\| \varphi_{k_2} \wt{G} \|}_{L^2}
	\Big\| \int_{\R^3} \frac{\varphi_{k_3}(\rho)}{\jrho} e^{is\jrho} \wt{H}(s,\rho) 
	\, \varphi_{k_4}(\eta-\s) \nu_1^R(-\eta+\s,\rho) d\rho \Big\|_{Sch},
\\
\label{AR0''}
& {\big\|   D_{k_2,k_3,k_4}^{GH}(s) \big\|}_{L^2_\eta} \lesssim {\| \varphi_{k_3} \wt{H} \|}_{L^2}
	\Big\| \int_{\R^3} \frac{\varphi_{k_2}(\s)}{\jsig} e^{is\jsig} \wt{G}(s,\s)
	\, \varphi_{k_4}(\eta-\s) \varphi_{\sim k_3}(\rho) \nu_1^R(-\eta+\s,\rho) d\s \Big\|_{Sch}.
\end{align}

\medskip
{\it Case 1: Estimate of $A_{R}^{hh}$}.
In this case we can proceed similarly to the proof of Lemma \ref{hhnuR};
compare the expressions \eqref{hhnuRop} and \eqref{M123a'}.
We fix $\delta>0$ small and look at a few subcases.


\smallskip
{\it Subcase 1.1: $k_2 \vee k_3 \leqslant -m/2 + \delta m$}.
In this configuration we use \eqref{nuRest} to see that
\begin{align}\label{ARnu}
{\big\| \varphi_{k_4}(x) \varphi_{k_3}(y) \nu_1^R(x,y) \big\|}_{Sch} \lesssim 2^{-2(k_3\vee k_4)} 2^{10M_0}
  \cdot 2^{(3/2)(k_3+k_4)}
	\lesssim 2^{k_3\wedge k_4} 2^{10M_0},
\end{align}
and \eqref{AR0'} to obtain 
\begin{align*}
{\| D_{k_2,k_3,k_4}^{hh}(s) \|}_{L^2_\eta} & \lesssim 
	{\| \varphi_{k_2} \wt{h} \|}_{L^2} \cdot {\| \varphi_{k_3} \wt{h} \|}_{L^1} \cdot 2^{k_3\wedge k_2} \cdot 2^{10M_0}
	\lesssim 2^{k_2} Z \cdot 2^{5k_3/2} Z  \cdot 2^{k_3\wedge k_2} 2^{10M_0}
\end{align*}
which, plugged into \eqref{AR0} is sufficient for $k_2,k_3 \leqslant -m/2 + \delta m$ with $\delta$ small enough.



\smallskip
{\it Subcase 1.2:  $k_2 \vee k_3 \geqslant -m/2 + \delta m$ and $k_2 \wedge k_3 \leqslant -m/2 + \delta m$}.
By symmetry we may assume that $k_2 \leqslant -m/2 + \delta m \leqslant  k_3$.
We then integrate by parts in $\rho$ in the formula \eqref{M123a'}.
This leads to a main term where $\partial_\rho$ hits the profile $\wt{h}(\rho)$.
Estimating in the same way that led to \eqref{AR0'}, we obtain an inequality like the second one in \eqref{AR0'}
with $\wt{H}$ replaced by $s^{-1} \rho/|\rho|^2 \cdot \nabla \wt{h}(\rho)$; 
this shows, using also \eqref{ARnu},  that for all $s \approx 2^m$,
\begin{align*}
& {\| D_{k_2,k_3,k_4 }^{hh}(s) \|}_{L^2_\eta} 
\\
& \lesssim 
	2^{-m-k_3} {\big\| \varphi_{k_3} \nabla \wt{h} \big\|}_{L^2} 
	\Big\| \int_{\R^3} \varphi_{k_2}(\s)  \jsig^{-1} e^{is\jsig} \wt{h}(s,\s) 
	\varphi_{k_4}(\sigma-\eta) \varphi_{\sim k_3}(\rho) \nu_1^R(-\eta+\s,\rho) d\s \Big\|_{Sch}
\\
& \lesssim 2^{-m-k_3} Z \cdot {\| \varphi_{k_2} \wt{h} \|}_{L^1} \cdot {\| \varphi_{k_4}(x) \varphi_{k_3}(y) \nu_1^R(x,y) \|}_{Sch}
\\
& \lesssim 2^{-m-k_3} Z \cdot 2^{5k_2/2} Z \cdot 2^{k_2} 2^{10M_0}
\end{align*}
which is sufficient. 

\smallskip
{\it Subcase 1.3: $k_2 \wedge k_3 \geqslant -m/2 + \delta m$}.
Without loss of generality, we may assume $k_2 \leqslant k_3$.
In this case we can integrate by parts in both $\rho$ and $\s$ and, estimating similarly to the previous cases, we see that
\begin{align*}
& {\| D_{k_2,k_3,k_4}^{hh}(s) \|}_{L^2_\eta} 
\\
& \lesssim 
	2^{-m-k_3} {\big\| \varphi_{k_3} \nabla \wt{h} \big\|}_{L^2} 
	\Big\| \int_{\R^3} \varphi_{k_2}(\s) e^{is\jsig} \frac{1}{s} \frac{\s}{|\s|^2} \cdot \nabla _\s \wt{h}(s,\s) \varphi_{k_4}(\s-\eta) \varphi_{\sim k_3}(\rho)  \nu_1^R(-\eta+\s,\rho) d\s \Big\|_{Sch}
\\
& \lesssim 2^{-m-k_3} Z \cdot 2^{-m-k_2} {\| \varphi_{k_2} \nabla \wt{h} \|}_{L^1} \cdot {\| \varphi_{k_4}(x) \varphi_{k_3}(y) \nu_1^R(x,y) \|}_{Sch}
\\
& \lesssim 2^{-m-k_3} Z \cdot 2^{-m+k_2/2} Z \cdot 2^{k_3} 2^{10M_0} \lesssim Z^2 2^{-2m + \beta m +10M_0}.
\end{align*}
Inserting this bound in \eqref{AR0} gives the desired estimate for $A_R^{hh}$.

\medskip
{\it Case 2: Estimate of $A_{R}^{gg}$.}
The case of $(g,g)$ interactions is easier than the previous one since, on the support of \eqref{M123a'},
when $G=H=g$, we must have $|\s|,|\rho|\approx 1$.
In particular, $\nu_1^R$ is pointwise upper bounded by $2^{10M_0}$.

We have, using Schur's test, and H\"older with $|\eta| \lesssim 2^{\beta m}$,
\begin{align}\label{ARgg}
\begin{split}
& {\big\| A_R^{gg}(t) \big\|}_{L^2} = 
  \Big\| \int_0^t s B(s) \int_{\R^9} e^{-is \Phi_{1}} 
  \, \wt{g}(\s)  \wt{g}(\rho) \, \varphi_{\sim 0}(\s)\varphi_{\sim 0}(\rho) 
  \\ & \times \varphi_{k_4}(\eta -\s) \nu_1^R(-\eta+\s,\rho) d\s d\rho \, K_{3;k,k_1,p}(\xi,\eta) d\eta  
  \, \tau_m(s) ds \Big\|_{L^2}
  \\
  & \lesssim 2^{2m} \rho(2^m)^{1/2} \cdot {\| K_{3;k,k_1,p} \|}_{Sch} \cdot
  2^{(3/2)\beta m}
  \\
  & \times \sup_{s\approx 2^m} \sup_{|\eta| \lesssim 2^{\beta m}} \Big| \int_{\R^6}
  \, e^{is\jsig} \wt{g}(\s) \, e^{is\jrho} \wt{g}(\rho) \, \varphi_{\sim 0}(\s)\varphi_{\sim 0}(\rho) 
  \varphi_{k_4}(\eta -\s) \nu_1^R(-\eta+\s,\rho) d\s d\rho \Big|.
\end{split}
\end{align}
Then, in view of \eqref{K_3bound}, it suffices to show that the quantity on the last line above has 
an upper bound by, say, $2^{-5m/3}$.
To see that this is the case, we insert cutoffs in the size of $\jsig - 2\lambda$ and $\jrho-2\lambda$, 
and aim to bound 
\begin{align}\label{M123J1}
\begin{split}
  \psi_{\ell_2,\ell_3}(\eta) & := 
  \int_{\R^6} 
  e^{is\jsig} (\varphi_{\sim 0}\wt{g})(\s)  e^{is\jrho} (\varphi_{\sim 0}\wt{g})(\rho)
  \, \varphi_{\ell_2}^{(\ell_0)}(\jsig - 2\lambda) \, 
  \\ & \times  \varphi_{\ell_3}^{(\ell_0)}(\jrho - 2\lambda)
\varphi_{k_4}(\eta -\s)  \nu_1^R(-\eta+\s, \rho) d\s d\rho,
  \qquad \ell_0 := \lfloor -m + \delta m \rfloor,
\end{split}
\end{align}
with $\delta>0$ sufficiently small.

In the case $\ell_2=\ell_3=\ell_0$ we use \eqref{inftyfreqg}, and integrate directly to get 
\begin{align*}
{\| \psi_{\ell_2,\ell_3} \|}_{L^\infty_\eta} 
  & \lesssim 2^{\ell_0} {\|\wt{g}\|}_{L^\infty} \cdot 2^{\ell_0} {\| \wt{g}\|}_{L^\infty} \cdot 2^{10M_0}
  \\ 
  & \lesssim 2^{-2m + 2\delta m + 10M_0} \big(2^m \rho(2^m)\big)^2,
\end{align*}
which suffices for $\delta$ small enough.

If $\ell_2,\ell_3 > \ell_0$ we can resort to integration by parts in both $\rho$ and $\s$
in the formula \eqref{M123J1}, and estimate, up to faster decaying terms,
\begin{align*}
\begin{split}
 \big| \psi_{\ell_2,\ell_3}(\eta) \big| & \lesssim 
 2^{-m} {\| \varphi_{\ell_2} \nabla_\s \wt{g} \|}_{L^1}  
 \cdot 2^{-m} {\| \varphi_{\ell_3} \nabla_\rho \wt{g} \|}_{L^1} \cdot 2^{10M_0} 
 \\ 
 & \lesssim 2^{-2m + 10M_0} \big(m2^m \rho(2^m)\big)^2,
\end{split}
\end{align*}
having used \eqref{dxiwtgest}. 

The cases $\ell_3 = \ell_0 < \ell_2$ and $\ell_2 = \ell_0 < \ell_3$ can be treated similarly using direct
integration in the variable corresponding to the index $\ell_0$,
and integration by parts in the other.

\medskip
{\it Case 3: Estimate of $A_{R}^{gh}$ and  $A_{R}^{hg}$.}
These terms are all treated using strategies already employed in previous two cases, 
and we only give the argument for the slightly worse case of $A_{R}^{gh}$. 
We start by looking at the formula \eqref{M123a'} with $G=g$ and $H=h$, and introduce
a cut-off $\varphi_{\ell_2}^{(\ell_0)}(\jsig - 2 \lambda)$, with the usual choice $\ell_0=\lfloor -m+\delta m \rfloor$.

\smallskip
{\it Subcase 1.1: $k_3 \leqslant -m/2 + \delta m$}.
In this case we use first \eqref{AR0} and then \eqref{AR0''} to obtain 
\begin{align}
\nonumber
{\big\| A_{R}^{gh}(t) \big\|}_{L^2_\xi} & \lesssim 2^{3m/2 + 10\beta m}
  \sup_{s\approx 2^m} {\big\|   D_{k_2,k_3, k_4}^{gh}(s) \big\|}_{L^2_\eta}
  \\
  \label{AR0gh}
  & \lesssim 2^{3m/2 + 10\beta m}
  \cdot {\| \varphi_{k_3} \wt{h} \|}_{L^2}
  \\ 
  & \times {\Big\| \int_{\R^3} \frac{\varphi_{k_2}(\s)}{\jsig} e^{is\jsig} \wt{g}(\s)
  \, \varphi_{\ell_2}^{(\ell_0)}(\jsig - 2 \lambda) \varphi_{k_4}(\eta-\s) \varphi_{k_3}(\rho) 
  \nu_1^R(-\eta+\s,\rho) d\s \Big\|}_{Sch}.
\nonumber
\end{align}
For the Schur norm in \eqref{AR0gh}, resorting to integration by parts if $\ell_2>\ell_0$,
and using \eqref{dxiwtgest} and \eqref{inftyfreqg}, we have a bound of
\begin{align}\label{AR0gh1}
C 2^{-m+\delta m} m\rho(2^m)2^m \cdot {\big\| \varphi_{k_4}(x)\varphi_{k_3}( y)\nu_1^R(x,y)  \big\|}_{Sch}
  \lesssim2^{-m+2\delta m} \cdot 2^{k_3}.
\end{align}
Using this and ${\| \varphi_{k_3} \wt{h} \|}_{L^2} \lesssim 2^{k_3}Z$ in \eqref{AR0gh1} 
gives a sufficient bound, since $k_3 \leqslant -m/2 +\delta m$.

\smallskip
{\it Subcase 1.2: $k_3 \geqslant -m/2 + \delta m$}.
In this case we can proceed as above, but first integrate by parts in $\rho$
(the frequency variable of $\wt{h}$) in the formula for $D^{gh}$ in \eqref{M123a'}.
Then we can bound using \eqref{AR0} and \eqref{AR0''} as done in \eqref{AR0gh} above,
and obtain the same estimate but with
$2^{-m-k_3} {\| \varphi_{k_3} \nabla \wt{h} \|}_{L^2}$ instead of ${\| \varphi_{k_3} \wt{h} \|}_{L^2}$;
since $2^{-m-k_3} {\| \varphi_{k_3} \nabla \wt{h} \|}_{L^2} \lesssim 2^{-m/2- \delta m} Z$ in the present case,
we can then conclude as above.



\medskip
\subsubsection{Estimate of $A^{\epss}_{Re}$} 
The estimates for $A^{\epss}_{Re}$ are substantially easier than those for $A^{\epss}_{R}$.
Looking at the formula \eqref{M123a}, and comparing with \eqref{M123a'}, 
we just need to replace $\mu^{Re}(\eta,\s,\rho)$ for $\nu_1^R(-\eta+\sigma,\rho)$ in \eqref{M123a'}.
Then we recall that, in view of the assumptions of Proposition \ref{propM1'}, 
on the support of the integral we must have  $\vert \eta \vert \gtrsim 2^{-5\beta m}$;
therefore, $\max \{\vert \eta \vert ,\vert \s \vert, \vert \rho \vert \} \gtrsim 2^{-5\beta m},$ 
and in view of \eqref{nuReest}, we have a negligible loss coming from the measure.
Therefore we can follow the same arguments in \S\ref{ssecAR} above, 
based on integration by parts in frequency. 
We skip the details.
\end{proof}

\medskip
\subsection{Estimate of $II_{m,k,k_1}(G), G \in \lbrace g,h \rbrace$}\label{ssecII}
We prove the following:

\begin{lemma}
Under the a priori assumptions of Proposition \ref{propboot},
and with the definition \eqref{M100}, we have, for all $m=0,1,\dots$,
\begin{align}
\sum_{k,k_1\in\Z} {\big\| II_{m,k,k_1}(G) \big\|}_{L^2} 
\lesssim 2^m \rho(2^m)^{1-\beta + \delta} \e^{\beta}, \qquad G \in \lbrace g,h \rbrace
\end{align}
for some $\delta \in (0,\beta/2)$.
\end{lemma}

\begin{proof}
To show estimates on $II_{m,k,k_1}$ we first consider a slightly more general term of the form
\begin{align*}
III_{m,k,k_1} :=\varphi_{k}(\xi)  \int_0^t B(s) \int_{\R^3} e^{-is\Phi(\xi,\eta)} 
  \varphi_{k_1}(\eta) \, \jeta^{-1}\wt{G}(s,\eta) \nu'(\xi,\eta) d\eta \, \tau_m(s) ds, 
    \qquad  G \in \lbrace g,h \rbrace,
\end{align*}
for some $\nu'$, and establish some general bounds for it.

First, we write that
\begin{align}
\notag
& \Vert III_{m,k,k_1} (t,\xi) \Vert_{L^2} 
\\
\notag
& \lesssim 2^m \rho(2^m)^{1/2} \cdot 2^{3k/2} \cdot \sup_{s \approx 2^m} \sup_{\vert \xi \vert \approx 2^k} \bigg \vert \int_{\R^3} e^{-is\Phi(\xi,\eta)} 
  \varphi_{k_1}(\eta) \, \jeta^{-1}\wt{G}(s,\eta) \nu'(\xi,\eta) d\eta  \bigg \vert 
  \\
 \label{lsk1}  
 & \lesssim 2^m \rho(2^m)^{1/2} \cdot 2^{3k/2} \Vert \varphi_{k_1}(\eta) \wt{G}(\eta) \Vert_{L^1_{\eta}} 
 \sup_{\vert \xi \vert \approx 2^k, \vert \eta \vert \approx 2^{k_1}} \vert \nu'(\xi,\eta) \vert.
\end{align}

Second, we can integrate by parts in $\eta$ and obtain
\begin{align*}
\Vert III_{m,k,k_1} \Vert_{L^2} & \lesssim \Vert III_{m,k,k_1}^{(1)} \Vert_{L^2} + \Vert III_{m,k,k_1}^{(2)} \Vert_{L^2},
\\
III_{m,k,k_1}^{(1)} &:= \varphi_{k}(\xi) \int_0 ^t \tau_m(s) s^{-1} B(s) 
  \int_{\mathbb{R}^3} \big(\frac{\eta}{\vert \eta \vert^2} \cdot \nabla \big) 
  \wt{G}(s,\eta) \nu'(\xi,\eta) \varphi_{k_1}(\eta) d\eta ds,
\\
III_{m,k,k_1}^{(2)} &:= \varphi_{k}(\xi) \int_0 ^t \tau_m(s) s^{-1} B(s) \int_{\mathbb{R}^3} 
  \wt{G}(s,\eta) \textrm{div} 
  \bigg( \frac{\eta}{\vert \eta \vert^2} \nu'(\xi,\eta) \varphi_{k_1}(\eta) \bigg) d\eta ds .
\end{align*}
Therefore we can write as above that
\begin{align}\label{medk1}
{\Vert III_{m,k,k_1}^{(1)} \Vert}_{L^2} \lesssim \rho(2^m)^{1/2} 
\bigg \Vert \varphi_{k_1}(\eta)\frac{ \nabla \wt{G}(s,\eta)}{\vert \eta \vert} 
\bigg \Vert_{L^1_{\eta}} \sup_{\vert \xi \vert \approx 2^k, \vert \eta \vert \approx 2^{k_1}} \vert \nu'(\xi,\eta) \vert.
\end{align}

In the same vein, we can write that 
\begin{align}  \label{medk1bis}
\begin{split} 
 \Vert III^{(2)}_{m,k,k_1}(t,\xi) \Vert_{L^2} & \lesssim \rho(2^m)^{1/2} 2^{3k/2} \max \bigg( \sup_{\vert \xi \vert \approx 2^k, 
  \vert \eta \vert \approx 2^{k_1}} \vert \nabla_{\eta} \nu'(\xi,\eta) \vert, \, 
  2^{-k_1} \sup_{\vert \xi \vert \approx 2^k, \vert \eta \vert \approx 2^{k_1}} \vert \nu'(\xi,\eta ) \vert \bigg) \\
& \times  \bigg \Vert \varphi_{k_1}(\eta) \frac{\wt{G}(s,\eta)}{\vert \eta \vert} \bigg \Vert_{L^1_{\eta}}  .
\end{split}
\end{align}

We need one more ingredient to deal with large outputs $|\xi|$.
Notice that
\begin{align*}
- \Delta_x \big( \partial_{\xi} \psi(x,\xi) \big) + V \partial_{\xi} \psi(x,\xi) 
  = 2 \xi \psi(x,\xi) + \vert \xi \vert^2 \partial_{\xi} \psi(x,\xi),
\end{align*}
so that, integrating by parts in $x$, we obtain
\begin{align}\label{meas-trans}
\begin{split}
& \vert \xi \vert^2 \partial_{\xi} \nu(\xi,\eta) 
\\
& = - \int_{\R^3} \Delta \big( \partial_{\xi} \overline{\psi(x,\xi)}\big) \psi(x,\eta) \phi(x) dx 
  + \int_{\R^3} \partial_\xi \overline{ \psi(x,\xi) } \psi(x,\eta) V(x) \phi(x) dx - 2 \xi \nu(\xi,\eta) 
  \\
& = -\int_{\R^3} \partial_{\xi} \overline{\psi(x,\xi)} \Delta \big( \psi(x,\eta) \phi(x) \big) dx 
  + \int_{\R^3} \partial_\xi \overline{\psi(x,\xi)}\psi(x,\eta) V(x) \phi(x) dx - 2\xi \nu(\xi,\eta) 
  \\
& := \nu_1 (\xi,\eta) + \nu_2 (\xi,\eta) + \nu_3 (\xi,\eta).
\end{split}
\end{align}
We are now ready to estimate $II_{m,k,k_1}(G), G \in \lbrace g,h \rbrace$.  
We distinguish two main cases, depending on whether the input is $g$ or $h.$ 

\medskip
\textit{Case $G=h$}.
We distinguish two main cases depending on the relation between $k$ and $k_1$.

\medskip
{\it Case 1: $k<k_1$}.
We look at two subcases.

\smallskip
\noindent
{\it Subcase 1.1: $k_1 > 10m/N$ or $k_1 \leqslant -5 m$}.
Using \eqref{lsk1} with $\nu' = \partial_{\xi} \nu$ 
together with \eqref{Mnuest} as well as the bootstrap assumption on the $H^N$ norm \eqref{boot0}, we find 
\begin{align*}
\Vert II_{m,k,k_1} (t,\xi) \Vert_{L^2} & \lesssim 2^m \rho(2^m)^{1/2} 
  \cdot 2^{3k/2} 2^{3 k_1 /2} 
  \Vert \varphi_{k_1}(\eta) \wt{h}(s,\eta) \Vert_{L^2} 
  \\
& \lesssim 2^m\rho(2^m)^{1/2} \cdot 2^{3k/2} 2^{3k_1/2} \cdot \big(2^{-Nk_1^{+}/2} \wedge 2^{k_1} \big) \cdot Z,
\end{align*}
which allows us to conclude in this case.

\smallskip
\noindent
{\it Case 1.2: $-5m< k_1 \leqslant 10m/N$}.
Integrating by parts in the formula \eqref{M100} for $II_{m\color{red},k \color{black},k_1}$,
and using \eqref{medk1} and \eqref{medk1bis} with $\nu' = \partial_{\xi} \nu$ together with \eqref{Mnuest}, we find 
\begin{align*}
\Vert II_{m,k,k_1} 
  (t,\xi) \Vert_{L^2} & \lesssim \rho(2^m)^{1/2} \cdot 2^{3k/2} 
  \cdot 
2^{-k_1} \cdot \big( 2^{-k_1} {\| \varphi_{\sim k_1} (\eta) \|}_{L^{6/5}} {\| \wt{h} (\eta) \|}_{L^6}
  + {\| \varphi_{k_1} \nabla \wt{h} (\eta) \|}_{L^1} \big)
  \\
  & \lesssim \rho(2^m)^{1/2} 2^{3k/2} 2^{k_1/2}  \cdot Z , 
\end{align*}
which is enough to conclude in this case.

\medskip
{\it Case 2: $k \geqslant k_1 $}. In this case we also distinguish two subcases between medium and other frequencies.

\smallskip
\noindent
{\it Case 2.1: $k_1 > 10m/N$ or $k_1 \leqslant -5 m$}.
In this case we consider $|\xi|^2 II_{m,k,k_1}$; using \eqref{meas-trans} we have three terms to estimate, 
corresponding to $\nu_i, \, i=1,2,3$, which we denote by $\vert \xi \vert^2 II_{m,k,k_1}^{(i)}$.
We write, using \eqref{psiLinfty}, \eqref{Mnuest} and \eqref{lsk1}
\begin{align*}
\big \Vert  \vert \xi \vert^2 II_{m,k,k_1}^{(i)} (t,\xi) \big \Vert_{L^2} \lesssim 2^{m} \rho(2^m)^{1/2} 
	2^{3k/2} 2^{3k_1/2} \big(2^{-(N/2)k_1^{+}} \wedge 2^{k_1} \big)  2^{2k_1^{+}} \cdot Z , \qquad i=1,2.
\end{align*}
For $i=3,$ we can bound using Plancherel similarly to \eqref{reuse}
\begin{align*}
\big \Vert \vert \xi \vert^2 II_{m,k,k_1}^{(3)} (t,\xi) \big \Vert_{L^2} \lesssim 2^{m} \rho(2^m)^{1/2} \cdot 2^k \cdot 2^{-(N/2)k^{+}} 
	\cdot \big( 2^{-(N/2) k_1^{+}} \wedge 2^{k_1} \big) \cdot Z .
\end{align*}
Putting the above two estimates together with an $L^2$-bound from \eqref{lsk1}, we conclude that 
\begin{align*}
\Vert II_{m,k,k_1}^{(i)} \Vert_{L^2} \lesssim 2^{-2k^{+}} \cdot 2^{m} \rho(2^m)^{1/2} 
	\cdot  \big(2^{-(N/2)k_1^{+}} \wedge 2^{k_1} \big) \big(  2^{3k/2 + 3k_1/2} 2^{2k_1^{+}} + 2^k 2^{-(N/2)k^{+}}  \big) \cdot Z.
\end{align*}
Given our assumptions on $k,k_1$ this is sufficient to conclude.

\smallskip
\noindent
{\it Case 2.2: $-5m < k_1 \leqslant 10m/N$}.
We make use of \eqref{medk1}, \eqref{medk1bis} with $\nu' = \nu_i, i=1,2 $ to write that
\begin{align*}
\Vert \vert \xi \vert^2 II_{m,k,k_1}^{(i)} \Vert_{L^2} \lesssim \rho(2^m)^{1/2} \cdot 2^{3k/2} 2^{k_1/2} 
	2^{2k_1^{+}} 
	\cdot Z . 
\end{align*}
For $i=3,$ we use the Strichartz estimate (Lemma \ref{Strichartz})
with parameters $(\wt{q},\wt{r},\gamma) = (2,\infty,0)$ and 
\eqref{dispbootf} to obtain
\begin{align*}
\Vert \vert \xi \vert^2 II_{m,k,k_1}^{(3)} \Vert_{L^2} \lesssim  2^k \cdot 2^{m/2} \rho(2^m)^{1/2} 
	\cdot 2^{-k_1^{+}} \Vert P_{k_1} e^{isL} h \Vert_{L^6_x} 
	\lesssim 2^k \cdot 2^{-m + 10 \delta_N m} \cdot Z. 
\end{align*}
The desired bound follows in all cases.

\medskip
\textit{Case $G=g$.}
Note that in this case $k_1 \sim 0,$ therefore we will disregard the sum over that index. 
We insert cutoffs $\varphi_{\ell}^{(\ell_0)}(\jxi - 2\lambda), \ell_0:=\lfloor -m+\delta m \rfloor,$ 
where $\delta>0$ is chosen suitably small.

\medskip 
\noindent
{\it Case 1: $k < m/10$}.  
If $\ell > \ell_0$ we integrate by parts in $\eta$ in the formula \eqref{M100}.
Using \eqref{medk1}, \eqref{medk1bis}, as well as \eqref{dxiwtgest} and \eqref{inftyfreqg}, we find 
\begin{align*}
\Vert II_{m,k}(t,\xi) 
  \Vert_{L^2_{\xi}} & \lesssim \rho(2^m)^{1/2} \cdot 2^{3k/2} 
  \big( \Vert \varphi_{\ell}(\jxi-2\lambda) \nabla_\eta \wt{g}(s,\eta) \Vert_{L^1_{\eta}} 
  + {\| \wt{g} \|}_{L^1} \big)
  \\
  & \lesssim \rho(2^m)^{1/2} \cdot 2^{3m/20} \cdot 2^m \rho(2^m) m^2,
\end{align*}
which suffices.
A similar estimate holds if $\ell=\ell_0,$ integrating directly using \eqref{inftyfreqg}.

\medskip
\noindent
{\it Case 2: $k \geqslant m/10$}.  
In this case we use the identity \eqref{meas-trans} and the same notation in Case 2.1 above. 
As above we integrate by parts in frequency if $\ell>\ell_0$ and integrate directly otherwise. 
Using \eqref{medk1}, \eqref{medk1bis} as well as \eqref{dxiwtgest} and \eqref{inftyfreqg}, we find, for $i=1,2$,
\begin{align*}
\big \Vert \vert \xi \vert^2 II_{m,k}^{(i)} \big \Vert_{L^2_{\xi}} 
& \lesssim \rho(2^m)^{1/2} 
  \cdot 2^{3k/2} \big( \Vert \varphi_{\ell}(\jxi-2\lambda) \nabla_{\eta} \wt{g}(s,\eta) \Vert_{L^1_{\eta}}
  + {\| \wt{g} \|}_{L^\infty} \big)
  \\
  & \lesssim \rho(2^m)^{1/2} \cdot 2^{3k/2} (2^m \rho(2^m) m^2 ).
\end{align*}
For the last case,  we can conclude as we did for the profile $h$ above using Strichartz 
estimates as well as \eqref{dispbootg}:
\begin{align*}
\big \Vert \vert \xi \vert^2  II_{m,k}^{(3)}(t,\xi) \big \Vert_{L^2_{\xi}} \lesssim 2^k 
  \cdot 2^{m/2} \rho(2^m)^{1/2} \Vert e^{itL} g \Vert_{L^6_x} \lesssim 2^k \cdot \rho(2^m) m.
\end{align*}
Upon dividing by $2^{2k}$ and summing over $k \geqslant m/10$, the bound above is more than sufficient to conclude.
\end{proof}

\medskip
\subsection{The remainder terms}\label{ssecRM}
Here we explain how to estimate the remainder terms of mixed type $RM$ in \eqref{RM};
since these are easier to treat compared to $M_1$ and $M_2$ above, we will leave some details to the reader.

\subsubsection*{The term \eqref{RM1}}
Using \eqref{defB}, the term \eqref{RM1} can be written as
\begin{align*}
\nonumber
RM_1 & = RM_{1,1} + RM_{1,2} + RM_{1,3},
\\
RM_{1,1} & := \int_0^t \overline{B(s)} \int_{\R^3} e^{-is(\jxi - \jeta + \lambda)} 
  \jeta^{-1}\wt{h}(s,\eta) \nu(\xi,\eta) d\eta \,  ds,
\\
RM_{1,2} & := \int_0^t \overline{B(s)} \int_{\R^3} e^{-is(\jxi - \jeta + \lambda)} 
  \jeta^{-1}\wt{g}(s,\eta) \nu(\xi,\eta) d\eta \, ds,
\\
RM_{1,3} & := \int_0^t \big( \overline{A(s)} - \overline{B(s)} \big) \int_{\R^3} e^{-is(\jxi - \jeta + \lambda)} 
  \jeta^{-1}\wt{f}(s,\eta) \nu(\xi,\eta) d\eta \, ds.
\end{align*}

The term $RM_{1,1}$ is similar to the term $M_1$ treated in subsections above, 
where the only difference is the opposite sign of the oscillation 
of the amplitude of the internal mode $e^{-is\lambda}$.
The arguments above are clearly insensitive to this change,
and we can then apply the same estimates verbatim.
The starting point is the analogue of \eqref{M1I}
with the phase now given by $\Phi(\xi,\eta) = \jxi  + \lambda -\jeta$.
For example, the contributions corresponding to \eqref{M123a} 
now have phases $\jxi + \lambda - \epsilon_1\jsig - \epsilon_2\jrho$.
This does not impact the estimates of the term \eqref{M123del} since \eqref{M123delph}
still holds by changing the sign of $\lambda$
(we are restricting to $|\jxi + \lambda - \jeta| \approx 2^p$ in this case).
The same goes for \eqref{M123pv} and \eqref{M123pvph}. 

The term $RM_{1,2}$ enjoys the same estimates satisfied by $M_2$,
which are based on Strichartz estimates and are not affected by the different sign of the oscillation
of the discrete component.

Using  $A-B = O(A^2)$, we see that $RM_{1,3}$ has the form
\begin{align}\label{RM13}
\int_0^t O(A^2(s)) \int_{\R^3} e^{-is(\jxi - \jeta - \lambda)} \jeta^{-1} \wt{f}(s,\eta) \nu(\xi,\eta) d\eta ds.
\end{align}
After taking a $\partial_\xi$ derivative at the cost of adding an $s$ factor, 
we can directly apply the usual Strichartz estimate from Lemma \ref{Strichartz},
as in the estimate of $M_2$,
and use the $L^6_x$ a priori decay estimate for $e^{itL}f$, see Lemma \ref{dispersive-bootstrap}, 
and $O(A^2(s)) = \js^{-1}$, to get 
\begin{align*}
\big\| \partial_\xi RM_{1,3} \big\|_{L^2} \lesssim \sum_{m} \| sA^2(s) \|_{L^2_{s\approx 2^m}} 
  \sup_{s\approx 2^m} \| \phi \, e^{isL} f \|_{L^1} \lesssim 
  \sum_m 2^{m/2} \cdot 2^{(-1+\beta + \delta_N)m} \varepsilon^{\beta} \lesssim 1.
\end{align*}

\subsubsection*{The term \eqref{RM2}}
Writing \eqref{RM2} in terms of $A$ gives expressions of the form
\begin{align*}
\int_0^t A_{\epsilon_1}(s) \int_{\R^3} e^{-is (\jxi - \epsilon_1\lambda + \jeta)} 
  \jeta^{-1} \overline{\wt{f}}(s,\eta) \nu(\xi,\eta) d\eta ds, 
\end{align*}
and we see that the oscillating phase is lower bounded by $\max(\jeta,\jxi)$.
Then this term is clearly easier to bound than $M_1$ and $M_2$ above.

\subsubsection*{The term \eqref{RM3}}
This is almost identical to the term $RM_{1,3}$ in \eqref{RM13}, and can be handled \color{red} as \color{black} above.


\medskip
\section{Energy estimates} \label{energyest} 

In this section we first prove energy estimates for $h$,
along with the $H^N$ scattering statement \eqref{mtscattHN} in Theorem \ref{maintheo}.
Then we prove Lemma \ref{lemmaH} about control of the weighted norm for high frequencies.

\subsection{Sobolev norms and scattering}

We prove the following:

\begin{proposition}\label{propHN}
Under the a priori assumption of Proposition \ref{propboot}, there exists $\delta>0$
such that, for any $1\leqslant t_1<t_2 \leqslant T$ we have
\begin{align}\label{hscatt}
{\| h(t_1) - h(t_2) \|}_{H^N} \lesssim \e^{1+\delta} \, t_1^{-\delta} \color{red}. \color{black}
\end{align}
\end{proposition}

The above proposition implies the conclusion \eqref{bootconc} for the $H^N$ norm, 
as well as scattering in $H^N$ for $h(t)$.
Recall that $h$ satisfies \eqref{decompositionh}.
In the next two subsections, we will then estimate the $H^N$ norm of the terms appearing in \eqref{decompositionh}.

\smallskip
\subsubsection{Source terms}
We start by inserting the usual $\tau_m$ cutoffs (see \eqref{timedecomp}) in the time integral defining
the source terms \eqref{S},
and denote $S_{i}^m$, $i=1,2,3$ the corresponding terms. 
The following lemma gives an estimate consistent with \eqref{hscatt} for these source terms.

\begin{lemma} \label{en-source}
Under the assumptions of Proposition \ref{propboot}, 
there exists $\delta>0$ such that, for all $m=0,1,\dots$, we have, for $i=1,2,3,$ 
\begin{align}
& \Vert S_i^m \Vert_{H^N} \lesssim 2^{-\delta m} \e^{1 + \delta}.
\end{align}
\end{lemma} 

\begin{proof}
We first look at 
\begin{align*}
S_1^m(t,\xi) := 
  \int_0^t e^{-is(\jxi-2\lambda)} (A^2(s)-B^2 (s)) \tau_m(s)ds \, \mf{\theta}(\xi),
\end{align*}
see \eqref{Scubic},
and split it as
\begin{align*}
S_1^m = S_-^m + S_+^m, \qquad \wt{S_-^m}(t,\xi) := \varphi_{\leqslant -m/10}(\jxi-2\lambda) \wt{S_1^m}(t,\xi).
\end{align*}
We then estimate directly by H\"older and \eqref{defB}
\begin{align*}
{\| S_-^m(t) \|}_{H^N_x} \lesssim 2^{-m/20} {\| \wt{S_1^m} \|}_{L^\infty} \lesssim 
  2^{-m/20} \cdot 2^m \rho^{3/2}(2^m),
\end{align*}
which suffices since $\rho \lesssim \e^2$.
For $S_+^m$ we instead integrate by parts in $s$ to obtain
\begin{align*}
{\big\| S_+^m(t) \big\|}_{H^N_x} & \lesssim 
  {\Big\| \frac{\varphi_{>-m/10}(\jxi-2\lambda)}{\jxi-2\lambda}
  \Big| \int_0^t e^{-is(\jxi-2\lambda)} \frac{d}{ds}\Big[ (A^2(s)-B^2 (s))
  \tau_m(s) \Big] ds \, \jxi^N \mf{\theta}(\xi) \Big| \Big\|}_{L^2_\xi}
  \\
  & \lesssim 2^{m/20} 2^m \rho^2(2^m),
\end{align*}
which is also sufficient.
For $S_2^m$, see \eqref{S1-chi}, we can also integrate by parts at no cost since $|\jxi-2\lambda|\gtrsim 1$,
and then use \eqref{defB} and \eqref{renorm-A} to obtain an even stronger bound:
\begin{align*}
{\| S_2^m(t) \|}_{H^N_x} \lesssim 2^m \rho^2(2^m).
\end{align*}

The same arguments apply to the two terms in \eqref{Snr}. Looking at the first one (the other can be treated identically) we write 
\begin{align*}
S_3^m 
  & = \int_0^t e^{-is\jxi} \big( |A|^2 (s) - |B|^2(s) \big) \, \tau_m(s) ds \mf{\theta}(\xi) 
  + \int_0^t e^{-is\jxi} |B|^2(s) \, \tau_m(s) ds \mf{\theta}(\xi) 
\end{align*}
and then proceed as for $S_1^m$ to deal with the first of the two terms above (this is in fact even easier 
since the phase does not vanish), and just integrate by parts, as we did for $S_2^m$, to deal with the second. 
\end{proof}

\smallskip
\subsubsection{Mixed terms}\label{EEmixed}
As in the previous paragraph, we localize the mixed terms \eqref{M}-\eqref{RM} in time by inserting time cutoffs. 
We denote the corresponding terms $M_i^m, i=1,2$ and $RM_i^m, i=1,2,3$ respectively. 
We prove a similar result as in the previous subsection:

\begin{lemma} \label{en-mix}
We have for $m=0,1, \dots$
\begin{align*}
& \Vert (M_1^m - M_2^m) \Vert_{H^N} \lesssim 2^{-\delta m} \e^{1+\delta}, 
\\
& \Vert RM_i^m \Vert_{H^N} \lesssim 2^{-\delta m} \e^{1+\delta}, 
  \quad i=1,2,3,
\end{align*}
for some $\delta > 0 .$ 
\end{lemma}

\begin{proof}
We first show how to bound the $H^N$ norm of the (time-localized) mixed term 
\begin{align}
(M_1^m - M_2^m)(t) 
  = \int_0^t e^{-isL} e^{-is\lambda} B(s) \big( \phi(x) \, L^{-1} \Im w(s) \big) \,\tau_m(s) ds, \quad m=0,1,2,\dots.
\end{align}
Recall that $f = h-g$ and $e^{itL}f = w = (\partial_t + iL)v$;
note also that with the same argument below we could estimate $M_1^m$ and $M_2^m$ separately as well.
The idea here is to use Strichartz estimates 
and the gain of one derivative on $w$
- due to the  `strongly semilinear' nature of \eqref{sysav} - 
followed by interpolation through the Gagliardo-Nirenberg inequality for $1 \leqslant N_1 \leqslant N$
\begin{align} \label{Gagliardo-L2}
\Vert w \Vert_{H^{N_1}} & \lesssim \Vert w \Vert_{H^{N}}^{\frac{N_1-1}{N-1}} 
  \Vert w \Vert_{L^{6}}^{\frac{N-N_1}{N-1}} 
  \lesssim \varepsilon^{\frac{N_1-1}{N-1}} \cdot 
  \rho(2^m)^{(1-\beta)\frac{N-N_1}{N-1}} 2^{ \frac{5}{3}m \delta_N \frac{N-N_1}{N-1}} 
  \varepsilon^{\beta \frac{N-N_1}{N-1}},
\end{align}
where we used \eqref{dispbootf} for the last part of the inequality. 

Relying on the boundedness of wave operators, followed by Lemma \ref{Strichartz}, and using \eqref{Gagliardo-L2}, 
we find
\begin{align*}
{\| (M_1^m-M_2^m)(t) \|}_{H^N} & \lesssim \sum_{n_1+n_2 \leqslant N } 
  {\Big\| \int_0^t B(s) e^{-isL} \big( D^{n_2} \phi(x) \cdot L^{-1} D^{n_1} \Im w(s) \big) \, \tau_m(s) ds \Big\|}_{L^2}
\\
& \lesssim \sum_{n_1+n_2 \leqslant N } 
  {\Big\| B(s) \tau_m(s) \, \langle D \rangle^{n_2}\phi \, 
  \cdot \langle D \rangle^{n_1-1} \Im w(s) \Big\|}_{L^2_tL^1_x}
\\
& \lesssim {\| B(s) \|}_{L^2_{s\approx 2^m}} 
  \cdot \sup_{s\approx 2^m } 
  {\| w(s) \|}_{H^{N-1}_x}
\\
& \lesssim {\| B(s) \|}_{L^2_{s\approx 2^m}} 
  \cdot \e^{\frac{N-2}{N-1}}  \cdot \Big( \rho(2^m)^{1-\beta} 
  2^{\frac{5}{3} \delta_N m} \e^{\beta} \Big)^{\frac{1}{N-1} }
  \\
& \lesssim \big(2^m \rho(2^m) \big)^{1/2} \cdot \e^{\frac{N-2}{N-1}} \cdot \e^{\frac{1-\beta}{N-1}}
  2^{-\frac{1-\beta}{2(N-1)} m } 
  2^{\frac{5\delta_N }{3(N-1)}m } \e^{\frac{\beta}{N-1}}
\\ 
  & 
 \lesssim \e \cdot  2^{-\frac{1-\beta}{2(N-1)} m} 2^{\frac{5\delta_N }{3(N-1)}m },
\end{align*}
where we used $\rho(t) \leqslant \e t^{-1/2}$ for the line before last.
Since $(1-\beta)/2 > 5\delta_N/3$, this concludes the proof of the first estimate.

The remainder terms can be treated following the same strategy decomposing
$A = A-B + B = O(\rho^{1/2})$ (recall \eqref{defB}), therefore we can skip the details.
\end{proof}

\smallskip
\subsubsection{Field interactions}
Finally, we look at the nonlinear interactions that are quadratic in the field $w$, 
and show  the following: 

\begin{lemma} \label{en-quad}
Under the a priori assumptions of Proposition \ref{propboot}, we have for $m=0,1,\dots$
\begin{align}\label{EEfields}
\Vert F_m \Vert_{H^N} & \lesssim 2^{-\delta m} \e^{1+\delta}, \qquad
  F_m := \int_0^t \tau_m(s) e^{-isL}  \big( L^{-1} \Im w(s) \big)^2 ds ,
\end{align}
for some $\delta >0.$
\end{lemma}

\begin{proof}
Using the boundedness of wave operators and distributing the derivatives, it suffices to show that
\begin{align}\label{EEfields1}
\begin{split}
{\Big\| \int_0^t e^{-isL} L^N \big( 
	L^{-1} \Im w(s) \big) \big( 
	L^{-1} \Im w(s) \big)
  \, \tau_m(s) ds\Big\|}_{L^2} \lesssim 2^{-\delta m}  \varepsilon^{1+\delta} \color{red}. \color{black}
\end{split}
\end{align}
Using Lemma \ref{Strichartz} with $(\tilde{q},\tilde{r},\gamma)=(6,3,2/3)$, 
together with the boundedness of wave operators and a Sobolev product estimate,
we have that the right-hand side of \eqref{EEfields1} is bounded by
\begin{align}\label{EEfields2}
\begin{split}
& C \sum_{k \geqslant 0} 2^{-k/3} {\big\| P_k L \big[ L^N \big( 
	L^{-1} \Im w \big) 
  	\big( 
	L^{-1} \Im w \big) \big] \big\|}_{L^{6/5}_{s\approx 2^m}L_x^{3/2}}
  \\
  & \lesssim {\big\| w \big\|}_{L^{6/5}_{s\approx 2^m} L^6_x} 
  {\big\| L^{N} w \big\|}_{L^{\infty}_{s\approx 2^m}L^2_x} 
  \\
  & \lesssim 2^{5m/6} \rho(2^m)^{1-\beta} 2^{\frac{5}{3} \delta_N m } \cdot \e.
\end{split}
\end{align}
Using that $\rho(2^m) \lesssim 2^{-\alpha m} \e^{2(1-\alpha)}$ 
with $\alpha = 5/6 + 10\beta$, with $\beta$ small enough, gives us the desired estimate.

\end{proof}

\medskip
\subsection{Weighted norm at high frequency: Proof of Lemma \ref{lemmaH}}\label{SsecH}
We focus on proving the hardest estimate \eqref{FH1} since \eqref{FH2} can be proven with similar arguments.
We split $F_H$ as a sum of the terms
\begin{align*}
\begin{split}
& F_{H,\epss}^{(i,j,k)}(a,b) := \int_0^t \int_{\R^6} \frac{e^{-is \Phi_\epss (\xi,\eta,\sigma)}}{\jeta \jsig} 
    \wt{a}(s,\eta) \wt{b}(s,\sigma) \mu_H^{(i,j,k)} (\xi,\eta,\sigma) d\eta d\sigma \, ds, 
    \\
    & \qquad \Phi_\epss (\xi,\eta,\sigma):= \jxi -\eps_1\jeta -\eps_2\jsig,
\end{split}
\end{align*}
with
\begin{align*}
\begin{split}
\mu_H^{(i,j,k)}(\xi,\eta,\s) := \mu(\xi,\eta,\s) \varphi_{i}(\jxi)\varphi_{j}(\jeta)\varphi_{k}(\jsig),
\end{split}
\end{align*}
where $i,j,k \in \lbrace +,- \rbrace,(i,j,k) \neq (+,+,+),$ 
and we have defined $\varphi_{+} := \varphi_{\leqslant M_0}$ (frequencies $\lesssim 2^{M_0}$)
and $\varphi_{-} := \varphi_{> M_0}$ (frequencies $\gtrsim 2^{M_0}$).
Define also 
\begin{align*}
\mathcal{B}(f,g):= \int_{\mathbb{R}^6} f(\eta) g(\s)  \mu(\xi,\eta,\s) d\eta d\s .
\end{align*}

We rely on the distributional identity in Lemma \ref{derM} 
(ignoring the easier case where the derivative falls on the outer localizer). 
Recall that $\mathcal{R}^{\alpha}$ denotes the euclidean Riesz transform. 
In the sequel we will slightly abuse notations and still write $\mathcal{R}^{\alpha}$ for the symbol of that operator,
i.e., $\mathcal{R}^{\alpha}_\xi = \xi^\alpha/|\xi|$.
Using Lemma \ref{derM} we estimate (note that we will mostly drop the dependence on the index $m$ below to improve legibility)
\begin{align*}
\Vert \partial_{\xi} F_{H}^{(i,j,k)} \Vert_{L^2} & \lesssim \sum_{m} \Vert I_{H}^{(i,j,k)} \Vert_{L^2}
  + \Vert II_{H}^{(i,j,k)} \Vert_{L^2} +  \Vert III_{H}^{(i,j,k)} \Vert_{L^2}, 
\end{align*}
with
\begin{align*}
I_{H}^{(i,j,k)}(\xi) & := \mathcal{R}_{\xi}^{\alpha} \varphi_i (\xi) \int_0^t 
 e^{-is \jxi} 
 \\ & \times \mathcal{B} \bigg( \partial_{\eta^{\beta}} 
  \bigg( \mathcal{R}_{\eta}^{\beta} \frac{e^{is \epsilon_1 \jeta}}{\jeta} \varphi_{j}(\eta) \wt{f_{\epsilon_1}} (s,\eta) \bigg),
  \frac{e^{is \epsilon_2 \jsigma}}{\jsigma} \varphi_{k}(\s) \wt{f_{\epsilon_2}} (s,\s) \bigg) \tau_m(s )ds,
\\
II_{H}^{(i,j,k)}(\xi) & := \mathcal{R}_{\xi}^{\alpha} \varphi_i (\xi) \int_0^t e^{-is \jxi} 
  \\ & \qquad \qquad \times \mathcal{B} 
  \bigg( E_{\eta} \bigg( \frac{e^{is \epsilon_1 \jeta}}{\jeta} \varphi_{j}(\eta) \wt{f_{\epsilon_1}} (s,\eta) \bigg), 
  \frac{e^{is \epsilon_2 \jsigma}}{\jsigma} \varphi_{k}(\s) \wt{f_{\epsilon_2}} (s,\s)  \bigg)  \tau_m(s) ds
\\
& + \mathcal{R}_{\xi}^{\alpha} E_{\xi} \varphi_i (\xi) \int_0^t e^{-is \jxi} \mathcal{B} 
  \bigg( \frac{e^{is \epsilon_1 \jeta}}{\jeta} \varphi_{j}(\eta) \wt{f_{\epsilon_1}} (s,\eta),
  \frac{e^{is \epsilon_2 \jsigma}}{\jsigma} \varphi_{k}(\s) \wt{f_{\epsilon_2}} (s,\s)\bigg)  \tau_m(s ) ds,
\\
III_{H}^{(i,j,k)}(\xi) & := \mathcal{R}_{\xi}^{\alpha} \varphi_i (\xi) \int_0^t is \frac{\xi^{\alpha}}{\jxi} e^{-is \jxi} 
  \\ 
  & \qquad \qquad \times \mathcal{B} \bigg( \frac{e^{is \epsilon_1 \jeta}}{\jeta} \varphi_{j}(\eta) \wt{f_{\epsilon_1}} (s,\eta), 
  \frac{e^{is \epsilon_2 \jsigma}}{\jsigma} \varphi_{k}(\s) \wt{f_{\epsilon_2}} (s,\s) \bigg) \tau_m(s)ds .
\end{align*}
Note that we can also alternatively differentiate the other profile and  write 
\begin{align*}
\Vert \partial_{\xi} F_{H}^{(i,j,k)} \Vert_{L^2} & \lesssim \sum_{m} \Vert I_{H}^{(i,j,k)'} \Vert_{L^2} 
  + \Vert II_{H}^{(i,j,k)'} \Vert_{L^2} +  \Vert III_{H}^{(i,j,k)} \Vert_{L^2}, 
  \\
I_{H}^{(i,j,k)'}(\xi) & := \mathcal{R}_{\xi}^{\alpha} \varphi_i (\xi) \int_0^t e^{-is \jxi} 
  \\
  & \times \mathcal{B} \bigg(  \frac{e^{is \epsilon_1 \jeta}}{\jeta} \varphi_{j}(\eta) \wt{f_{\epsilon_1}} (s,\eta),
  \partial_{\sigma^{\beta}} \bigg( \mathcal{R}_{\s}^{\beta} \frac{e^{is \epsilon_2 \jsigma}}{\jsigma} \varphi_{k}(\s) \wt{f_{\epsilon_2}} (s,\s) \bigg) \bigg) \tau_m(s )ds,  
\\
II_{H}^{(i,j,k)'}(\xi) & := \mathcal{R}_{\xi}^{\alpha} \varphi_i (\xi) \int_0^t e^{-is \jxi} 
  \\
  & \times \mathcal{B} \bigg( \frac{e^{is \epsilon_1 \jeta}}{\jeta} \varphi_{j}(\eta) \wt{f_{\epsilon_1}} (s,\eta), 
  E_{\s} \bigg( \frac{e^{is \epsilon_2 \jsigma}}{\jsigma} \varphi_{k}(\s) \wt{f_{\epsilon_2}} (s,\s) \bigg) \color{red} \bigg) \color{black} \tau_m(s ) ds, 
\\
& + \mathcal{R}_{\xi}^{\alpha} E_{\xi} \varphi_i (\xi) \int_0^t e^{-is \jxi} 
  \\
  & \times \mathcal{B} \bigg( \frac{e^{is \epsilon_1 \jeta}}{\jeta} \varphi_{j}(\eta) \wt{f_{\epsilon_1}} (s,\eta),
  \frac{e^{is \epsilon_2 \jsigma}}{\jsigma} \varphi_{k}(\s) \wt{f_{\epsilon_2}} (s,\s)\bigg)  \tau_m(s ) ds.
\end{align*}

\smallskip
\textit{Estimate of $II_H^{(i,j,k)}$.} 
We use either dispersive bounds or energy bounds depending on whether the input is low or high frequency.  
More precisely, we use for $2 < p \leqslant 6$ and $s \approx 2^m$ that, for $j=+,-$,
\begin{align} \label{Lpen}
\Vert \wt{\mathcal{F}}^{-1} \big( \varphi_{j} (\eta) e^{is \jeta} \wt{f}(s,\eta) \big) \Vert_{L^p} 
  \lesssim 2^{\frac{6-3p}{2p} m + 5 M_0} \cdot Z,
\end{align}
which is a direct consequence of \eqref{dispbootf}, and of Sobolev's embedding with the Sobolev a priori bound in \eqref{boot}
and \eqref{eng}.
Then, using Lemma \ref{ER}, we obtain that
\begin{align*}
\Vert II_{H}^{(i,j,k)} \Vert_{L^2} & \lesssim \int_0^t \bigg \Vert 
  \wtF^{-1} \big( e^{is \epsilon_1 \jeta} \frac{\varphi_{j}(\eta) }{\jeta}
  \wt{f_{\epsilon_1}}(s,\eta) \big) \bigg \Vert_{L^{3-}} \bigg \Vert 
  \wtF^{-1} \big(e^{is \epsilon_2 \jsigma}\frac{ \varphi_{k}(\sigma)}{\jsigma} 
  \wt{f_{\epsilon_2}}(s,\s) \big) \bigg \Vert_{L^6} \tau_m (s) ds 
  \\
& \lesssim 2^m \cdot 2^{20 M_0} 2^{-3m/2} \cdot Z^2,
\end{align*}
which is acceptable. 

\smallskip
\textit{Estimate of $I_H^{(i,j,k)}$.} 
The two main cases to consider are when the derivative falls on the exponential and on the profile:
\begin{align*}
\Vert I_H^{(i,j,k)} \Vert_{L^2} & \lesssim \Vert I_{H,1}^{(i,j,k)} \Vert_{L^2} + \Vert I_{H,2}^{(i,j,k)} \Vert_{L^2},
\\
I_{H,1}^{(i,j,k)} &:= \mathcal{R}_{\xi}^{\alpha} \varphi_i (\xi)  \int_0^t  e^{-is \jxi} 
  \\ & \qquad \times
  \mathcal{B} \bigg(\frac{e^{is \epsilon_1 \jeta}}{\jeta} \varphi_{j}(\eta) \mathcal{R}_{\eta}^{\beta} \partial_{\eta^{\beta}} \wt{f_{\epsilon_1}} (s,\eta),
  \frac{e^{is \epsilon_2 \jsigma}}{\jsigma} \varphi_{k}(\s) \wt{f_{\epsilon_2}} (s,\s)  \bigg) \tau_m(s) ds,
  \\
I_{H,2}^{(i,j,k)} &:= \mathcal{R}_{\xi}^{\alpha} \varphi_i (\xi) \int_0^t e^{-is\jxi} 
  \\ & \qquad \times
  \mathcal{B} \bigg( is \epsilon_1 \frac{\eta^{\beta}}{\jeta} \frac{e^{is \epsilon_1 \jeta}}{\jeta} 
  \varphi_{j}(\eta) \mathcal{R}_{\eta}^{\beta} \wt{f_{\epsilon_1}} (s,\eta), 
  \frac{e^{is \epsilon_2 \jsigma}}{\jsigma} \varphi_{k}(\s) \wt{f_{\epsilon_2}} (s,\s) \bigg) \tau_m(s) ds.
\end{align*}

Assume first that $k=-.$ 
In this case we use Sobolev embedding as well as the $H^N$ bound in the bootstrap and \eqref{eng} to write that,
for $p \geqslant 2$,
\begin{align*}
& \Vert \wt{\mathcal{F}}^{-1} \big( \varphi_{-} (\eta) e^{is \jeta} \wt{f}(s,\eta) 
  \big) \Vert_{L^p} 
  \\
  & \lesssim \Vert \wt{\mathcal{F}}^{-1} \big( \varphi_{-} (\eta) e^{is \jeta} \wt{f}(s,\eta) \big) \Vert_{H^2} 
  \lesssim 2^{-M_0 (N-2)} \Vert f \Vert_{H^N} \lesssim 2^{-\frac{5(N-2)}{N}m} \varepsilon^{1-}.
\end{align*}
Therefore we can bound
\begin{align*}
\Vert I_{H,1}^{(i,j,-)} \Vert_{L^2} \lesssim  \int_0^t \big \Vert 
  	\nabla_{\eta} \wt{f}(s,\eta) 
  \big \Vert_{L^2} \Vert \widetilde{\mathcal{F}}^{-1} e^{is\epsilon_1 \jsigma} \varphi_{-} (\sigma) \wt{f}(s,\s) \Vert_{L^{\infty}} \tau_m(s) ds 
  \\ \lesssim 2^{m} \cdot Z \cdot 2^{-\frac{5(N-2)}{N} m} \cdot \e^{1-},
\end{align*}
which yields the desired conclusion.  

We can bound $I_{H,2}^{(i,j,-)}$ 
similarly and obtain
\begin{align*}
\Vert I_{H,2}^{(i,j,-)} \Vert_{L^2} \lesssim 2^{2m} \cdot Z \cdot 2^{-\frac{5(N-1)}{N} m} .
\end{align*}

Note that since we can choose which profile to differentiate, 
the case $j=-$ is handled identically estimating $I_{H}^{(i,-,k)'}$ instead.

If $j,k = +$ then necessarily $i=-$ and in this case we write, 
using boundedness of wave operators and distributing derivatives, 
\begin{align*}
& \Vert I_{H,2}^{(-,+,+)} \Vert_{L^2} 
\\
& \lesssim 2^{-5m}  \int_0^t \bigg \Vert D^N \widetilde{\mathcal{F}}^{-1} 
  \mathcal{B} \bigg( is \epsilon_1 \frac{\eta^{\alpha}}{\vert \eta \vert} \frac{e^{is \epsilon_1 \jeta}}{\jeta} 
  \varphi_{+}(\eta) \wt{f_{\epsilon_1}} (s,\eta) , \frac{e^{is \epsilon_2 \jsigma}}{\jsigma} 
  \varphi_{+}(\s) \wt{f_{\epsilon_2}} (s,\s) \bigg)  \bigg \Vert_{L^2}  \tau_{m}(s)  ds 
  \\
& \lesssim 2^{-5m} \sum_{n_1 + n_2 \leqslant N} \int_0^t s \Vert D^{n_1} L^{-1} P_{\leqslant M_0}w \Vert_{L^2} 
  \Vert D^{n_2} L^{-1} P_{\leqslant M_0}w \Vert_{L^{\infty}} \tau_m(s) ds  
  \\
& \lesssim 2^{-5m} \cdot 2^{2m} \cdot 2^{M_0} \e^{2-},
\end{align*}
where for the last line we used 
the $H^N$ assumption in the bootstrap \eqref{boot}, and \eqref{eng}. This suffices for the desired estimate.

For the remaining term $I_{H,1}^{(-,+,+)}$ we split $\partial_\eta f = \partial_\eta g + \partial_\eta h$.
For the $h$ part, we use the Strichartz estimate in Lemma \ref{Strichartz} with parameters $(6,3,2/3)$,
Bernstein's inequality to gain a derivative, and then distribute this derivative over the inputs;
for the $g$ part we use 
the boundedness of
wave operators and distribute derivatives. Overall we obtain
\begin{align*}
& \Vert I_{H,1}^{(-,+,+)} \Vert_{L^2} 
\\
& \lesssim \sum_{\ell \geqslant M_0}  2^{-\ell/3} 
  \Bigg( \Bigg \Vert \langle \nabla \rangle \wtF^{-1}_{\xi\rightarrow x} \mathcal{B} \bigg(\frac{e^{is \epsilon_1 \jeta}}{\jeta} 
  \varphi_{+}(\eta) \mathcal{R}_{\eta}^{\alpha} \partial_{\eta^{\alpha}} \wt{h_{\epsilon_1}} (s,\eta),
  \frac{e^{is \epsilon_2 \jsigma}}{\jsigma} \varphi_{+}(\s) \wt{f_{\epsilon_2}} (s,\s) \bigg)  
  \Bigg \Vert_{L^{6/5}_{s \approx 2^m} L^{3/2}_x}  
  \\
& + \sum_{\ell \geqslant M_0} 2^{-N \ell} \cdot 2^m \cdot \sum_{n_1 + n_2 \leqslant N} \big \Vert D^{n_1} L^{-1} \varphi_{+}(D) 
  \whF^{-1} \partial_{\eta} \wt{g}(s,\eta) 
  \big \Vert_{L^2} \sup_{s \approx 2^m} \Vert D^{n_2} L^{-1} w(s) \Vert_{L^2}.
\end{align*}
The first term is bounded by $2^{5m/6} \sup_{s \approx 2^m} \Vert \varphi_{+}(L) e^{isL} f(s) \Vert_{L^6_x} \cdot Z,$ 
which is sufficient for 
the desired result.
Using \eqref{eng} we see that the last term is bounded by $2^{-N M_0 +m} \cdot 2^{m/2} \e^{1-}$, 
which is sufficient given our choice of parameter $M_0,$ see \eqref{HLsplit}.

\smallskip
\textit{Estimate of $III_{H}^{(i,j,k)}$.} 
This term can be handled following the same strategy as for $I_{H,2}^{(i,j,k)},$
so we skip the details.
The proof of the proposition is concluded. $\hfill \Box$

\appendix

\medskip
\section{Basic results about the NSD}\label{Appmu}
We recall several results derived in \cite{GHW} for the interaction distribution of three distorted complex exponentials
as defined in \eqref{intromu}.
We first define some useful operators and state bounds that they satisfy:

\begin{lemma}[Theorem 3.4 in \cite{GHW}] \label{ER}
With the definitions \eqref{Wodef}, let
\begin{align*}
R^{\alpha} := \mathcal{W} \frac{x^{\alpha}}{\vert x \vert} \mathcal{W}^{\ast}, 
\ \ \ \mathcal{E} := \big[\vert x \vert , \mathcal{W} \big]  \mathcal{W}^{\ast},
\end{align*}
and 
\begin{align*}
\mathcal{R}^{\alpha} := \widetilde{\mathcal{F}} R^{\alpha} \widetilde{\mathcal{F}}^{-1}, 
\ \ \ E:= \widetilde{\mathcal{F}} \mathcal{E} \widetilde{\mathcal{F}}^{-1}. 
\end{align*}
Both are bounded from $L^p$ to $L^{p+\delta}$ for any $1 \leqslant p < \infty$ and $\delta>0$ small enough. 
\end{lemma}

\begin{remark}
$\mathcal{R}^{\alpha}$ is the euclidean Riesz transform. It commutes with differentiation.
\end{remark}


The next result is an identity for the distributional derivative of $\mu$.

\begin{lemma}[Identity (3.56) in \cite{GHW}] \label{derM}
We have, in the distributional sense that
\begin{align*}
\partial_{\xi^{\alpha}} \mu &= \mathcal{R}_{\xi}^{\alpha} \mathcal{R}_{\eta}^{\beta} \partial_{\eta^{\beta}} \mu +\mathcal{R}_{\xi}^{\alpha} E_{\eta}^{\ast} \mu - \mathcal{R}_{\xi}^{\alpha} E_{\xi}^{\ast} \mu \\
&=\mathcal{R}_{\xi}^{\alpha} \mathcal{R}_{\s}^{\beta} \partial_{\s^{\beta}} \mu + \mathcal{R}_{\xi}^{\alpha} E_{\s}^{\ast} \mu - \mathcal{R}_{\xi}^{\alpha} E_{\xi}^{\ast} \mu.
\end{align*}
Therefore, we have the following identity
\begin{align} \label{measure-weight}
\begin{split}
\partial_{\xi}^{\alpha} \big( \wt{fg} \big) & = \mathcal{R}_{\xi}^{\alpha} \int_{\mathbb{R}^6}  \mathcal{R}_{\eta}^{\beta} 
\partial_{\eta^{\beta}} \wt{f} (\eta) \wt{g}(\sigma) \mu(\xi,\eta,\sigma) d\eta d\sigma + \mathcal{R}_{\xi}^{\alpha} 
  \int_{\mathbb{R}^6} E \wt{f}(\eta) \wt{g}(\sigma) \mu(\xi,\eta,\sigma) d\eta d\sigma  \\
&- E \mathcal{R}_{\xi}^{\alpha} \int_{\mathbb{R}^6} \wt{f} (\eta) \wt{g}(\sigma)  \mu(\xi,\eta,\sigma) d\eta d\sigma 
\\
&= \mathcal{R}_{\xi}^{\alpha} \int_{\mathbb{R}^6}   \wt{f} (\eta) \mathcal{R}_{\sigma}^{\beta} \partial_{\sigma^{\beta}} 
  \wt{g}(\sigma) \mu(\xi,\eta,\sigma) d\eta d\sigma + \mathcal{R}_{\xi}^{\alpha} \int_{\mathbb{R}^6} \wt{f}(\eta) 
  E \wt{g}(\sigma) \mu(\xi,\eta,\sigma) d\eta d\sigma  \\
&- E \mathcal{R}_{\xi}^{\alpha} \int_{\mathbb{R}^6} \wt{f} (\eta) \wt{g}(\sigma)  \mu(\xi,\eta,\sigma) d\eta d\sigma .
\end{split}
\end{align} 
\end{lemma}


\begin{remark}
Note that we slightly abused notations above and denoted $\mathcal{R}^{\alpha}_{\xi}$ the symbol of the euclidean Riesz transform.
\end{remark}

\begin{remark}
A convenient feature in \eqref{measure-weight} is that one can pick which profile to differentiate. 
\end{remark}

As a direct consequence we obtain a proof of Lemma \ref{verylowfreq}
and nonlinear estimates for short times in the next two subsections.

\subsection{Proof of Lemma \ref{verylowfreq}}\label{Ssecverylowfreq}
Note that at least one of the profiles is localized at very low frequency. 
We pick one of them and differentiate the other profile, using \eqref{measure-weight} 
as well as Bernstein's inequality:
\begin{align*}
\Vert \partial_{\xi} F_{L,\epsilon_1,\epsilon_2} \Vert_{L^2} & 
  \lesssim \int_0^t \big( \Vert \mathcal{R}_{\eta}^{\alpha} \langle \eta \rangle^{-1} 
  \partial_{\eta^{\alpha}} \wt{f} \Vert_{L^2} +  \bigg \Vert \mathcal{R}_{\eta}^{\alpha} 
  \frac{s \eta^{\alpha}}{\jeta^{2}} e^{is\jeta} \wt{f} \bigg \Vert_{L^2} 
  + \Vert E f \Vert_{L^2}  \bigg) \Vert f_{\leqslant s^{-5}} \Vert_{L^{\infty}} ds 
  \\
& \lesssim \int_0^t \big(\langle s \rangle^{1/2+} + s \e + \e \big)  \langle s \rangle^{-15/2} ds \, \e,
\end{align*}
which is sufficient to conclude. We used the crude bound $\Vert \partial_{\eta} \wt{f}  \Vert_{L^2} 
\lesssim \langle s \rangle^{1/2+} ,$ which follows from \eqref{growth} and the bootstrap assumption \eqref{boot}.
$\hfill \Box$

\smallskip
\subsection{Local existence}\label{Ssecloc}
Here we prove that \eqref{locboota} implies \eqref{locbootc}.
Since $t \leqslant \e^{-1},$ ($t\geqslant 1$) 
the proof of the bootstrap estimates is much simpler than in the general case.
Therefore, we only give an outline of the proof, typically discarding similar and lower order terms.

First, we have from Lemma \ref{decay} that
\begin{align} \label{disp-shortt} 
\Vert e^{itL} f \Vert_{L^q} \lesssim t^{(-1 + 10 \delta_N)\frac{3(q-2)}{2q}} \e^{1-\beta} .  
\end{align}
We bound the terms that appear in \eqref{decompositionh}. The source terms are treated as in Lemma \ref{source},
therefore we skip the details.
For the mixed terms, we write that (treating the worse case where the derivative falls on the phase)
using Strichartz estimates with parameters $(q,r,\gamma) = (2,\infty,0)$
\begin{align*}
\Vert \partial_{\xi} M(t) \Vert_{L^2} & \lesssim \bigg \Vert \int_0^t a(s) s \frac{\xi}{\jxi} e^{is \jxi} 
  \int_{\mathbb{R}^3} \frac{e^{-is \jeta}}{\jeta} \wt{f}(s,\eta)  \nu(\xi,\eta) d\eta ds  \bigg \Vert_{L^2}
  \\
& \lesssim {\big\| s \Vert e^{itL} f \Vert_{L^6_x} \big\|}_{L^2_t} \e 
  \lesssim t^{1/2+20 \delta_N} \e^{2-\beta} \leqslant \e^{3/2-2\beta},
\end{align*}
where we used the dispersive estimate \eqref{disp-shortt},
and disregarded the simpler contribution to the integral for $s\in[0,1]$\color{red}. \color{black}

We come to the quadratic terms. Using Lemma \ref{derM} we find, neglecting the lower order terms,
\begin{align*}
\Vert \partial_{\xi} \wt{F}(t) \Vert_{L^2} & \lesssim \bigg \Vert \int_0^t s \int_{\mathbb{R}^6} 
\frac{\eta^{\alpha}}{\jeta^2} e^{is \jeta} \wt{f}(\eta) \frac{e^{is \jsigma}}{\jsigma} \wt{f}(\s)
\mu(\xi,\eta,\s) d\eta d\s ds \bigg \Vert_{L^2} 
\\
& + \bigg \Vert \int_0^t \int_{\mathbb{R}^6} \frac{e^{is \jeta}}{\jeta} \partial_{\eta^{\alpha}} 
\wt{f}(\eta) \frac{e^{is \jsigma}}{\jsigma} \wt{f}(\s) \mu(\xi,\eta,\s) d\eta d\s ds \bigg \Vert_{L^2} 
\\
& \lesssim \int_0^t s \Vert e^{itL} L^{-1} f  \Vert_{L^6} \Vert e^{itL} L^{-1} f \Vert_{L^3} ds
+ \int_0^t \Vert L^{-1} \mathcal{F}^{-1} \big( \partial_{\xi} \wt{f} \big) \Vert_{L^3} \Vert e^{itL} L^{-1} f \Vert_{L^6} ds 
\\
& \lesssim \int_0^t s \cdot \langle s \rangle ^{-1+10 \delta_N} \cdot 
\langle s \rangle^{-1/2 + 5 \delta_N} ds \, \e^{2-2\beta} + \int_0^t \langle s \rangle ^{-1} ds \, \e^{3/2-2\beta} 
\\
& \lesssim \langle t \rangle^{1/2 + 15 \delta_N} \e^{2-2\beta} + \e^{3/2-3\beta} ,
\end{align*}
where we used the bound, coming from the bootstrap assumption and \eqref{growth},
\begin{align*}
\Vert \partial_{\xi} \wt{f} \Vert_{L^2} \lesssim  \Vert \partial_{\xi} \wt{h} \Vert_{L^2} 
  + \Vert \partial_{\xi} \wt{g}  \Vert_{L^2} \lesssim \e^{1-\beta} + t^{3/2} \e^2 \vert \log \e \vert^2 \leqslant \e^{1/2-},
\end{align*}
which yields the desired result given that $t \leqslant \e^{-1}$ provided $\beta$ is small enough.
$\hfill \Box$

\medskip
\section{Structure of the NSD and product estimates}\label{Appmu2}
In order to deal with nonlinear problems and understand oscillations in distorted Fourier space,
we need a more refined analysis of (the singularities of) the NSD. 
We first recall the decomposition result of \cite{PS}, under the more generous assumption $V\in \mathcal{S}$
as in Proposition \ref{mudecomp}; we refer the reader to \cite{PS} for the precise statement under
finite regularity and localization assumptions on the potential $V$.

\begin{proposition}[\cite{PS}] \label{decomp-meas-PS}
The distribution \eqref{muDuh} can be written as
\begin{align} \label{mudecomp-0}
\begin{split}
(2\pi)^{9/2} \mu(\xi,\eta,\sigma) & = (2\pi)^{3/2} \delta_0 (\xi-\eta-\sigma) - \frac{1}{4\pi} \mu_1(\xi,\eta,\sigma) 
\\ 
& - \frac{1}{(4\pi)^2} \mu_2(\xi,\eta,\sigma) - \frac{1}{(4\pi)^3} \mu_3(\xi,\eta,\sigma),
\end{split}
\end{align}
where:

\setlength{\leftmargini}{1em}

\begin{itemize}
\item The distribution $\mu_1$ is given by
\begin{align}
\label{mu1}
\mu_1(k,\ell,m) & = \nu_1(-k+\ell,m) + \nu_1(-k+m,\ell) + \overline{\nu_1(-\ell-m,k)}
\end{align}
where
\begin{align}
\label{Propnu+1}
\nu_1(p,q) = \nu_0(p,q) + \nu_L(p,q) + \nu_R(p,q),
\end{align}
and:
\medskip
\begin{itemize}
\item[(1)] The leading order is
\begin{align}
\label{Propnu+1.0}
\nu_0(p,q) :=  \frac{b_0(p,q)}{|p|} \Big[ i\pi \, \delta(|p|-|q|) + \pv \frac{1}{|p|-|q|} \Big]
\end{align}
with $b_0$ satisfying the bounds 
\begin{align} \label{Propnu+1.1}
\big \vert \varphi_P(p) \varphi_Q(q) \, \nabla_p^\alpha \nabla_q^\beta b_0(p,q) \big \vert 
  \lesssim 2^{-\vert \alpha \vert P} \big(2^{ \vert \alpha \vert Q} + 2^{(1- \vert \beta \vert) Q_{-}}\big) \cdot \mathbf{1}_{\lbrace \vert P-Q \vert < 5 \rbrace},
\end{align}
for all $P,Q \leqslant A$,
and $\vert \alpha \vert + \vert \beta \vert \leqslant  N_2$,
for any arbitrarily large integer\footnote{The value of $N_2$ can be related to 
the localization and smoothness of the potential $V$, and it can be chosen to be approximately $N/2$
when $\jx^N V \in H^N$.
Since we assume, for simplicity, that $V \in \mathcal{S}$, we can let $N_2$ be arbitrarily large.} $N_2$.
Recall our notation $Q^- = \min(Q,0)$.

\medskip
\item[(2)] 
The lower order terms $\nu_L(p,q)$ can be written as 
\begin{align}
\label{Propnu+2}
\nu_L(p,q) = \frac{1}{|p|} 
  \sum_{a=1}^{N_2} \sum_{J\in\Z} b_{a,J}(p,q) \cdot 2^J K_a\big(2^J(|p|-|q|)\big) 
\end{align}
with 
$K_a \in \mathcal{S}$ 
and $b_{a,J}$ satisfying
\begin{align}
\label{Propnu+2.1}
\sum_{J\in\Z} 
  \big|  \varphi_P(p) \varphi_Q(q) \nabla_p^\alpha \nabla_q^\beta  b_{a,J}(p,q) \big| \lesssim 
  2^{-|\alpha|P} \big(2^{|\alpha |Q}  + 2^{(1-|\beta|)Q_-}\big) \cdot \mathbf{1}_{\{|P-Q|< 5\}},
\end{align}
for all $P,Q \leqslant A$, $|\alpha|+|\beta| \leqslant N_2$. 

\medskip
\item[(3)] The remainder term $\nu_R$ satisfies the estimates 
\begin{align}\label{Propnu+3}
\begin{split}
\big| \varphi_P(p) \varphi_Q(q) \nabla_p^\alpha \nabla_q^\beta \nu_R(p,q) \big| 
  \lesssim 2^{-2\max(P,Q)} \cdot  2^{-(|\alpha|+|\beta|)\max(P,Q)} \cdot 2^{(|\alpha|+|\beta|+2)5A}
\end{split}
\end{align}
for $|P-Q| < 5$, and 
\begin{align}\label{Propnu+3imp}
\begin{split}
\big| \varphi_P(p) \varphi_Q(q) \nabla_p^\alpha \nabla_q^\beta \nu_R(p,q) \big| 
  \lesssim 2^{-2\max(P,Q)} \cdot 2^{-|\alpha|\max(P,Q)} \max(1,2^{(1-|\beta|)Q_-}) \cdot 2^{(|\alpha|+|\beta|+2)5A}
\end{split}
\end{align}
for $|P-Q| \geqslant 5$,
for all $P,Q \leqslant A$, $|\alpha|+|\beta| \leqslant N_2$. 

\end{itemize}

\item The measure $\mu_2$ is given by
\begin{align}
\label{mu2}
\mu_2(k,\ell,m) & = \nu_2^1(k,\ell,m) + \nu_2^2(k,\ell,m) + \nu_2^2(k,m,\ell)
\end{align}
with the following properties:

\begin{itemize}
\item 
Let $k,\ell,m \in\R^3$ with $|k|\approx 2^K, |\ell|\approx 2^L$ and $|m|\approx 2^M$, 
and assume that $K,L,M \leqslant A$ for some $A>0$.
Then we can write
\begin{align}\label{Propnu210}
\nu_2^1(k,\ell,m) = 
  \nu_{2,0}^1(k,\ell,m) + \nu_{2,R}^1(k,\ell,m),
\end{align}
where:
\begin{itemize}
\medskip
\item[(1)] 
$\nu_{2,0}^1(k,\ell,m)$ can be written as
\begin{align}\label{Propnu211}
\nu_{2,0}^1(k,\ell,m) = \frac{1}{|k|} 
  \sum_{i=1}^{N_2} \sum_{J\in\Z} b_{i,J}(k,\ell,m) \cdot K_i\big(2^J(|k|-|\ell|-|m|)\big) 
\end{align}
with 
$K_i \in \mathcal{S}$ and the symbols $b_{i,J}$ satisfy
\begin{align}\label{Propnu212}
\begin{split}
  & \big| \varphi_K(k) \varphi_L(\ell) \varphi_M(m) 
  \nabla_k^a \nabla_\ell^\alpha \nabla_m^\beta  b_{i,J}(k,\ell,m) \big| 
  \\ 
  & \lesssim 2^{-|a|K} \cdot \big( 2^{|a|\max(L,M)} + 2^{(1-|\alpha|)L} 2^{(1-|\beta|)M} \big)
  \mathbf{1}_{\{ |K-\max(L,M)|<5 \}},
\end{split}
\end{align}
for all $K,L,M \leqslant A$, and $|a|+|\alpha|+|\beta| \leqslant N_2$. 

\medskip
\item[(2)] For $M \leqslant L$, the remainder term $\nu_{2,R}^1$ satisfies 
\begin{align}\label{Propnu213}
\begin{split}
\big| 
  \nabla_k^a \nabla_\ell^\alpha \nabla_m^\beta \nu_{2,R}^1(k,\ell,m) \big| 
  & \lesssim 2^{-2\max(K,L)} \cdot 2^{-|a|\max(K,L)}
  \\ & \times 2^{-|\alpha|\max(K,L)} \max(1,2^{-(|\beta|-1)M}) \cdot 2^{(|a|+|\alpha|+|\beta|+2)5A}
\end{split}
\end{align}
for all $K,L,M \leqslant A$ and $|a|+|\alpha|+|\beta| \leqslant N_2$. 
A similar statement holds when $L \leqslant M$ exchanging the roles of $L$ and $M$ (and $\alpha$ and $\beta$).


\end{itemize}
\bigskip
\item 
Let $k,\ell,m \in\R^3$ with $|k|\approx 2^K, |\ell|\approx 2^L$ and $|m|\approx 2^M$, 
and $K,L,M \leqslant A$ for some $A>0$.
Then we can write
\begin{align}\label{Propnu220}
\nu_2^2(k,\ell,m) = 
  \nu_{2,+}^2(k,\ell,m) + \nu_{2,-}^2(k,\ell,m) + \nu_{2,R}^2(k,\ell,m),
\end{align}
where:


\begin{itemize}
 
\item[(1)] The leading order is
\begin{align}\label{Propnu221}
\nu_{2,\pm}^2(k,\ell,m) = \frac{1}{|\ell|} 
  \sum_{i=1}^{N_2} \sum_{J\in\Z} b_{i,J}^\pm(k,\ell,m) \cdot K_i\big(2^J(|k|\pm|\ell|-|m|)\big) 
\end{align}
with
\begin{align}\label{Propnu222}  
\begin{split}
  & \big| \varphi_K(k) \varphi_L(\ell) \varphi_M(m) 
  \nabla_k^a \nabla_\ell^\alpha \nabla_m^\beta  b_{i,J}^\pm(k,\ell,m) \big| 
  \\ 
  & \lesssim 2^{-|\alpha|L} \cdot \big( 2^{|\alpha|\max(K,M)} + 2^{(1-|a|)K} 2^{(1-|\beta|)M} \big)
  \mathbf{1}_{\{\max(K,L,M)-\mbox{\tiny$\mathrm{med}$}(K,L,M)<5\}},
\end{split}
\end{align}
for all $K,L,M \leqslant A$, and $|a| + |\alpha|+|\beta| \leqslant N_2$.

\medskip
\item[(2)] The remainder term satisfies, for $K \leqslant M$,
\begin{align}\label{Propnu223}
\begin{split}
\big| \nabla_k^a \nabla_\ell^\alpha \nabla_m^\beta \nu_{2,R}^2(k,\ell,m) \big| 
  & \lesssim 2^{-2\max(L,M,K)} \cdot 2^{-(|\alpha|+|\beta|)\max(L,M,K)}
  \\ & \times \max(1,2^{-(|a|-1)K}) \cdot 2^{(|a|+|\alpha|+|\beta|+2)5A}
\end{split}
\end{align}
for all $K,L,M \leqslant A$ and $|a|+|\alpha|+|\beta| \leqslant N_2$. 
A similar estimate holds when $M \leqslant K$ by exchanging the roles of $M$ and $K$ (and $a,\beta$).
\end{itemize}

\end{itemize}

\item 
Let $k,\ell,m \in\R^3$ with $|k|\approx 2^K, |\ell|\approx 2^L$ and $|m|\approx 2^M$, 
and assume that $K,L,M \leqslant A$ for some $A>0$.
Then we can write
\begin{align}\label{Propmu30}
\mu_3(k,\ell,m) = \mu_{3,0}(k,\ell,m) + \mu_{3,R}(k,\ell,m),
\end{align}
where:

\begin{itemize}

\medskip
\item[(1)] 
The leading order has the form
\begin{align}\label{Propmu31}
\mu_{3,0}(k,\ell,m) =
  \sum_{i=0}^{N_2} \sum_{J\in\Z} b_{i,J}(k,\ell,m) \cdot K_i\big(2^J(|k|-|\ell|-|m|)\big) 
\end{align}
with $K_i \in \mathcal{S}$, and
\begin{align}\label{Propmu31'}  
\begin{split}
& \big| \varphi_K(k) \varphi_L(\ell) \varphi_M(m) 
  \nabla_k^a \nabla_\ell^\alpha \nabla_m^\beta  b_{i,J}(k,\ell,m) \big| 
  \\ & \lesssim \max(1,2^{-(|a|-1)K})  \max(1,2^{-(|\alpha|-1)L}) \max(1,2^{-(|\beta|-1)M})
  \mathbf{1}_{\{|K-\max(L,M)| < 5\}}
\end{split}
\end{align}
for all $K,L,M \leqslant A$, and $|a| + |\alpha|+|\beta| \leqslant N_2$. 

\medskip
\item[(2)] The remainder term satisfies
\begin{align}\label{Propmu3R}
\begin{split}
\big| 
  \nabla_k^a \nabla_\ell^\alpha \nabla_m^\beta \mu_{3,R}(k,\ell,m) \big| 
  \lesssim 
  \max(1,2^{-(|a|-1)K})  \max(1,2^{-(|\alpha|-1)L}) \\ \cdot \max(1,2^{-(|\beta|-1)M}) \cdot 2^{4A(|a|+|\alpha|+|\beta|+1)}
\end{split}
\end{align}
for all $K,L,M \leqslant A$ and $|a|+|\alpha|+|\beta| \leqslant N_2$. 
\end{itemize}

\end{itemize}

\end{proposition}

\smallskip
Next, we collect all the needed bilinear 
estimates for the operators associated to the various pieces 
appearing in the decomposition of Proposition \ref{decomp-meas-PS}.   
Again these results are borrowed from \cite{PS}. 

Given a symbol $b=b(\xi,\eta,\s)$, we define the bilinear operators
\begin{align}
\label{theomu10b}
T_0[b](g,h)(x) & := \widehat{\mathcal{F}}^{-1}_{\xi\rightarrow x} \iint_{\R^3\times\R^3} g(\eta) h(\s) 
  \,b(\xi,\eta,\s) \, \delta_0(\xi-\eta-\s) \, d\eta d\s
\end{align}
and
\begin{align}
\label{theomu11bb}
T_1^1[b](g,h)(x) & :=  \widehat{\mathcal{F}}^{-1}_{\xi\rightarrow x} \iint_{\R^3\times\R^3} g(\xi-\eta) h(\sigma) 
  \,b(\xi,\eta,\sigma)\, \nu_1(\eta,\sigma) \, d\eta d\s,
\\
\label{theomu12bb}
T_1^2[b](g,h)(x) &:= \widehat{\mathcal{F}}^{-1}_{\xi\rightarrow x}\iint_{\R^3\times\R^3} g(-\eta-\sigma) h(\sigma) 
  \,b(\xi,\eta,\sigma)\, \overline{\nu_1(\eta,\xi)} \, d\eta d\s, 
  \\
  \label{theomu2bb}
  T_2[b](g,h)(x) & :=  \widehat{\mathcal{F}}^{-1}_{\xi\rightarrow x} \iint_{\R^3\times\R^3} g(\eta) h(\sigma) 
  \,b(\xi,\eta,\sigma)\, \mu_2(\xi,\eta,\sigma) \, d\eta d\s,
  \\
   \label{theomu3bb}
  T_3[b](g,h)(x) & :=  \widehat{\mathcal{F}}^{-1}_{\xi\rightarrow x}\iint_{\R^3\times\R^3} g(\eta) h(\sigma) 
  \,b(\xi,\eta,\sigma)\, \mu_3(\xi,\eta,\sigma) \, d\eta d\s.
\end{align}

\def\AfactorBB1{C_0}
\def\FfactorBB{{\max(L,M)}}
\def\FfactorBB2{{\max(L,K)}}

\medskip
\begin{theorem}[Bilinear bounds 1]\label{theomu1}
Assume that:

\setlength{\leftmargini}{2em}
\begin{itemize}

\medskip
\item The symbol $b$ is such that
\begin{align}\label{theomu1asb1}
\textrm{supp}(b) \subseteq \big\{ (\xi,\eta,\sigma) \in \mathbb{R}^9\,:\, |\xi|+|\eta|+|\sigma| \leqslant 2^A, 
  \, |\eta| \approx 2^{k_1}, \, |\sigma| \approx 2^{k_2}\big\},
\end{align}
for some $A \geqslant 1$.


\medskip
\item For all $|\xi| \approx 2^{k}$, $|\eta|\approx 2^{k_1}$ and $|\sigma| \approx 2^{k_2}$
\begin{align}\label{theomu1asb2}
\vert \nabla_{\xi}^a \nabla^\alpha_\eta \nabla^\beta_\sigma b(\xi,\eta,\sigma) \vert 
	\lesssim 2^{-\vert  a \vert  k} 2^{-\vert \alpha \vert  k_1}2^{- \vert  \beta \vert  k_2} 
	\cdot 2^{(\vert  a \vert +\vert \alpha \vert +\vert  \beta \vert )A}, 
	\qquad \vert  a \vert ,\vert  \alpha \vert ,\vert  \beta \vert \leqslant 5.
\end{align}


\medskip
\item There is $10A \leqslant D \leqslant  2^{A/10}$ such that
\begin{align}
\label{theomu1asgh}
\begin{split}
\mathcal{D}(g,h) := \Vert g \Vert_{L^2} \Vert  h \Vert_{L^2} 
	+ \min\big( {\Vert \partial_k g \Vert}_{L^2} {\Vert h \Vert}_{L^2}, 
	{\Vert g \Vert}_{L^2} {\Vert \partial_k h \Vert}_{L^2} \big) \leqslant 2^{D}.
\end{split}
\end{align}

\end{itemize}

Then, the following estimates hold:
For any $p,q \in [1,\infty]$ and $r> 1$ with 
\begin{align}
\frac{1}{p} + \frac{1}{q} > \frac{1}{r}, \qquad 
\end{align}
we have
\begin{align}
\label{theomu1conc}
\begin{split}
{\big\| T_1^1[b]\big(g,h\big) \big\|}_{L^r} 
  & \lesssim {\big \Vert \widehat{g} \big \Vert}_{L^p} {\big \Vert \widehat{h} \big \Vert}_{L^q} \cdot 2^{\max(k_1,k_2)}
  \cdot 2^{C_0 A} + 2^{-D} \mathcal{D}(g,h), 
\end{split}
\end{align}
and
\begin{align}
\label{theomu1concT12}
\begin{split}
{\big \Vert P_K T_1^2[b](g,h) \big \Vert}_{L^r} 
  & \lesssim {\big \Vert \widehat{g} \big \Vert}_{L^p} {\big \Vert  \widehat{h} \big \Vert}_{L^q} \cdot 2^{\max(k_1,k_2)}
  \cdot 2^{C_0 A} + 2^{-D} \mathcal{D}(g,h),
\end{split}
\end{align}
where\footnote{This is a convenient value of the absolute constant $C_0$ 
but it can certainly be improved. 
In the course of nonlinear estimates for the evolution equation
we 
impose conditions on the smallness of $C_0 \delta_N$ (or similar quantities); see \eqref{mteps} for the definition of $\delta_N$. Then, a smaller value of the constant $C_0$ in Theorem \ref{theomu1} would reduce the total number of derivatives $N$ required for our initial data.
} $C_0:=65$.
In the above $P_K$ denotes the flat/regular Littlewood-Paley projection.
Moreover, we have
\begin{align} \label{theomuconcT023}
\Vert T_0[b](g,h) \Vert_{L^r}, \, \Vert T_2[b](g,h) \Vert_{L^r}, 
  \, \Vert T_3[b](g,h) \Vert_{L^r} \lesssim \Vert \widehat{g} \Vert_{L^p} \Vert \widehat{h} \Vert_{L^q} \cdot 2^{(C_0+1)A}.
\end{align}

\end{theorem}

We also recall the following statement from \cite{PS}, which is used in the proof of Theorem \eqref{theomu1}:

\begin{lemma}\label{PSvol}
Let $j \geqslant 1,$ and consider the bilinear operator 
\begin{align*}
B_j[b](g,h)(x) = \mathcal{F}^{-1}_{\xi \mapsto x} \int_{\mathbb{R}^6} g(\eta-\xi) h(\sigma) b(\xi,\eta,\sigma) \chi \big(2^j(\vert \eta \vert - \vert \sigma \vert) \big) \, d\eta d\sigma ,
\end{align*}
where $\chi$ is a Schwartz function. Assume that:

\begin{itemize}
\item For some $A \geqslant 1$, $K,L,M \in \Z$, with $L \gg -j$, 
we have
\begin{align*}
\textrm{supp}(b) \subseteq \big\{ (\xi,\eta,\sigma) \in \mathbb{R}^9\,:\, |\xi|+|\eta|+|\sigma| \lesssim 2^A,
  \, |\xi| \approx 2^K, \, |\eta| \approx 2^L, \, \vert \s \vert \approx 2^M \}; 
\end{align*}
\item The following estimates hold:
\begin{align*}
\vert \nabla_{\xi}^a \nabla_{\eta}^{\alpha} \nabla_{\sigma}^{\beta} b(\xi,\eta,\sigma) \vert 
  \lesssim 2^{-\vert a \vert K - \vert \alpha \vert L - \vert \beta \vert M}, 
  \qquad \vert a \vert, \vert \alpha \vert , \vert \beta \vert \leqslant 4.
\end{align*}
\end{itemize}
Then for $p,q,r \in [1,\infty],$ we have
\begin{align*}
\Vert B_j[b] (g,h) \Vert_{L^r} \lesssim 2^{-j} \cdot 2^{2L} \cdot 2^{8 A} \cdot \Vert \widehat{g} \Vert_{L^p} 
  \Vert \widehat{h} \Vert_{L^q}, \qquad \frac{1}{p} + \frac{1}{q} = \frac{1}{r}.
\end{align*}
\end{lemma}

\end{document}